\documentclass[smallextended]{svjour3}   
\smartqed  % flush right qed marks, e.g. at end of proof

\usepackage{amsmath,amsfonts,amssymb,graphics,color,accents,enumerate,epsfig}

\usepackage{amsthm}
\usepackage{hyperref}
\usepackage{makeidx}
\usepackage{mathtools}  
\theoremstyle{plain}
\newtheorem{thm}{Theorem}
\numberwithin{thm}{section}
\newtheorem{lem}[thm]{Lemma}
\newtheorem{prop}[thm]{Proposition}
\newtheorem{cor}[thm]{Corollary}

\theoremstyle{definition}
\newtheorem{defn}[thm]{Definition}
\theoremstyle{remark}
\newtheorem{rem}[thm]{Remark}

\numberwithin{equation}{section}

\newcommand{\Int}{\Omega}
\newcommand{\Ext}{{{\mathcal S}_0}}

\def\curl{\operatorname{curl}}
\def\dive{\operatorname{div}}

\providecommand{\R}{\mathbb{R}}
\providecommand{\C}{\mathbb{C}}
\providecommand{\N}{\mathbb{N}}

\providecommand{\eps}{\varepsilon}
\renewcommand{\leq}{\leqslant}
\renewcommand{\geq}{\geqslant}
\def\longrightharpoonup{\relbar\joinrel\rightharpoonup} 

\spnewtheorem*{xproof}{}{\itshape}{\rmfamily}% the label is assigned later

\renewenvironment{proof}[1][\proofname]
 {\renewcommand\xproofname{#1}\xproof}
 {\endxproof}
 
 \renewcommand{\Re}{{\rm Re}\,}
\makeindex

\begin{document}

\date{\today}
\title{Point vortex dynamics as zero-radius limit of the motion of a rigid body in an irrotational fluid}
\author{Olivier Glass \and Alexandre Munnier \and Franck Sueur} 
\institute{O. Glass \at CEREMADE, UMR CNRS 7534, Universit\'e Paris-Dauphine, 
Place du Mar\'echal de Lattre de Tassigny, 75775 Paris Cedex 16, France
\and 
A. Munnier \at
Universit\'e de Lorraine, Institut Elie Cartan de Lorraine, UMR 7502, Nancy-Universit\'e, Vandoeuvre-l\`es-Nancy, F-54506, France
\& CNRS, Institut Elie Cartan de Lorraine, UMR 7502, Nancy-Universit\'e, Vandoeuvre-l\`es-Nancy, F-54506, France
\and
F. Sueur \at
Institut de Math\'ematiques de Bordeaux, UMR CNRS 5251, Universit\'e de Bordeaux, 351 cours de la Lib\'eration, F-33405 Talence Cedex, France.
}
\titlerunning{Point vortex dynamics as zero-radius limit of a rigid body's motion}
\authorrunning{Glass, Munnier \& Sueur}
\maketitle
%\newpage

% Abstract =============================================
%
\begin{abstract}

The point vortex system is usually considered as an idealized model where the
vorticity of an ideal incompressible two-dimensional  fluid is concentrated in a
finite number of moving points. 
In the case of  a single vortex in an otherwise irrotational ideal fluid occupying 
a bounded and simply-connected  two-dimensional  domain the motion is given by the so-called Kirchhoff-Routh velocity which depends only on the domain.
The main result of  this paper establishes that this dynamics can also be  obtained as the limit of the motion of a rigid body immersed in such a fluid 
when the body shrinks to a massless point particle  with fixed circulation. 
The rigid body is assumed to be only accelerated by the force exerted by the  fluid pressure on its boundary, the fluid velocity and pressure being given by the incompressible Euler equations, with zero vorticity. 
The circulation of the fluid velocity around the particle is conserved as time proceeds according to Kelvin's theorem and gives the strength of the limit point vortex. 
We  also prove that  in the different regime where the body shrinks with a fixed mass the limit dynamics is governed by a second-order differential equation involving a Kutta-Joukowski-type lift force.

To prove these results, in a first step we reformulate the dynamics of the body in order to make more explicit different kind of interactions with the fluid. 
Precisely we establish that the Newton-Euler equations of translational and rotational dynamics of the body can be seen as a $3$-dimensional ODE with coefficients solving an auxiliary problem for the fluid. When the circulation around the body is zero, this equation is a geodesic equation  for a metric associated with the well-known ``added inertia'' phenomenon; with a nonzero circulation, an additional term similar to the Lorentz force of electromagnetism appears. 
Then, in the zero-radius limit, surprising relations between leading and subprincipal orders of various terms and  modulation variables show up and allow us to establish a normal form with a gyroscopic structure.  This leads to   uniform estimates on the body's dynamics thanks to a modulated energy, and therefore allows us to describe the transition of the dynamics in the limit.

\end{abstract}
%

%
%====================================================
%
\setcounter{tocdepth}{3}

\let\oldtocsection=\tocsection

\let\oldtocsubsection=\tocsubsection

\let\oldtocsubsubsection=\tocsubsubsection

\newcommand{\tocsection}[2]{\hspace{0em}\oldtocsection{#1}{#2}}
\newcommand{\tocsubsection}[2]{\hspace{1em}\oldtocsubsection{#1}{#2}}
\newcommand{\tocsubsubsection}[2]{\hspace{2em}\oldtocsubsubsection{#1}{#2}}
\tableofcontents
%====================================================

%
%
\printindex
%\newpage
%
%
%

%
%########################################################################################################
%
%%
%
%########################################################################################################
%
%%
%
\section{Introduction}

The point vortex system is a classical topic which originates from fluid mechanics and goes back to Helmholtz \cite{Helmholtz}, Kirchhoff \cite{Kirchhoff}, Poincar\'e \cite{Poincare}, Routh \cite{Routh}, Kelvin \cite{Kelvin}, and Lin \cite{Lin1,Lin2}.
It appeared as an idealized model where the vorticity of an ideal incompressible two-dimensional fluid is concentrated in a finite number of points.
Although it does not constitute a solution to the Euler equations in the sense of distributions, it is now understood that point vortices can be viewed as limits of concentrated smooth vortices which evolve according to the Euler equations.
In the case of  a single vortex moving in a bounded and simply-connected domain this was proved by Turkington in \cite{Turk}.
An extension to the case of several vortices was given by Marchioro and Pulvirenti; see \cite{MP}. 
Recently  Gallay has proven in \cite{Gallay} that the point vortex system can also be obtained as vanishing viscosity limits of concentrated smooth vortices evolving according to the incompressible Navier-Stokes equations.

The main goal of this paper is to prove that the point vortex system can also be viewed as the limit of the dynamics of a solid, shrinking into a massless point particle with fixed circulation, in free motion in an irrotational bounded flow.
By free motion we mean that the rigid body is only accelerated by the force exerted by the fluid pressure on its boundary, the fluid velocity and pressure being given by the incompressible Euler equations with zero vorticity.
 In a different regime, we also derive a different ``massive'' point vortex system  evoked  (in the case of two point vortices in the whole plane)  by Friedrichs in  \cite[Chapter 3]{Friedrichs} 
 under the terminology of {\it bound vortices} (as opposed to {\it free vortices}).
 In this case the dynamics is given by a second-order differential equation involving a gyroscopic force similar to the 
celebrated Kutta-Joukowski-type lift force revealed in the case of a single body in a irrotational unbounded flow at the beginning of the 20th century during the first mathematical investigations in aeronautics; see for example \cite{Lamb}. 
This result extends the one  obtained in  \cite{GLS} to the case where the solid-fluid system is bounded. 

To distinguish  these two limits  of the dynamics of the solid when its size goes to $0$
 we therefore introduce two cases: 
 Case (i) when the mass of the solid is fixed (and then the solid tends to a point-mass particle), and 
Case (ii): when the mass tends to 0 along with the size (and then the solid tends to a massless point particle). In particular Case (ii) encompasses the case of fixed density.
The main results in this paper establish the massive point vortex system  in Case (i), see Theorem~\ref{theo:3}, and the 
{\it classical point vortex system} in Case (ii), see  Theorem~\ref{theo:1}, 
as  limits  of the dynamics of a shrinking solid in a fluid  in a cavity. 
In both cases the strength of the point vortex obtained in the limit is given by the circulation around the body.  
This circulation is supposed held fixed independently of the size of the body and  is conserved as time proceeds according to Kelvin's theorem. 

From the fluid viewpoint the circulation around the body somehow encodes the amount of vorticity hidden in the body.
The limit where the body has a diameter tending to zero therefore corresponds to a singular perturbation problem (in space). 
Indeed it is well understood since the work \cite{ILN}, see also \cite{Lopes}, that when a solid obstacle with a nonzero given circulation is held fixed in a perfect incompressible fluid, with possibly nonzero vorticity, then in the limit the obstacle shrinks into a fixed point particle and the Euler equation driving the fluid evolution has to be modified: in the Biot-Savart law providing the fluid velocity generated by the fluid vorticity, a point vortex placed at the fixed position of the point obstacle has to be added to the fluid vorticity, with a strength equal to the circulation previously mentioned. 
 
Still the dynamics of an immersed rigid body requires a more precise analysis, in particular because it is driven by the fluid pressure on the boundary of the solid, a quantity which depends in a nonlinear and non local way on the fluid velocity and hides a remote interaction between the moving body and the exterior boundary. 
Moreover, in the zero-radius limit, the pressure field on the boundary of the solid is expected to be singular and  the Newton-Euler equations driving the particle's dynamics involve a singular perturbation problem in time (in addition to the singularity in space), in a particularly intricate way for asymmetric particles (actually for any other form than a disk) and even more so for light particles whose mass and moment of inertia go to zero (Case (ii)). 
   Our analysis relies on a detailed treatment of the structure of these singularities, first for any positive body radius and then for vanishingly small radius, 
 to describe the transition of the dynamics in the limit.

In a first step, see Theorem~\ref{THEO-intro} below, we reformulate the dynamics of the body in order to make more explicit different kind of interactions with the fluid. 
 Indeed the fluid velocity can be recovered from the solid position and velocity by an elliptic-type
problem, so that the fluid state may be seen as solving an auxiliary steady problem,
where time only appears as a parameter, instead of the unsteady incompressible Euler equation. 
We establish a reformulation of the Newton-Euler equations as a second-order differential equation on
the solid position which is determined by three degrees of freedom  (two for the translation and one for the rotation) with coefficients obtained by solving
the auxiliary  fluid problem.
Indeed we establish that the dynamics  of the body  may be recast as a geodesic equation  with an applied force similar to the Lorentz force of electromagnetism. 
 The metric associated with the geodesic part of the equation is given by the total inertia, that is the inertia of the solid to which one adds the so-called ``added inertia'': a symmetric nonnegative  matrix depending only on the body's shape and position, encoding the amount of incompressible 
fluid that the rigid body has also to accelerate around itself. 
 The magnetic part of the Lorentz force is a gyroscopic force, proportional to the circulation around the body, which can be seen as an extension of the Kutta-Joukowski lift force. In particular the contribution of this force to the energy variation vanishes. 
On the contrary, the electric part of the Lorentz force leads to an energy exchange between the fluid and the solid.

To be able to describe the transition of the dynamics in the zero-radius limit we would like to deduce uniform estimates from this geodesic-Lorentz reformulation in order to pass to the limit. 
 Still the equation contains also the electric-type force which prevents from obtaining such bounds from an energy estimate. 
  Indeed even if the whole system is hamiltonian, see \cite{GS-Arnold}, nothing excludes a priori some sharp energy exchange between the fluid and the immersed rigid particles. 
  To overcome this difficulty the second part of our analysis exploits the structure of the various terms of the geodesic-Lorentz reformulation in  the limit where the size of the solid goes to $0$. 
  These terms involve integrals  of  functions describing the part of the fluid velocity due to the body's  velocity and to the body's circulation. 
 These functions are given as solutions to some elliptic-type problems in a domain which is the complement to the vanishingly small body  in the cavity, which entails some small-scale variations of these functions. 
Multi-scale expansions allow us to precisely describe their asymptotic behaviour and to deduce asymptotic expansions of every term of the geodesic-Lorentz reformulation. 
In Case (i) a rough expansion is sufficient as the only leading term is gyroscopic and therefore allows to obtain some uniform estimates. 
In Case (ii) a few striking combinations allow us to transfer the bad electric-type term as a modulation of the particle velocity appearing in the other terms of the equation (i.e. the geodesic and magnetic terms).
This leads to a geodesic-gyroscopic asymptotic normal form 
 of the Newton-Euler equations where a modulation of the unknown is used. 
 This normal form  is tailored to obtain uniform estimates on the dynamics thanks to some energy-type quantities modulated by the limit dynamics. 
 These uniform estimates then allow us to pass to the limit in both cases.

%
%
%
%
%
%
%
%%%%%%%%%%%%%%%%%%%%%%%%%%%%%%%%%%%%%%%%%%%%%%%%%%%%%%%%%%%%%%%%%%%%%%%%%%%%%%%%%%%%%%%%%%%%%%%%%%%%%%%%%%%%%%%%%%%%%%%%%%%%%%%%%%%%%%
%
%
%
%
%
%
%
\section{Main results}
\label{MR}
\subsection{Dynamics of a solid with fixed size and mass}
\label{Subsec:FixedSize}
To begin with, let us recall the dynamics of a solid with fixed size and mass in a perfect incompressible fluid. We denote by $\Omega$ 
\index{BGrecZO@$\Omega$: fixed domain occupied by the whole  system}
the bounded open smooth and simply connected\footnote{the simple connectedness is a simplifying assumption but is actually not essential in the analysis.}  
domain of $\mathbb R^2$ occupied by the fluid-solid system.
At the initial time, the domain of the solid is a non-empty closed smooth and simply connected domain $\mathcal S_0 \subset \Omega$ 
\index{AS1@$\mathcal S_0$: domain initially occupied by the solid}
and
$\mathcal F_0:=\Omega\setminus {\mathcal S}_0 $             
\index{AF0@$\mathcal F_0$:  domain initially occupied by the fluid}
is the domain of the fluid. 
There is no loss of generality in assuming that the center of mass of the solid coincides at the initial time with the origin. 

The rigid motion of the solid is represented at every moment by a rotation matrix
\begin{equation*} 
R(\vartheta(t)) := 
\begin{bmatrix}
\cos  \vartheta (t) & - \sin \vartheta (t) \\
\sin  \vartheta (t) & \cos  \vartheta (t)
\end{bmatrix},
\index{BGrecFT1@$\vartheta$: rotation angle of the solid}
\index{AR1@$R(\vartheta)$: $2\times2$  rotation matrix of angle $\vartheta$}
\end{equation*}
describing the rotation of the solid with respect to its original position
and a vector $h(t) $ in $\mathbb R^2$ describing the position of the center of mass. 
\index{AH1@$h$: position of the center of mass}
The domain of the solid at every time $t>0$ is therefore 
$\mathcal S(t) :=R(\vartheta(t))\mathcal S_0+h(t), $ 
while the domain of the fluid is 
$\mathcal F(t) := \Omega\setminus {\mathcal S}(t)$ (see Fig.~\ref{Fig1}).
%
%------------------------------------------
\begin{figure}[htb]
\def\svgwidth{0.5\textwidth}
\centerline{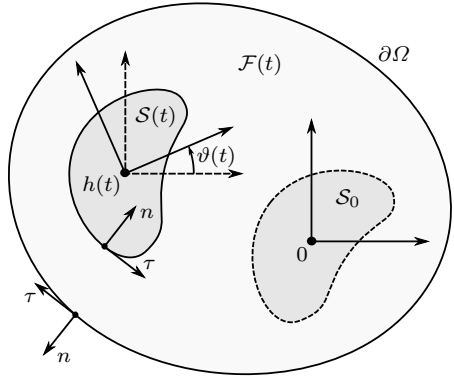}
\caption{\label{Fig1}The domains $\Omega$, $\mathcal S(t)$ and $\mathcal F(t):=\Omega\setminus\overline{\mathcal S(t)}$ of the problem.}
\end{figure}
%-----------------------------------------
%
The fluid-solid system is governed by the following set of coupled equations:
\begin{subequations} \label{SYS_full_system}
\begin{alignat}{3}
\nonumber \text{Fluid equations:}& \\
\frac{\partial u}{\partial t}+(u\cdot\nabla)u +\nabla \Pi&=0&\quad&\text{in }\mathcal F(t) , \label{EEE1} \\
\dive u&=0&&\text{in }\mathcal F(t). \label{E2} \\
\nonumber \text{Solid equations:}& \\
mh''&=\int_{\partial\mathcal S(t)}\Pi n\, {\rm d}s,    \label{EqTrans} \\
\mathcal J\vartheta''&=\int_{\partial\mathcal S(t)}(x-h(t))^\perp\cdot \Pi n\, {\rm d}s.  \label{EqRot} \\
\nonumber \text{Boundary conditions:}& \\
u\cdot n&=\big(\vartheta'  (\cdot-h)^\perp+ h' \big) \cdot n&&\text{on }\partial\mathcal S(t) ,\label{souslab}\\
u\cdot n&=0&&\text{on }\partial\Omega. \label{souslabis} 
\end{alignat}
\end{subequations}
Above $u$ and $\Pi$ denote the velocity and pressure fields in the fluid,
\index{AU1@$u$: fluid velocity}
\index{BGrecP1@$\Pi$: fluid pressure}
$m>0$ and $\mathcal{J}>0$
\index{AM0@$m$: solid's mass}
\index{AJ1@$\mathcal{J}$: solid's moment of inertia}
denote respectively the mass and the moment of inertia of the body  while the fluid  is supposed to be  homogeneous of density $1$, to simplify the notations.
When $x=(x_1,x_2)$
the notation $x^\perp $ stands for $x^\perp =( -x_2 , x_1 )$, 
$n$ \index{AN1@$n$: normal vector}
denotes the unit normal vector pointing outside of the fluid.
Let us also emphasize that ${\rm d}s$ will stand for the arc length without any distinction between $\partial \Omega$,
$\partial \mathcal S(t)$ and $\partial \mathcal S_{0}$. \index{AD1@${\rm d}s$: arc length} \par
In this paper, we consider irrotational solutions, that is, solutions satisfying
\begin{equation} \label{irr}
\curl u (t,\cdot) = 0 \text{ in } \mathcal F(t).
\end{equation}
Actually introducing the initial data:
\begin{subequations}  \label{ID_full_system}
\begin{gather}
u_{|t=0}=u_0 \text{ in }\mathcal F_0,  \label{fluid_initial_cond} \\
\vartheta(0)=0, \ h(0)=0, \ h'(0)=\ell_0, \  \vartheta' (0)=\omega_0,
\end{gather}
\end{subequations}
the condition~\eqref{irr} merely depends on the initial vorticity $\curl u_0$: if the flow is irrotational at the initial time, that is $\curl u_0 = 0$ in $\mathcal F_0$,
it will remain irrotational for every time as in \eqref{irr}, according to Helmholtz's third theorem.
On the other hand the circulation around the body is constant in time:
\begin{equation} \label{circ}
\int_{\partial\mathcal S(t)} u(t)\cdot\tau \, {\rm d}s = \gamma ,
\end{equation}
with
\begin{equation} \nonumber
\gamma := \int_{\partial\mathcal S_0} u_0\cdot\tau \, {\rm d}s,
\index{BGrecAG1@$\gamma$: circulation} 
\end{equation}
according to Kelvin's theorem.
Here $\tau$
\index{BGrecT@$\tau$: tangential vector}
denotes the unit counterclockwise tangential vector so that $n = \tau^\perp$.
Let us mention here that we will also use the notation $\tau$ on  $\partial \Omega$ such that $n := \tau^\perp$ so that it is clockwise in this case (see Fig.~\ref{Fig1}).
\par
\ \par
In the irrotational case, System \eqref{SYS_full_system}-\eqref{irr} can be recast as an ODE whose unknowns are the degrees of freedom of the solid, namely $\vartheta$ and $h$.
In particular, given \eqref{circ}, the motion of the fluid is completely determined by the solid position and velocity.
To state this rigorously, let us introduce the variables
$$
h := (h_{1}, h_{2}) \quad\text{and}\quad
q:=(\vartheta,h_1,h_2)\in\mathbb R^3 .  \index{AQ1@$q$: body position}$$ 
Since the domains $\mathcal S(t)$ and $\mathcal F(t)$ depend on $q$ only, we will rather denote them $\mathcal S(q)$ and $\mathcal F(q)$ in the rest of the paper.
\index{AS2@$\mathcal S(q)$: solid domain associated with the solid position $q$}
\index{AF1@$\mathcal F(q)$: fluid domain associated with the solid position $q$}
Since throughout this paper we will not consider any collision, we introduce:
\begin{equation} \label{DefQ}
\mathcal Q :=\{q\in\mathbb R^3\ :\ d(\mathcal S(q),\partial\Omega)>0\}, 
\index{AQ3@$\mathcal Q$: set of body positions without collision}
\end{equation}
where $d(A,B)$ denotes the minimal distance between the two sets $A$ and $B$ in the plane 
\begin{equation} \label{MinimalDistance}
d(A,B) := \inf \, \{  | x-y  |_{\R^{2}} , \ x \in A, \ y \in B \}.
\end{equation}
Above and throughout the paper we use the notation $ | \cdot |_{\R^{d}}$ for the Euclidean norm in $\R^{d}$.
Since $\mathcal S_0$ is a closed subset in the open set $ \Omega$, the initial position $q(0)=0$ of the solid belongs to $\mathcal Q$. \par
\ \par
Now we need to introduce various objects depending on the geometry and on the constants $m$, ${\mathcal J}$, $\gamma$, in order to make the aforementioned ODE explicit. \par
\ \par
\noindent
{\it Kirchhoff potentials $\varphi_{j}$.}
Consider the functions $\xi_{j}$, for $j=1,2,3$, defined for $(q,x) $ in $\cup_{q \in \mathcal Q} \;  \Big( \{ q  \} \times  \mathcal F(q)  \Big)$, by the formula
\begin{equation} \label{def-xi-j}
\xi_{1}  (q,x) := (x-h)^\perp 
\text{ and }
\xi_{j}  (q,x) := e_{j-1}, \text{ for } j=2,3,
\index{BGrecNX@$\xi_{j}$: elementary rigid velocities}
\end{equation}
where $e_1$ and $e_2$ are the unit vectors of the canonical basis.
\index{AE0@$e_1, e_2$: unit vectors of the canonical basis}
For any $j=1,2,3$,  for  any $q$ in $\mathcal Q$, 
we denote by $K_j(q,\cdot)$ the normal trace of $\xi_{j} $ on $\partial\Omega  \cup \partial \mathcal S(q)$, that is:
\begin{equation} \label{Def-Kj}
K_j (q,\cdot) := n \cdot  \xi_{j} (q,\cdot) \text{ on } \partial\Omega  \cup \partial \mathcal S(q) .
\index{AK1@$K_{j}(q,\cdot)$: normal trace of elementary rigid velocities}
\end{equation}
Now $q$ being fixed in ${\mathcal Q}$, we introduce the Kirchhoff potentials $\varphi_j(q,\cdot)$, for $j=1,2,3$, as the unique (up to an additive constant) solutions in $\mathcal F(q)$ of the following Neumann problem:
\begin{subequations} \label{Kir}
\begin{alignat}{3}
\Delta \varphi_j (q, \cdot) &= 0 & \quad & \text{ in } \mathcal F(q),\\ \label{Kir-b}
\frac{\partial \varphi_j}{\partial n} (q,\cdot)&=K_{j} (q,\cdot) & \quad & \text{ on }\partial\mathcal S(q), \\\label{Kir-c}
\frac{\partial\varphi_j}{\partial n}(q,\cdot) & =0 & &\text{ on } \partial \Omega.
\end{alignat}
\index{BGrecTF1@$\varphi_j(q,\cdot)$ ($j=1,2,3$) :  Kirchhoff's potentials}
\end{subequations}
We concatenate the $K_j$ and $\varphi_j$ into the vectors:  
\begin{subequations} \label{Kir-gras}
\begin{align} 
\boldsymbol{K}(q,\cdot) &:=(K_1(q,\cdot),K_2(q,\cdot),K_3(q,\cdot))^{t}  \text{ and } \\
\boldsymbol\varphi(q,\cdot) &:=(\varphi_1(q,\cdot),\varphi_2(q,\cdot),\varphi_3(q,\cdot))^{t} ,
\index{BGrecTF2@$\boldsymbol\varphi(q,\cdot)$: vector containing the three Kirchhoff potentials}
\end{align}
\end{subequations}
where the exponent $t$ denotes the transpose of the vector. 

\ \par
\noindent
{\it Stream function $\psi$ for the circulation term.} 
For every $q $ in $\mathcal Q$, there exists a unique $C(q)\in\mathbb R$ such that the unique solution 
$\psi(q,\cdot)$ of the Dirichlet problem:
\begin{subequations} 
\label{def_stream}
\begin{alignat}{3}
\Delta \psi(q,\cdot) & =0 & \quad & \text{ in } \mathcal F(q) \\
\psi(q,\cdot) & =C(q) & \quad & \text{ on } \partial \mathcal S(q)\\
\psi(q,\cdot) & =0 & & \text{ on } \partial\Omega,
\end{alignat}
satisfies
\begin{equation} \label{circ-norma}
\int_{\partial\mathcal S(q)} \frac{\partial\psi}{\partial n} (q,\cdot) \, {\rm d}s=-1.
\index{BGrecYP1@$\psi$: circulatory part of the stream function}
\end{equation}
\end{subequations}
This can be seen easily by defining the corresponding harmonic function $\tilde{\psi}(q,\cdot)$ with $\tilde{\psi}(q,\cdot)=-1$ on $\partial \mathcal S(q)$ and  $\tilde{\psi}(q,\cdot)=0$ on $\partial \Omega$ and renormalizing it. 
Indeed the strong maximum principle ensures that $\frac{\partial\tilde{\psi}}{\partial n} (q,\cdot)<0$ on $\partial \mathcal S(q)$, so that $\int_{\partial {\mathcal S}(q)} \,\frac{\partial\tilde{\psi}}{\partial n} (q,\cdot) \,{\rm d}s <0$. \par
The function $C(q)$ is actually minus the inverse of the condenser capacity  of $\mathcal S(q)$ in $\Omega$, that is, of $\int_{ \mathcal F(q)} |  \nabla \tilde{\psi}   (q,\cdot) |^2 \, {\rm d}x$.
Observe that
\begin{gather}
\label{signeC}
\forall q \in \mathcal Q, \quad C(q) = - \int_{ \mathcal F(q)} |  \nabla \psi   (q,\cdot) |^2 \, {\rm d}x < 0 , \\
\label{reg-C}
C \in C^\infty   ( \mathcal Q ; (-\infty ,0)) \text{ and depends on }  \mathcal S_0  \text{ and } \Omega.
\end{gather}
Concerning \eqref{reg-C} and other similar properties below about the regularity with respect to the shape, we refer to \cite{CM,HP,M,SZ}. \par
\ \par
\noindent
{\it Decomposition of the fluid velocity.}
The fluid velocity $u$ satisfies a div-curl type system in the doubly-connected domain $\mathcal F(q)$, constituted of 
\eqref{E2},
\eqref{souslab},
\eqref{souslabis},
\eqref{irr}
and of 
\eqref{circ}. 
When the solid position $q $ in $\mathcal Q$, 
and the  right hand sides of these equations, including the solid velocity $q'= (q'_{1} , q'_{2} , q'_{3}) \in \R^3$ are given,
the fluid velocity $u $  is determined in a unique way and we will therefore denote it by $u (q,\cdot)$.
Moreover, thanks to \eqref{Kir}, \eqref{Kir-gras} and \eqref{def_stream}, the solution  $u (q,\cdot)$  takes the form:
\begin{equation} \label{EQ_irrotational_flow}
u (q,\cdot) = u_{q'} (q,\cdot) + u_\gamma (q,\cdot) , 
\end{equation}
with 
\begin{subequations} \label{DecompU}
\begin{align} 
\label{DecompUa}
u_{q'}  (q,\cdot) &:=\nabla (\boldsymbol\varphi(q,\cdot)\cdot q') = \nabla \left( \sum_{j=1}^3 \varphi_j (q,\cdot)q'_j \right) \\
\label{DecompUb} \text{ and }\quad u_\gamma  (q,\cdot) &:=\gamma \nabla^\perp\psi (q,\cdot) .
\end{align}
\end{subequations}
So besides the dependence with respect to  $\mathcal S_0$, to $\Omega$ and to the space variable, $u_{q'}$ 
depends on $q$ and linearly on $q'$ while $u_\gamma$ depends on $q$ and linearly on $\gamma$. 
Notice that the initial data \eqref{fluid_initial_cond} for system of equations \eqref{SYS_full_system} is no longer required since it can be deduced from the given circulation $\gamma$ and the initial data of the solid through the functions $\boldsymbol\varphi(0,\cdot)$ and $\psi(0,\cdot)$. \par
\ \par
\noindent
{\it Inertia matrices.}
We can now define the mass matrices
\begin{subequations} \label{def_M_Gamma}
\begin{align}
M_g  &:={\rm diag}(\mathcal J,m,m),\\
\index{AM2@$M_{g}$: genuine solid inertia} 
\label{def-MAq}
M_a (q)  &:= \int_{\partial \mathcal S(q)}\boldsymbol\varphi(q,\cdot)\otimes\frac{\partial\boldsymbol\varphi}{\partial n}(q,\cdot) \,{\rm d}s 
= \Big( \int_{ {\mathcal F}(q) } \nabla \varphi_{i}  \cdot  \nabla\varphi_{j} {\rm d}x 
  \Big)_{1  \leqslant i,j  \leqslant  3}, \\
\index{AM3@$M_a$: added inertia} 
M(q) &:= M_g  + M_a (q) .
\index{AM4@$M$: total inertia of the solid}
\end{align}
\end{subequations}
The matrix $M(q)$ corresponds to the sum of the genuine inertia $M_{g}$ of the body and the so-called added inertia $M_{a} (q)$, which, loosely speaking,  measures how much the surrounding fluid resists the acceleration of the  body motion (since the fluid undergoes an acceleration as well). Both $M_{g}$ and $M_{a} (q)$ are symmetric and positive-semidefinite, and $M_{g}$ is positive definite.  \par
\ \par
\noindent
{\it Christoffel symbols.} A bilinear symmetric mapping $ \Gamma (q)$ associated with $M(q)$ can be defined as follows, %
for every $p=(p_1,p_2,p_3)\in\mathbb R^3$:
\begin{subequations} 
\label{def_vraiment-Gamma}
\begin{equation} \label{Christo2}
\langle\Gamma (q),p,p\rangle :=\left(\sum_{1\leq i,j\leq 3} \Gamma^k_{i,j}(q) p_i p_j  \right)_{1\leq k\leq 3}\in\mathbb R^3 ,
\end{equation}
where, for every $i,j,k \in\{1,2,3\}$, we denote by 
\begin{equation} \label{Christo1}
\Gamma^k_{i,j} (q)  :=  \frac12
\Big(  (M_a)_{k,j}^{i} + (M_a)_{k,i}^{j} - (M_a)_{i,j}^{k}  \Big) (q) 
\index{BGrecAG2@$\Gamma$: Christoffel symbols}
\end{equation}
the Christoffel symbols of the first kind associated with the mass matrix.
In this identity, the notation $(M_a)_{i,j}^{k} $ stands for the partial derivative of the entry of indices $(i,j)$ of the matrix $M_a$ with respect to $ q_{k}$, that is
\begin{equation} \label{cri}
(M_a)_{i,j}^{k} := \frac{\partial (M_a)_{i,j}}{\partial q_{k}} .
\end{equation}
\end{subequations}
We underline that since the genuine inertia $M_{g}$ of the body is independent of the position $q$ of the solid, only the added inertia is involved in the Christoffel symbols. \par
\ \par
\noindent
{\it Force term.}
Eventually, we also define the column vectors:
\begin{subequations} 
\label{def_EetB}
\begin{align}
\label{B-def}
B(q) & :=  \int_{\partial\mathcal S(q)} \left( \frac{\partial\psi}{\partial n}  \left( \frac{\partial\boldsymbol\varphi}{\partial n}  \times 
\frac{\partial\boldsymbol\varphi}{\partial \tau} \right)  \right)(q,\cdot) \, {\rm d}s,\\
\label{E-def}
E(q) &:= - \frac{1}{2} \int_{\partial\mathcal S(q)}  
\left( \left| \frac{\partial\psi}{\partial n} \right|^2 \frac{\partial\boldsymbol\varphi}{\partial n} \right) (q,\cdot)\,  {\rm d}s ,
\index{AB1@$B(q)$: magnetic-type field acting on the solid}
\index{AE1@$E(q)$: electric-type field acting on the solid }
\index{AF2@${F}(q,p)$: total force acting on the solid}
\end{align}
and for $p$ in $\R^3$ the force term
\begin{equation} \label{def-upsilon}
F (q,p) :=  \gamma^2 E(q) +  \gamma \, p \times B(q) .
\end{equation}
\end{subequations}
We recall that $\gamma$ denotes the circulation around the body. \par
\begin{rem}
The notations $E$ and $B$ are chosen on purpose to highlight the analogy with the Lorentz force acting on a charged particle moving under the influence of an electromagnetic field $(E,B)$. This force vanishes if $\gamma = 0$.
\end{rem}
It can be checked that
\begin{subequations} 
\begin{align}
\label{reg-M}
 &M \in  C^{\infty}(\mathcal Q ; S^{++}_3 (\R)) \text{ and depends on }  \mathcal S_0, m, \mathcal{J}  \text{ and } \Omega,  
\\ &{F}  \in C^{\infty}(\mathcal Q \times \R^3 ; \R^3) \text{ and depends on }  \mathcal S_0, \gamma \text{ and } \Omega,\nonumber
\\ &\hspace{4.5cm}\text{  and vanishes when  } \gamma =0,
\\  &\Gamma  \in C^{\infty}(\mathcal Q; \mathcal{BL} (\R^3 \times \R^3 ; \R^3 ) ) \text{ and depends on }  \mathcal S_0  \text{ and } \Omega.
\end{align}
\end{subequations}
Above $S^{++}_3 (\R)$ denotes the set of real symmetric positive-definite $3\times 3$ matrices, $\mathcal{BL} (\R^3 \times \R^3 ; \R^3 )$ denotes the space of bilinear mappings from  $ \R^3 \times \R^3 $ to $ \R^3$.

We stress that $M$ does not depend on the circulation $\gamma$ whereas $F$ does not depend on $m$ and $\mathcal{J}$  and $ \Gamma $ does not depend on  $m$, $\gamma$ and $\mathcal{J}$.
In the following, when specifying these dependences is relevant, we will denote 
\begin{equation} \label{crochet}
M [\mathcal S_0 , m ,\mathcal{J},\Omega],  \ \Gamma [\mathcal S_0 ,\Omega]
\text{ and } 
F [\mathcal S_0 ,\gamma,\Omega]  \text{ instead of } M, \ \Gamma  \text{ and }  F.
\end{equation}
Now our first result is a rephrasing of System \eqref{SYS_full_system}-\eqref{irr} as an ordinary differential equation. \par
\begin{thm} \label{THEO-intro}
For smooth solutions, System~\eqref{SYS_full_system}-\eqref{irr} can be recast, up to the first collision, as the second order ODE 
\begin{equation} \label{ODE_intro}
% q' &=p,\\
 M(q)  q''  + \langle \Gamma (q),q',q'\rangle = {F} (q,q') .
\end{equation}
\end{thm}
Let us emphasize that 
 \eqref{ODE_intro} only determines the body motion. 
The fluid velocity $u (q,\cdot)$ 
is then deduced by 
\eqref{EQ_irrotational_flow} and \eqref{DecompU}.
The proof of Theorem~\ref{THEO-intro} is postponed to Section \ref{tard}. \par
\begin{rem}
\label{rem-geod}
Theorem~\ref{THEO-intro} extends  \cite [Theorem 1.1]{Munnier} where 
 the potential case, i.e. the case where $\gamma=0$,   is obtained. 
In that case  the ODE \eqref{ODE_intro} means that the particle is moving along the geodesics associated with the Riemann metric induced on $\mathcal Q$ by the matrix $M(q)$.
  
On the other hand in  \cite{GS-Arnold} another reinterpretation  in term of  geodesics is given: 
classical solutions to the PDEs  driving the fluid-rigid body system 
are the geodesics of a Riemannian manifold of infinite dimension, in the sense that
they are the critical points of an action, which is the integral over time of the total kinetic energy of
the system. 
This result was stated in the 3D case in \cite{GS-Arnold} but holds in the present setting as well, even in the case where $\gamma \neq 0$ and where the fluid is rotational. 
Theorem  \ref{THEO-intro-without} somehow establishes that for irrotational flows this geodesic structure can be projected on the degrees of freedom corresponding to the immersed rigid particle only, at the prize of an extra Lorentz-type force for $\gamma \neq 0$.
\end{rem}
\begin{rem}
The analysis leading to Theorem~\ref{THEO-intro} is extended in \cite{GKS} in the case where a control acts on a part of the
external boundary. Then the remote influence of the external boundary control translates into  additional
force terms.
\end{rem}
According to classical ODE theory, given $\gamma$, there exists a maximal time $T>0$ and a unique maximal solution $q$ in $C^\infty([0,T) ; \mathcal Q)$  to \eqref{ODE_intro} 
with Cauchy data $q(0)=(0,0)$, $q'(0)=(\omega_0,\ell_0)$.

Moreover,  $T$ is the time of the first collision of the solid with the outer boundary of the fluid domain. 
If there is no collision, then $T=+\infty$. 
This follows from Corollary~\ref{bd-loin} below, which itself relies on an energy argument. 
Indeed an important feature of the system \eqref{ODE_intro} is that it is conservative. 
More precisely,  let us define
 for  $(q,p) $ in $\mathcal Q  \times  \mathbb R^3$, 
\begin{equation}
\label{pader} 
\mathcal{E} (q,p) := \frac{1}{2} M(q) p \cdot p + U(q) ,
\index{AE1@$\mathcal{E} (q,p)$: total energy}
\index{AU4@$U (q)$: potential energy}
\end{equation}
where the potential energy $U(q)$ is given by
\begin{equation} \nonumber
U(q) := - \frac{1}{2} \gamma^{2}  C (q),
\end{equation}
with $C(q)$ given by \eqref{def_stream}. We will prove the following in Subsection \ref{section122}.
\begin{lem}\label{LEM-pot-energy}
The derivative $DC(q)$ of $C(q)$ with respect to $q$ satisfies:
\begin{equation} \label{pot-energy}
\forall q \in \mathcal Q ,  \quad \frac{1}{2} DC (q) =E(q)   ,
\end{equation}
where $E(q)$ is defined by \eqref{E-def}.
\end{lem}
As a corollary of Lemma~\ref{LEM-pot-energy} we have the following result regarding the energy conservation.
\begin{prop} \label{energy}
For any $q $ in $C^\infty([0,T] ;\mathcal Q )$ satisfying  \eqref{ODE_intro}, 
\begin{equation} \label{conserv}
\frac{d}{dt} \mathcal{E} (q,q') = 0 ,
\end{equation}
\end{prop}
We will prove Proposition  \ref{energy} in  Subsection \ref{sectionProp1}. \par
Let us emphasize that the energy function $\mathcal{E}$ belongs to $C^\infty(\mathcal Q  \times  \mathbb R^3 ;\R)$ and is the sum of two nonnegative terms.
 In addition to its dependence on $q$ and $p$, the energy $ \mathcal{E}$ depends on $\mathcal S_0 , m ,\mathcal{J}, \gamma$ and $\Omega$. \par
If we assume that the body remains at a distance greater than $\delta > 0$ from the boundary we may infer a bound for the body velocity.
Indeed  we have the following immediate corollary of Proposition \ref{energy}.
\begin{cor} \label{bd-loin}
Let $\mathcal S_0$ a subset of $\Omega$, $p_0$ in  $\mathbb R^3$ and $(\gamma, m, \mathcal J)$ in $\mathbb R \times (0,+\infty) \times (0,+\infty)$.
Let $\delta > 0$ and $q$ in $C^\infty([0,T] ;\mathcal Q \times  \mathbb R^3)$ satisfying  \eqref{ODE_intro} with the Cauchy data
$(q , q')(0)= (0, p_0)$ and such that $d(\mathcal S(q(t)),\partial\Omega)> \delta$ for $t$ in $[0,T]$.
Then there exists $K>0$ depending only on $\mathcal S_0$, $\Omega$, $p_0$, $\gamma$, $m$, $\mathcal J$ and $\delta$ such that 
$ | q' |_{\R^{3}}  \leq K$ on $[0,T]$.
\end{cor}
Let us refer here to \cite{HM} and \cite{Munnier:2015aa} for a study of the collision of a solid moving in a potential flow (that is in the case where $\gamma=0$) with the fixed boundary of the fluid domain. \par
\ \par
Let us now turn our attention back to the Christoffel symbols defined in \eqref{def_vraiment-Gamma}. They actually can be split into two parts: one taking into account the effect of the solid rotation and the other part encoding the effect of the exterior boundary.
First, we introduce the impulses $\rho_g$, $\rho_a$, $\rho$ in $\R$ and $P_g$, $P_{a}$  and
$P$ in $\R^2$ by the following relations
\begin{equation} \label{onoublie}
\begin{pmatrix}
\rho_g  \\
P_g 
\end{pmatrix} :=M_{g} p
, \quad 
\begin{pmatrix}
\rho_a \\
P_a
\end{pmatrix} := M_a(q)p
, \quad 
\begin{pmatrix}
\rho \\
P
\end{pmatrix} :=  
\begin{pmatrix}
\rho_g  +\rho_a \\
P_g  + P_a
\end{pmatrix} .
\index{BGrecR1@$\rho_a$: added angular impulse}
\index{AP2@$P_a$: added translation impulse}
\end{equation}
Then for every $q $ in $\mathcal Q$ and for every $p = (\omega , \ell )$ in $\mathbb R^3$, we let:
\begin{equation} \label{zozo}
\langle \Gamma^{{\rm rot}} (q), p, p\rangle := 
- \begin{pmatrix}0\\
P_a \end{pmatrix}
\times p
-
\omega M_a (q)\begin{pmatrix}
0\\
\ell^\perp
\end{pmatrix}       \in\mathbb R^3 .
\index{BGrecAG3@$\Gamma^{{\rm rot}} $: Christoffel tensor related to the solid rotation}
\end{equation}
We can notice that one also has 
$$
\langle \Gamma^{{\rm rot}} (q), p, p \rangle = 
- \begin{pmatrix} 0 \\ P \end{pmatrix} \times p
- \omega M (q) \begin{pmatrix} 0 \\ \ell^\perp \end{pmatrix}  \in \R^3,
$$
since the extra terms cancel out. 
\par
Next, for every $j,k,l\in\{1,2,3\}$, we set
\begin{multline} \label{Def_Gamma}
(\Gamma^{\partial \Omega})^j_{k,l}(q) :=\\
\frac{1}{2} \int_{\partial\Omega} \left(\frac{\partial\varphi_j}{\partial\tau}  \frac{\partial\varphi_k}{\partial\tau} K_l 
+ \frac{\partial\varphi_j}{\partial\tau} \frac{\partial\varphi_l}{\partial\tau} K_k
- \frac{\partial\varphi_k}{\partial\tau} \frac{\partial\varphi_l}{\partial\tau} K_j \right) (q,\cdot)\, {\rm d}s ,
\end{multline}
and we associate correspondingly $\Gamma^{\partial \Omega} (q)$ a symmetric bilinear mapping in ${\mathcal B}{\mathcal L} (\R^{3} \times \R^{3};\R^{3})$ so that for $p=(p_1,p_2,p_3)\in\mathbb R^3$:
\begin{equation} \nonumber
\langle \Gamma^{\partial \Omega} (q), p, p \rangle
:= \left(\sum_{1\leq k,l\leq 3} (\Gamma^{\partial \Omega})^j_{k,l}(q) \, p_k p_l \right)_{1\leq j\leq 3} \in \mathbb R^3.
\index{BGrecAG4@$\Gamma^{\partial \Omega} $:  Christoffel tensor omitting the solid rotation}
\end{equation}
The Christoffel symbols satisfy the following relation.
\begin{prop} \label{splitting-christo}
For every $q $ in $\mathcal Q$
and for every $p$ in $\mathbb R^3$, 
\begin{equation} \label{NiouGamma}
\langle \Gamma (q),p,p\rangle =   \langle \Gamma^{{\rm rot}} (q),p,p\rangle  + \langle \Gamma^{\partial \Omega} (q),p,p\rangle .
\end{equation}
\end{prop}
The proof of Proposition \ref{splitting-christo} is given in Subsection~\ref{Sect-splitting-christo}.
We emphasize that in \eqref{NiouGamma} and the expressions above,  unlike \eqref{cri}, there is no derivative with respect to $q$, that is, no more shape derivative.

We will see that in the decomposition \eqref{NiouGamma}, the term $\Gamma^{\partial \Omega}$ obeys a scaling law with respect to $\varepsilon$ different from the one of $\Gamma^{{\rm rot}}$ (compare \eqref{exp-GammaS} and \eqref{expGammaOmega} below), which makes it of lower order. \par 
\ \par
\noindent
The case where $\mathcal S_0$ is a disk is peculiar, and we focus on it for the rest of Subsection~\ref{Subsec:FixedSize}. In particular several degeneracies appear in this case:
\begin{itemize}
\item the added mass matrix $M_{a}(q)$ degenerates (it becomes of rank $2$),
\item the potentials $\varphi_{2}$, $\varphi_{3}$ and $\psi$ depend on $q$ only through the position $h_{c}$ of the center of the disk ${\mathcal S}(q)$, and in particular so do $E_{2}$, $E_{3}$, $B_{1}$ and $(M_{a,i,j})_{i,j=2,3}$,
\item the dynamics of the solid also degenerates in the sense that it satisfies ${\mathcal J} \vartheta'' = m h'' \cdot (h_{c}-h)^{\perp}$.
\end{itemize}
Above and in the sequel we denote with an index (in normal font type) the coordinates of $E$ or $B$ (we will sometimes use italic type indices for other purposes, in a way that should not be ambiguous). 
Note that in particular if the solid is homogeneous, $h=h_{c}$ and $\vartheta'$ is constant over time. \par
As a consequence, in this case where ${\mathcal S}_{0}$ is a disk, we establish a particular reduction of the dynamics. We will use a block decomposition of the matrix  $M_{a}=M_{a}(q)$:
\begin{equation} \label{def-Ma-Omega}
M_{a} =:
\begin{pmatrix}
m_\#   & \mu^t \\
\mu & M_{\flat}   
\end{pmatrix} ,
\end{equation}
where $M_\flat$ is a symmetric $2 \times 2$ matrix.
This matrix is useful in the case where $\mathcal S_0$ is a disk of center $\zeta$ (we will use this notation later in a broader context).
Of course the position $h_c$ of the center of $\mathcal S(t)$ is related to $q= (h,\vartheta)$ by 
\begin{equation} \label{Defhc}
h_c = h + R(\vartheta) \zeta.
\end{equation}
It is easy to see that $M_\flat (q)$ depends on $q$ only through $h_c$ so that we define 
\begin{equation} \label{DefTildeMb}
\tilde{M}_{\flat} (h_c) := M_{\flat} (q).	
\end{equation}
We associate with the matrix field $ \tilde{M}_{\flat} (h_c)$  a bilinear symmetric mapping $\Gamma_{\flat} (h_c)$  defined as follows: for $p_{\flat} \in \R^{2}$,
\begin{subequations} \label{def_vraiment-Gamma-disk}
\begin{equation} \label{Christo2-disk}
\langle\Gamma_{\flat} (h_c), p_{\flat}, p_{\flat} \rangle :=\left(\sum_{1\leq i,j\leq 2} (\Gamma_{\flat})^k_{i,j}(h_c) p_{\flat,i} p_{\flat,j}  \right)_{1\leq k\leq 2}\in\mathbb R^2 ,
\end{equation}
where, for every $i,j,k \in\{1,2\}$, we set
\begin{equation} \label{Christo1-disk}
(\Gamma_{\flat})^k_{i,j} (h_c)  :=  \frac12
\Big(  (\tilde{M}_{\flat})_{k,j}^{i} + (\tilde{M}_{\flat})_{k,i}^{j} - (\tilde{M}_{\flat})_{i,j}^{k}  \Big) (h_c).
\end{equation}
In this identity, the notation $(\tilde{M}_{\flat})_{i,j}^{k} $ stands for the partial derivative of the entry of indices $(i,j)$ of the matrix $\tilde{M}_{\flat}$ with respect to $ (h_c)_{k}$, that is
\begin{equation} \label{cri-disk}
(\tilde{M}_{\flat})_{i,j}^{k} := \frac{\partial (\tilde{M}_{\flat})_{i,j}}{\partial (h_c)_{k}} .
\end{equation}
\end{subequations}
The field $E(q)$ also depends on $q$ only through $h_c$ and we define $E_{\flat} (h_c)$ in $\R^2$ by 
\begin{equation} \label{Eflat}
E_{\flat} (h_c) := (E_2 ,E_3)^t (q) .
\end{equation}
In the same way, $B(q)$ depends on $q$ only through $h_c$ and we define $\tilde{B}_{1} (h_c)$ in $\R$ by the relation
\begin{equation} \label{RelB1}
\tilde{B}_1(h_{c}) = B_{1}(q).
\end{equation}
With these settings, the dynamics can be described as follows.
\begin{thm} \label{THEO-intro-disk}
In the case where $\mathcal S_0$ is a disk of center $\zeta$, for smooth solutions, System~\eqref{SYS_full_system}-\eqref{irr} can be recast up to the first collision, as the following differential system
\begin{subequations} \label{ODE_intro-disk}
\begin{alignat}{3}
\label{OID1}
m h'' + \tilde{M}_{\flat} (h_c) h_c''  + \langle \Gamma_{\flat} (h_c), h_c', h_c' \rangle 
  & = \gamma^2 E_{\flat}  (h_c) - \gamma \tilde{B}_1 (h_c) (h'_c)^\perp, &  \\
\label{OID2}
h_c - h  &= R(\vartheta) \zeta ,& \\
\label{OID3}
\mathcal J\vartheta'' &=  (h_c - h)^\perp \cdot  mh'' ,&
\end{alignat}
\end{subequations}
with Cauchy data
$(h,\vartheta)(0)=0$,  $ (\vartheta,h)'(0)=(\omega_0,\ell_0)$ in  $\mathbb R \times \mathbb R^2$.
\end{thm}
As for Theorem~\ref{THEO-intro}, the fluid velocity $u (q,\cdot)$ is then deduced by \eqref{EQ_irrotational_flow} and \eqref{DecompU}. \par
The proof of Theorem~\ref{THEO-intro-disk} is given in Subsection~\ref{tard-disk}. \par
\ \par
\ \par
\noindent
We have now at our disposal all the material to deal with the limit of the dynamics when the size of the solid goes to $0$.
As mentioned in the introduction, we distinguish two cases:
\begin{itemize}
\item Case (i):  the mass of the solid is fixed (and then the solid tends to a massless point particle), and 
\item Case (ii):  the mass tends to $0$ along with the size (and then the solid tends to a point-mass particle).
\end{itemize}

%
%
%%%%%%%%%%%%%%%%%%%%%%%%%%%%%%%%%%
%
%
\subsection{Case (i): Dynamics of a solid shrinking  to a point-mass particle}
From now on, we suppose that $0 \in \Omega$ and we scale ${\mathcal S}_{0}$ around $0$.
Precisely, for every $\varepsilon $ in $(0,1]$,
\index{BGrecE1@$\varepsilon$: typical size of the solid}
we define 
\begin{equation} \label{piti}
\mathcal S_{0,\varepsilon} :=\varepsilon\mathcal S_0 ,
\index{AS2@$\mathcal S_{0,\varepsilon}$: initial position of the shrinking solid}
\end{equation}
and for every $q=(\vartheta,h)\in\mathbb R^3$, 
\begin{equation} \label{def-solide-scaled}
\mathcal S_\varepsilon(q) :=R(\vartheta)\mathcal S_{0,\varepsilon}+h \quad\text{ and }\quad\mathcal F_\varepsilon(q)=\Omega\setminus {\mathcal S}_\varepsilon(q).
\index{AS3@$\mathcal S_{\varepsilon}(q)$: position of the shrinking solid} 
\end{equation}
We recall that $h(0)=0$ so that \eqref{piti} represents a homothety centered at $h(0)$. 
Without loss of generality, we suppose that for any $\varepsilon \in (0,1]$, $\mathcal S_{0,\varepsilon} \subset \Omega$; it suffices to consider some ${\mathcal S}_{0,\varepsilon}$ as the initial solid position ${\mathcal S}_{0}$, if necessary. \par
\ \par
In Case (i) the solid occupying the domain $\mathcal S_\varepsilon(q)$ is assumed to have a mass and a moment of inertia of the form
\begin{equation} \label{mass-inertie}
m_\varepsilon=  m  \quad\text{ and }\quad \mathcal J_\varepsilon=\varepsilon^2 \mathcal J_{1},
\index{AM1@$m_\varepsilon$: mass of the shrinking solid}
\index{AJ5@$\mathcal{J}_{\varepsilon}$: moment of inertia of the shrinking solid}
\end{equation}
where $m>0$  and  $\mathcal J_{1}>0$ are fixed. \par
With these settings, we denote by $q_\varepsilon$ 
 \index{AQ2@$q_\varepsilon$: position  of the shrinking solid}
the solution to the ODE  \eqref{ODE_intro}  associated with
\begin{equation}
M_\varepsilon  :=  M [\mathcal S_{0,\varepsilon}, m_\varepsilon ,\mathcal{J}_\varepsilon , \Omega ], \ 
\Gamma_\varepsilon  := \Gamma[\mathcal S_{0,\varepsilon} ,\Omega ]
\text{ and } {F}_\varepsilon  := {F} [{\mathcal S}_{0,\varepsilon} , \gamma ,\Omega ],
\end{equation}
in place of $M$, $\Gamma$ and $F$, respectively, 
defined on the maximal time interval $[0,T^\varepsilon)$. 
Accordingly, we define $E_{\varepsilon}$ and $B_{\varepsilon}$ so that the equivalent of \eqref{def-upsilon} is true for all $p \in \R^3$.
As before we decompose $q_\varepsilon$ into
$$q_\varepsilon =(\vartheta_\varepsilon, h_\varepsilon) \in \mathbb R \times \mathbb R^2 .$$
We emphasize that the circulation $\gamma$ and the Cauchy data  $(q_0, p_0)$ do not depend on $\varepsilon$.
The latter are decomposed into 
$$q_0 = (0, 0)  \text{ and }p_0 = (\omega_0 , \ell_0 ).$$
Our first result is the convergence, in this setting, of $h_\varepsilon$ to the solution to a massive point vortex equation.
Let us introduce this limit equation.
Let  $(\overline{h},\overline{T})$ be the maximal solution to the ODE:
\begin{subequations}\label{ODE-mass}
\begin{alignat}{3} 
m \overline{h}''  &=\gamma \left( \overline{h}^\prime- \gamma u^\Omega (\overline{h}) \right)^\perp \, &\quad&\text{ for } t \in [0,\overline{T}), \\
 \text{ with }\quad
\overline{h} (0)&= 0 \text{ and } \overline{h}'(0)=\ell_0,
\end{alignat}
\end{subequations}
where $u^\Omega$ is the Kirchhoff-Routh velocity defined as follows.
Consider first $\psi^{\Int}_{{\it 0}} (h,\cdot)$, the solution to the following Dirichlet problem:
\begin{equation} \label{depsi0}
\begin{array}{l}
\Delta \psi^{\Int}_{{\it 0}} (h,\cdot) =  0  \text{ in } \Omega, \ \ 
\psi^{\Int}_{{\it 0}} (h,\cdot) =  G(\cdot-h )   \text{  on } \partial\Omega,	
\end{array}
\end{equation}
where
\begin{equation} \label{NewtonianPotential}
G(r) := - \frac{1}{2\pi} \ln |r| .
\end{equation}
\index{AG1@$G$: Newtonian potential}
The Kirchhoff-Routh stream function $\psi^\Omega$ is defined as
\begin{equation} \label{def-KRS}
\psi^\Omega (x) :=  \frac{1}{2} \psi^{\Int}_{{\it 0}} (x,x) , \quad x\in\Omega,
\index{BGrecYP2@$\psi^\Omega $: Routh' stream function}
\end{equation}
and the Kirchhoff-Routh velocity $u^\Omega$ is defined by
\begin{equation} \label{DefUOmega}
u^\Omega := \nabla^\perp \psi^\Omega ,
\index{AU2@$u^\Omega $: Routh' velocity} 
\end{equation}
where $ \nabla^\perp := (-\partial_2 , \partial_1 )$. 

The existence of  $(\overline{h},\overline{T})$ follows from classical ODE theory. 
Moreover 
for any $h $ in $C^\infty([0,T]; \Omega)$ satisfying \eqref{ODE-mass}  the following energy conservation holds: 
\begin{equation} \label{conserv(i)}
\frac{d}{dt} \mathcal{E}_{(i)} (h,h') = 0, \quad \text{ with }
\mathcal{E}_{(i)} (h,h') := \frac{1}{2} m h' \cdot h' - \gamma^2 \psi^\Omega (h).
\end{equation}
\begin{rem}
Indeed  \eqref{ODE-mass} is a Hamiltonian system associated with the energy  \eqref{conserv(i)}. 
For further Hamiltonian aspects related to System \eqref{ODE-mass}  we refer for instance to \cite{VKM}.
\end{rem}
From \eqref{conserv(i)} and 
the regularity of the Kirchhoff-Routh stream function $\psi^\Omega$ in $\Omega$  we deduce 
that $\overline{T}$ is the time of the first collision of $\overline{h}$  with the outer boundary $\partial \Omega$ of the fluid domain. 
If there is no collision, then $\overline{T}=+\infty$. \par

The precise statement of our first convergence result is as follows.
\begin{thm} \label{theo:3}
Let $\mathcal S_0$ a subset of $\Omega$ as above, $p_0$ in $ \mathbb R^3$ and  $(\gamma, m, \mathcal J) $ in $\mathbb R \times (0,+\infty) \times (0,+\infty)$.
Let $(\overline{h},\overline{T})$ be the maximal solution to \eqref{ODE-mass}.
For every $\varepsilon$ in $(0,1]$, let $(q_\varepsilon ,T_\varepsilon)$ be  the maximal solution to \eqref{ODE_intro} with respectively
$M_\varepsilon  =  M [\mathcal S_{0,\varepsilon}, m_\varepsilon ,\mathcal{J}_\varepsilon , \Omega ]$,
$\Gamma_\varepsilon  = \Gamma[\mathcal S_{0,\varepsilon} ,\Omega ]$
and  ${F}_\varepsilon  = {F} [{\mathcal S}_{0,\varepsilon} , \gamma ,\Omega ] $
in place of $M$, $\Gamma$ and  $F$, respectively,  where $m_\varepsilon ,\mathcal{J}_\varepsilon$ are given by \eqref{mass-inertie},
and with the initial data 
$q_\varepsilon (0) = 0$ and  $q_\varepsilon' (0)= p_0 .$
Then, as $\varepsilon \to 0^{+}$,  $\liminf T_\varepsilon \geq \overline{T}$, and for all $T$ in $(0,\overline{T})$, 
 $h_\varepsilon \rightharpoonup \overline{h}$ in $W^{2,\infty}([0,T];\mathbb R^2)$ weak-$\star$
and $\varepsilon\vartheta_\varepsilon \rightharpoonup 0$ in $W^{2,\infty}([0,T];\mathbb R)$ weak-$\star$.
\end{thm}
Theorem~\ref{theo:3} is proved in Section~\ref{Sec:PreuveTheo3}.
%
%
%
%##############################################################
%
%
%
%
\subsection{Case (ii): Dynamics of a solid shrinking to a massless point particle}
In this section the solid is still assumed to occupy initially the domain ${\mathcal S}_{0,\varepsilon}$ given by \eqref{piti} (satisfying the same assumptions as above) but we assume now that it has a mass and a moment of inertia given by
\begin{equation} \label{mass-inertie2}
m_\varepsilon= \alpha_\varepsilon m_{1}  \text{ and } \mathcal J_\varepsilon=  \alpha_\varepsilon  \varepsilon^2 {\mathcal J}_{1},
\end{equation}
where $\alpha_{1}=1$ and $\alpha_\varepsilon \rightarrow 0^{+}$ as ${\varepsilon} \rightarrow 0^{+}$, and where $m_{1}>0$ and ${\mathcal J}_{1}>0$ are fixed. %\par
To simplify the notations we will assume that $ \alpha_\varepsilon$ is of the form 
\begin{equation} \label{mass-inertie-deception}
\alpha_\varepsilon =\varepsilon^{\alpha} ,
\end{equation}
with $ \alpha > 0$.
The particular case where $ \alpha =2$ corresponds to the case of a fixed solid density.
Case (i) corresponded to the case where  $ \alpha = 0$. 

In this setting, we denote by $q_\varepsilon =(\vartheta_\varepsilon, h_\varepsilon) $ in $\mathbb R \times \mathbb R^2 $ the solution to the ODE \eqref{ODE_intro} associated with $M_\varepsilon  :=  M [\mathcal S_{0,\varepsilon}, m_\varepsilon ,\mathcal{J}_\varepsilon , \Omega ]$,
$\Gamma_\varepsilon  := \Gamma[\mathcal S_{0,\varepsilon} ,\Omega ]$
and ${F}_\varepsilon  := {F} [{\mathcal S}_{0,\varepsilon} , \gamma ,\Omega ]$ in place of $M$, $\Gamma$ and $F$, respectively, 
defined on the maximal time interval $[0,T^\varepsilon)$.
We stress that the circulation $\gamma$ and the Cauchy data are still assumed independent of $\varepsilon$.
Moreover we will assume here that
$$ \gamma \neq 0.$$

Our second result is the convergence of $h_\varepsilon$ as ${\varepsilon} \rightarrow 0^{+}$ to the solution to the point vortex equation:
\begin{alignat}{3} \label{argh}
\overline{h}^\prime = \gamma u^\Omega (\overline{h})  \text{ for }t>0,   \text{ with } \overline{h} (0)= 0.
\end{alignat}

It is well-known that \eqref{argh}  is a Hamiltonian system associated with the energy 
\begin{equation*}
\mathcal{E}_{(ii)} (h) :=  \gamma^2    \psi^\Omega (h) ;
\end{equation*}
see for instance to \cite{MP,Newton}.
In particular for any $h $ in $C^\infty([0,T]; \Omega)$  satisfying \eqref{argh},  
\begin{equation} \label{conserv(ii)}
\frac{d}{dt} \mathcal{E}_{(ii)} (h) = 0 .
\end{equation}
It follows from \eqref{conserv(ii)} and the fact that 
$\psi^\Omega (h) \rightarrow +\infty$ when $h$ gets close to $\partial \Omega$, see for instance \cite[Eq. (1.27)]{Turk}, that 
the solution $\overline{h} $ is global in time, and in particular that there is no collision of the point vortex with the external boundary $\partial \Omega$. \par
Our result in this situation is the following. 
\begin{thm} \label{theo:1}
Let $\mathcal S_0$ a subset of $\Omega$ as above, different from a disk.
Let $\gamma \neq 0$, $p_0$ in $ \mathbb R^3 $ and  $(m_1 ,\mathcal J_1) $ in $ (0,+\infty) \times (0,+\infty)$. 
Let us consider the global solution $\overline{h}$ to \eqref{argh} and for every $\varepsilon $ in $(0,1]$, let
$(q_\varepsilon ,T_\varepsilon)$ the maximal solution to  \eqref{ODE_intro} with respectively
$M_\varepsilon  =  M [\mathcal S_{0,\varepsilon}, m_\varepsilon ,\mathcal{J}_\varepsilon , \Omega ]$, 
$\Gamma_\varepsilon  = \Gamma[\mathcal S_{0,\varepsilon} ,\Omega ]$
and 
${F}_\varepsilon  = {F} [\mathcal S_{0,\varepsilon} , \gamma ,\Omega ] $
in place of $M, \ \Gamma \text{ and } F$, 
where $m_\varepsilon$, $\mathcal{J}_\varepsilon$ are given by \eqref{mass-inertie2},
and with the initial data
$q_\varepsilon (0) =0$ and  $q_\varepsilon' (0) =  p_0 .$
Then, as $\varepsilon \rightarrow 0^{+}$, 
$T_\varepsilon \rightarrow +\infty$ 
and $h_\varepsilon \rightharpoonup \overline{h}$ in $W^{1,\infty}([0,T];\mathbb R^2)$ weak-$\star$ for all $T>0$.
\end{thm}
Theorem~\ref{theo:1} is proved in Section~\ref{Sec:PreuveTheo1}. \par
In the case where $\mathcal S_0$ is a disk, the statement needs a slight modification.
\begin{thm} \label{theo:1disk}
Let $\mathcal S_0$ a disk in $\Omega$, with center $\zeta$ and center of mass $0$.
Let $\gamma\neq 0$, $(m_{1}, \mathcal J_{1}) $ in $(0,+\infty)^2$ and $p_0$ in $ \mathbb R^3$.
Let us consider the global solution $\overline{h}$ to \eqref{argh} and for every $\varepsilon$ in $(0,1]$, let $(q_\varepsilon ,T_\varepsilon)$ be as in Theorem~\ref{theo:1}.
Let $h_{c,\varepsilon}=h_{\varepsilon} + \varepsilon R(\vartheta_{\varepsilon}) \zeta$ the center of the disk.
Then, as $\varepsilon \rightarrow 0^{+}$, $T_\varepsilon \rightarrow +\infty$ 
and $h_{c,\varepsilon} \rightharpoonup \overline{h}$ in $W^{1,\infty}([0,T];\mathbb R^2)$ weak-$\star$ for all $T>0$.
\end{thm}
Theorem~\ref{theo:1disk} is proved in Subsection~\ref{Subsec:TCONHD}. 
It is straightforward to check that this involves the convergence of $h_{\varepsilon}$ to $\overline{h}$ in $L^{\infty}$ strong for all $T>0$. Actually, one even gets the convergence in $W^{\beta,\infty}(0,T)$ weak-$\star$ for all $T>0$ with $\beta:=\min\left(1,\frac{2}{\alpha}\right)$; see again Subsection~\ref{Subsec:TCONHD}. As we will see during the proofs, if either ${\mathcal S}_{0}$ is homogeneous (so that the center of the disk and the center of gravity coincide) or if $\alpha \leq 2$, Theorem~\ref{theo:1} is actually valid without change. 
%
%
%
%
%
%
%
%%%%%%%%%%%%%%%%%%%%%%%%%%%%%%%%%%%%%%%%%%%%%%%%%%%%%%%%%%%%%%%%%%%%%%%%%%%%%%%%%%%%%%%%%%%%%%%%%%%%%%%%%%%%%%%%%%%%%%%%%%%%%%%%
%
%
%
%
%
\section{Comments and organization of the paper}
\label{fewComments}
\subsection{Comments}

The limit systems obtained in Cases (i) and (ii) do not depend on the body shape nor on the value of $\alpha >0$.
Still the proof is simpler in the case where the body is a homogeneous disk. 
Indeed if  $\mathcal S_0$ is a disk, in both Cases (i)  and (ii), it follows directly from \eqref{EqRot}
that the rotation $\vartheta_\varepsilon $ satisfies, for any $\varepsilon $ in $(0,1)$, 
$\vartheta_\varepsilon (t) = t \omega_0$ as long as the solution exists.  \par
We notice that the convergences in Theorem~\ref{theo:3} and Theorem~\ref{theo:1} cannot be improved unless some compatibility condition holds at initial time.
One may however wonder if these convergences could be improved in the open interval $(0,T)$.
It seems that some strong oscillations in time show up when $\eps \rightarrow 0^{+}$ which prevent strong convergence from happening. 
We plan to study this phenomenon by a multi-scale approach of the solution to the ODE \eqref{ODE_intro} in a forthcoming work. 
Once again the case where the body is a homogeneous disk is likely to simplify the discussion. \par
Let us also underline that the convergence of the dynamics of the solid involves some convergence in the fluid. It is not difficult to check that for any $k \in \N$ we have the weak-$\star$ convergence in $W^{1,\infty}(0,T;C^{k}(K))$ for any $T < \overline{T}$ in Case~(i) and in $L^{\infty}(0,T;C^{k}(K))$ for any $T>0$ in Case (ii), of the fluid velocity $u_{\varepsilon}$ toward $\gamma u^{\Omega}(\cdot) + \frac{\gamma}{2\pi} \frac{(\cdot - \overline{h}(t))^{\perp}}{|\cdot - \overline{h}(t)|^{2}}$, for any compact set $K \subset \Omega$ not intersecting $\{ \overline{h}(t), t \in [0,T] \}$. \par
In Case~(i), one may also raise the question whether it is possible that $\liminf T_\varepsilon > \overline{T}$. This problem should be connected to the behaviour of the potentials and stream functions as the body approaches the boundary; see for instance \cite{BDTV}, \cite{CNS}, \cite{MNP} and \cite{Munnier:2015aa} and references therein for this question. \par
The analysis performed in this paper can be easily adapted in order to cover the case where the circulation $\gamma$ depends on $\eps$ under the form 
$\gamma^{\eps} = \eps^{\beta} \; \gamma^{1}$ with $\beta >0$ in Case (i) and  $\beta $ in $(0,1)$ in Case (ii). One obtains respectively at the limit the trivial equations 
$\overline{h}'' = 0$ and $\overline{h}^\prime = 0$. \par
Our analysis should hold as well in the case of several bodies moving in the full plane or in a multiply-connected domain, as long as there is no collision.
This will be tackled in a forthcoming work. \par
Another natural question is whether or not one may extend the results of Theorem \ref{theo:3} and Theorem \ref{theo:1} to rotational flows. 
These extensions are respectively tackled in \cite{GLS}  and  \cite{AvecCricri} in the case without external boundary. 
%
%

%%%%%%%%%%%%%%%%%%%%%%%%

\subsection{Organization of the paper} 
The paper is organized as follows. \par
%
%\ \par
%
In Section~\ref{Subsec-withoutOmega} we deal with the simpler case when there is no external boundary.  This case is well-known in the literature and has been in particular studied with complex analysis, using Blasius' lemma (see e.g. \cite{MP}). We shall use an alternative approach based on a lemma due to Lamb, see Lemma~\ref{blasius}, allowing one to exchange normal and tangential components in some trilinear integrals over the body boundary. The case without outer boundary is actually very important to tackle the general one. \par
Then Theorem~\ref{THEO-intro}, Proposition~\ref{energy} and Proposition~\ref{splitting-christo}, which are independent of $\varepsilon$, are proved in Section~\ref{tard}. \par
In Sections~\ref{Sec:PreuveTheo3} and \ref{Sec:PreuveTheo1}, respectively, we prove Theorem \ref{theo:3} concerning the limit of a massive particle (Case (i)) and Theorem \ref{theo:1} concerning the limit of a massless particle (Case (ii)). These proofs rely on some asymptotic normal forms \eqref{fnorm1} and \eqref{fnorm2}, respectively. These normal forms are the key point of the demonstration and allow us to establish some renormalized and modulated energy estimates and prove the passage to the limit. They are established in the last two sections. \par
In Section~\ref{sec-dev-stream} we establish some asymptotic expansions of stream and potential functions with respect to $\eps$.  
These expansions involve two scales corresponding respectively to variations over length $O(1)$ and $O(\eps)$ respectively on  $\partial \Omega$ and $\partial  \mathcal S_{\eps} (q)$. 
The profiles appearing in these expansions are obtained by successive corrections, considering alternatively at their respective scales
the body boundary from which the external boundary seems remote and the external boundary from which the body seems tiny, 
so that good approximations are given respectively by the case without external boundary and  without the body. \par
Then in Section~\ref{Sec:FormesNormales}, we prove the normal forms \eqref{fnorm1} and \eqref{fnorm2}.
To do so we plug the expansions obtained in Section~\ref{sec-dev-stream} into the expressions of the inertia matrix $M_{\varepsilon}(q)$, of the Christoffel symbol $\langle \Gamma_{\eps} (q),{p},{p}\rangle$ and of the force fields $E_{\varepsilon}$ and $p \times B_{\eps}$ and compute the leading terms of the resulting expansions. These expansions can themselves be plugged into the ODE \eqref{ODE_intro} of Theorem~\ref{THEO-intro}. In particular, thanks to Lamb's lemma  we will make appear in several terms of the expansions of $E_{\varepsilon}$ and $B_{\varepsilon}$ some coefficients of the added inertia of the solid corresponding to the case without external boundary. Strikingly this allows us to combine the subprincipal terms of the expansions of $E_{\varepsilon}$ and $B_{\varepsilon}$ with the leading term of the expansion of $\Gamma_{\varepsilon}$ thanks to a modulation of the variable; see Lemma~\ref{lemma-grav}. This fact is essential in the proof.

%
%
%
%
%
%%%%%%%%%%%%%%%%%%%%%%%%%%%%%%%%%%%%%%%%%%%%%%%%%%%%%%%%%%%%%%%%%%%%%%%%%%%%%%%%%%%%%%%%%%%%%%%%%%%%%%%%%%%%%%%%%%%%%%%%%%%%%%%%%%%%%%
%
%
%
%
%
%
\section{Case without external boundary}
\label{Subsec-withoutOmega}
In this section, we consider a simplified version of the problem that we are interested in. The simplification consists in assuming that the domain occupied by the 
fluid-solid system is the whole plane, i.e. $\Omega =\R^2$, the fluid being at rest at infinity. We aim to give (in this simplified unbounded configuration) the counterparts of Theorems~\ref{THEO-intro}, \ref{theo:3} and \ref{theo:1}.
The rephrasing of the equations driving the dynamics as an ODE (that is 
a result similar to Theorem~\ref{THEO-intro}), has been known since the investigations of Blasius, Kutta, Joukowski, Chaplygin and Sedov.
Their analysis, relying on a complex-analytic approach, was then revisited following another approach which seems to date back to  Lamb.
Since we will elaborate on the latter in order to deal with the bounded case, we will first establish the counterpart of Theorem~\ref{THEO-intro} relying on Lamb's analysis.
 We will deduce from this ODE an energy conservation providing an analogous to Proposition \ref{energy}. 
Then we will investigate the passage to the limit of the dynamics when the size of the solid goes to $0$ in both Cases (i) and (ii). Hence we will establish the counterparts of Theorem~\ref{theo:3} and Theorem~\ref{theo:1}.  

\begin{rem}
The purpose of this section is not only to provide a ``warm-up'' for the much more involved ``bounded'' configuration. It turns out that several objects that will come up in the analysis 
will be also of central importance in the sequel.  
\end{rem}
\subsection{Recasting of the system as an ODE}
\label{Subsec-recasting}
In the case where $\Omega=\mathbb R^2$, 
the fluid-solid system is governed by the following set of coupled equations, quite similar to System \eqref{SYS_full_system}:
\begin{subequations}
\label{SYS_full_system-without}
\begin{gather} 
\frac{\partial u}{\partial t}+(u\cdot\nabla)u +\nabla \Pi=0 
\quad\text{ and } \quad
\dive u=0\text{ in }\R^2 \setminus {\mathcal S}(t),\\
(\mathcal J\vartheta'' , mh'' ) = 
\Big(\int_{\partial\mathcal S(t)}(x-h(t))^\perp\cdot \Pi n\, {\rm d}s , 
\int_{\partial\mathcal S(t)}\Pi n\, {\rm d}s \Big),  \\
u\cdot n=\big( \vartheta' (\cdot-h)^\perp+ h' \big) \cdot n\text{ on }\partial\mathcal S(t)   \text{ and } 
u(x) \rightarrow 0 \text{ when } \vert x\vert  \rightarrow +\infty . 
\end{gather}
\end{subequations}
%
%  ,\\

%
%
We still consider irrotational solutions:
\begin{equation} \label{irr-without}
\curl u=0\text{ in }\R^2 \setminus {\mathcal S}(t).
\end{equation}
Again, this depends solely on the fluid part of the initial data
\begin{equation*}
 u \vert_{t=0}=u_0\text{ in }   \R^2 \setminus\     \mathcal S_0  \text{ and }  (\vartheta,h)(0)=(0,0),\  (\vartheta',h')(0)=(\omega_0,\ell_0).
\end{equation*}
To state the aforementioned reformulation, we introduce the 
Kirchhoff potentials, the inertia matrices, the Christoffel symbols and the force term corresponding to this simplified case. \par
\ \par
\noindent
{\it Kirchhoff potential.}
Let us first denote by $\varphi^{\Ext}_{j}$, for $j=1,2,3$, the Kirchhoff potentials in $\R^{2} \setminus \mathcal S_0$ which are the functions  that satisfy the following Neumann problem:
\begin{subequations} \label{Kir-exte}
\begin{alignat}{2}
& \Delta \varphi^{\Ext}_{j}  
=0  \quad & \text{ in } \R^{2} \setminus \mathcal S_0, \\
 \label{Kir-ext-boundary}
& \frac{\partial \varphi^{\Ext}_{j}}{\partial n} 
= 
\left\{ \begin{array}{l}
x^{\perp} \cdot n  \ \text{ for } j= 1 \\
e_{j-1} \cdot n , \text{ for } j=2,3 \\  
\end{array} \right.
\quad & \text{ on }\partial\mathcal S_0,  \\
& \nabla \varphi^{\Ext}_{j}(x)  \rightarrow 0 & \text{ at infinity. }
\end{alignat}
\index{BGrecTF3@$\varphi^{\Ext}_{j}$ ($j=1,2,3$) :  Kirchhoff's potentials when $ \Omega =  \R^{2}$}
\end{subequations}
We also define  
\begin{equation} \label{Kir-gras-exte}
\boldsymbol\varphi^\Ext
:=(\varphi^{\Ext}_{1} , \, \varphi^{\Ext}_{2},\, \varphi^{\Ext}_{3})^{t} .
\end{equation}
\index{BGrecTF4@$\boldsymbol\varphi^\Ext$: vector  containing the three Kirchhoff potentials when $ \Omega =  \R^{2}$}
\ \par
\noindent
{\it Stream function for the circulation term. } In the same spirit as \eqref{def_stream}, we first introduce the function $\psi^{\Ext}_{{\it -1}}$ as the solution to 
\begin{subequations} \label{model_stream-1}
\begin{alignat}{3}
\label{model_stream-1-a}
-\Delta \psi^{\Ext}_{{\it -1}}  &=0& &\text{in }\R^{2} \setminus \mathcal S_0 ,\\
\label{model_stream-1-b}
\psi^{\Ext}_{{\it -1}} &=C^{\Ext} &\quad&\text{on }\partial\mathcal S_0 ,
\\ \label{model_stream-1-c}
\psi^{\Ext}_{{\it -1}} &= O(\ln |x|) &\quad&\text{at infinity,}
\end{alignat}
where the constant $C^{\Ext}$
\index{AC2@$C^{\Ext}$: value of $\psi^{\Ext}_{{\it -1}}$ on $\partial\mathcal S_0$}
is such that:
\begin{equation} \label{53d-1}
\int_{\partial\mathcal S_0}\frac{\partial\psi^{\Ext}_{{\it -1}}}{\partial n}\, {\rm d}s = -1 .
\end{equation}
\end{subequations}
The existence and uniqueness of  $\psi^{\Ext}_{{\it -1}}$ will be recalled below in Proposition~\ref{coro-cap}; one can still identify $\frac{\partial\psi^{\Ext}_{{\it -1}}}{\partial n}$ as being the equilibrium measure of $\partial\mathcal S_0$. \par
\ \par
\noindent
{\it Change of frame and decomposition of the fluid velocity.} The vector field $\underline{u}$ defined from the fluid velocity $u$ by
\begin{equation} \label{chgt}
\underline{u}(t,x) :=R(\vartheta(t))^t \, u(t,R(\vartheta(t))x+h(t)),
\end{equation}
satisfies the following  div-curl type system in the doubly-connected domain $\mathcal F_0$:
\begin{subequations}
\begin{gather} 
\label{divcurlv-a}
\dive \underline{u}  =0     \text{ and }    \curl \underline{u} = 0      \text{ in }\R^{2} \setminus \mathcal S_0 , \\
\label{divcurlv-c}
\underline{u}\cdot n = \left(\underline{\ell} + \vartheta' x^\perp\right)\cdot n \quad\text{ on }  \partial \mathcal{S}_0,  \\
\label{divcurlv-d}
\underline{u}  \longrightarrow 0\quad\text{at infinity,} \\
\label{divcurlv-e}
\int_{\partial\mathcal S_0}  \underline{u} \cdot\tau\, {\rm d}s =  \gamma ,
\end{gather}
\end{subequations}
where $\gamma := \int_{\partial\mathcal S_0} u_{0} \cdot\tau\, {\rm d}s$
and  $\underline{\ell}(t) :=R(\vartheta(t))^t \, h' (t)$.

When the right hand sides of these equations are given
the auxiliary velocity field $\underline{u}$ is determined in a unique way and takes the form:
\begin{subequations}
\label{decomp_underline}
\begin{equation} 
\underline{u}  = \underline{u}_{(\vartheta',\underline{\ell})}  + \underline{u}_{\gamma} 
\end{equation}
where, for $\underline{\ell}=(\underline{\ell}_{1},\underline{\ell}_{2})$, 
\begin{equation} 
\underline{u}_{(\vartheta',\underline{\ell})}   :=  \nabla \big(   \varphi^{\Ext}_1  \vartheta'     +   \varphi^{\Ext}_2  \underline{\ell}_1  +   \varphi^{\Ext}_3  \underline{\ell}_2  \big) 
\quad\text{ and }\quad \underline{u}_{\gamma}   :=\gamma \nabla^\perp\psi^{\Ext}_{\mathit{-1}}  .
\end{equation}
\end{subequations}

Below we define $ M^\Ext_{a, \vartheta} $, $\Gamma^{\Ext}_{\vartheta}$ and ${F}^{\Ext}_{\vartheta}$ such that 
\begin{gather*}
 M_{g}   +  M^\Ext_{a, \vartheta} = M [\mathcal S_0 , m^1 ,\mathcal{J}^1 ,{ \R^2}] (q)  ,
\ \Gamma^{\Ext}_{\vartheta} =  \Gamma [\mathcal S_0 ,{ \R^2}] (q)
 \\  \text{  and  } 
{F}^{\Ext}_{\vartheta} (p) = F [\mathcal S_0 ,\gamma,{ \R^2}]  (q,p)  .
\end{gather*}
\begin{rem} \label{RemNotTheta}
In particular, since the dependence on  $q$ of $M$, $ \Gamma$ and $F$ reduces to a dependence on the rotation $R(\vartheta)$ only in the case without external boundary, 
this dependence on $\vartheta$  is mentioned through an index.
\end{rem}
\ \par
\noindent
{\it Inertia matrices.}
We  define the  added mass  matrix 
\begin{equation} \label{def-Maa}
M^{\Ext}_{a}
:= \int_{\partial \mathcal S_0} \boldsymbol\varphi^{\Ext} \otimes\frac{\partial\boldsymbol\varphi^\Ext}{\partial n} \, {\rm d}s
= \Big( \int_{ \R^{2} \setminus \mathcal S_0 } \nabla \varphi^{\Ext}_{i} \cdot  \nabla \varphi^{\Ext}_{j} {\rm d}x 
  \Big)_{1  \leqslant i,j  \leqslant  3} .
\index{AM5@$M^{\Ext}_{a}$: added inertia of the solid when $ \Omega =  \R^{2}$}
\end{equation}
This matrix  is symmetric positive-semidefinite and depends only on $\mathcal S_0$.
Actually it is positive definite if and only if  $\mathcal S_0$ is not a disk. 
Moreover, when $\mathcal S_0$ is a disk centered at the center of mass, this matrix reads
$M^{\Ext}_{a}  =: \text{diag }(0,m_{a}^{\Ext} ,m_{a}^{\Ext} )$ with  $m_{a}^{\Ext} >0$. 

Then we can introduce the mass matrix $M^\Ext_{a, \vartheta} $ taking the rotation into account by
\begin{equation} \label{g+a_ext}
M^\Ext_{a, \vartheta}  := \mathcal{R}(\vartheta) M^{\Ext}_{a} \mathcal{R}(\vartheta)^t .
\index{AM6@$M^{\Ext}_{{a},\vartheta}$: conjugate matrix of $M_{a}^{\Ext}$ by the rotation matrix of angle $\vartheta$}
\end{equation}
Above we used the notation
\begin{equation}
\label{anum}
\mathcal{R}(\vartheta) :=
\begin{pmatrix}
1 & 0 
\\ 0 & {R}(\vartheta) 
\end{pmatrix}   \in{\rm SO}(3).
\index{AR3@$\mathcal{R}(\vartheta)$: $3\times3$  rotation matrix of angle $\vartheta$}
\end{equation}
{\it Christoffel symbols.} Let $\Gamma^{\Ext}_{\vartheta}$ the bilinear symmetric mapping  defined 
 for all ${p} = (\omega , {\ell} ) \in \R \times \R^{2}$ by
\begin{equation} \label{Christo-exter}
\langle \Gamma^{\Ext}_{\vartheta} ,{p},{p}\rangle  :=
- \begin{pmatrix} 0 \\
P^\Ext_{a, \vartheta}
\end{pmatrix}
 \times {p}
-  \omega M^\Ext_{a, \vartheta}
\begin{pmatrix}
 0 \\
 {\ell}^\perp  \end{pmatrix},
\index{BGrecAG5@$\Gamma^{\Ext}_{\vartheta}$: Christoffel symbols when $\Omega= \R^{2}$}
\end{equation}
where $P^\Ext_{a, \vartheta}$  \index{AP3@$P^\Ext_{a, \vartheta} $:  translation impulses when $\Omega= \R^{2}$}
denotes the last two coordinates of the impulse  $M^\Ext_{a, \vartheta} \, {p} $. 
We associate with $\Gamma^{\Ext}_{\vartheta}$ the coefficients $(\Gamma^{\Ext}_{\vartheta})^k_{i,j}$ such that 
for any  $\vartheta $ in $\R$, for any $p$ in $\R^3$,  
\begin{equation*}
\langle \Gamma^{\Ext}_{\vartheta} ,p,p\rangle :=\left(\sum_{1\leq i,j\leq 3} (\Gamma^{\Ext}_{\vartheta})^k_{i,j} p_i p_j  \right)_{1\leq k\leq 3} .
\end{equation*}
Then a tedious computation reveals that for every $i,j,k \in\{1,2,3\}$, 
\begin{equation*} 
(\Gamma^{\Ext}_{\vartheta})^k_{i,j}  =  \frac12
\Big(  (M^\Ext_{a, \vartheta})_{k,j}^{i} + (M^\Ext_{a, \vartheta})_{k,i}^{j} - (M^\Ext_{a, \vartheta})_{i,j}^{k}  \Big)  ,
\end{equation*}
where $(M^\Ext_{a, \vartheta})_{i,j}^{k} $ denotes the partial derivative with respect to $ q_{k}$ of the entry of indexes $(i,j)$ of the matrix $M^\Ext_{a, \vartheta}$.
Thus the coefficients $(\Gamma^{\Ext}_{\vartheta})^k_{i,j}$, for $i,j,k \in\{1,2,3\}$, are the Christoffel symbols of the first kind associated with   the  matrix $M^\Ext_{a, \vartheta}$.
There is naturally no counterpart of the term $\Gamma^{\partial \Omega}$ here.
\par
\ \par
\noindent
{\it Force term.} Let us define  the geometric constant 
\begin{equation} \label{def-zeta}
{\zeta} := - \int_{\partial\mathcal S_0}x \frac{\partial\psi^{\Ext}_{{\it -1}}}{\partial n}\, {\rm d}s  
\index{BGrecEZ1@$\zeta$: conformal center of $\mathcal S_0$}
\end{equation}
which is a vector of  $\R^2$, depending only on $\mathcal S_0$,  usually referred  to as the conformal center  of $\mathcal S_0$. 
Remark that if ${\mathcal S}_{0}$ is a disk, $\psi^{\Ext}_{{\it -1}}(x) = \frac{1}{2\pi} \ln |x-h_{c}(0)|$ where $h_{c}(0)$ is the center of the disk (recalling that the initial position of the center of mass is $h(0)=0$); it follows that $\zeta=h_{c}(0)$ and the definition of $\zeta$ in \eqref{def-zeta} is coherent with the notation in Theorem~\ref{THEO-intro-disk}. \par
Next we define the force term, for all ${p} = (\omega , {\ell} ) \in \R^{3}$, as being:
\begin{equation} \label{force-ext}
F^{\Ext}_{\vartheta} ({p}) :=
 \gamma 
\begin{pmatrix}
\zeta_\vartheta \cdot  {\ell} \\
 {\ell}^\perp - \omega \zeta_\vartheta
\end{pmatrix},
\index{AF3@$F^{\Ext}_{\vartheta} $: force term  when $\Omega= \R^{2}$}
\end{equation}
where 
\begin{equation} \label{def-zeta_vartheta}
\zeta_\vartheta := R(\vartheta)  {\zeta} .
\index{BGrecEZ2@$\zeta_\vartheta$: conformal center of $\mathcal S_0$ rotated of $\vartheta$}
\end{equation}
An important feature of the force vector field $(\vartheta, {p}) \mapsto {F}^{\Ext}_{\vartheta} ({p}) $ is that it is gyroscopic, in the sense of the following definition; see for instance \cite[p. 428]{arnold}. Note that we use here again the convention of Remark \ref{RemNotTheta} to put $\vartheta$ as an index.
\begin{defn} \label{def-gyro}
We say that a vector field $F $ in $C^\infty (\R \times  \R^3 ;  \R^3)$ is gyroscopic if for any $(\vartheta,p) $ in  $\R \times  \R^3$, 
$p \cdot {F}  (\vartheta,p) = 0 $. 
\end{defn}
Indeed, for any $(\vartheta ,p)$ in  $\R \times  \R^3$, the force ${F}^{\Ext}_{\vartheta} (p) $ can be written as 
\begin{equation} \label{def-Bext}
{F}^{\Ext}_{\vartheta} (p) = \gamma p \times B^{\Ext}_{\vartheta}  
\text{ with }  
B^{\Ext}_{\vartheta} :=
\begin{pmatrix}
- 1 \\  \zeta_{\vartheta}^{\perp}
\end{pmatrix} .
\end{equation}
%
%
%
%%%%%%%%%%%%%%%%%%%%%%%% TH 1 EXT %%%%%%%%%%%%%%%%%%%%%%%%%%%%%%%%%%%%%%%%%%%%%
%
%
The next result is a reformulation of System~\eqref{SYS_full_system-without} as an ordinary differential equation and the counterpart of Theorem~\ref{THEO-intro} in the case without external boundary. 
\begin{thm} \label{THEO-intro-without}
For smooth solutions, System~\eqref{SYS_full_system-without}-\eqref{irr-without} can be recast, up to the first collision, as the second order ODE 
\begin{equation}  \label{ODE_ext}
(M_{g}   +  M^\Ext_{a, \vartheta} ) \, q''   +  \langle  \Gamma^{\Ext}_{\vartheta} ,q',q'\rangle
= {F}^{\Ext}_{\vartheta} (q').
\end{equation}
\end{thm}
Let us underline that it is understood that we associate with the body equation given by  \eqref{ODE_ext} the fluid velocity $u$ such that 
the vector field $\underline{u}$ given by \eqref{chgt} satisfies \eqref{decomp_underline}.

The proof of Theorem  \ref{THEO-intro-without} is postponed to Subsection \ref{Sub:samsoir}.

In addition we have  the following energy conservation property.
\begin{prop} \label{Prop-conserv-no}
If $q$ is a smooth solution  to \eqref{ODE_ext} then during its lifetime the quantity $(M_{g}   +  M^\Ext_{a, \vartheta} ) \, q' \cdot q' $ is conserved. 
\end{prop}
The quantity above corresponds to twice the sum of the kinetic energy of the solid associated with the genuine inertia and of the one associated with the added inertia.
\begin{proof}
For any $\vartheta $ in $\R$, for any $p$ in $\R^3$,  we define the matrix 
\begin{equation*} 
S^{\Ext}_{\vartheta} (p) := \left(\sum_{1\leq i\leq 3} (\Gamma^{\Ext}_{\vartheta} )^k_{i,j} p_i  \right)_{1\leq k,j\leq 3}  \text{ so that  }
\langle \Gamma^{\Ext}_{\vartheta} , p, p \rangle = S^{\Ext}_{\vartheta} (p)  p .
\end{equation*}
Then, an explicit computation (another method, more theoretical, is given in the proof of Proposition~\ref{energy}; see Lemma~\ref{antis}) proves that for any  $\vartheta $ in $\R$, for any $p$ in $\R^3$,  
\begin{equation}
\label{EZ1}
\frac{1}{2}    \frac{\partial  M^\Ext_{a, \vartheta}}{\partial q}    (\vartheta) \cdot  {{p}}
- S^{\Ext}_{\vartheta} ( {{p}})   \text{ is skew-symmetric.}
\end{equation}
Then Proposition  \ref{Prop-conserv-no} follows from  \eqref{EZ1} and from the fact that 
the force ${F}^{\Ext}_{\vartheta} (p)$ is gyroscopic in the sense of  Definition~\ref{def-gyro}.
\end{proof}

As a consequence, given $\gamma$ and $\mathcal S_0$ and some Cauchy data, there exists a unique global smooth solution to \eqref{ODE_ext}.

\subsection{Lamb's lemma}
To prove Theorem~\ref{THEO-intro-without}, we start with recalling a technical result borrowed from \cite[Article 134a. (3) and (7)]{Lamb} and which is a cornerstone of our analysis. 
We recall that $\xi_{j}$ and $K_{j}$, for  $j=1,2,3$,  were defined in \eqref{def-xi-j} and \eqref{Def-Kj} respectively.
\begin{lem} \label{blasius}
For any pair of vector fields $u$, $v$ in $C^\infty (\overline{\R^2 \setminus \mathcal S_0} ; \R^2 )$ satisfying 
\begin{itemize}
\item $\dive u = \dive v = \curl u  = \curl v = 0$, 
\item $u(x) = O(1/|x|)$ and $v(x) = O(1/|x|)$ as $|x| \rightarrow + \infty$,
\end{itemize}
 for any $j=1,2,3$, 
\begin{equation} \label{blasius-formu}
\int_{\partial \mathcal S_0} (u\cdot v) K_{j} (0,\cdot ) \, {\rm d}s 
= \int_{\partial  \mathcal S_0} \xi_{j} (0,\cdot )  \cdot \Big(  (u \cdot n) v +  (v \cdot n)  u  \Big) \, {\rm d}s .
\end{equation}
\end{lem}
\begin{proof}
Let us start with the case where $j=2$ or $3$. Then 
\begin{equation} \label{Bla1}
\int_{\partial  \mathcal S_0} (u\cdot v) K_{j} (0,\cdot ) \, {\rm d}s  
= \int_{\R^2 \setminus\mathcal S_0} \dive \big( (u\cdot v) \xi_{j} (0,\cdot )\big)    \, {\rm d}x ,
\end{equation}
since  $u(x) = O(1/|x|)$ and $v(x) = O(1/|x|)$ when $|x| \rightarrow + \infty$.
Therefore
\begin{align} \nonumber
\int_{\partial  \mathcal S_0} (u\cdot v) K_{j} (0,\cdot ) \, {\rm d}s  
&= \int_{\R^2 \setminus\mathcal S_0} \xi_{j} (0,\cdot )\cdot \nabla (u\cdot v)     \, {\rm d}x 
\\&= \int_{\R^2 \setminus\mathcal S_0} \xi_{j} (0,\cdot ) \cdot (  u \cdot \nabla  v + v\cdot \nabla u )  \, {\rm d}x ,\label{Bla2}
\end{align}
since  $ \curl u  = \curl v = 0$.
Now, integrating by parts, using that $\dive u = \dive v = 0$ and once again that $u(x) = O(1/|x|)$ and $v(x) = O(1/|x|)$ as $|x| \rightarrow + \infty$, we obtain \eqref{blasius-formu} when $j=2$ or $3$. \par
We now tackle the case where $j=1$.
We follow the same lines as above, with two precisions.
First we observe that there is no contribution at infinity in \eqref{Bla1} and \eqref{Bla2} when $j=1$ as well.
Indeed $ \xi_{1}$ and the normal to a centered circle are orthogonal. 
Moreover there is no additional distributed term coming from the integration by parts in \eqref{Bla2} when $j=1$  since
\begin{equation*}
\int_{\R^2 \setminus\mathcal S_0}
v \cdot (  u \cdot \nabla_x  \xi_{j} (0,\cdot ) )+  u \cdot (v\cdot \nabla_x  \xi_{j} (0,\cdot ) )     \, {\rm d}x
= \int_{\R^2 \setminus\mathcal S_0} ( v \cdot   u^\perp  +  u \cdot v^\perp )   \, {\rm d}x = 0.
\end{equation*}
The proof is then complete.
\qed 
\end{proof}
\subsection{Proof of Theorem~\ref{THEO-intro-without}}
\label{Sub:samsoir}
To transfer the equations in the body frame we recast the equations in terms of 
the vector field $\underline{u}$ (defined from the fluid velocity $u$ by
\eqref{chgt}), of $\underline{\ell}(t)=R(\vartheta(t))^t \, h' (t)$ and of the auxiliary pressure $\underline{\Pi}$,
defined from the fluid pressure $\Pi$ by $\underline{\Pi} (t,x)=\Pi(t,R(\vartheta(t))x+h(t))$.
Equations~\eqref{SYS_full_system-without} become
\begin{subequations}
\begin{gather}
\label{Euler11}
\displaystyle \frac{\partial \underline{u}}{\partial t}
+ \left[(\underline{u}-\underline{\ell}-\vartheta' x^\perp)\cdot\nabla\right]\underline{u} 
+ \vartheta' \underline{u}^\perp+\nabla \underline{\Pi} =0   \text{ for }  x\in \R^2 \setminus\mathcal S_0 \\
\dive \underline{u} = 0  \text{ for }  x\in \R^2 \setminus\mathcal S_0 , \\
\label{Euler13}
\underline{u}\cdot n = \left(\underline{\ell} +\vartheta' x^\perp\right)\cdot n  \text{ for }  x\in \partial \mathcal{S}_0
\ \text{ and } \ \underline{u}(x) \rightarrow 0 \text{ as } |x| \rightarrow +\infty, \\
\label{Solide11}
(\mathcal{J} \vartheta''  , m \underline{\ell}' ) =
\left(\int_{\partial \mathcal{S}_0} x^\perp \cdot \underline{\Pi} n \ ds ,  
\int_{\partial \mathcal{S}_0} \underline{\Pi} n \ ds-m \vartheta' \underline{\ell}^\perp \right) . 
\end{gather}
\end{subequations}
Let us consider the right hand side of \eqref{Solide11}.
Using an integration by parts (which we can justify by the decay at infinity of the solution, see \cite[p. 520]{GLS}), we obtain 
\begin{equation*}
\left( \int_{ \partial \mathcal{S}_0} \underline{\Pi} x^{\perp} \cdot n \,ds , \ \int_{ \partial \mathcal{S}_0} \underline{\Pi} n \, {\rm d}s
 \right) = 
\left( \int_{ \R^2 \setminus\mathcal S_0} \nabla \underline{\Pi} \cdot \nabla \varphi^{\Ext}_{i} \, dx \right)_{i=1,2,3} .
\end{equation*}
Moreover  since  the vector field $\underline{u} $ is irrotational we infer from the equation \eqref{Euler11} that 
\begin{gather*}
 \nabla  \underline{\Pi} = - \frac{\partial \underline{u}}{\partial t} - \nabla {Q}   ,   \text{ with }
{Q} :=  \frac{1}{2} |\underline{u}|^2 - (\underline{\ell} + \vartheta' x^\perp)\cdot \underline{u}.
\end{gather*}
Thus the solid equations \eqref{Solide11} now read 
\begin{multline*}
M_{g} \,  \underline{p} ' 
 + m \vartheta' \begin{pmatrix} 0 \\ \ell^{\perp} \end{pmatrix}
\\ = - \left(   \int_{ \R^2 \setminus\mathcal S_0}   \frac{\partial \underline{u}}{\partial t}      \cdot \nabla \varphi^{\Ext}_{i} \, dx   \right)_{i=1,2,3} 
- \left(   \int_{ \R^2 \setminus\mathcal S_0}   \nabla {Q}       \cdot \nabla \varphi^{\Ext}_{i} \, dx   \right)_{i=1,2,3} .
\end{multline*}
Integrating by parts, taking into account the  boundary condition \eqref{Euler13},  and 
recalling that the  added mass  matrix $M^{\Ext}_{a}$ is defined  in \eqref{def-Maa} 
we observe that 
the contribution of $\frac{\partial \underline{u}}{\partial t}$ is
\begin{equation*}
\left(   \int_{ \R^2 \setminus\mathcal S_0}   \frac{\partial \underline{u}}{\partial t}      \cdot \nabla \varphi^{\Ext}_{i} \, dx   \right)_{i=1,2,3} 
 = M_{a}^{\Ext} \, \underline{p}', 
\text{ where } 
\underline{p} := \begin{pmatrix} \vartheta'   \\ \underline{\ell} \end{pmatrix} .
\end{equation*}
By another integration by parts we also have, for $i=1,2,3$, that 
\begin{equation*}
  \int_{ \R^2 \setminus\mathcal S_0}   \nabla {Q}       \cdot \nabla \varphi^{\Ext}_{i} \, dx  
=   \frac{1}{2} \int_{\partial  \mathcal{S}_{0}  } |{\underline{u}}|^2  K_i \, {\rm d}s -  \int_{\partial  \mathcal{S}_{0}  } (\underline{\ell} + \vartheta' x^\perp)\cdot {\underline{u}}  K_i \, {\rm d}s , 
\end{equation*}
where we simplified the notation by writing $K_i$ instead of $K_{i}(0,\cdot)$ (defined in \eqref{Def-Kj}), which corresponds to the initial position of the solid, that is $q=0$. In the same way, we will write $\xi_i$ for the vector fields $\xi_{i}(0,\cdot)$ (defined in \eqref{def-xi-j}) in the computations below.
Thus the solid equations \eqref{Solide11} can be rewritten in the form
\begin{multline*}
(M_{g} + M_{a}^{\Ext} ) \underline{p} ' 
 + m \vartheta' \begin{pmatrix} 0 \\ \ell^{\perp} \end{pmatrix}
\\= - \left( \frac{1}{2} \int_{\partial  \mathcal{S}_{0}  } |{\underline{u}}|^2  K_i \, {\rm d}s
- \int_{\partial  \mathcal{S}_{0}  } (\underline{\ell} + \vartheta' x^\perp)\cdot {\underline{u}}  K_i \, {\rm d}s\right)_{i=1,2,3} .
\end{multline*}
We compute the first term in the right hand side  by using Lamb's lemma and the boundary conditions. 
For $i =1,2,3$, 
\begin{multline*}
 \frac{1}{2} \int_{\partial  \mathcal{S}_{0}  } |{\underline{u}}|^2  K_i \, {\rm d}s  = 
\int_{\partial  \mathcal{S}_{0}  } \big( (\underline{\ell} + \vartheta' x^\perp)\cdot  n \big)  \big( \underline{u} \cdot n \big)  K_{i}  \, {\rm d}s \\
+ \int_{\partial  \mathcal{S}_{0}  } \big( (\underline{\ell} + \vartheta' x^\perp)\cdot  n \big)  \big( \underline{u} \cdot \tau  \big) \big(  \xi_{i} \cdot  \tau \big)  \, {\rm d}s .
\end{multline*}
It follows that the solid equations \eqref{Solide11}  can be recast as follows: 
\begin{multline} \label{roof1}
(M_{g} + M_{a}^{\Ext} ) \underline{p}'
+ m \vartheta' \begin{pmatrix} 0 \\ \underline{\ell}^{\perp} \end{pmatrix} \\
= - \Big( \sum_{k=1}^3  \, \underline{p}_{k}  \,  \int_{\partial  \mathcal{S}_{0}  } \,   \big( \underline{u}\cdot \tau  \big) 
 [ \big(  \xi_{i} \cdot  \tau \big) K_{k}  -  \big(  \xi_{k} \cdot  \tau \big)   K_{i}   ] \, {\rm d}s \Big)_{i=1,2,3} 
 \end{multline}

We observe that the brackets above are either vanishing (for $i=k$) or given by the following identities:
\begin{gather}
\nonumber
   \big(  \xi_{2} \cdot  \tau \big) K_{3}  -  \big(  \xi_{3} \cdot  \tau \big)   K_{2}   
= 1 , \quad 
 \big(  \xi_{2} \cdot  \tau \big)  K_{1} 
   -  \big(  \xi_{1} \cdot  \tau \big)  K_2
 =  x^\perp \,  \cdot  e_{2} ,
 \\  \label{roof2} \text{ and }
\big(  \xi_{3} \cdot  \tau \big) K_{1}  -  \big(  \xi_{1} \cdot  \tau \big)   K_{3} 
= -  x^\perp \,  \cdot  e_{1} .
\end{gather}
Thanks to these equalities we compute the previous integrals in terms of the entries of the matrix $M_{a}^{\Ext}$.
Precisely we decompose $M_{a}^{\Ext}$ into
\begin{equation} \label{def-Ma}
M_{a}^{\Ext}  =:
\begin{pmatrix}
m^{\Ext}_\#   & (\mu^{\Ext})^t \\
\mu^{\Ext} & M^{\Ext}_{\flat}   
\end{pmatrix} ,
\end{equation}
where $M^\Ext_\flat$ is a symmetric $2 \times 2$ matrix. Then we have the following result.
\begin{lem}
There holds: 
\begin{equation} \label{azaz}
\int_{\partial  \mathcal{S}_{0}  } \,   \big( \underline{u}\cdot \tau  \big) x^\perp \, {\rm d}s 
= \gamma \zeta^{\perp} - \vartheta'  \mu - M^\Ext_{\flat} \, \underline{\ell} .
\end{equation}
\end{lem}
\begin{proof}
We first use \eqref{decomp_underline}  and \eqref{def-zeta} to get 
\begin{equation*} 
\int_{\partial  \mathcal{S}_{0}  } \,   \big( \underline{u}\cdot \tau  \big) x^\perp \, {\rm d}s 
=  \gamma \zeta^{\perp}
+ \sum_{j=1}^3 \,  \underline{p}_{j}  \int_{\partial  \mathcal{S}_{0}  } \, \big(   \nabla \varphi^{\Ext}_{j}  \cdot  \tau \big)  x^\perp\, {\rm d}s.
\end{equation*}
Therefore, it only remains to observe that an integration by parts yields: 
$$
\int_{\partial  \mathcal{S}_{0}  }  \big(   \nabla \varphi^{\Ext}_{j}  \cdot  \tau \big)  x^\perp\, {\rm d}s 
= - \int_{\partial  \mathcal{S}_{0}  } \varphi^{\Ext}_{j}  n \, {\rm d}s  
= - \big( (M_{a}^{\Ext})_{i,j} \big)_{i=2,3} .
$$
\qed
\end{proof} 

Combining now  \eqref{divcurlv-e}, \eqref{roof1}, \eqref{roof2} and \eqref{azaz}
we end up with the identity:
\begin{equation*}
\big[  M_{g} +  M^{\Ext}_{a} \big] \underline{p}' 
+ m \vartheta' \begin{pmatrix} 0 \\ \underline{\ell}^{\perp} \end{pmatrix}
+ \begin{pmatrix} \underline{\ell}^{\perp} \cdot M^\Ext_{\flat} \underline{\ell}  \\ \vartheta' (M^\Ext_{\flat} \underline{\ell})^{\perp}\end{pmatrix} 
+ \vartheta' \underline{p} \times \begin{pmatrix} 0  \\ \mu  \end{pmatrix}
 =  \gamma \underline{p} \times  \begin{pmatrix} -1  \\ \zeta^{\perp}  \end{pmatrix} .
\end{equation*}
Going back to the original frame  we arrive at \eqref{ODE_ext}. This concludes the proof of Theorem~\ref{THEO-intro-without}. \par
Now in the remaining subsections we will make use of the reformulation of the system established in 
Theorem~\ref{THEO-intro-without} to  investigate the passage to the limit of the dynamics when the size of the solid goes to $0^{+}$ in both Cases (i) and (ii).
We will establish the counterparts of Theorem~\ref{theo:3} and Theorem~\ref{theo:1}.
\subsection{Scaling with respect to $\varepsilon$}

In this subsection we investigate the scaling of several objects with respect to $\varepsilon$ in the absence of outer boundary. 
We will treat at once both Cases (i) and (ii). Recall that the scaling of the inertia is given by the relations \eqref{mass-inertie2} and \eqref{mass-inertie-deception}, that is to say, 
Case (i) corresponds to $\alpha=0$ in \eqref{mass-inertie-deception} while Case (ii) occurs when $\alpha>0$.
We consider for $\varepsilon $ in $(0,1]$, the scaled solid occupying initially the domain $\mathcal S_{0,\varepsilon}$ given by \eqref{piti}.
We recall that $\gamma$ and the Cauchy data $p_0 = (\omega_0 , \ell_0 )$ are supposed to be independent of $\varepsilon$ whereas $q_0 =0$. 
We denote by $q_\varepsilon :=(\vartheta_\varepsilon, h_\varepsilon) $ in $\mathbb R \times \mathbb R^2 $ the solution to the ODE obtained from \eqref{ODE_ext} by rescaling the coefficients following the relations \eqref{piti}, \eqref{mass-inertie2} and \eqref{mass-inertie-deception}. \par
\ \par
\noindent
{\it Genuine inertia matrix and kinetic energy.}  Under the relation \eqref{mass-inertie2}  and \eqref{mass-inertie-deception}, the matrix of genuine inertia reads
\begin{equation} \label{Eq:Mgeps}
M_{g,\varepsilon} = \varepsilon^{\alpha} I_\eps  M_g   I_\eps, 
\end{equation}
where we define
\begin{equation} \label{Eq:Ieps}
I_\eps :={\rm diag}(\varepsilon,1,1) \text{ and } M_g :={\rm diag}(\mathcal J_1,m_1,m_1).
\index{AI1@$I_\eps$: diagonal matrix ${\rm diag}(\varepsilon,1,1)$}
\end{equation}
Therefore, it is natural to introduce the vector
\begin{equation} \label{def-hat}
{p}_\varepsilon  := I_{\varepsilon} q_{\varepsilon}'
 = \begin{pmatrix} \varepsilon \vartheta_\varepsilon' \\ h_\varepsilon'  \end{pmatrix}.
\index{AP1@${p}_\varepsilon$: solid velocity with rescaled angular velocity}
\end{equation}
In particular the solid kinetic energy of the solid can be recast as
\begin{equation} \label{RecastedEnergy}
\frac{1}{2} M_{g,\varepsilon} \, q_{\varepsilon}' \cdot q_{\varepsilon}' 
= \frac{1}{2}  \varepsilon^{\alpha} M_g \, p_{\varepsilon}  \cdot p_{\varepsilon} .
\end{equation}
Hence the natural counterpart to $h_\varepsilon'  $ for what concerns the angular velocity is rather $\eps \vartheta_\varepsilon'  $ than $\vartheta_\varepsilon' $. %\par
This can also be seen on the boundary condition \eqref{souslab}: when $x$ belongs to $\partial\mathcal S_\varepsilon (t)$, 
the term $ \vartheta_\varepsilon'  (x-h_\varepsilon)^\perp$ is of order $\eps \vartheta_\varepsilon'  $ and is added to $ h_\varepsilon' $. \par
\ \par
\noindent
{\it Added inertia matrix.} 
We recall that $M^{\Ext}_{a}$ is the added inertia matrix for ${\mathcal S}_{0}$ defined in \eqref{def-Maa}, while $M_{a,\vartheta}^{\Ext}$ is the one corresponding to ${\mathcal S}(q)=R(\vartheta)\mathcal S_0+h$ (with $q=(\vartheta,h)$), given in \eqref{g+a_ext}. 
The scaled version of these matrices are $M_{a,\varepsilon}^{\Ext}$ corresponding to the solid ${\mathcal S}_{0,\varepsilon}$ and $M_{a,\vartheta,\varepsilon}^{\Ext}$ corresponding to ${\mathcal S}_{\varepsilon}(q)$.  
\index{AM7@$M^{\Ext}_{{a},\varepsilon}$: added inertia of ${\mathcal S}_{0,\varepsilon}$ when $\Omega=\R^{2}$}
Then one easily sees after suitable scaling/rotation arguments that
\begin{equation} \label{ScalingAddedInertia}
M_{a,\varepsilon}^{\Ext} = \varepsilon^{2} I_\eps  M^{\Ext}_{a}  I_\eps
\text{ and } 
M_{a,\vartheta,\varepsilon}^{\Ext} = \varepsilon^{2} I_\eps  M^{\Ext}_{a,\vartheta}  I_\eps.
\end{equation}
So the dependence of the added inertia matrices with respect to $\varepsilon$ is quite simple. It will not be the case any longer in the case $\Omega$ bounded, where this dependence will be much more intricate; see Proposition  \ref{dev-added} below.\par
\ \par
\noindent
{\it Other terms and scaled equation.} The other terms in \eqref{ODE_ext} have also a simple scaling with respect to $\varepsilon$ in the unbounded case.
Concerning the Christoffel symbols \eqref{Christo-exter}, it is not hard to check that 
$\langle \Gamma^\Ext_{\vartheta,\varepsilon}, q_{\varepsilon}', q_{\varepsilon}' \rangle =
\varepsilon I_{\varepsilon} \langle \Gamma^{\Ext}_{\vartheta}, p_{\varepsilon}, p_{\varepsilon} \rangle$,
and concerning the force term \eqref{force-ext} that
${F}^\Ext_{\vartheta,\varepsilon} ( q_{\varepsilon}' ) 
= I_{\varepsilon} {F}^{\Ext}_{\vartheta} ({p}_{\varepsilon}).$ 
The counterpart of Equation~\eqref{ODE_ext} for a shrinking solid is therefore
\begin{equation} \label{fnormEXTE}
\Big( \eps^{\alpha}\,  M_g + \eps^2 \, M^{\Ext}_{a, \vartheta_\varepsilon} \Big)
 p'_\varepsilon 
+ \varepsilon \langle \Gamma^{\Ext}_{\vartheta_\varepsilon} , p_\varepsilon, p_\varepsilon\rangle 
 = F^{\Ext}_{\vartheta_\varepsilon} ( p_{\varepsilon} )  .
\end{equation}
As we can see here, a difficulty is that the term ${F}^{\Ext}_{\vartheta_\varepsilon}$ depends on  the unknown $\varepsilon \vartheta_\varepsilon$ through $ \vartheta_\varepsilon$, that is singularly.
This difficulty, which is still present in the general case, will be overcome by using some averaging effect; see  \eqref{2-g} and Lemma~\ref{pS} below. \par
\ \par
\noindent
{\it Rescaled energy.} Let, for any $\vartheta $ in $\R$ and for any $\eps $ in $(0,1)$,
\begin{equation} \label{TrueMatrix}
{M}_\vartheta  (\eps):=
\left\{
\begin{array}{ccc}
 M_g + \eps^{2-\alpha}\, M^{\Ext}_{a, \vartheta} &   \text{ if } & \alpha \leq 2 ,  \\
 {M}^{\Ext}_{a,\vartheta} + \eps^{\alpha -2}\, M_g &  \text{ if }   &   \alpha > 2 .
  \end{array}
\right.
\index{AM8@${M}_\vartheta  (\eps)$: Rescaled total inertia}
\end{equation}
Observe that
\begin{equation} \label{agreable}
\eps^{\alpha}\,  M_g + \eps^2 \, M^{\Ext}_{a, \vartheta}
= \eps^{\min(2,\alpha)} \,  M_{ \vartheta }(\varepsilon) ,
\end{equation}
and that this matrix is of order ${O}(1)$ with respect to $\varepsilon$.
Using \eqref{def-hat}, \eqref{RecastedEnergy} and \eqref{ScalingAddedInertia} 
 we obtain that the total energy is 
$\eps^{\min(2,\alpha)} \,  {\mathcal{E}}_\vartheta (\eps , p_\eps) $, where  for any $p $ in $\R^3$, 
\begin{equation} \label{TrueNRJ}
{\mathcal{E}}_\vartheta (\eps , p)  :=  \frac{1}{2}  {M}_\vartheta (\eps)  p \cdot p .
\end{equation}
The use of ${\mathcal{E}}_\vartheta (\eps , p)$ is motivated by the following elementary result.
\begin{lem} \label{coerc}
Suppose that ${\mathcal S}_{0}$ is not a disk or $\alpha \leq 2$.
There exists $K>0$ depending only on ${\mathcal S}_{0}$, $m_{1}$ and $\mathcal J_{1}$ such that, for any $(\eps ,\vartheta, p) $ in $(0,1) \times \R \times \R^3$, 
\begin{equation*}
K |p|^{2}_{\R^{3}}  \leq  {\mathcal{E}}_\vartheta (\eps, p)  \leq  K^{-1}  |p|^{2}_{\R^{3}} .
\end{equation*}
\end{lem}
Using the energy conservation provided by Proposition~\ref{Prop-conserv-no} 
we deduce that, unless ${\mathcal S}_{0}$ is a disk and $\alpha>2$, $(| p_{\varepsilon}| )_{\varepsilon \in (0,1)}$ is bounded uniformly on $ [0,+\infty)$ in both Cases (i) and (ii). 
(Obtaining such an a priori estimate when $\Omega$ is a bounded domain will be much more involved in particular in Case (ii).) \par
In the degenerate case when ${\mathcal S}_{0}$ is a disk and $\alpha>2$ (hence in Case (ii)), the problem is that $M^{\Ext}_{a, \vartheta}$ is both the principal part of $M_{\vartheta}(\varepsilon)$ and degenerate. However using \eqref{Kir1Disk} below, we can check that for $p=(\omega,\ell)$, 
$$M^{\Ext}_{a, \vartheta} p \cdot p = M^{\Ext}_{\flat}(\ell + \omega \zeta_{\vartheta}^{\perp}) \cdot (\ell + \omega \zeta_{\vartheta}^{\perp}),$$
where $M^{\Ext}_{\flat}$ is given in \eqref{def-Ma}. We obtain that
\begin{equation*}
{\mathcal{E}}_{\vartheta_{\varepsilon}} (\eps , p_{\varepsilon}) 
= \frac{1}{2} \left( \varepsilon^{\alpha-2} m_{1} |h'_{\varepsilon}|^{2} +\varepsilon^{\alpha} {\mathcal J}_{1} |\vartheta'_{\varepsilon}|^{2} 
+ M^{\Ext}_{\flat} h'_{c,\varepsilon} \cdot h'_{c,\varepsilon} \right),
\end{equation*}
where we recall that $h_{c,\varepsilon}=h_{\varepsilon} + \varepsilon R(\vartheta_{\varepsilon}) \zeta$. \par
We deduce in this case that the families $(h'_{c,\varepsilon})$, $(\varepsilon^{\frac{\alpha}{2}-1} h'_{\varepsilon})$ and $(\varepsilon^{\frac{\alpha}{2}} \vartheta'_{\varepsilon})$ for $\varepsilon $ in $(0,1)$,  are bounded uniformly on $ [0,+\infty)$. \par
\ \par
Now our goal is to pass to the limit in each term of \eqref{fnormEXTE}.
We  distinguish Case (i) (with $m_\varepsilon$, $\mathcal{J}_\varepsilon$ given by \eqref{mass-inertie}) and  Case (ii) (with $m_\varepsilon$, $\mathcal{J}_\varepsilon$ given by \eqref{mass-inertie2}) . 
%
%
%
%
%%%%%%%%%%%%%%%%%%%%%%%%%%%%%%%%%%%%%%%%%%%%%%%%%%%%%%%%%%%%%%%%%%%%%%%%%%%%%%%%%%%%%%%%%
%
%
%
\subsection{Case (i) without external boundary}
In Case (i) we show that $h_\varepsilon$ converges to the solution to a massive point vortex equation similar to \eqref{ODE-mass} with a vanishing Kirchhoff-Routh velocity. 
For $\gamma \neq 0$ we let for $t \geq 0$,
\begin{equation*}
\overline{h}  (t) :=  \frac{ m }{  \gamma } \left[R\left(\frac{ \gamma t }{ m} - \frac{ \pi }{ 2} \right) \ell_0 + \ell_{0}^{\perp} \right],
\end{equation*}
and $\overline{h}  (t) := \ell_{0} t$ for $\gamma=0$.
Of course $\overline{h}$ satisfies 
\begin{equation} \label{ODE-massEXTE}
m \overline{h}''  =\gamma \left( \overline{h}^\prime  \right)^\perp  \  \text{ for } t \geq 0, \ 
 \text{ with }\ 
\overline{h} (0) = 0 \ \text{ and } \ \overline{h}'(0)=\ell_0.
\end{equation}
The precise statement of our first convergence result is as follows.
\begin{thm} \label{theo:3EXTE}
Let $\mathcal S_0$ a subset of $\R^2$ as above, $p_0 $ in $ \mathbb R^3$ and  $(\gamma ,m ,\mathcal J) $ in $\mathbb R \times (0,+\infty) \times (0,+\infty)$.
Let $\overline{h}$ be the global solution to \eqref{ODE-massEXTE}.
For every $\varepsilon $ in $(0,1]$,
let $q_\varepsilon $ be  the global solution to \eqref{def-hat} and \eqref{fnormEXTE} where $m_\varepsilon$ and $\mathcal{J}_\varepsilon$ are given by \eqref{mass-inertie},
and with the initial data 
$q_\varepsilon (0) = 0$ and  $q_\varepsilon' (0)= p_0 .$
Then, for all $T> 0$, as $\varepsilon \to 0^{+}$, 
$h_\varepsilon \rightharpoonup \overline{h}$ in $W^{2,\infty}([0,T];\mathbb R^2)$ weak-$\star$
and $\varepsilon\vartheta_\varepsilon \rightharpoonup 0$ in $W^{2,\infty}([0,T];\mathbb R)$ weak-$\star$.
\end{thm}
\begin{proof}
Let $T> 0$.
Using that $(p_{\varepsilon} )_{\varepsilon \in (0, 1)}$ is bounded uniformly  on $ [0,T]$ and Equation~\eqref{fnormEXTE},
we deduce uniform $W^{2,\infty}$ bounds on $h_{\varepsilon}$ and $\varepsilon \vartheta_{\varepsilon}$. 
This entails the existence of a converging subsequence $(\varepsilon_{n} \vartheta_{\varepsilon_{n}}, h_{\varepsilon_{n}})$
of $(\varepsilon \vartheta_{\varepsilon}, h_{\varepsilon})$:
\begin{equation} \label{SousSuiteEXTE}
(\varepsilon_{n} \vartheta_{\varepsilon_{n}}, h_{\varepsilon_{n}}) \longrightharpoonup (\Theta_{*}, h_{*}) \ \text{ in } W^{2,\infty} \text{ weak--}  \star .
\end{equation}
We now aim at characterizing the limit.
First it is clear that the left hand side of  \eqref{fnormEXTE} converges to 
$ M_{g}  ( \Theta_{*}'' , h_{*}'' )^t$            
 in $ L^{\infty}   \text{ weak} -\star.$
Now consider the force term 
${F}^{\Ext}_{\vartheta} ({p}_{\varepsilon_{n}}) 
= \gamma 
({\zeta}_{\vartheta_{\varepsilon_{n}}} \cdot h'_{\varepsilon_{n}} ,
(h'_{\varepsilon_{n}})^\perp - \varepsilon_{n} \vartheta_{\varepsilon_{n}}' {\zeta}_{\vartheta_{\varepsilon_{n}}})^t$
as defined in  \eqref{force-ext}.
On the one hand using \eqref{def-zeta_vartheta} we see that
\begin{equation} \label{2-g}
\varepsilon_{n} \vartheta_{\varepsilon_{n}}' {\zeta}_{\vartheta_{\varepsilon_{n}}} \longrightharpoonup 0 \text{ in } W^{-1,\infty}  \text{ weak--}\star.
\end{equation}
On the other hand since the weak-$\star$ convergence in $W^{2,\infty}$ entails the strong $W^{1,\infty}$ one,  we deduce that
$h'_{\varepsilon_{n}} \rightarrow h_{*}'$ in $L^{\infty}$
and $ \big( h_{*}(0) , h_{*}'(0)  \big) =  \big( 0, \ell_{0}  \big) $.
Hence from the last two coordinates of System \eqref{fnormEXTE}
we deduce
$m_{1} h_{*}''= \gamma (h_{*}') ^{\perp}.$
Due to the uniqueness of the solution to Equation~\eqref{ODE-massEXTE}, this establishes that $h_{*} = \overline{h}$ and  the convergence as $\varepsilon \rightarrow 0^{+}$ of the whole sequence (not merely of a subsequence). 
This concludes the proof of  Theorem \ref{theo:3EXTE} for the part concerning the position of the center of mass.

We now turn to the part concerning the angle, that is the convergence of $\varepsilon \vartheta_{\varepsilon}$.
We will use the following lemma (see \cite{GLS}).
\begin{lem} \label{pS}
Let $(\omega_{n})_{n \in \N} $ in $W^{1,\infty}(0,T;\R)^{\N}$, $(\varepsilon_{n})_{n \in \N}$ in $(0,+\infty)^{\N}$ with $\varepsilon_{n} \rightarrow 0^{+}$ as $n \rightarrow +\infty$, such that
$\varepsilon_{n} \omega_{n} \longrightharpoonup \overline{\rho} $ in  $ W^{1,\infty}(0,T;\R) $ 
 weak-$\star $ as  $n \rightarrow +\infty$.
 Let $(w_{n})_{n \in \N} $ in $L^{\infty}(0,T;\C)^{\N}$ such that
$w_{n} \longrightarrow w$ in  $ L^{\infty}(0,T;\C) $  as $n \rightarrow +\infty$.
Let $\vartheta_{n}:=\int_{0}^{t} \omega_{n}$.
Suppose that, on $(0,T)$, 
$$\varepsilon_{n} \omega_{n}'(t) = \Re [w_{n}(t) \exp(- i \vartheta_{n}(t))].$$
Then $\overline{\rho}$ is constant on $[0,T]$.
\end{lem}
We consider the first coordinate of  the system \eqref{fnormEXTE} (recall that $\alpha=0$ here). By the uniform estimate of $p_{\varepsilon}$ in $W^{1,\infty}(0,T)$, the system can be written as:
\begin{equation} \nonumber
M_g p'_\varepsilon  = F^{\Ext}_{\vartheta_\varepsilon} ( p_{\varepsilon} )  + R_{\varepsilon},
\end{equation}
where the remainder $R_{\varepsilon}$ is of order ${\mathcal O}(\varepsilon)$ in $L^{\infty}(0,T)$.
We apply Lemma~\ref{pS} to $\omega_{n}=\vartheta_{\varepsilon_{n}}'$ and
$w_{n} := \frac{1}{{\mathcal J}_{1}} \left[ \gamma (\zeta \cdot h'_{\varepsilon_{n}}) - i \gamma (\zeta \cdot (h'_{\varepsilon_{n}})^{\perp}) + R_{\varepsilon_{n},1} e^{i \vartheta_{\varepsilon_{n}}}\right]$, where $R_{\varepsilon,1}$ denotes the first coordinate of $R_{\varepsilon}$.
The assumptions are satisfied thanks to \eqref{SousSuiteEXTE}. Hence $\Theta_{*}$ is constant; using the initial data, we infer that $\Theta_{*} =0$. This concludes the proof of Theorem \ref{theo:3EXTE}. 
\qed
\end{proof}
%
%
%
%%%%%%%%%%%%%%%%%%%%%%%%%%%%%%%%%%%%%%%%%%%%%%%%%%%%%%%
%
%
\subsection{Case (ii) without external boundary}
In this situation we show that the limit system is the point vortex system with a vanishing Kirchhoff-Routh velocity field, that is the trivial system $\overline{h}'=0$.
Our result is the following. 
%
%
%##############################################
%
\begin{thm} \label{theo:1EXTE}
Let $\mathcal S_0 $ a subset of $\R^2$ as before, different from a disk (or $\alpha \leq 2$), $\gamma \neq 0$, $p_0$ in $ \mathbb R^3 $ and  $(\gamma ,m ,\mathcal J) $ in $ \mathbb R \times (0,+\infty) \times (0,+\infty)$. 
For every $\varepsilon $ in $(0,1]$, let $q_\varepsilon $ be  the global solution to \eqref{def-hat} and \eqref{fnormEXTE} where $m_\varepsilon$ and $\mathcal{J}_\varepsilon$ are given by \eqref{mass-inertie2},
and with the initial data $q_\varepsilon (0) =0$ and  $q_\varepsilon' (0) =  p_0$.
Then, as $\varepsilon \rightarrow 0^{+}$, $h_\varepsilon \rightharpoonup 0$ in $W^{1,\infty}([0,T];\mathbb R^2)$ weak-$\star$ for all $T>0$.
\end{thm}
\begin{proof}
First, thanks to the energy estimate, $\eps M^{\Ext}_{a, \vartheta_\varepsilon}$ is bounded in $W^{1,\infty}$ uniformly with respect to $\varepsilon$.
Since moreover $M_{g}$ is constant and $( p_\varepsilon )'$ is bounded in $W^{-1,\infty}$, we can conclude that the first term in the  left hand side of \eqref{fnormEXTE}
converges to $0$  in  $W^{-1,\infty}$ (due to the extra powers of $\varepsilon$).
Next,  the second term of the left hand side converges to $0$ in $L^{\infty}$ since the terms inside the brackets are bounded. 
Now  the last two  coordinates of  the right hand side of the equation \eqref{fnormEXTE} correspond to the vector $\gamma (h_\varepsilon')^\perp - \gamma  \varepsilon \vartheta_\varepsilon' {\zeta}_{\vartheta_\varepsilon }$.
The last term converges weakly to $0$ in $W^{-1,\infty}$ as seen in Case (i); see \eqref{2-g}.
Hence we infer that ${h}'_{\varepsilon}$ converges weakly-$\star$ to $0$ in $W^{-1,\infty}$. 
Because ${h}'_{\varepsilon}$ is bounded, this convergences also occurs in $L^{\infty}$ weak-$\star$.
This is sufficient to deduce  the strong convergence of $h_{\varepsilon}$ toward some $h_{*}$   in $L^{\infty}$, and that
$h_{*}'=0$  and $ h_{*} (0) = 0 $.
\qed \end{proof} 
In the case of disk and for $\alpha>2$, we have the equivalent convergence for the center of the disk $h_{c,\varepsilon}$. It suffices to pass to the weak limit in \eqref{OID1} where, since there is no external boundary, the right hand side is simplified by $E_{\flat}=0$ and $\tilde{B}_1=-1$. Note that in that case $\tilde{M}_{\flat}$ is constant, so $\Gamma_{\flat}=0$ (while the ``original'' $\Gamma$ is not, since $M_{a}$ has one more dimension and a different set of variables). 
When the disk is homogeneous, $h_{\varepsilon}=h_{c,\varepsilon}$, so that $h_{\varepsilon}$ converges itself in $W^{1,\infty}$ weak-$\star$.
 We omit the details.

%
%
%
%
%
%
%
%
%%%%%%%%%%%%%%%%%%%%%%%%%%%%%%%%%%%%%%%%%%%%%%%%%%%%%%%%%%%%%%%%%%%%%%%%%%%%%%%%%%%%%%%%%%%%%%%%%%%%%%%%%%%%%%%%%%%%%%%%%%%%%%%%%%%%%%%%%%%%%%%%%%%%%%%%%%%%%%%%
%
%
%
%
%
\section{Recasting the system: Proofs of Theorem~\ref{THEO-intro}, Lemma~\ref{LEM-pot-energy}, Proposition \ref{energy}, Proposition \ref{splitting-christo} and Theorem~\ref{THEO-intro-disk}}
\label{tard}
In this Section, we prove the results of Subsection~\ref{Subsec:FixedSize} concerning the dynamics of a solid with fixed size and mass, that is, Theorem~\ref{THEO-intro}, Lemma~\ref{LEM-pot-energy}, Proposition~\ref{energy}, Proposition~\ref{splitting-christo} and Theorem~\ref{THEO-intro-disk}.
\subsection{Splitting the proof of Theorem~\ref{THEO-intro}}
The pressure $\Pi$ can be recovered by means of Bernoulli's formula which is obtained by combining \eqref{EEE1} and  \eqref{irr}, and which reads:
\begin{equation} \label{EQ_bernoulli}
\nabla \Pi =-\left(\frac{\partial u}{\partial t}+\frac{1}{2}\nabla|u^2 |\right)\quad\text{in }\mathcal F(q).
\end{equation}
Given $q$, $p$ and $\gamma$, the pair $(u,\Pi)$ where $u$ is given by \eqref{EQ_irrotational_flow} and $\Pi$ by \eqref{EQ_bernoulli} yields a solution to \eqref{EEE1}, \eqref{E2}, \eqref{souslab} and \eqref{souslabis}. 
 \par
Now, equations \eqref{EqTrans} and  \eqref{EqRot}  can be summarized in the variational form:
\begin{equation} \label{bern_1}
mh'' \cdot \ell^\ast + \mathcal J \vartheta'' \omega^\ast
= \int_{\partial\mathcal S(q)} \Pi (\omega^\ast(x-h)^\perp + \ell^\ast) \cdot n \, {\rm d}s,
\end{equation}
for all $p^\ast :=(\omega^\ast,\ell^\ast) $ in $ \mathbb R \times \R^2$.
Let us associate with $(q,p^\ast)$ in $ \mathcal Q \times \R^3$ the potential vector field 
\begin{equation} \label{pot*}
u^\ast :=\nabla (\boldsymbol\varphi(q,\cdot)\cdot p^\ast) ,
\end{equation}
which is defined on $\mathcal F(q)$.
By definition of the 
Kirchhoff potentials, see \eqref{Kir-b} and \eqref{Kir-c}, 
we have the following equalities: 
$$
u^\ast\cdot n=\big(\omega^\ast  (\cdot-h)^\perp+ \ell^\ast  \big) \cdot n \quad \text{on }\partial\mathcal S(q)  \text{ and }
u^\ast\cdot n=0 \quad \text{on }\partial\Omega. 
$$
According to Bernoulli's formula  \eqref{EQ_bernoulli} and upon an integration by parts, identity \eqref{bern_1} can be turned into:
\begin{equation}
\label{bern_2}
mh''\cdot\ell^\ast+\mathcal J\vartheta''\omega^\ast=-\int_{\mathcal F(q)}\left(\frac{\partial u}{\partial t}+
\frac{1}{2}\nabla|u|^2\right)\cdot u^\ast{\rm d}x,
\end{equation}
for all $p^\ast :=(\omega^\ast,\ell^\ast) $ in $ \mathbb R \times \R^2.$
Therefore substituting the decomposition  \eqref{EQ_irrotational_flow} into \eqref{bern_2} we arrive at
\begin{multline}
\label{bern_2bis} 
mh''\cdot\ell^\ast+\mathcal J\vartheta''\omega^\ast 
+ \int_{\mathcal F(q)} \left( \frac{\partial u_{q'}}{\partial t}+\frac{1}{2}\nabla|u_{q'}|^2 \right) \cdot u^\ast{\rm d}x
= \\
- \int_{\mathcal F(q)} \left( \frac{1}{2}\nabla|u_\gamma|^2 \right) \cdot u^\ast{\rm d}x 
-
\int_{\mathcal F(q)}\left(\frac{\partial u_\gamma}{\partial t} + \nabla (u_{q'} \cdot u_{\gamma} ) \right) \cdot u^\ast{\rm d}x ,
\end{multline}
for all $p^\ast :=(\omega^\ast,\ell^\ast) $ in $ \mathbb R \times \R^2.$

Then the reformulation of Equations (\ref{SYS_full_system})-\eqref{irr} stated in Theorem~\ref{THEO-intro} will follow from the three following lemmas which deal respectively with the left hand side of \eqref{bern_2bis} and the two terms in the right hand side.
%
%Lemma_3=================================================
%
%
\begin{lem} \label{LEM_3}
For every smooth curve $t\mapsto q(t)$ in $\mathcal Q$ and every $p^\ast =(\omega^\ast,\ell^\ast)$ in $\mathbb R^3$, one has
\begin{multline} \label{lagrange_1}
mh''\cdot\ell^\ast+\mathcal J\vartheta''\omega^\ast
+ \int_{\mathcal F(q)} \left(\frac{\partial u_{q'}}{\partial t}+\frac{1}{2}\nabla|u_{q'}|^2\right)\cdot u^\ast{\rm d}x \\
= M(q)q'' \cdot p^\ast  + \langle \Gamma(q),q',q'\rangle \cdot p^\ast ,
\end{multline}
where $q=(\vartheta,h)$,
$u^\ast$ is given by \eqref{pot*}, $u_{q'}$ is given by  \eqref{DecompUa},
$M(q)$ and $\Gamma(q)$ are defined in \eqref{def_M_Gamma} and \eqref{def_vraiment-Gamma}.
\end{lem}
%
%
%
%Lemma4##############################################
\begin{lem} \label{LEM_easy}
For every $q\in\mathcal Q$ and every $p^\ast =(\omega^\ast,\ell^\ast)\in\mathbb R^3$, 
\begin{equation} \label{rab}
- \int_{\mathcal F(q)}\left(\frac{1}{2}\nabla|u_\gamma|^2\right)\cdot u^\ast{\rm d}x=\gamma^2 E(q)\cdot p^\ast,
\end{equation}
where  $u^\ast$ is given by \eqref{pot*},  $u_{\gamma}$ is given by  \eqref{DecompUb}, $E(q)$ is defined in \eqref{E-def}. 
\end{lem}
%########################################################
% 
%
%
%Lemma4##############################################
\begin{lem} \label{LEM_4}
For every smooth curve $t\mapsto q(t)$ in $\mathcal Q$  and every $p^\ast =(\omega^\ast,\ell^\ast)$ in $\mathbb R^3$, 
\begin{equation} \label{last_one}
- \int_{\mathcal F(q)}\left(\frac{\partial u_\gamma}{\partial t}+
\nabla (u_{q'} \cdot u_{\gamma} )\right)\cdot u^\ast{\rm d}x=\gamma\big(q' \times B(q)\big)\cdot p^\ast,
\end{equation}
where $q=(\vartheta,h)$, $u^\ast$ is given by \eqref{pot*}, $u_{\gamma}$ and  $u_{q'}$ are given by \eqref{DecompU}, $B(q)$ is defined in \eqref{B-def}. 
\end{lem}
%########################################################
%
%
%
Lemma~\ref{LEM_easy} simply follows from an integration by parts.
If we consider Lemmas~\ref{LEM_3} and \ref{LEM_4} as granted, then gathering the results of Lemmas~\ref{LEM_3}, \ref{LEM_easy} and \ref{LEM_4} with \eqref{bern_2bis}, the conclusion of Theorem~\ref{THEO-intro} follows. \qed \par
\ \par
Lemmas~\ref{LEM_3} and \ref{LEM_4} are proved in Subsections~\ref{SecPL3} and \ref{SecPL4}, respectively.

%
%
%
%%%%%%%%%%
%
%
%
\subsection{Reformulation of the potential part: Proof of Lemma \ref{LEM_3}}
\label{SecPL3}
We start with observing that 
\begin{equation} \label{trotriv}
mh'' \cdot \ell^\ast + \mathcal J \vartheta'' \omega^\ast =  M_g q''\cdot p^\ast .
\end{equation} 
Now in order to deal with the last term of the left hand side of \eqref{lagrange_1} we use a Lagrangian strategy.
For any  $q$ in $\mathcal Q$ and every $p  $ in $\mathbb R ^3$, let us define
\begin{equation} \label{cococo}
\mathcal{E}_1(q,p) :=\frac{1}{2}\int_{\mathcal F(q)}|  \nabla (\boldsymbol\varphi(q,\cdot)\cdot p) |^2{\rm d}x .
\end{equation} 
Thus $\mathcal{E}_1(q,q')$ denotes  the kinetic energy  of the potential part $u_{q'}$ of the flow.
It follows from classical shape derivative theory (see  \cite{CM,HP,M,SZ}) that $\mathcal{E}_1 $ is in $C^{\infty} \big(\mathcal Q \times  \mathbb R ^3 ; \lbrack 0,+\infty ) \big)$.

Now the crucial quantity here is the Euler-Lagrange function:
\begin{equation} \label{def-EL}
\mathcal E  \mathcal L  
 = \left(\frac{d}{dt}\frac{\partial \mathcal{E}_1 }{\partial p} (q,q') -\frac{\partial \mathcal{E}_1 }{\partial q} (q,q')\right)\cdot p^\ast ,
\end{equation}
associated with  any smooth curve $t\mapsto q(t)$ in $\mathcal Q$ and any $p^\ast\in\mathbb R ^3$.
\begin{lem} \label{LEM_1}
For any smooth curve $t\mapsto q(t)$ in $\mathcal Q$, for  every $p^\ast\in\mathbb R ^3$,
\begin{equation} \label{lagrange}
\int_{\mathcal F(q)}
\left(\frac{\partial u_{q'}}{\partial t}+\frac{1}{2}\nabla|u_{q'}|^2\right)\cdot u^\ast{\rm d}x =
\mathcal E  \mathcal L 
\end{equation}
where   $u_{q'}$ is given by  \eqref{DecompU}, $u^\ast$ is given by \eqref{pot*} and $\mathcal E  \mathcal L $ is given by   \eqref{def-EL}.
\end{lem}
%
%
%
%================================================================

\begin{proof}

Let us  make use of the following slight abuse of notations which simplifies the presentation of the proof of Lemma \ref{LEM_1}.
For a smooth function $I(q,p)$, where $(q,p)$ is running into $\mathcal Q \times \mathbb R ^3$, and  a smooth curve $t\mapsto q(t)$ in $\mathcal Q$ let us define
\begin{equation*}
\left( \frac{\partial}{\partial q} \frac{d}{dt} I(q,q') \right) (t) := \left(\frac{\partial}{\partial q} J\right) \big(q(t),q'(t),q''(t)\big), 
\end{equation*}
where, for $(q,p,r)$ in  $\mathcal Q \times \mathbb R ^3 \times \mathbb R ^3$,
\begin{equation} \label{chain}
J(q,p,r) = p  \frac{\partial I}{\partial q} (q,p) + r \frac{\partial I}{\partial p} (q,p) .
\end{equation}
Observe in particular that
\begin{align}
\nonumber
\frac{d}{dt} \big(   I (q(t),q'(t)) \big) &=  J \big(q(t),q'(t),q''(t)\big), 
\intertext{and}
 \label{ty}
\frac{d}{dt} \left( \frac{\partial I}{\partial q}   (q(t),q'(t)) \right) 
&=  \left( \frac{\partial}{\partial q} \frac{d}{dt} I(q,q') \right) (t) .
\end{align}
Below, in such circumstances, it will be comfortable to write 
\begin{equation*}
\frac{\partial}{\partial q}  \left[ J \big(q(t),q'(t),q''(t)\big)  \right] \text{ instead of }
\left(\frac{\partial J}{\partial q} \right) \big(q(t),q'(t),q''(t)\big) , 
\end{equation*}
and it will be understood that $J$ is  extended  from $\big(q(t),q'(t),q''(t)\big) $ to general $(q,p,r)$ by \eqref{chain}. \par
Let us also recall  the Reynolds transport formula (see \cite[pages 12--13]{Reynolds:1903}), that we shall use to express the derivative of functionals of the kind:
$$\mathcal I(q)=\int_{\mathcal F(q)} f(q,\cdot)\,{\rm d}x,$$
where $f(q,\cdot)$ is a smooth function defined in $\mathcal F(q)$. The formula reads, for every $q\in\mathcal Q$ and every $p^\ast$ in $\mathbb R^3$:
\begin{equation}
\label{eq:reynolds}
\frac{\partial\mathcal I}{\partial q}(q)\cdot p^\ast=\int_{\mathcal F(q)} \frac{\partial f}{\partial q}(q,\cdot)\cdot p^\ast\,{\rm d}x
+\int_{\partial\mathcal S(q)} f(q,\cdot)(u^\ast\cdot n)\,{\rm d}s,
\end{equation}
where  
$u^\ast$ is given by \eqref{pot*}.\par
We start with manipulating the first term of $\mathcal E  \mathcal L$. 
Upon an integration by parts we arrive at
$$
\frac{\partial \mathcal{E}_1 }{\partial p}\cdot p^\ast =\int_{\mathcal F(q)} u_{q'}\cdot u^\ast{\rm d}x
= \int_{\partial\mathcal S(q)}(\boldsymbol\varphi\cdot q' )(u^\ast\cdot n)\,{\rm d}s .
$$
Then, thanks to formula \eqref{eq:reynolds}, 
\begin{equation}
\label{theV}
\frac{\partial \mathcal{E}_1 }{\partial p} \cdot p^\ast 
=\frac{\partial}{\partial q}\left(\int_{\mathcal F(q)}\left(\boldsymbol\varphi\cdot q' \right)\,{\rm d}x\right)\cdot p^\ast-
\int_{\mathcal F(q)}\left(\frac{\partial \boldsymbol\varphi }{\partial q}\cdot q'\right)\cdot p^\ast{\rm d}x.
\end{equation}
We differentiate  \eqref{theV} with respect to $t$ to arrive at
\begin{multline} \label{PPP}
\frac{d}{dt}\frac{\partial \mathcal{E}_1 }{\partial p} \cdot p^\ast =\frac{d}{dt}\frac{\partial}{\partial q}\left(\int_{\mathcal F(q)}(\boldsymbol\varphi\cdot q' )\,{\rm d}x\right)\cdot p^\ast\\
-\frac{d}{dt}\left(\int_{\mathcal F(q)}\left(\frac{\partial \boldsymbol\varphi }{\partial q}\cdot q'\right)\cdot p^\ast{\rm d}x\right) .
\end{multline}
With the abuse of notations  mentioned above we exchange the derivatives involved in the first term of the right hand side, so that 
 the identity \eqref{PPP} can be rewritten as follows:
\begin{multline} \label{DD0}
\frac{d}{dt}\frac{\partial \mathcal{E}_1 }{\partial p} \cdot p^\ast 
= \frac{\partial}{\partial q} \frac{d}{dt}\left(\int_{\mathcal F(q)}(\boldsymbol\varphi\cdot q' )\,{\rm d}x\right)\cdot p^\ast \\
- \frac{d}{dt}\left(\int_{\mathcal F(q)}\left(\frac{\partial \boldsymbol\varphi }{\partial q}\cdot q'\right)\cdot p^\ast{\rm d}x\right).
\end{multline}
Moreover, using again the Reynolds transport formula \eqref{eq:reynolds}, we deduce that 
\begin{align} \nonumber
\frac{d}{dt}\left(\int_{\mathcal F(q)}(\boldsymbol\varphi\cdot q' )\,{\rm d}x\right) &= \int_{\mathcal F(q)}\frac{\partial}{\partial t}(\boldsymbol\varphi\cdot q' )\,{\rm d}x + \int_{\partial\mathcal S(q)}(\boldsymbol\varphi\cdot q' )(u_{q'}\cdot n)\,{\rm d}s \\
\label{DD1}
&=\int_{\mathcal F(q)}\frac{\partial}{\partial t}(\boldsymbol\varphi\cdot q' )\,{\rm d}x + 2 \mathcal{E}_1 (q,q'),
\end{align}
by integration by parts.

We infer from \eqref{DD0} and \eqref{DD1}, again with the abuse of notations  mentioned above, that:
\begin{multline}
\label{6}
 \mathcal E  \mathcal L 
=
\frac{\partial \mathcal{E}_1}{\partial q} 
+ \frac{\partial}{\partial q} 
\left[ \int_{\mathcal F(q)}
\frac{\partial}{\partial t}(\boldsymbol\varphi\cdot q' )\,{\rm d}x\right]\cdot p^\ast
\\-\frac{d}{dt}\left(\int_{\mathcal F(q)}\left(\frac{\partial \boldsymbol\varphi }{\partial q}\cdot q'\right)\cdot p^\ast{\rm d}x\right).
\end{multline}
Thanks to formula \eqref{eq:reynolds}, we translate the second term of the right hand side into
\begin{multline} \label{8}
\frac{\partial}{\partial q} \left[ \int_{\mathcal F(q)}
\frac{\partial}{\partial t}(\boldsymbol\varphi\cdot q' )\,{\rm d}x\right] \cdot p^\ast
= \int_{\mathcal F(q)}\frac{\partial}{\partial q}\left(\frac{\partial}{\partial t}(\boldsymbol\varphi\cdot q')\right)\cdot p^\ast\,{\rm d}x
\\+ \int_{\partial\mathcal S(q)}\frac{\partial}{\partial t}(\boldsymbol\varphi\cdot q' )(u^\ast\cdot n)\,{\rm d}s,
\end{multline}
and  the last one into 
\begin{align} \nonumber
\frac{d}{dt} \left( \int_{\mathcal F(q)} \left(\frac{\partial \boldsymbol\varphi }{\partial q}\cdot q'\right) \cdot p^\ast{\rm d}x\right)
&= \int_{\mathcal F(q)}\frac{\partial}{\partial t}\left( \left(\frac{\partial \boldsymbol\varphi }{\partial q}\cdot q'\right)\cdot p^\ast\right)\,{\rm d}x
\\ &\quad+\int_{\partial\mathcal S(q)}\left(\left(\frac{\partial \boldsymbol\varphi }{\partial q}\cdot q'\right)\cdot p^\ast\right)(u_{q'} \cdot n)\,{\rm d}s .
\end{align}
Using again \eqref{ty} for the first term and integrating by parts the second one, we arrive at 
\begin{multline} \label{9} 
\frac{d}{dt}\left(\int_{\mathcal F(q)} \left(\frac{\partial \boldsymbol\varphi }{\partial q}\cdot q'\right)\cdot p^\ast{\rm d}x\right)
\\=  \int_{\mathcal F(q)} \frac{\partial }{\partial q} \left( \frac{\partial}{\partial t} (\boldsymbol\varphi\cdot q' ) \right)\cdot p^\ast \,{\rm d}x
+\int_{\mathcal F(q)} \left(\frac{\partial u_{q'}}{\partial q}\cdot p^\ast\right) \cdot u_{q'} \,{\rm d}x .
\end{multline}
On the other hand, thanks to formula \eqref{eq:reynolds}, 
 \begin{equation}
\label{7}
\frac{\partial \mathcal{E}_1}{\partial q}  \cdot p^\ast=\int_{\mathcal F(q)}\left(\frac{\partial u_{q'}}{\partial q}\cdot p^\ast\right)\cdot u_{q'}\,{\rm d}x+\frac{1}{2}\int_{\partial\mathcal S(q)}|u_{q'}|^2(u^\ast\cdot n)\,{\rm d}s .
\end{equation}
Substituting the expressions  \eqref{8}, \eqref{9} and \eqref{7} into \eqref{6} and simplifying, we end up with:
\begin{equation*}
\mathcal E  \mathcal L 
= \int_{\partial\mathcal S(q)}\left[\frac{\partial}{\partial t}(\boldsymbol\varphi\cdot q' )+\frac{1}{2}|u_{q'}|^2\right](u^\ast\cdot n)\,{\rm d}s.
\end{equation*}
Upon an integration by parts, we recover \eqref{lagrange} and the proof of Lemma~\ref{LEM_1} is completed.
\qed 
\end{proof} 
%
%
%##################################################################
%

Now, we observe that  the kinetic energy $\mathcal{E}_1(q,q')$ associated with the potential part of the flow, as defined by \eqref{cococo}, can be rewritten as:
\begin{equation} \label{E1}
\mathcal{E}_1 (q,q') =\frac{1}{2} M_a(q) q' \cdot q',
\end{equation}
where $ M_a (q)$ is  defined by \eqref{def_M_Gamma}.
This allows us to prove the following result.
%Lemma niou################################################
%
%
\begin{lem} \label{niou}
For any smooth curve $t\mapsto q(t)$ in $\mathcal Q$, for  every $p^\ast\in\mathbb R ^3$, one has
\begin{equation} \label{full_RHS}
\mathcal E  \mathcal L  = M_a(q)q'' \cdot p^\ast + \langle \Gamma (q),q',q'\rangle\cdot p^\ast ,
\end{equation}
with $\mathcal E \mathcal L$ is given by \eqref{def-EL}, $ M_a (q)$ defined by \eqref{def_M_Gamma} and  $\Gamma (q)$ associated with $M(q)$ by the Christoffel formula \eqref{Christo2}-\eqref{Christo1}.
\end{lem}
%
%
%
%#########################################################

%
\begin{proof}[Proof of Lemma~\ref{niou}.]
Using \eqref{E1} in the definition   \eqref{def-EL} of $ \mathcal E  \mathcal L $, we arrive at 
\begin{align*}
 \mathcal E  \mathcal L 
&= M_a (q)q''  \cdot p^\ast
+  \Big( \big( DM_a (q)\cdot q' \big) q' \Big) \cdot p^\ast
- \frac12 \Big( \big( DM_a  (q)\cdot p^\ast \big) q'\Big)  \cdot q' .
\end{align*}
Let us recall  the notation $(M_a)_{i,j}^{k} (q)$   in \eqref{cri}
and let the notation  $\sum$ stand for $\sum_{1 \leq i,j,k \leq 3 } $ for the rest of this proof.
Then
\begin{align*}
 \mathcal E  \mathcal L 
= M_a \, q''  \cdot p^\ast
+  \sum  (M_a)_{i,j}^{k}  \,  q'_{k} q'_{j}  p^\ast_{i} 
- \frac12  \sum  (M_a)_{i,j}^{k}  \,  q'_{i} q'_{j}  p^\ast_{k} .
\end{align*}
A symmetrization of the second term of the right hand side  above leads to 
\begin{align*}
 \mathcal E  \mathcal L 
&= M_a  \,  q''  \cdot p^\ast
\\ &\quad + \frac12 \Big( \sum (M_a)_{i,j}^{k}  \,    q'_{k} q'_{j}  p^\ast_{i} 
+ \sum  (M_a)_{i,k}^{j}  \,  q'_{k} q'_{j}  p^\ast_{i} 
-  \sum (M_a)_{i,j}^{k}  \,  q'_{i} q'_{j}  p^\ast_{k} \Big) ,
\end{align*}
Then the result follows by exchanging $i$ and $k$ in the last sum.
\qed 
\end{proof} 
Finally Lemma~\ref{LEM_3} straightforwardly results from the combination of \eqref{trotriv}, Lemmas \ref{LEM_1} and \ref{niou}. This concludes the proof of Lemma \ref{LEM_3}.
\subsection{Reformulation of the cross part: Proof of Lemma~\ref{LEM_4}}
\label{SecPL4}
We now turn to the proof of Lemma~\ref{LEM_4}.
Write $q:=(\vartheta,h)$ and recall that $\psi(q,\cdot)$ and $C(q)$ are defined in \eqref{def_stream} and that 
$$u_\gamma :=\gamma\nabla^\perp\psi(q,\cdot),  \quad 
u_{q'} :=\nabla ( \boldsymbol\varphi(q,\cdot) \cdot q' )   \text{ and }u^\ast :=\nabla ( \boldsymbol\varphi(q,\cdot) \cdot p^\ast) .$$ 
Upon an integration by parts, since $\psi (q, \cdot)$ vanishes on $\partial \Omega$, 
\begin{equation} \label{cla} 
\int_{\mathcal F(q)}\left(\frac{\partial u_\gamma}{\partial t}\right)\cdot u^\ast{\rm d}x
= - \gamma\int_{\partial\mathcal S(q)}\left(   \frac{\partial }{\partial t} \big( \psi (q, \cdot)\big)  \right)
\left(\frac{\partial\boldsymbol\varphi}{\partial \tau}(q,\cdot)\cdot p^\ast \right) \, {\rm d}s . 
\end{equation}
\begin{lem} \label{wehave}
On $\partial\mathcal S(q)$, 
\begin{equation} \label{tderiv}
\frac{\partial }{\partial t} \big( \psi(q, \cdot)\big) 
= - \frac{\partial\psi}{\partial n}(q,\cdot) \left(\frac{\partial\boldsymbol\varphi}{\partial n}(q, \cdot) \cdot q'\right) 
+ DC(q)\cdot q' .
\end{equation}
\end{lem}
\begin{proof}
We start with the observation that 
\begin{equation}
\label{tderiv1}
 \frac{\partial}{\partial t} \big( \psi(q, \cdot)\big)   =  \frac{\partial\psi}{\partial q}(q,\cdot) \cdot q' 
 \end{equation}
 is the derivative of the function $\psi(q, \cdot)$ when the boundary $\partial\mathcal S(q)$ undergoes a rigid displacement of velocity 
$w=\omega(x-h)^\perp+\ell$ where $q'=(\omega,\ell)$. 
Then we differentiate the identity:
$$\psi(q,R(\vartheta)X+h)=C(q),\qquad   \text{ for  } X\in\partial\mathcal S_0,$$
with respect to $q$ in the direction $q'$. We obtain:
\begin{equation} \label{tderiv2}
\frac{\partial\psi}{\partial q}(q,x)\cdot q'+ \nabla\psi(q,x)\cdot w=DC(q)\cdot q',\qquad   \text{ for  } x\in\partial\mathcal S(q).
\end{equation}
Since $\psi(q,\cdot)$ is constant on $\partial\mathcal S(q)$, its tangential derivative is zero. 
Besides, on $\partial\mathcal S(q)$ we have $w\cdot n=u_{q'}\cdot n= \frac{\partial\boldsymbol\varphi}{\partial n}(q,\cdot) \cdot q'$ so 
\begin{equation} \label{tderiv3}
\nabla \psi(q,x) \cdot w = \frac{\partial\psi}{\partial n}(q,x) 
\left( \frac{\partial\boldsymbol\varphi}{\partial n}(q,x) \cdot q' \right)  \text{ for  } x\in\partial\mathcal S(q).
\end{equation}
Gathering \eqref{tderiv1}, \eqref{tderiv2} and \eqref{tderiv3} we
obtain \eqref{tderiv}.
\qed 
\end{proof} 
Substituting now \eqref{tderiv} into \eqref{cla}, we arrive at 
\begin{equation} \label{circu_1} 
\int_{\mathcal F(q)}\left(\frac{\partial u_\gamma}{\partial t}\right)\cdot u^\ast{\rm d}x 
=\gamma\int_{\partial\mathcal S(q)}
\frac{\partial \psi}{\partial n}\left(\frac{\partial\boldsymbol\varphi}{\partial n}\cdot q'\right)\left(\frac{\partial\boldsymbol\varphi}{\partial \tau}\cdot p^\ast\right) \, {\rm d}s.
\end{equation}
On the other hand, by an integration by parts, 
\begin{equation} \label{circu_2}
\int_{\mathcal F(q)}\nabla(u_{q'}\cdot u_\gamma)\cdot u^\ast{\rm d}x = \gamma\int_{\partial\mathcal S(q)}(u_{q'}\cdot u_\gamma)(u^\ast\cdot n) \, {\rm d}s .
\end{equation}
Adding \eqref{circu_1} and \eqref{circu_2}  and using that $u_\gamma=-\gamma \frac{\partial\psi}{ \partial n} \tau$ on $\partial\mathcal S(q)$, we arrive at 
\begin{multline*}
\int_{\mathcal F(q)}\left(\frac{\partial u_\gamma}{\partial t}+\nabla(u_{q'}\cdot u_\gamma)\right)\cdot u^\ast{\rm d}x = \\
\gamma\int_{\partial\mathcal S(q)} \frac{\partial \psi}{\partial n}\left[\left(\frac{\partial\boldsymbol\varphi}{\partial n}\cdot q'\right) 
\left(\frac{\partial\boldsymbol\varphi}{\partial \tau}\cdot p^\ast\right)
- \left(\frac{\partial\boldsymbol\varphi}{\partial n}\cdot p^\ast\right)\left(\frac{\partial\boldsymbol\varphi}{\partial \tau}\cdot q' \right)\right] \, {\rm d}s,
\end{multline*}
which is \eqref{last_one}. This ends the proof of Lemma~\ref{LEM_4}. 
%\qed
%
%
%
%
%
%%%%%%%%%%%%%%%%%%%%%%%%%%%%%
%
%
%
%
\subsection{Decomposition of the Christoffel symbols: Proof of Proposition~\ref{splitting-christo}}
\label{Sect-splitting-christo}

This subsection is devoted to the proof of Proposition \ref{splitting-christo}.
We will use the matrix ${\underline{M}}_a(q)$  given by
\begin{equation} \label{massea}
{\underline{M}}_a(q) := \mathcal R(\vartheta)^t \, \, M_a(q)\mathcal R(\vartheta) ,
\end{equation}
where we recall that $\mathcal R(\vartheta)$ is the rotation matrix defined by \eqref{anum}.
We also introduce the following real-valued functions depending on the variables $q=(\vartheta, h) \in \mathcal Q$, $p$ in ${\mathbb R}^3$ and $p^\ast$ in ${\mathbb R}^3$:
\begin{align*}
\Xi_{1} (q,p,p^\ast ) &:=
\left[\left(\frac{\partial {\underline{M}}_a}{\partial q}(q)\cdot p^\ast \right)\mathcal R(\vartheta)^t \, p\right]\cdot\mathcal R(\vartheta)^t \, p , \\
\Xi_{3} (q,p,p^\ast ) &:=
\left[\left(\frac{\partial {\underline{M}}_a}{\partial q}(q)\cdot p\right)\mathcal R(\vartheta)^t \, p\right]\cdot\mathcal R(\vartheta)^t \, p^\ast .
\end{align*}
The indices in $\Xi_{1}$ and $\Xi_{3}$ above are chosen in order to recall the position where $p^\ast$ appears (the ``first $p$'' in the expression of $\Xi_{1}$ is $p^{\ast}$ and so on.)
Similarly we define, for $p= (\omega,\ell)$ and $p^\ast = (\omega^\ast ,\ell^\ast) $ the functions: 
\begin{multline*}
\Upsilon_{1} (q,p,p^\ast )  :=\omega^\ast
M_a(q) p \cdot \begin{pmatrix}0\\  
\ell^\perp \end{pmatrix} , \quad
 \Upsilon_{2} (q,p,p^\ast ) :=\omega
M_a(q)p^\ast \cdot \begin{pmatrix}0\\  
\ell^\perp \end{pmatrix} , \\
\text{ and } \  \Upsilon_{3} (q,p,p^\ast ) :=\omega
M_a(q)p \cdot \begin{pmatrix} 0 \\  
\ell^{\ast\perp} \end{pmatrix} .
\end{multline*}
The proof of Proposition~\ref{splitting-christo} is then split into the proof of the following three ancillary lemmas.
%
%
%Lemma niou################################################
%
%
\begin{lem} \label{niou-split}
For any $(q,p,p^\ast)$ in $\mathcal Q \times \mathbb R ^3 \times \mathbb R ^3$, 
\begin{multline} \label{full_full}
\langle \Gamma (q),p,p\rangle\cdot p^\ast
= \Upsilon_{1} (q,p,p^\ast ) 
- \Upsilon_{2} (q,p,p^\ast ) 
- \Upsilon_{3} (q,p,p^\ast )  \\+ \Xi_3 (q,p,p^\ast )
- \frac{1}{2}\Xi_1 (q,p,p^\ast ) .
\end{multline}
\end{lem}
%#########################################################

\begin{lem} \label{nioubis}
For any $(q,p,p^\ast)$ in $\mathcal Q \times \mathbb R ^3 \times \mathbb R ^3$, 
\begin{equation*}
\Upsilon_{1} (q,p,p^\ast ) - \Upsilon_{2} (q,p,p^\ast ) - \Upsilon_{3} (q,p,p^\ast ) 
= \langle \Gamma^{{\rm rot}} (q),p,p\rangle\cdot p^\ast ,
\end{equation*}
where $\Gamma^{{\rm rot}} (q)$ is defined in \eqref{zozo}.
\end{lem}
\begin{lem} \label{LEM_2}
For any $(q,p,p^\ast)$ in $\mathcal Q \times \mathbb R ^3 \times \mathbb R ^3$, 
\begin{equation} \label{expre_quadratic}
\Xi_3 (q, p, p^\ast) - \frac{1}{2}\Xi_1 (q,p,p^\ast )
= \langle \Gamma^{\partial \Omega} (q),p,p\rangle\cdot p^\ast,
\end{equation}
where $\Gamma^{\partial \Omega}( q)$ is defined in \eqref{Def_Gamma}.
\end{lem}
%#########################################################

Before proving these three lemmas, we introduce a few notations.
For every $q=(\vartheta,h)\in\mathcal Q$, we define the change of variables $y = R(\vartheta)^t(x-h)$, the domains
$$
{\underline{\Omega}}(q) :=R(\vartheta)^t(\Omega-h),  \quad {\underline{\mathcal F}}(q) :=R(\vartheta)^t(\mathcal F(q)-h)={\underline{\Omega}}(q)\setminus {\mathcal S}_0 ,
$$
 and the functions ${\underline{\varphi}}_i(q,y)$ such that 
 $${\underline{\boldsymbol\varphi}}(q,\cdot) :=({\underline{\varphi}}_1(q,\cdot),{\underline{\varphi}}_2(q,\cdot),{\underline{\varphi}}_3(q,\cdot)) ,$$
satisfies 
$${\underline{\boldsymbol\varphi}}(q,y) =
\mathcal R(\vartheta)^t{\boldsymbol\varphi}(q,x),\qquad y\in{\underline{\mathcal F}}(q) .
$$
For every $j=1,2,3$, the functions ${\underline{\varphi}}_j(q,\cdot)$ are harmonic in ${\underline{\mathcal F}}(q)$ and satisfy:
\begin{subequations}
\begin{alignat}{3}
\frac{\partial{\underline{\varphi}}_j}{\partial n}(q,y)&=\begin{cases}
y^\perp\cdot n&j=1;\\
n_{j-1}&j=2,3\\
\end{cases}&\quad&
\text{ on }\partial\mathcal S_0,\\
\label{trala}
\frac{\partial{\underline{\varphi}}_j}{\partial n}(q,y)&=0\quad(j=1,2,3)&&\text{ on }\partial{\underline{\Omega}}(q).
\end{alignat}
\end{subequations}
Therefore the matrix ${\underline{M}}_a(q)$ defined in  \eqref{massea} can be recast as 
\begin{equation}
{\underline{M}}_a(q) =\int_{\partial\mathcal S_0}{\underline{\boldsymbol\varphi}}(q)\otimes\frac{\partial{\underline{\boldsymbol\varphi}}}{\partial n}(q) \, {\rm d}s .
\end{equation}

\begin{proof}[Proof of Lemma \ref{niou-split}.]
Let  $t\mapsto q(t)$  be a smooth curve  in $\mathcal Q$, defined in a neighborhood of $0$ such that $q(0) = q$ and $q'(0)= p$. 
On the one hand, it follows from  \eqref{E1} and \eqref{massea} that for every $p^\ast=(\omega^\ast,\ell^\ast)\in\mathbb R^3$,  
\begin{align}
\frac{d}{dt}\frac{\partial \mathcal{E}_1}{\partial p}\cdot p^\ast
&=M_a(q)p' \cdot p^\ast - \Upsilon_{2} (q,p,p^\ast ) 
+ \Xi_3 (q,p,p^\ast ) - \Upsilon_{3} (q,p,p^\ast )  .
\label{explict_1}
\end{align}
On the other hand, 
\begin{equation} \label{explict_2}
\frac{\partial \mathcal{E}_1}{\partial q}\cdot p^\ast
= -\Upsilon_{1} (q,p,p^\ast )  + \frac{1}{2}\Xi_1 (q,p,p^\ast ).
\end{equation}
Gathering \eqref{def-EL}, \eqref{full_RHS}, \eqref{explict_1} and \eqref{explict_2},  we deduce \eqref{full_full}.
\qed
\end{proof} 
%
%
%
%##################################################################

\begin{proof}[Proof of Lemma~\ref{nioubis}.]
On the one hand, since $M(q)$ is  symmetric, 
\begin{equation} \label{rel_2}
\Upsilon_{2} (q,p,p^\ast )  = \omega M_a(q)\begin{pmatrix} 0 \\ \ell^{\perp} \end{pmatrix} \cdot p^\ast.
\end{equation}
On the other hand, with the notation of \eqref{onoublie}, 
\begin{equation} \label{rel_3}
- \Upsilon_{3} (q,p,p^\ast ) 
+ \Upsilon_{1} (q,p,p^\ast ) 
=\begin{pmatrix}
-P_a^\perp\cdot \ell\\
\omega  P_a^\perp
\end{pmatrix}\cdot p^\ast =-\left[\begin{pmatrix}0\\
P_a \end{pmatrix}\times
p\right]\cdot p^\ast .
\end{equation}
Now gathering \eqref{rel_2} and \eqref{rel_3} we conclude the proof of Lemma~\ref{nioubis}.
\qed
\end{proof}

\begin{proof}[Proof of Lemma \ref{LEM_2}.]

We will use the following lemma.
%Lemma forme################################################
%
\begin{lem} \label{LEM_forme}
For $i,j=1,2,3$, for every  $\hat q = (\hat\vartheta,\hat h) $ in $\mathcal Q$ and every $p^\ast =(\omega^\ast,\ell^\ast)$ in $\mathbb R ^3$, 
\begin{equation} \label{ovv-eq}
\frac{\partial}{\partial q}\left.\left(\int_{{\underline{\mathcal F}}(q)} \nabla {\underline{\varphi}}_i(q) \cdot \nabla{\underline{\varphi}}_j(q)
{\rm d}y \right) \right|_{q=\hat q} \cdot p^\ast
= - \int_{\partial{\underline{\Omega}}(\hat q)} \frac{\partial{\underline{\varphi}}_i}{\partial \tau}
\frac{\partial{\underline{\varphi}}_j}{\partial \tau} (\hat w^\ast\cdot n) \, {\rm d}s ,
\end{equation}
with 
\begin{equation} \label{HatwAst}
\hat w^\ast (\hat q ,  p^\ast ,\cdot) :=-\omega^\ast \cdot^\perp-R(\hat\vartheta)^t \, \ell^\ast.
\end{equation}
\end{lem}
%
%
%#########################################################
Let us first take Lemma~\ref{LEM_forme} for granted and conclude the proof of Lemma~\ref{LEM_2}.
Applying the change of variables $x=R(\hat\vartheta)y+\hat h$, we deduce that:
\begin{align*}
\frac{\partial{\underline{M}}_a}{\partial q}(\hat q)\cdot p^\ast 
&= \left(\frac{\partial}{\partial q}\left.\left(\int_{{\underline{\mathcal F}}(q)}\nabla{\underline{\varphi}}_i(q)\cdot\nabla{\underline{\varphi}}_j(q){\rm d}y\right)\right|_{q=\hat q}\cdot p^\ast \right)_{1\leq i,j\leq 3} \\
&=-\left(\int_{\partial{\underline{\Omega}}(\hat q)}\frac{\partial{\underline{\varphi}}_i}{\partial \tau}\frac{\partial{\underline{\varphi}}_j}{\partial \tau}(\hat w^\ast\cdot n) \, {\rm d}s\right)_{1\leq i,j\leq 3} \\
&=\mathcal R(\hat\vartheta)^t \, \left(\int_{\partial\Omega}\frac{\partial\varphi_i}{\partial \tau}(\hat q)\frac{\partial\varphi_j}{\partial \tau}(\hat q)( w^\ast\cdot n) \, {\rm d}s\right)_{1\leq i,j\leq 3}\mathcal R(\hat\vartheta) ,
\end{align*}
with $ w^\ast :=\omega^\ast(x-h)^\perp +\ell^\ast$. 
Therefore, applying this with $(\hat q , p^\ast ) = (q,  p^\ast ) $ and with $(\hat q , p^\ast ) = (q,  p ) $,  we arrive at
\begin{equation*}
\Xi_3 (q,p,p^\ast ) -\frac{1}{2}\Xi_1 (q,p,p^\ast )
= \sum \left[  \Lambda^l_{ij}(q)p_lp_j-\frac{1}{2}\Lambda^i_{jl}(q)p_lp_j \right] p^\ast_i ,
\end{equation*}
where as before, the notation  $\sum$ stands for $\sum_{1 \leq i,j,l \leq 3 }$ in this proof and where
for $k=1,2,3,$ the matrices $\Lambda^k(\hat q)$ are given by:
\begin{equation*}
\Lambda^k(\hat q) = \left(\int_{\partial\Omega}\frac{\partial\varphi_i}{\partial \tau}(\hat q)\frac{\partial\varphi_j}{\partial \tau}(\hat q) K_k (\hat q,\cdot) \, {\rm d}s\right)_{1\leq i,j\leq 3}.
\end{equation*}
We recall that $K_1(\hat q,\cdot)=(x-\hat h)^\perp\cdot n$ and $K_j(\hat  q,\cdot)=n_{j-1}$ ($j=2,3$) on $\partial\Omega$. %\par
Now the quadratic form in $p$ can be symmetrized  as follows:
\begin{equation*}
\sum \left[\Lambda^l_{ij}(q)p_lp_j-\frac{1}{2}\Lambda^i_{jl}(q)p_lp_j\right]p^\ast_i
=\frac{1}{2} \sum \left[\Lambda^l_{ij} +\Lambda_{il}^j-\Lambda_{jl}^i\right](q) \, p_l \,p_j \, p^\ast_i,
\end{equation*}
which leads to \eqref{expre_quadratic}. This concludes the proof of Lemma~\ref{LEM_2}.
\qed 
\end{proof} 
%
%Lemma_3=================================================
%
%
%
Now we give the proof of Lemma~\ref{LEM_forme}.
\begin{proof}[Proof of Lemma~\ref{LEM_forme}.]
The quantity 
\begin{equation}
\label{shape_derive}
\frac{\partial}{\partial q}\left.\left(\int_{{\underline{\mathcal F}}(q)}\nabla{\underline{\varphi}}_i(q)\cdot\nabla{\underline{\varphi}}_j(q){\rm d}y\right)\right|_{q=\hat q}\cdot p^\ast 
\end{equation}
can be interpreted as the time derivative of the quantity between parentheses, when the outer boundary $\partial{\underline{\Omega}}(\hat q)$ undergoes a rigid displacement of velocity $w^\ast$.

More precisely,  denote by $\chi$ a cut-off function, compactly supported, valued in $[0,1]$ and such that $\chi=1$ in a  neighborhood of $\partial{\underline{\Omega}}(\hat q)$ and $\chi=0$ in a neighborhood of $\mathcal S_0$. Then, denote by $\xi(t,\cdot)$ the flow associated with the ODE:
\begin{equation} \label{ODExi}
\xi'(t,y)=\chi(\xi(t,y))\hat w^\ast(t,\xi(t,y)),\quad   \text{ for } t>0 ,   \text{  with }
\xi(0,y)=y,
\end{equation}
where $\hat w^\ast$ was introduced in \eqref{HatwAst}. Notice that: 
\begin{equation*}
\xi(t,y)=R(-t\omega^\ast)y-tR(\hat\vartheta)^t \, \ell^\ast,
\end{equation*}
in a neighborhood of $\partial{\underline{\Omega}}(\hat q)$ and $\xi(t,y)=y$ in a neighborhood of $\partial\mathcal S_0$.

For every $t$ small, define
$${\underline{\Omega}}_t :=\xi(t,{\underline{\Omega}}(\hat q))  \text{ and }{\underline{\mathcal F}}_t:=\xi(t,{\underline{\mathcal F}}(\hat q)).$$
For $j=1,2,3$, let ${\underline{\varphi}}_{j,t}$ be harmonic in $\underline{\mathcal F}_t$ and satisfy the Neumann boundary conditions:
\begin{subequations}
\begin{alignat}{3}
\frac{\partial{\underline{\varphi}}_{j,t}}{\partial n}&=\begin{cases}
(y-\hat h)^\perp\cdot n&j=1;\\
n_{j-1}&j=2,3\\
\end{cases}&\quad&
\text{ on }\partial\mathcal S_0,\\
\frac{\partial{\underline{\varphi}}_{j,t}}{\partial n}&=0\quad(j=1,2,3)&&\text{ on }\partial{\underline{\Omega}}_t.
\end{alignat}
\end{subequations}
With these settings, the quantity \eqref{shape_derive} can be rewritten as:
\begin{equation*}
\frac{d}{dt} \left. \left(\int_{\underline{\mathcal F}_t} \nabla{\underline{\varphi}}_{i,t}\cdot\nabla{\underline{\varphi}}_{j,t} \, {\rm d}x\right)\right|_{t=0}.
\end{equation*}
According to the Reynolds transport formula \eqref{eq:reynolds}, it can be expanded as follows:  
\begin{multline} \label{reynolds}
\frac{d}{dt} \left.\left(\int_{\underline{\mathcal F}_t}\nabla{\underline{\varphi}}_{i,t}\cdot\nabla{\underline{\varphi}}_{j,t} \, {\rm d}x\right)\right|_{t=0}
=
\int_{{\underline{\mathcal F}}(\hat q)}\nabla{\underline{\varphi}}'_i\cdot\nabla{\underline{\varphi}}_j \, {\rm d}x
+ \int_{{\underline{\mathcal F}}(\hat q)}\nabla{\underline{\varphi}}_i\cdot\nabla{\underline{\varphi}}'_j \, {\rm d}x \\
+ \int_{\partial{\underline{\Omega}}(\hat q)}(\nabla{\underline{\varphi}}_i\cdot\nabla{\underline{\varphi}}_j)(\hat w^\ast\cdot n) \, {\rm d}s,
\end{multline}
where 
$${\underline{\varphi}}'_j :=  \frac{\partial  {\underline{\varphi}}_{j,t}}{\partial t} \Big|_{t=0} .$$

%SubLemma forme################################################
\begin{lem} \label{SUBLEM_forme}
For $j=1,2,3$, the function ${\underline{\varphi}}'_j$ is  harmonic in ${\underline{\mathcal F}}(\hat q)$, satisfies 
\begin{equation} \label{bd-int}
\frac{\partial{\underline{\varphi}}_j'}{\partial n}=0\text{ on }\partial\mathcal S_0
\end{equation}
and 
\begin{equation} \label{bd-ext}
\frac{\partial{\underline{\varphi}}_j'}{\partial n}=
\frac{\partial}{\partial \tau} 
\Big( 
(\hat w^\ast\cdot n) \frac{\partial {\underline{\varphi}}_j}{\partial \tau} 
\Big)
\text{ on }\partial{\underline{\Omega}}(\hat q).
\end{equation}
\end{lem}
%
%#########################################################
%
Once Lemma~\ref{SUBLEM_forme} is proved, \eqref{ovv-eq} follows from \eqref{reynolds} and an integration by parts.
\qed 
\end{proof} 
\begin{proof}[Proof of Lemma~\ref{SUBLEM_forme}.]
The function  ${\underline{\varphi}}'_j$ is defined and harmonic in ${\underline{\mathcal F}}(\hat q)$ and the boundary conditions are obtained by differentiating with respect to $t$, at $t=0$,  the identities on the fixed boundaries $\partial\mathcal S_0$ and $\partial{\underline{\Omega}}(\hat q)$:
\begin{subequations}
\begin{alignat}{3}
\frac{\partial{\underline{\varphi}}_{j,t}}{\partial n} \big(\xi(t,\cdot)\big)&=\begin{cases}
(y-\hat h)^\perp\cdot n&j=1; \\
n_{j-1}&j=2,3; \\
\end{cases}&\quad&
\text{ on }\partial\mathcal S_0, \\
\label{ccla}
\frac{\partial{\underline{\varphi}}_{j,t}}{\partial n}\big(\xi(t,\cdot)\big)&=0\quad(j=1,2,3)&&\text{ on }\partial{\underline{\Omega}}(\hat q).
\end{alignat}
\end{subequations}
Let us focus on the proof of \eqref{bd-ext}, the proof of \eqref{bd-int} being quite similar with some simplifications.
On $\partial{\underline{\Omega}}(\hat q)$, using \eqref{ODExi}, we can write that:
\begin{equation} \label{ddtnablavarphi}
\frac{d}{dt} \left( \frac{\partial{\underline{\varphi}}_{j,t}}{\partial n} \big(\xi(t,\cdot)\big) \right)|_{t=0}=
\frac{\partial{\underline{\varphi}}_j'}{\partial n} 
 + \langle D^2{\underline{\varphi}}_j,\hat w^\ast,n\rangle + \omega^\ast
\frac{\partial{\underline{\varphi}}_j}{\partial \tau},
\end{equation}
where the last term is obtained by noticing that $n(\xi(t,\cdot))=R(-t\omega^\ast)n$.
Therefore by taking the derivative  at $t=0$ of the identity \eqref{ccla} and using \eqref{ddtnablavarphi} we obtain
\begin{align}
\nonumber
\frac{\partial{\underline{\varphi}}_j'}{\partial n}  &=
- \langle D^2{\underline{\varphi}}_j,\hat w^\ast,n\rangle 
- \omega^\ast
\frac{\partial{\underline{\varphi}}_j}{\partial \tau} 
\\ \label{tro}
 &= - \frac{\partial^2{\underline{\varphi}}_j}{\partial n^2}(\hat w^\ast\cdot n)
-\langle D^2{\underline{\varphi}}_j,\tau,n\rangle (\hat w^\ast\cdot\tau)
- \omega^\ast
\frac{\partial{\underline{\varphi}}_j}{\partial \tau} ,
\end{align}
by decomposing $w^\ast$ into normal and tangential parts.
Taking the tangential derivative of the identity \eqref{trala} and using the relation $\frac{\partial n}{\partial \tau}=\kappa\tau$ with $\kappa$  the local curvature of $\partial{\underline{\Omega}}(\hat q)$, we arrive at 
\begin{equation} \label{tangent_derivative}
\langle D^2{\underline{\varphi}}_j,\tau,n\rangle+\kappa\frac{\partial {\underline{\varphi}}^j}{\partial \tau}=0\quad\text{on }\partial\underline{\Omega}(\hat q) .
\end{equation}
We substitute \eqref{tangent_derivative} into \eqref{tro} to deduce:
\begin{align}
\nonumber
\frac{\partial{\underline{\varphi}}_j'}{\partial n}  &=
- \frac{\partial^2{\underline{\varphi}}_j}{\partial n^2}(\hat w^\ast\cdot n)
+ \Big( \kappa ( \hat w^\ast\cdot\tau ) - \omega^\ast \Big)
\frac{\partial{\underline{\varphi}}_j}{\partial \tau} 
\\  &=
- \frac{\partial^2{\underline{\varphi}}_j}{\partial n^2}(\hat w^\ast\cdot n)
+  \left( \frac{\partial}{\partial \tau}(\hat w^\ast\cdot n) \right) \frac{\partial{\underline{\varphi}}_j}{\partial \tau}  .
\end{align}
On $\partial{\underline{\Omega}}(\hat q)$, we have with local coordinates:
$$
\Delta{\underline{\varphi}}_j= \frac{\partial}{\partial\tau}\left(\frac{\partial {\underline{\varphi}}_j}{\partial \tau }\right)
+ \kappa \frac{\partial{\underline{\varphi}}_j}{\partial n} + \frac{\partial^2{\underline{\varphi}}_j}{\partial n^2} .
$$
Since ${\underline{\varphi}}_j$ is harmonic and $\frac{\partial{\underline{\varphi}}_j}{\partial n}=0$ on $\partial{\underline{\Omega}}(\hat q)$, we 
deduce that 
$\frac{\partial^2{\underline{\varphi}}_j}{\partial n^2}  = -
\frac{\partial^2{\underline{\varphi}}_j}{\partial \tau^2}$ on $\partial{\underline{\Omega}}(\hat q)$,  and therefore
\begin{equation*}
\frac{\partial{\underline{\varphi}}_j'}{\partial n}  
= (\hat w^\ast\cdot n) \frac{\partial}{\partial\tau}\left(\frac{\partial {\underline{\varphi}}_j}{\partial \tau }\right)
+  \left( \frac{\partial}{\partial \tau}(\hat w^\ast\cdot n) \right) \frac{\partial{\underline{\varphi}}_j}{\partial \tau}  ,
\end{equation*}
which is \eqref{bd-ext}.
\qed 
\end{proof}
\subsection{Identification of the electric potential: Proof of Lemma~\ref{LEM-pot-energy}}
\label{section122}
We now establish Lemma~\ref{LEM-pot-energy}. By definition, 
\begin{equation*}
C(q) = - \int_{\mathcal   F(q)} \nabla \psi (q, \cdot) \cdot \nabla \psi (q, \cdot)   \, {\rm d}x .
\end{equation*}
Thus, by Reynolds transport formula \eqref{eq:reynolds} we infer that for every $p$ in $\mathbb R^3$,
\begin{equation} \label{12novA}
D C(q) \cdot p =  - 2 \int_{\mathcal   F(q)} \nabla \left( \frac{\partial \psi}{\partial q }\cdot p \right)  \cdot \nabla \psi    \, {\rm d}x 
- \int_{\partial \mathcal   S(q)}  |\nabla  \psi|^2 u_{q'} \cdot n \, {\rm d}s  ,
\end{equation}
where $u_{q'}$ was defined in \eqref{DecompU}.
Upon an integration by parts, 
\begin{equation} \label{col1}
\int_{\mathcal   F(q)} \nabla \left( \frac{\partial \psi}{\partial q } \cdot p \right) \cdot \nabla \psi    \, {\rm d}x
= \int_{\partial \mathcal   S(q)} \left( \frac{\partial \psi}{\partial q } \cdot p \right)  \frac{\partial \psi}{\partial n}   \, {\rm d}s .
\end{equation}
Moreover gathering \eqref{tderiv2} and \eqref{tderiv3}  we arrive at 
\begin{equation} \label{col2}
\frac{\partial\psi}{\partial q}(q,x)\cdot p
= DC(q)\cdot p -  \frac{\partial\psi}{\partial n}(q,x) \frac{\partial\boldsymbol\varphi}{\partial n}(q,x)\cdot p .
 \qquad   \text{ for } x\in\partial\mathcal S(q) .
\end{equation}
Now combining \eqref{col1} and \eqref{col2} we obtain 
\begin{align}
\nonumber
\int_{\mathcal   F(q)} \nabla \left( \frac{\partial \psi}{\partial q } \cdot p \right) \cdot \nabla \psi    \, {\rm d}x
&= \int_{\partial \mathcal   S(q)} (DC(q)\cdot p )  \frac{\partial \psi}{\partial n}  \, {\rm d}s 
  - \int_{\partial \mathcal   S(q)} \left| \frac{\partial \psi}{\partial n} \right|^2  \frac{\partial\boldsymbol\varphi}{\partial n}\cdot p   \, {\rm d}s  \\ 
\label{12novB} 
&= - DC(q)\cdot p 
- \int_{\partial \mathcal   S(q)} \left| \frac{\partial \psi}{\partial n} \right|^2  \frac{\partial\boldsymbol\varphi}{\partial n}\cdot p   \, {\rm d}s ,
\end{align}
thanks to \eqref{circ-norma}. 
On the other hand since $\psi (q,\cdot )$ is constant on $\partial \mathcal   S(q)$, 
\begin{equation} \label{12novC}
\int_{\partial \mathcal   S(q)}  |\nabla  \psi|^2 u_{q'} \cdot n\, {\rm d}s 
=
\int_{\partial \mathcal   S(q)}  \left| \frac{\partial \psi}{\partial n} \right|^2  \frac{\partial\boldsymbol\varphi}{\partial n}\cdot p   \, {\rm d}s .
\end{equation}
Gathering \eqref{12novA}, \eqref{12novB} and \eqref{12novC} and recalling the definition \eqref{E-def} of $E$, we obtain \eqref{pot-energy}. This concludes the proof of Lemma~\ref{LEM-pot-energy}. \qed
\subsection{Conservation of energy: Proof of Proposition \ref{energy}}
\label{sectionProp1}

We prove Proposition \ref{energy}.
We start with the observation that 
\begin{equation} \label{conserv-proof1}
\big( \mathcal{E} (q,q') \big)' = M(q) q'' \cdot q' 
+ \frac{1}{2} (DM(q) \cdot q') q' \cdot q'  - \frac{1}{2} \gamma^{2}   DC (q) \cdot q'  .
\end{equation}
Now, thanks to \eqref{ODE_intro} and \eqref{def-upsilon},  
\begin{equation} \label{conserv-proof2}
M(q) q'' \cdot q' = - \langle \Gamma (q),q',q'\rangle  \cdot q' + F(q,q')  \cdot q' ,
\end{equation}
and
\begin{equation}
\label{dima}
F (q,q')   \cdot q'=  \gamma^2 E(q)   \cdot q' .
\end{equation}
We introduce the matrix for any $(q,p) $ in $\mathcal Q \times \R^{3}$, 
\begin{equation} \label{pad1}
S(q,p) := \left(\sum_{1\leq i\leq 3} \Gamma^k_{i,j}(q) p_i  \right)_{1\leq k,j\leq 3} ,
\end{equation}
so that
\begin{equation}
\label{pad2}
\langle \Gamma (q), p, p \rangle = S(q,p) p .
\end{equation}
Combining \eqref{conserv-proof1}, \eqref{conserv-proof2}, \eqref{dima},  \eqref{pad1} and  \eqref{pad2} we obtain
\begin{equation} \label{conserv-proof3}
\big( \mathcal{E} (q,q') \big)'  = \gamma^2  \left(  E(q)  - \frac{1}{2} DC (q) \right)   \cdot q'
+ \left( \frac{1}{2}  DM(q) \cdot q' - S(q,q') \right) q' \cdot q'   .
\end{equation}
The first term of the right hand side vanishes thanks to Lemma~\ref{LEM-pot-energy}.
Proposition~\ref{energy} follows then from the following result.
\begin{lem} \label{antis}
For any $(q,p) $ in $\mathcal Q \times \R^{3}$, the matrix
$\frac{1}{2}  DM(q) \cdot p - S(q,p) $ is skew-symmetric.
\end{lem}
%

%%%%%%%%%%%%%%%%%%%
%
\begin{proof}
We first observe that $DM(q) \cdot p$ is the $3 \times 3$ matrix containing the entries
\begin{equation*}
\sum_{1\leq k \leq 3} (M_a)_{i,j}^{k}  (q)  \, p_{k} , \text{ for } 1\leq i,j\leq 3,
\end{equation*}
where $(M_a)_{i,j}^{k} (q)$ is defined in \eqref{cri}. 
On the other hand, the matrix $S(q,p)$ contains the entries
\begin{equation*}
\frac{1}{2} \sum_{1\leq k \leq 3} \Big(  (M_a)_{i,j}^{k}  +    (M_a)_{i,k}^{j}   -    (M_a)_{k,j}^{i} \Big) (q) \, p_{k} ,
\end{equation*}
for $1\leq i,j\leq 3$.
Therefore, the matrix $DM(q) \cdot p - S(q,p)$  contains  the entries
\begin{equation*}
c_{ij} (q,p)  = -  \frac12 \sum_{1\leq k \leq 3}    \Big(   (M_a)_{i,k}^{j}   -    (M_a)_{k,j}^{i}  \Big) (q) \,  p_{k},
\end{equation*}
for $1\leq i,j\leq 3$.
Since the matrix $M(q)$ is symmetric,  $c_{ij}(q,p)  = - c_{ji} (q,p) $ for $1\leq i,j\leq 3$.
\qed 
\end{proof}
\subsection{The case of a disk: proof of Theorem~\ref{THEO-intro-disk}}
\label{tard-disk}
In this subsection, we suppose that ${\mathcal S}_{0}$ is a disk,  of center $\zeta$ and radius $r_{0}$.
We start by observing that, as noticed in Subsection~\ref{Subsec-recasting}, $\zeta = h_{c}(0) = h_{c}(0) - h(0)$, so with \eqref{def-zeta_vartheta} we deduce
$\zeta_{\vartheta} = h_{c} -h$,
that is, \eqref{OID2}. Relation \eqref{OID3} is an immediate consequence of the decomposition $x-h= (x-h_{c}) + (h_{c} -h)$,
of the fact that $(x-h_{c})\cdot n=0$ on $\partial\mathcal S(t)$, and of \eqref{EqTrans}-\eqref{EqRot}.
% and \eqref{OID2}.
Hence only \eqref{OID1} needs proving. \par
Concerning the added mass matrix, due to \eqref{Kir},  in this case,  $\partial_{n} \varphi_{1}(q,\cdot) = (h_{c}-h)^{\perp} \cdot n(\cdot)$, and consequently, 
\begin{equation} \label{Kir1Disk}
\varphi_{1} (q,\cdot) = - \zeta_{\vartheta,2} \, \varphi_{2}(q,\cdot) + \zeta_{\vartheta,1} \, \varphi_{3}(q,\cdot).
\end{equation}
We underline that $\varphi_{2}$ and $\varphi_{3}$ depend merely on $h_{c}$ while $\varphi_{1}$ depends also on $\vartheta$. 
From \eqref{Kir1Disk}, one deduces that for any $q$ in ${\mathcal Q}$ such that $d(B(h,r_{0}), \partial \Omega) >0$ and any $p=(\omega,\ell)$:
\begin{equation} \label{ReducEnergieAjoutee}
M_{a}(q) p \cdot p = \tilde{M}_{\flat}(h_{c}) p_{\flat} \cdot p_{\flat}
\end{equation}
with
\begin{equation} \label{Defpb}
p_{\flat} = p_{\flat}(\vartheta,\omega,\ell) := \ell + \omega \zeta_{\vartheta}^{\perp}.
\end{equation}
Notice that for a solution $(p,q): [0,T] \rightarrow {\mathcal Q} \times \R^{3}$ of the system with $p=q'$, for all times, 
\begin{equation} \label{Proppb}
p_{\flat}(\vartheta(t),\vartheta'(t),h'(t)) = h_{c}'(t).
\end{equation}
Now to establish \eqref{OID1} we rely on the following adaptation of Lemma~\ref{niou}.
\begin{lem} \label{niou-disk}
For any smooth curve $t\mapsto q(t)$ in $\mathcal Q$, for  every $p^\ast = (\omega^{\ast},\ell^{\ast})\in\mathbb R ^3$, 
\begin{equation} \label{full_RHS_disk}
\mathcal E  \mathcal L  = \tilde{M}_{\flat}(h_{c})p_{\flat}' \cdot p_{\flat}^\ast + \langle \Gamma_{\flat}(h_{c}),p_{\flat},p_{\flat}\rangle\cdot p_{\flat}^\ast ,
\end{equation}
where
\begin{equation} \label{Defpbstar}
p_{\flat}^\ast=p_{\flat}^\ast(\vartheta,\omega^{\ast},\ell^{\ast}) := \ell^{\ast} + \omega^{\ast} \zeta_{\vartheta}^{\perp},
\end{equation}
with $\mathcal E \mathcal L$ is given by \eqref{def-EL}, $\tilde{M}_{\flat}(h_{c})$ is defined by \eqref{def-Ma-Omega}-\eqref{DefTildeMb} and $\Gamma_{\flat} (h_{c})$ is defined in \eqref{def_vraiment-Gamma-disk}.
\end{lem}
\begin{proof}
Mimicking the proof of Lemma~\ref{niou}, using $\frac{d}{dt} p_{\flat}^{\ast} = - \vartheta' \omega^{\ast} \zeta_{\vartheta}$, \eqref{DefTildeMb} and \eqref{ReducEnergieAjoutee}
we obtain
\begin{equation*}
\frac{d}{dt} \frac{\partial {\mathcal E}_{1}}{\partial p} \cdot p^{\ast} 
= \tilde{M}_{\flat} p_{\flat}' \cdot p_{\flat}^{\ast} + (D_{h_{c}} \tilde{M}_{\flat} \cdot h_{c}') p_{\flat} \cdot p_{\flat}^{\ast} - \tilde{M}_{\flat} p_{\flat} \cdot (\vartheta' \omega^{\ast} \zeta_{\vartheta}).
\end{equation*}
On the other hand, 
\begin{equation*}
\frac{\partial {\mathcal E}_{1}}{\partial q} \cdot p^{\ast} =  \frac{1}{2} (D_{h_{c}} \tilde{M}_{\flat} \cdot \omega^{*} \zeta_{\vartheta}^{\perp}) p_{\flat} \cdot p_{\flat}
+ \tilde{M}_{\flat} (\vartheta' \omega^{\ast} \zeta_{\vartheta}) \cdot p_{\flat}.
\end{equation*}
We deduce that 
\begin{equation*}
\mathcal E  \mathcal L  = \tilde{M}_{\flat} p_{\flat}' \cdot p_{\flat}^{\ast} + (D_{h_{c}} \tilde{M}_{\flat} \cdot h_{c}') p_{\flat} \cdot p_{\flat}^{\ast} - \frac{1}{2} (D_{h_{c}} \tilde{M}_{\flat} \cdot \omega^{*} \zeta_{\vartheta}^{\perp}) p_{\flat} \cdot p_{\flat},
\end{equation*}
and  conclude as in Lemma~\ref{niou}, by symmetrizing the second term.
\qed
\end{proof}
Theorem~\ref{THEO-intro-disk} finally follows from Lemmas~\ref{LEM_3}, \ref{LEM_easy}
and a rewriting of $E(q)$ and $B(q)$ taking \eqref{Kir1Disk} into account. 
%
%
%%%%%%%%%%%%%%%%%%%%%%%%%%%%%%%%%%%%%%%%%%%%%%%%%%%%%%%%%%%%%%%%%%%%%%%%%%%%%%%%%%%%%%%%%%%%%%%%%%%%%%%%%%%%%%%%%%%%%%%%%%%%%%%%%%%%%%
%
%
%
%
%
%
\section{Convergence to the massive point vortex system in Case (i): Proof of Theorem~\ref{theo:3}}
\label{Sec:PreuveTheo3}
In this section we prove Theorem \ref{theo:3} which corresponds to Case (i). We will rely on intermediate results (Proposition~\ref{Pro-fnormA}, Lemma~\ref{drift} and Lemma~\ref{kin-eps}). The proofs of these intermediate results will be postponed to the last sections. \par
We work on Equation~\eqref{ODE_intro} with a small solid, that is
\begin{align}
 \label{ODE_intro_eps}
M_\varepsilon (q_\varepsilon)   q_\varepsilon'' +  \langle \Gamma_\varepsilon  (q_\varepsilon),q_\varepsilon',q_\varepsilon'\rangle &= F_\varepsilon (q_\varepsilon,q_\varepsilon') ,
\end{align}
where
$$ M_\varepsilon  :=  M [\mathcal S_{0,\varepsilon}, m_\varepsilon ,\mathcal{J}_\varepsilon , \Omega ], \
\Gamma_\varepsilon :=   \Gamma [\mathcal S_{0,\varepsilon}  ,\Omega ]
\text{ and } {F}_\varepsilon  := {F} [\mathcal S_{0,\varepsilon} , \gamma ,\Omega ],$$
with
$m_\varepsilon ,\mathcal{J}_\varepsilon$  given by \eqref{mass-inertie}. 
The functions $M_\varepsilon (q)$, $ \langle \Gamma_\varepsilon (q)  ,p,p\rangle $ and ${F}_\varepsilon (q,p)$ are defined for $q$ in ${\mathcal Q}^{\varepsilon}$ and for $p$ in $\R^3$. \par
We begin by introducing some notations. Given $\delta > 0$ and $\eps_{0} $ in $(0,1)$, we let 
\begin{gather}
\label{Qeps}
\mathfrak Q :=  \{(\varepsilon, q) \in (0,1) \times \mathbb R^3 \ :\  d(\mathcal S_\eps(q),\partial\Omega)>0 \}  , \\
\index{AQ4@$\mathfrak Q$: bundle of shrinking body positions without collision}
\label{MathfrakQdelta}
\mathfrak Q^{\delta} := \{(\varepsilon, q) \in (0,1) \times \mathbb R^3 \ :\  d(\mathcal S_\eps(q),\partial\Omega)> \delta \}  , \\
\index{AQ5@$\mathfrak Q^{\delta}$: bundle of shrinking body positions at distance $\delta$ from the boundary}
\label{deqQed}
\mathfrak Q_{\delta,\eps_{0}} :=  \{(\varepsilon, q) \in (0,\varepsilon_{0}) \times \mathbb R^3 \ :\  d(\mathcal S_\eps(q),\partial\Omega)> \delta \}  .
\index{AQ6@$\mathfrak Q_{\delta,\varepsilon_{0}}$: bundle of shrinking body positions at distance $\delta$ from the boundary with $\varepsilon < \varepsilon_{0}$}
\end{gather}
For $\delta > 0$, we also introduce

\begin{equation} \label{Omega_delta}
\Omega_{\delta} := \{x \in  \Omega \ :\  d (x, \partial \Omega ) > \delta \} .
\end{equation}
\index{BGrecZO2@$\Omega_{\delta}$: set of  points at distance $\delta$ from the boundary}
Observe that despite the fact that the center of mass $h_{\eps}$ does not necessarily belong to ${\mathcal S}_{\eps}(q)$, we have the following elementary result whose proof is left to the reader.
\begin{lem} \label{banane}
Let $\delta > 0$. There exists $\delta_0 $ in $(0, \delta )$ and $\eps_{0} $ in $(0,1]$ such that for any $(\eps, q)$ in $\mathfrak Q_{\delta ,\eps_{0}}$, with $q=(\vartheta , h)$,
necessarily $h$ belongs to $\Omega_{\delta_{0}}$. 
\end{lem}
%
%
%%%%%%%%%%%%%%%%%%%%%%%%%%%%%%%%%%%%%%%%%%%%%%%%%%%%%%%%%%%%
%
%
\subsection{Normal form}
%
%
%In the Case (i) considered in this section, 
We will rephrase \eqref{ODE_intro_eps} to be able   to pass to the limit as $\varepsilon$ goes to 0.
The following definition will be useful to deal with the remainder.
\begin{defn} \label{def-weakly-nonlinear}Let $\delta>0$ and $\varepsilon_0\in(0,1)$ be given. 
We say that a vector field $F$ in $L^{\infty}_\mathrm{loc} (\mathfrak Q_{\delta,\eps_{0}} \times \R^3; \R^3)$ is weakly nonlinear if 
there exists $K>0$   depending on  $\mathcal S_0$, $m$, ${\mathcal J}$, $\gamma$, $\Omega$ and $\delta$
such that for any $(\eps, q, p) $ in $\mathfrak Q_{\delta,\eps_{0}} \times \R^3 $, 
\begin{equation} \label{ineq-H_r}
| F  (\eps , q, p) |_{\R^{3}} \leq K ( 1 + | p|_{\R^{3}}  + \eps | p|_{\R^{3}}^{2} ) .
\end{equation}
\end{defn}
The normal form is as follows.
\begin{prop} \label{Pro-fnormA}
There exists $F_{r}:\mathfrak Q \times \R^3 \rightarrow \R^3$ depending on  $\mathcal S_0$, $\gamma$  and $\Omega$ such that,
for any $\delta > 0$, there exists $\eps_{0} $ in $(0,1)$ such that
$F_{r}$ belongs to $L^{\infty}_\mathrm{loc}(\mathfrak Q_{\delta,\eps_{0}} \times \R^3; \R^3)$ and is weakly nonlinear in the sense
of Definition~\ref{def-weakly-nonlinear}
and such that Equation \eqref{ODE_intro_eps} can be recast as
\begin{align} \label{fnorm1} 
M_g {p}_\varepsilon ' 
= {F}^{\Ext}_{\vartheta_\varepsilon} (\varepsilon \vartheta_\varepsilon',  h_\varepsilon'  - \gamma  u^\Omega (h_\varepsilon ) )
+ \varepsilon {F}_{r} (\eps,q_\varepsilon, {p}_\varepsilon ) .
\end{align}
%\\
\end{prop}
Recall that $u^\Omega$ was defined in \eqref{DefUOmega}, that the force term ${F}^{\Ext}_{\vartheta_\varepsilon}(p)$ was defined in the section dealing with the case without outer boundary, see \eqref{force-ext}, and that ${p}_\varepsilon=(\varepsilon \vartheta_{\varepsilon}',h'_{\varepsilon})^{t}$ was defined in \eqref{def-hat}.

The normal form  \eqref{fnorm1} will be useful in order to pass to the limit.
To get Proposition~\ref{Pro-fnormA}, we will perform expansions of the inertia matrix, of the Christoffel symbols and of the force terms with respect to $\eps$.
Roughly speaking the leading terms coming from the force terms will be gathered into the first term of the right hand side of \eqref{fnorm1}; see \eqref{EB1}. 
\subsection{Renormalized energy estimates}
We will of course need uniform estimates
as $\varepsilon \rightarrow 0^{+}$ in order to pass to the limit in \eqref{fnorm1}. The energy is the natural candidate to yield such estimates. Hence we are led to consider the behavior of the energy with respect to $\varepsilon$. 
We therefore now index the energy (introduced in  in \eqref{pader}) as follows:
\begin{equation} \label{conserv-eps}
\mathcal{E}_{\varepsilon} (q,p) := \frac{1}{2} M_{\varepsilon}(q) p \cdot p
+ U_{\varepsilon}(q) ,
\index{AE2@$\mathcal{E}_{\varepsilon} (q,p)$:  total energy of the shrinking solid}
\end{equation}
where the potential energy $U_{\varepsilon}$ is given by
\begin{equation} \label{pot-eps}
U_{\varepsilon}(q) := - \frac{1}{2} \gamma^{2}  C_{\varepsilon} (q)  .
\end{equation}
Of course Proposition~\ref{energy} can be applied for each $\varepsilon $ in $(0,1)$ so that the energy associated with a solution 
$q_\eps  $ as in Theorem \ref{theo:3} is conserved along time until its maximal time of existence $T_{\eps}$.
We will establish in Subsection \ref{proof-deux}
the following result regarding the expansion of $C_{\varepsilon} (q)$  with respect to $\varepsilon$. 
The expansion is uniform, in the sense that the remainder is uniformly bounded, as long as the solid stays at a positive distance from the external boundary.
Let us recall that the Newtonian potential $G$ was introduced in \eqref{NewtonianPotential}, the Kirchhoff-Routh stream function $\psi^\Omega$ was defined in \eqref{def-KRS} and
the constant $C^{\Ext}$ in \eqref{model_stream-1}. 
We will also use the function defined for $q:= (\vartheta , h)$ in $\R \times \Omega$, 
\begin{equation} \label{PsiC}
\psi_c (q) :=  D_h \psi^{\Omega} (h) \cdot {\zeta}_\vartheta .
\end{equation}
\index{BGrecYP3@$\psi_c$: corrector stream function}
Above $D_h$ denotes the derivative with respect to $h$.
\begin{lem} \label{drift}
There exists a function $C_{r} :\mathfrak Q \rightarrow \R$ such that, for any  $\delta > 0$, there exists $\eps_{0}$ in $(0,1)$ such that $C_{r}$ is in $L^\infty (\mathfrak Q_{\delta,\eps_{0}} ; \R)$ and that for any $(\varepsilon ,q) $ in $ \mathfrak Q_{\delta,\eps_{0}}$, 
\begin{equation} \label{expansion_mu-intro}
C_{\varepsilon}  (q) = 
- G(\varepsilon) + C^{\Ext} 
+ 2  \psi^\Omega (h)
+ 2  \eps  \psi_c (q)
+ \eps^{2} C_r (\varepsilon ,q) .
\end{equation}
\end{lem}
%
%We recall that $C^{\Ext}$ was introduced in \eqref{model_stream-1}.
%
This result establishes that the potential energy  $U_{\varepsilon}(q)$ diverges logarithmically as $\eps \rightarrow 0^{+}$.
However  since they do not depend on the solid position and velocity
the contributions of the two first terms of \eqref{expansion_mu-intro} can be discarded from the energy in  \eqref{conserv-eps} without altering its conservation property.

On the other hand an immediate consequence of Proposition~\ref{dev-added} below
is the following result regarding the kinetic energy part.
\begin{lem} \label{kin-eps}
There is a function $M_{r}:\mathfrak Q \rightarrow \R^{3 \times 3}$ depending on $\mathcal S_0$ and $\Omega$, such that,
for any $\delta > 0$, there exists $\eps_{0} $ in $(0,1)$ such that $M_{r}$ is in $ L^\infty (\mathfrak Q_{\delta,\eps_{0}} ; \R^{3 \times 3})$
and such that for all $(\varepsilon ,q) $ in $\mathfrak Q$,  for all $p $ in $\R^3$, 
\begin{equation*}
\frac{1}{2} M_\varepsilon (q)  p \cdot p  
= 
\frac{1}{2} ( (M_g + \eps^{2}\, M^{\Ext}_{a, \vartheta})  I_{\varepsilon} p ) )\cdot  (I_{\varepsilon} p ) 
+ \frac{1}{2} \eps^4  (M_{r}   (\eps ,q )    I_{\varepsilon} p)   \cdot  ( I_{\varepsilon} p)  .
\end{equation*}
\end{lem}
Recall that $M^{\Ext}_{a, \vartheta}$ was defined in \eqref{g+a_ext} and $I_{\varepsilon}$ in \eqref{Eq:Ieps}.
Combining Lemma~\ref{drift} and Lemma~\ref{kin-eps} we obtain the following.
\begin{cor} \label{energy-eps-2}
Let $q_\eps$ and $T_{\eps}$ be as in Theorem \ref{theo:3}.
Then the renormalized energy 
\begin{multline*} \label{conserv-eps-reno-2}
 \frac{1}{2} ( (M_g + \eps^{2}\, M^{\Ext}_{a, \vartheta}) {p}_\varepsilon  )\cdot  {p}_\varepsilon
 - \gamma^2 \psi^\Omega (h_{\varepsilon}) \\
+ \frac{1}{2} \eps^4   M_{r}   (\eps ,q_{\varepsilon} )   {p}_{\varepsilon}  \cdot  {p}_{\varepsilon} 
- \eps \gamma^{2}  \psi_c (q_{\varepsilon} ) 
- \frac{1}{2} \eps^{2} \gamma^{2} C_r (\varepsilon ,q_{\varepsilon} )   ,
\end{multline*}
 is constant in time  until $T_{\eps}$.
\end{cor}
The two most important terms in the renormalized energy  above are the first and second ones which are respectively of order $O(   | {p}_\varepsilon |_{\R^{3}}^{2} )$ and $O(1)$ as long as there is no collision. Hence we deduce the following counterpart of Corollary~\ref{bd-loin}.
\begin{cor} \label{bd-loin-(i)}
Let $(q_\eps, T_{\varepsilon})$  as in Theorem~\ref{theo:3}.
Let $\delta>0$. There exists $K>0$ (depending on $\mathcal S_0$, $\Omega$, $p_0$, $\gamma$, $m_{1}$, $\mathcal J_{1}$, $\delta$) and $\varepsilon_{0}>0$ such that for any $\eps$ in $(0,\eps_{0})$,
as long as $(\eps,q_\eps)$ belongs to $\mathfrak Q_{\delta,\eps_{0}}$,
one has $|{p}_\varepsilon|_{\R^{3}} \leq K$.
\end{cor}
\subsection{Passage to the limit}
\label{Subsec:PTTLi}
We deduce from Corollary \ref{bd-loin-(i)} two different results. 
The first result concerns the lifetime $T_{\eps}$ of the solution $(q_\varepsilon,{p}_\varepsilon)$, which can be only limited by a possible encounter between the solid and the boundary $\partial \Omega$.
\begin{lem} \label{nrj-liminf}
There exist $\varepsilon_{0}>0$, $\underline{T}>0$ and $\underline{\delta} >0$, such that for any $\varepsilon $ in $(0,\varepsilon_{0})$, 
\begin{equation} \label{Eq:MinimalTime}
T^{\varepsilon} \geq \underline{T} \ \text{ and  on } [0,\underline{T} ], \ 
(\eps,q_\eps )  \in  \mathfrak Q_{\underline{\delta}  ,\eps_{0}} .
\end{equation}
\end{lem}
\begin{proof}
Let us introduce
\begin{equation} \label{rayon}
R_{0} := \max \{ |x|, \ x \in \partial {\mathcal S}_{0} \},
\end{equation}
so that, whatever $t\geq0$, $\vartheta(t)$ in $\R$ and $\varepsilon $ in $(0,1)$, 
\begin{equation} \label{dsuneboule}
{\mathcal S}_{\varepsilon}(q_{\varepsilon} (t) ) \subset \overline{B}(h_{\varepsilon}(t),\varepsilon R_{0}).
\end{equation}
We introduce $\underline{\delta}:= \frac{1}{4} d(0, \partial \Omega)$
and $\varepsilon_{0} $ in $(0,1)$ (which may be reduced later) such that 
$\varepsilon_{0} R_{0} \leq \underline{\delta}$ and
\begin{equation} \label{secuInit}
\forall \varepsilon \in (0,\varepsilon_{0}] ,\quad
d\big( \overline{B} (0,\varepsilon R_{0}), \partial \Omega \big)
\geq \frac{3}{4} d(0, \partial \Omega).  
\end{equation}
We apply Corollary~\ref{bd-loin-(i)} with $\underline{\delta}$ to deduce that there exists $K>0$  such that, reducing $\varepsilon_{0}$ if necessary, 
\begin{equation} \label{EstVit}
|h'_{\varepsilon}| \leq K, \text{ for all } t \text{ for which } d\big({\mathcal S}_{\varepsilon}(q_{\varepsilon} (t) ), \partial \Omega\big) \geq \underline{\delta}.
\end{equation}
We introduce $\underline{T} := \min\left(1, \frac{d( 0, \partial \Omega)}{2K} \right)$,
and for $\varepsilon $ in $(0,\varepsilon_{0}]$,
\begin{equation*}
{\mathcal I}^{\varepsilon} =  \left\{ t \in [0,1] \ : \ \forall s \in [0,t], \ 
d\big(\overline{B}(h_{\varepsilon}(s), \varepsilon R_{0}), \partial \Omega\big) \geq \underline{\delta} \right\} .
\end{equation*}
The set ${\mathcal I}^{\varepsilon}$ is a closed interval containing $0$, according to  \eqref{secuInit}.
Consider $\tilde{T}_{\varepsilon} := \max \, {\mathcal I}^{\varepsilon}$, and let us show that $\tilde{T}_{\varepsilon} \geq \underline{T}$. 
Of course, if $\tilde{T}_{\varepsilon}=1$, then this is clear; let us suppose that $\tilde{T}_{\varepsilon} <1$.
This involves that  
$d \big( \overline{B}(h_{\varepsilon}(\tilde{T}_{\varepsilon}), \varepsilon R_{0}), \partial \Omega \big) = \underline{\delta} .$
Using $\varepsilon_{0} R_{0} \leq \underline{\delta}$ we deduce
$d(h_{\varepsilon}(\tilde{T}_{\varepsilon}), \partial \Omega) \leq 2 \underline{\delta} .$
With the triangle inequality and $\underline{\delta}= \frac{1}{4} d(0, \partial \Omega) $  we infer that 
$d \big( h_{\varepsilon}(\tilde{T}_{\varepsilon}), 0 \big)  \geq \frac{1}{2} d(0, \partial \Omega).$
Now the relation \eqref{dsuneboule} implies that for all $t  $ in $ [0,\tilde{T}_{\varepsilon}]$,
\begin{equation} \label{inclus}
d \big( {\mathcal S}_{\varepsilon}(q_{\varepsilon} (t) ), \partial \Omega \big) \geq
d \big( \overline{B}(h_{\varepsilon}(t), \varepsilon R_{0}), \partial \Omega \big)
\geq \underline{\delta}  ,
\end{equation}
so that \eqref{EstVit} is satisfied during $[0,\tilde{T}_{\varepsilon}]$.
We deduce that $K \tilde{T}_{\varepsilon} \geq d(0, \partial \Omega)/2$, so $\tilde{T}_{\varepsilon} \geq \underline{T}$. %\par
Therefore for any $t$ in $[0,\underline{T}]$, for any $\varepsilon $ in $[0,\varepsilon_{0}]$,  \eqref{inclus} holds true.
This concludes the proof of Lemma~\ref{nrj-liminf}.
\qed
\end{proof} 
The second result establishes the desired convergence on any time interval during which we have a minimal distance between ${\mathcal S}_{\varepsilon}(q)$ and $\partial \Omega$, uniform for small $\varepsilon$.
Let us recall that $(\overline{h},\overline{T})$ denotes the maximal solution to \eqref{ODE-mass}.
\begin{lem} \label{nrj-cvi}
Let $\varepsilon_{1}>0$, $\check{\delta}>0$ and $\check{T} > 0$ with  $\check{T} < \overline{T}$,
and suppose that for any $\varepsilon $ in $(0,\varepsilon_{1} )$, 
\begin{equation} \label{HypUniformite}
d({\mathcal S}_{\varepsilon}(q_{\varepsilon}(t)), \partial \Omega) \geq \check{\delta} \; \text{ on } [0,\check{T}] .
\end{equation}
Then
$(h_\varepsilon, \varepsilon \vartheta_\varepsilon) \longrightharpoonup (\overline{h}, 0)$ in $W^{2,\infty}([0,\check{T}];\mathbb R^3)$ weak-$\star$.
\end{lem}
The proof of Lemma~\ref{nrj-cvi} consists in passing to the weak limit, with the help of all a priori bounds, in each term of \eqref{fnorm1}. 
It is a straightforward extension of the proof of Theorem \ref{theo:3EXTE} and it is therefore omitted. \par
\ \par
%
%%%%%%%%%%%%%%%%%%%%%%%%%%%%%%%%%%%%%%%%%%%%%%%%%%%%%%%%%%%%
%
%
We now finish the proof of Theorem~\ref{theo:3}. It only remains to extend the time interval on which the above convergences are valid to any closed subinterval of $[0,\overline{T})$. 
Hence let $T \in (0,\overline{T})$, and let us prove that for small $\varepsilon>0$ the time of existence $T^{\varepsilon}$ is larger than $T$ and establish the convergences on the time interval $[0,T]$. 
For such a $T$, we know that there exists $\overline{d}>0$ such that 
\begin{equation} \label{Eq:DBarre}
\forall t \in \left[0,\frac{T+\overline{T}}{2}\right], \ \ d(\overline{h}(t), \partial \Omega) \geq \overline{d}. 
\end{equation}
We let
$\overline{T}_{\varepsilon}:= \max \left\{ t>0 \: \ d\big(B(h_{\varepsilon}(t),\varepsilon R_{0}), \partial \Omega\big) \geq \overline{d}/2 \right\}.$
Let us recall that $R_{0}$ is defined in \eqref{rayon}.
Using Lemma~\ref{nrj-liminf} we deduce that, reducing $\overline{d}$ if necessary,
for some $\overline{\varepsilon}>0$,  $\inf_{ \varepsilon \in (0,\overline{\varepsilon}]} \overline{T}_{\varepsilon} >0$.
Therefore 
$\tilde{T}:=\liminf_{\varepsilon \rightarrow 0^+} \overline{T}_{\varepsilon} $ satisfies $ \tilde{T}>0 .$
Due to Corollary~\ref{bd-loin-(i)}, there exists $K>0$ and $\varepsilon_{0}$ such that for all $t $ in $[0,T+1]$ and $\varepsilon $ in $(0,\varepsilon_{0})$, 
\begin{equation} \label{enj2}
|h'_{\varepsilon}| + | \varepsilon \vartheta_{\varepsilon}' | \leq K \ \text{ as long as } d({\mathcal S}_{\varepsilon}(q_{\varepsilon}(t)), \partial \Omega) \geq \overline{d}/2.
\end{equation}
Now we claim that
\begin{equation} \label{TTilde}
\tilde{T}  \geq T + \frac{\overline{T}-T}{4}.
\end{equation}
Suppose that this is not the case, so that there is a sequence $\varepsilon_{n} \rightarrow 0^{+}$ such that $\overline{T}_{\varepsilon_{n}} \rightarrow \tilde{T} <  T + \frac{\overline{T}-T}{4}$.
Now for any $\eta $ in $(0,\tilde{T})$, on the interval $[0,\tilde{T}-\eta]$, the condition $ d(\overline{B}(h_{\varepsilon_{n}}(t),\varepsilon R_{0}), \partial \Omega) \geq \overline{d}/2$ is satisfied for $n$ large enough.
Moreover, for such $n$,  for all $t  $ in $ [0,\tilde{T}-\eta]$, \eqref{dsuneboule} implies that 
\begin{equation*}
d({\mathcal S}_{\varepsilon_{n}}(q_{\varepsilon_{n}}(t)), \partial \Omega) \geq
d \big( \overline{B}(h_{\varepsilon_{n}}(t), \varepsilon_{n} R_{0}), \partial \Omega \big) \geq   \overline{d}/2 .
\end{equation*}
Hence applying Lemma~\ref{nrj-cvi}, we deduce the uniform convergence of $(h_{\varepsilon_{n}} )$ to $\overline{h}$ on $[0,\tilde{T}-\eta]$.
In particular, as $n \rightarrow +\infty$,

$$d \big( h_{\varepsilon_{n}}(\tilde{T}-\eta) , \partial \Omega \big) 
\rightarrow d \big( \overline{h}(\tilde{T}-\eta), \partial \Omega \big) \geq \overline{d} ,$$
according to \eqref{Eq:DBarre}.
On the other hand by definition of $\overline{T}_{\varepsilon_{n}}$ 
$$d\big(B(h_{\varepsilon_{n}}(\overline{T}_{\varepsilon_{n}}),\varepsilon_{n} R_{0}), \partial \Omega\big) = \overline{d}/2 .$$
Using the triangle inequality and 
$\overline{T}_{\varepsilon_{n}} \rightarrow \tilde{T}$,  we get a contradiction with \eqref{enj2} for $\eta$ small enough.
Hence \eqref{TTilde} is valid, so that, reducing $\overline{\varepsilon}$ if necessary, 
 $\inf_{ \varepsilon \in (0,\overline{\varepsilon}]} \overline{T}_{\varepsilon} \geq T$.
Now, applying again \eqref{dsuneboule} and Lemma~\ref{nrj-cvi}, we reach the conclusion. This ends the proof of Theorem~\ref{theo:3}. 
%\qed
%
%
%
%
%
%%%%%%%%%%%%%%%%%%%%%%%%%%%%%%%%%%%%%%%%%%%%%%%%%%%%%%%%%%%%%%%%%%%%%%%%%%%%%%%%%%%%%%%%%%%%%%%%%%%%%%%%%%%%%%%%%%%%%%%%%%%%%%%%
%
%
%
%
%
\section{Convergence to the point vortex system in Case (ii): Proof of Theorem \ref{theo:1}}
\label{Sec:PreuveTheo1}
In this section we prove Theorem \ref{theo:1} which corresponds to Case (ii). We work again on Equations~\eqref{ODE_intro_eps} but here $m_\varepsilon$, and $\mathcal{J}_\varepsilon$ are given by \eqref{mass-inertie2}. Our analysis relies on some normal forms (see Propositions~\ref{Pro-fnorm} and \ref{Pro-fnorm-ball-inh} below) which proved in subsequent sections. 
We begin this section with the case where ${\mathcal S}_{0}$ is a homogeneous disk because in this situation it is simpler to deduce the convergence from the normal form. Then we treat the case when ${\mathcal S}_{0}$ is not a disk, and finish with the case of a non-homogeneous disk. 
\subsection{The case of a homogeneous disk}
\label{SubsecHomDisk}
When ${\mathcal S}_{0}$ is a homogeneous disk, the dynamics of the angle $\vartheta_{\varepsilon}$ is trivial; we expand the dynamics of the center of mass $h_{\varepsilon}$ in powers of $\varepsilon$.

Without loss of generality, we assume that ${\mathcal S}_{0}=\overline{B}(0,1)$. %\par
We first modify a bit Definition~\ref{def-weakly-nonlinear} as follows.
\begin{defn} \label{def-WNL2}
Let $\delta>0$ and $\varepsilon_0\in(0,1)$ be given. 
We say that a vector field $F_{\flat}$ in $L^{\infty}_\mathrm{loc}((0,\varepsilon_0) \times \Omega \times \R^2; \R^2)$ is weakly nonlinear if 
there exists $K>0$ depending on  $m$, ${\mathcal J}$, $\gamma$, $\Omega$ and $\delta$
such that for any $\varepsilon$ in $(0,\varepsilon_{0})$, any $h$ in $\Omega$ such that $d(B(h,\varepsilon),\partial \Omega) > \delta$ and any $\ell \in \R^{2}$, 
\begin{equation} \label{ineq-WNL2}
| F_{\flat} (\eps , h, \ell) |_{\R^{2}} \leq K ( 1 + | \ell |_{\R^{2}}  + \eps | \ell|_{\R^{2}}^{2} ) .
\end{equation}
\end{defn}
\noindent
We used the notation \eqref{MinimalDistance}.
In this simpler situation we obtain the following result.
\begin{prop} \label{Pro-fnorm-ball}
There exists $F_{\flat,r}$ in $L^{\infty}_\mathrm{loc}((0,1) \times\Omega \times \R^2; \R^2)$ such that, for any $\delta>0$, there exists $\varepsilon_{0}$ in $(0,1)$ for which $F_{\flat,r}$ is weakly nonlinear in the sense of Definition~\ref{def-WNL2} and such that Equation~\eqref{ODE_intro-disk} can be recast as
\begin{equation*}
\vartheta_{\varepsilon}'' =0,
\end{equation*}
and
\begin{multline} \label{fnorm2-ball}
\Big( \eps^{\alpha}\,  m_{1} + \eps^2 \pi \Big) \big( h_\varepsilon' - u^{\Omega}(h_{\varepsilon}) \big)'
= \gamma ( h_{\varepsilon}' - u^{\Omega}(h_{\varepsilon}) )^{\perp} \\
+ \varepsilon^{\min(2,\alpha)}  F_{\flat,r} (\eps,h_\varepsilon, h_\varepsilon' )  .
\end{multline}
\end{prop}
Since this proposition is a particular case of Proposition~\ref{Pro-fnorm} (or of Proposition~\ref{Pro-fnorm-ball-inh}), we make no specific proof. \par
Note that the coefficient $\pi$ comes from the fact that $M^{\Ext}_{\flat} =\pi \mbox{\rm Id}_{2}$ in this case (where $M^{\Ext}_{\flat}$ is defined in \eqref{def-Ma}).
As a consequence, Corollary~\ref{bd-loin-(i)} has the following counterpart.
\begin{cor} \label{bd-disk}
Let $h_\eps$ satisfy the assumptions of Theorem~\ref{THEO-intro-disk}.
Let $\delta>0$. There exists $K>0$ (depending on  $\Omega$, $\ell_0$, $\gamma$, $m_{1}$, $\delta$) and $\varepsilon_{0}>0$ such that for any $\eps$ in $(0,\eps_{0})$, for any $t>0$, as long as $d(B(h_{\varepsilon},\varepsilon),\partial \Omega) > \delta$
one has $\varepsilon^{\min(1,\frac{ \alpha }{2 })} \, |{h}'_\varepsilon|_{\R^{2}} \leq K$.
\end{cor}
\begin{proof}
The proof is almost the same as for Corollary~\ref{bd-loin-(i)}, but we have to take into account that the added mass matrix is degenerate and that dynamics of the rotation angle is trivial. \par
In the case of a homogeneous disk,the energy given in \eqref{conserv-eps} can be described by using the following function, where $q=(\vartheta,h)$ and $p=(\omega,\ell)$:
\begin{equation} \label{conserv-eps-disk}
\tilde{\mathcal{E}}_{\varepsilon} (h,\omega,\ell) = \frac{1}{2} \left(\varepsilon^{\alpha} m_{1} \mbox{\rm Id}_{2} + \varepsilon^{2} \tilde{M}_{\flat}(h) \right) \ell \cdot \ell
+ \varepsilon^{\alpha+2}{\mathcal J}_{1} \omega^{2}
+ \tilde{U}_{\varepsilon}(h) .
\end{equation}
Here we wrote $\tilde{U}_{\varepsilon}(h):=U_{\varepsilon}(q)$ for $q=(\vartheta,h)$ since actually it does not depend on the angle $\vartheta$. \par
Now using the same analysis as for Corollary~\ref{energy-eps-2} and taking into account that $\vartheta_{\varepsilon}'(t)=\omega_{0}$ for all times, we deduce that the following quantity is conserved over time:
\begin{multline*}
\frac{1}{2} \left(\varepsilon^{\alpha} m_{1} \mbox{\rm Id}_{2} + \varepsilon^{2} \tilde{M}_{\flat}(h) \right) h'_{\varepsilon} \cdot h'_{\varepsilon}
- \gamma^2 \psi^\Omega (h_{\varepsilon}) \\
+ \frac{1}{2} \eps^4 \tilde{M}_{\flat,r} (\eps, h_{\varepsilon}) h'_{\varepsilon} \cdot h'_{\varepsilon} 
+ \eps \tilde{C}_r (\varepsilon ,h_{\varepsilon} ) ,
\end{multline*}
where $\tilde{C}_r (\varepsilon, h_{\varepsilon}):= C_r (\varepsilon ,q_{\varepsilon})$ since it does not depend on the angle, and $\tilde{M}_{\flat,r}$ is a bounded function on sets for which $d(B(h_{\varepsilon},\varepsilon), \partial \Omega) > \delta$. The conclusion follows.
\qed
\end{proof}
To improve the estimates on ${h}'_\varepsilon$, we obtain modulated energy estimates. The following lemma relies on straightforward computations.
\begin{lem} \label{nrj11-disk}
Let $h_{\varepsilon}$ satisfy the assumptions of Theorem~\ref{THEO-intro-disk}.
Then, during their lifetime, regular solutions satisfy 
\begin{equation} \label{Leibencore2-disk}
\frac{d}{dt} \frac{1}{2} \left( \varepsilon^{2} \pi + \varepsilon^{\alpha} m_{1}\right) | h'_{\varepsilon} - u^{\Omega}(h_{\varepsilon}) |^{2}
= \varepsilon^{\min(2,\alpha)} (h'_{\varepsilon} - u^{\Omega}(h_{\varepsilon})) \cdot F_{\flat, r} (\eps , h_\varepsilon, h'_\varepsilon)  .
\end{equation}
\end{lem}
Moreover we have the following immediate estimate.
\begin{lem} \label{drifterisok-disk}
Let $\delta > 0$, there exists $\eps_{0}$ in $(0,1)$ and $K>0$ such that for any $\varepsilon \in (0,\varepsilon_{0})$, any $h \in \Omega$ such that $d(B(h,\varepsilon), \partial \Omega) > \delta$, we have $|u^\Omega(h)|_{\R^{2}} \leq K$.
\end{lem}
We conclude that Corollary~\ref{bd-disk} can be improved into the following. 
\begin{cor} \label{bd-disk2}
Let $h_\eps$ satisfy the assumptions of Theorem~\ref{THEO-intro-disk}.
Let $\delta>0$. There exists $K>0$ (depending on $\mathcal S_0$, $\Omega$, $\ell_0$, $\gamma$, $m_{1}$, $\delta$) and $\varepsilon_{0}>0$ such that for any $\eps$ in $(0,\eps_{0})$, for any $t>0$ such that $d(B(h_{\varepsilon} (t) ,\varepsilon),\partial \Omega) > \delta$
one has $|{h}'_\varepsilon (t)|_{\R^{2}} \leq K$.
\end{cor}
Thanks to Corollary~\ref{bd-disk2}, Lemma~\ref{nrj-liminf} remains true in the case under view, that it, the time of existence $T_{\varepsilon}$ of $h_{\varepsilon}$ is bounded from below by some $\underline{T}>0$. Now we have the following local convergence result of $h_{\varepsilon}$ toward $\overline{h}$, where we recall that $\overline{h}$ is the global solution to \eqref{argh}.
\begin{lem} \label{nrj-cv-disk}
Let $\varepsilon_{1}>0$,  ${\delta}>0$ and  ${T} >0$, and suppose that for any $\varepsilon $ in $(0,\varepsilon_{1} )$,  
\begin{equation} \label{HypUniformite2-disk}
d(B(h_{\varepsilon}(t),\varepsilon), \partial \Omega) > {\delta}  \; \text{ on  } [0,{T} ]  .
\end{equation}
Then $h_\varepsilon \longrightharpoonup \overline{h}$ in $W^{1,\infty}([0,\check{T}];\mathbb R^2)$ weak-$\star$.
\end{lem}

\begin{proof}
Given $\delta>0$, $T>0$ and $\varepsilon_{1}>0$ as above we apply Corollary~\ref{bd-disk2} on $[0,T]$ so that reducing $\varepsilon_{1}>0$ if necessary, 
\begin{equation} \label{EstBase-disk}
(h'_{\varepsilon})_{\varepsilon \in (0, \varepsilon_{1})}  \text{ is bounded uniformly }   \text{ on } [0,T].
\end{equation}
Our goal is to pass to the limit in each term of \eqref{fnorm2-ball}. 
For what concerns the left hand side it is obvious that
$\big( \eps^{\alpha}\,  m_{1} + \eps^2 \, \pi \big) \big( h_\varepsilon' - u^{\Omega}(h_{\varepsilon}) \big)'
\longrightarrow 0$  in $W^{-1,\infty}$. 
Next, the term $\varepsilon^{\min(2,\alpha)}  F_{\flat,r} (\eps,h_\varepsilon, h_\varepsilon' )$ converges to $0$ in $L^{\infty}$.
Hence we infer that $h_\varepsilon' - \gamma u^\Omega (h_\varepsilon)$ converges weakly to $0$ in $W^{-1,\infty}$. 
Due to the a priori estimate, this convergences occurs in $L^{\infty}$ weak-$\star$.
By \eqref{EstBase-disk}, there is a subsequence $(h_{\varepsilon_{n}})_n$ of $(h_{\varepsilon})$ satisfying $h_{\varepsilon_{n}} \rightharpoonup h_{*}$ in $W^{1,\infty}$ weak-$\star$.
In particular convergence of $h_{\varepsilon}$ toward $h_{*}$ is strong in $L^{\infty}$, and with the convergence of $h_\varepsilon' - \gamma u^\Omega (h_\varepsilon)$ we deduce that
$h_{*}'=\gamma u^{\Omega}(h_{*})$  and $ h_{*}(0) = 0$.
The uniqueness of the solution to this Cauchy problem gives $h_{*} = \overline{h}$ and that the whole sequence $(h_{\varepsilon})$ converges toward $\overline{h}$ as $\varepsilon \rightarrow 0^{+}$. 
This concludes the proof of Lemma~\ref{nrj-cv-disk}. 
\qed
\end{proof} 
\ \par
Finally we briefly conclude the proof of Theorem~\ref{theo:1} in the case of a homogeneous disk, which is the same as the one of Theorem~\ref{theo:3}, except that in the case (ii) under view, as mentioned below \eqref{argh}, the solution $\overline{h} $ is global in time, and in particular that there is no collision of the point vortex with the external boundary $\partial \Omega$. Hence here, compared to the end of the proof of Theorem~\ref{theo:3} at the end of Subsection~\ref{Subsec:PTTLi}, we can pick any $T>0$, and then we define $\overline{d}$ so that for all 
$t \in [0,T+1]$, $d(\overline{h}(t), \partial \Omega) \geq \overline{d}$ rather than by \eqref{Eq:DBarre} and prove $\tilde{T} \geq T + \frac{1}{2}$ rather than \eqref{TTilde}. Of course we rely on Corollary~\ref{bd-disk2} rather than Corollary~\ref{bd-loin-(i)} and on Lemma~\ref{nrj-cv-disk} rather than Lemma~\ref{nrj-cvi}. This ends the proof of Theorem~\ref{theo:1} in the case of a homogeneous disk. 
\subsection{Geodesic-gyroscopic normal form}
In Case (ii)
we will establish that \eqref{ODE_intro_eps} can be put into a normal form whose structure looks like \eqref{ODE_ext} up to a refined modulation of the velocity of the center of mass. 
Indeed, in the same way as we defined the Kirchhoff-Routh velocity $u^\Omega $ by  $u^\Omega =\nabla^\perp   \psi^\Omega $
we introduce the corrector velocity $u_c$ corresponding to the stream function $\psi_c$ defined in \eqref{PsiC}
by 
\begin{equation} \label{def-uc-perp}
u_c (q) :=\nabla^\perp_{h} \psi_c (q).
\index{AU3@$u_c$: corrector velocity}
\end{equation}
Observe that the function $ u_c $ depends on $\Omega$, $ \mathcal S_0$,  $\vartheta$ and on $h$, whereas $u^\Omega$ depends only on $\Omega$ and $h$.
We will make use in a crucial way of  the following second order modulation: 
\begin{equation}  \label{dmodu}
{\tilde{p}}_{\varepsilon} := \big(\varepsilon \vartheta_\varepsilon' , h_\varepsilon'  - \gamma [ u^\Omega (h_\varepsilon ) + \varepsilon   u_c (q_\varepsilon) ] \big)^t,
\end{equation}
which can be compared to the expression \eqref{def-hat} of $p_\varepsilon$.
We  note that  $ \gamma (u^\Omega (h) + \eps u_c (q)) $ is the beginning  of the expansion  of $ - \frac{1}{\gamma} \nabla^\perp_{h}  U_{\eps} (q)$ where $U_{\eps} (q)$ is the electric-type  potential energy defined in \eqref{pot-eps}; see \eqref{expansion_mu-intro}. This modulation is therefore driven by the leading terms of the electric-type potential. 
Observe also that, as long as the solid does not touch the boundary, the  drift term in the velocity of the center of mass
is bounded. Indeed the following refinement of Lemma~\ref{drifterisok-disk} is a direct consequence of Lemma~\ref{banane} and of the definitions of $u^\Omega$ and $u_{c}$.

\begin{lem} \label{drifterisok}
Let $\delta > 0$, there exists $\eps_{0} $ in $(0,1)$ and $K>0$ such that for any $(\varepsilon ,q) $ in $ \mathfrak Q_{\delta,\eps_{0}} $
with $q=(\vartheta,h)$, $| u^\Omega (h) + \varepsilon   u_c (q) |_{\R^{3}} \leq K$.
\end{lem}
Before stating our normal form, we introduce the following definition.
\begin{defn} \label{wg}
Let $\delta>0$ and $\varepsilon_0\in (0,1)$ be given. 
We say that a vector field $F$ in $C^{\infty}(\R \times \Omega ; \R^3)$ is weakly gyroscopic if  there exists $K>0$ depending on  $\mathcal S_0$, $\Omega$, $\gamma $ and $\delta$ such that for any smooth curve $q(t) = (\vartheta (t) , h(t))$ in $ \R\times  \Omega_{\delta}$,  for any $t\geq 0$ and any $\varepsilon \in (0,\varepsilon_{0})$
\begin{equation} \label{weakylgyro}
\Big| \int_0^t \, 
\tilde{p} \cdot F (q) \Big|
 \leq \varepsilon K  \left( 1 + t +  \int_0^t \, |  \tilde{p}|^{2}_{ \R^3} \right) ,
\end{equation}
with $\tilde{p}= (\eps \vartheta', h' - \gamma [u^\Omega (h) + \eps u_{c} (q)])^t$.
\end{defn}
The weakly gyroscopic vector fields in the sequel have the form $F=(F_{1},0,0)^{t}$.
The normal form that we use in Case (ii) is as follows.
\begin{prop} \label{Pro-fnorm}
%
%Let 
There exist
\begin{itemize}
\item  $F_{r}: \mathfrak Q \times \R^3 \rightarrow \R^3$ depending on  $\mathcal S_0$, $\gamma$  and $\Omega$, such that for any $\delta > 0$, there exists $\eps_{0} $ in $(0,1)$ such that $F_{r} \in L^{\infty}_\mathrm{loc}(\mathfrak Q_{\delta,\eps_{0}} \times \R^3; \R^3)$ and is weakly nonlinear in the sense of Definition~\ref{def-weakly-nonlinear},
\item $\mathsf{E}_{{\it 1}}^b$ in $C^{\infty}(\R \times \Omega ; \R^3)$ depending on  $\mathcal S_0$  and $\Omega$, weakly gyroscopic in the sense of Definition~\ref{wg},
\end{itemize}
such that Equation \eqref{ODE_intro_eps} can be recast as
\begin{multline}
\label{fnorm2}
\Big( \eps^{\alpha}\,  M_g + \eps^2 \, M^{\Ext}_{a, \vartheta_\varepsilon} \Big)
 \tilde{p}_\varepsilon '
+ \varepsilon \langle \Gamma^{\Ext}_{\vartheta_\varepsilon} ,{\tilde{p}}_\varepsilon,{\tilde{p}}_\varepsilon\rangle 
 = F^{\Ext}_{\vartheta_\varepsilon} ( {\tilde{p}}_{\varepsilon} )
+ \varepsilon \gamma^{2} \mathsf{E}_{{\it 1}}^b (q_\varepsilon) 
\\  + \varepsilon^{\min(2,\alpha)}  F_{r} (\eps,q_\varepsilon, {\tilde{p}}_{\varepsilon} )  .
\end{multline}
\end{prop}
We recall that $M^{\Ext}_{a, \vartheta}$, $\Gamma^{\Ext}_{\vartheta}$ 
and the force term $F^{\Ext}_{\vartheta}$  were defined in the case without outer boundary in \eqref{g+a_ext}, \eqref{Christo-exter} and \eqref{force-ext} (see also \eqref{def-Bext}) respectively. Moreover the term $\mathsf{E}_{{\it 1}}^b$ is explicit; see \eqref{DefH1b}.  The proof of Proposition~\ref{Pro-fnorm} is given in Section~\ref{Sec:FormesNormales}. \par
\begin{itemize}
\item \textbf{Motivations}.  
We will use the normal form  \eqref{fnorm2} both  in order to get a uniform bound of the velocity and to pass to the limit in Case (ii). 
It would be actually possible to deal with the case where $\alpha $ is small with a less accurate normal form and still get an energy estimate. 

In particular in order to get  a uniform bound of the velocity in Case (ii) we will perform an estimate on an energy adapted to the normal form \eqref{fnorm2}.
Observe that should the right hand side vanish the  normal form  \eqref{fnorm2} would be the geodesic equation associated
with the metric ${M}_{ \vartheta }(\varepsilon)$ defined in  \eqref{TrueMatrix}.
On the other hand the right hand side is the sum of  terms with a quite remarkable structure: 
the leading term $F^{\Ext}_{\vartheta_\varepsilon} ({\tilde{p}}_{\varepsilon})$ is gyroscopic in the sense of Definition~\ref{def-gyro},
the electric-type term  $\mathsf{E}_{{\it 1}}^b (q_\varepsilon)$ is  weakly gyroscopic in the sense of Definition \ref{wg};
and the remainder $F_{r}$ is weakly nonlinear  in the sense
of Definition~\ref{def-weakly-nonlinear}.
\item  \textbf{Ideas of the proof  of Proposition~\ref{Pro-fnorm}}. 
As for Proposition~\ref{Pro-fnormA}
the proof of  Proposition~\ref{Pro-fnorm} relies on expansions of the inertia matrix, of the Christoffel symbols and of the force terms with respect to $\eps$.
A striking and crucial phenomenon is that some subprincipal contributions (that is, of order $\varepsilon$) of the force terms will be gathered with the leading part of the term involving the Christoffel symbols to become a part of the second term of the left hand side of \eqref{fnorm2}; see Lemma~\ref{lemma-grav}.
The leading part of the contribution coming from the Christoffel symbols  will be provided by the $\Gamma^{{\rm rot}}$-part of the decomposition  \eqref{NiouGamma}. \par
\end{itemize}

\begin{rem}
The normal forms above are inspired by the case without external boundary (see Equation~\eqref{ODE_ext}) and by 
the paper \cite{BG} where the authors consider the motion of a light charged particle in a slowly varying electromagnetic field.
The equation of motion for the particle is an ordinary differential equation involving a small parameter in front of the term with the highest time derivative.
To restore some uniformity with respect to the small parameter they use a modulation, subtracting from the particle velocity the $|B|^{-2}   E \times B$ drift, and a normal form, see \cite[Eq. (3.5)]{BG}, where 
the only remaining singular term appears through a Lorentz gyroscopic force. This allows us to tackle the convergence of the particle motion to the so-called guiding center motion despite the fast oscillations induced by the gyroscopic force. 

In the case where the solid ${\mathcal S}_{0}$ is a homogeneous disk (Subsection~\ref{SubsecHomDisk}), 
$u^{\Omega}(h_{\varepsilon})$ is the vector 
whose coordinates are the  last two coordinates of this drift.
However in the case where ${\mathcal S}_{0}$ is not a disk our drift term $(0,\gamma (u^\Omega (h) + \eps u_c (q) ))$ does not enter this framework.
Actually the use of the $|B|^{-2} E \times B$ drift could give a modulated energy estimate only in the case $\alpha \leq 1$, and in particular not in the case of a solid with a fixed homogeneous density ($\alpha=2$). Moreover it would not be adapted to the passage to the limit.
\end{rem}
\subsection{Modulated energy estimates}
\label{sec-modudu}
As mentioned above, 
in Case (i), Corollary \ref{bd-loin-(i)} provides a uniform bound of $h_\varepsilon'$ as long as the body stays at a positive distance from the external boundary.
In Case (ii), the same analysis can be carried on so that Corollary~\ref{bd-loin-(i)} and Corollary~\ref{bd-disk} have the following counterpart.
\begin{cor} \label{bd-loin-(i)-Bis}
Let $q_\eps$ satisfy the assumptions of Theorem~\ref{theo:1}.
Let $\delta>0$. There exists $K>0$ (depending on $\mathcal S_0$, $\Omega$, $p_0$, $\gamma$, $m_{1}$, $\mathcal J_{1}$, $\delta$) and $\varepsilon_{0}>0$ such that for any $\eps$ in $(0,\eps_{0})$, such that as long as $(\eps,q_\eps)$ belongs to $\mathfrak Q_{\delta,\eps_{0}}$,
one has $\varepsilon^{\min(1,\frac{ \alpha }{2 })} \, |{p}_\varepsilon|_{\R^{3}} \leq K$.
\end{cor}
Therefore this analysis does no longer provide a uniform bound of the solid velocity. 
An important part of the proof consists in finding an appropriate substitute  which allows us a better control on the body velocity.
This will be accomplished below by a modulated energy, which, roughly speaking, consists in applying the energy ${\mathcal{E}}_\vartheta$ (see \eqref{TrueNRJ}) to $\tilde{p}_{\varepsilon}$ (defined in \eqref{dmodu}) rather than to $p_{\varepsilon}$ (defined in \eqref{def-hat}). 

 The structure established in Proposition~\ref{Pro-fnorm}
will allow us to obtain an estimate of the modulated energy 
${\mathcal{E}}_\vartheta (\eps , \tilde{p}_{\varepsilon}) $.
Since Equation~\eqref{fnorm2} looks like Equation~\eqref{ODE_ext} of the case without external boundary 
for which the total energy is the kinetic energy alone (as defined in Proposition~\ref{Prop-conserv-no}), 
one may hope to have a good behaviour of the modulated  energy ${\mathcal{E}}_\vartheta (\eps , \tilde{p}_{\varepsilon}) $ 
as time proceeds. 
Indeed  we have the following result on the time derivative  of the modulated energy. 
\begin{lem} \label{nrj11}
Let  $(q_\eps,T_{\varepsilon})$ as in Theorem~\ref{theo:1}.
Then 
\begin{equation} \label{Leibencore2}
\frac{d}{dt}  {\mathcal{E}}_{\vartheta_{\varepsilon}} (\eps , \tilde{p}_{\varepsilon})
= \eps^{\max(1-\alpha ,-1)} \,   \gamma^{2}  {\tilde{p}}_{\varepsilon} \cdot  \mathsf{E}_{{\it 1}}^b (q_\varepsilon) 
+ {\tilde{p}}_{\varepsilon} \cdot F_{r} (\eps , q,\tilde{p}_\varepsilon)  .
\end{equation}
\end{lem}
\begin{proof}
Since the matrix ${M}_{\vartheta_\varepsilon}$ defined in \eqref{TrueMatrix} is symmetric, 
\begin{equation} \label{Leib}
\frac{d}{dt} \mathcal{E}_{\vartheta_{\eps}} (\eps , \tilde{p}_{\varepsilon}) 
=  {\tilde{p}}_{\varepsilon} \cdot  {{M}}_{ \vartheta_\varepsilon }  (\eps) 
   {\tilde{p}}_{\varepsilon}'
+ \frac{1}{2} {\tilde{p}}_{\varepsilon} \cdot 
\left( \frac{d}{dt} {{M}}_{ \vartheta_\varepsilon }  (\eps) \right) {\tilde{p}}_{\varepsilon} .
\end{equation}
Recalling \eqref{agreable} and substituting \eqref{fnorm2}   in  \eqref{Leib} we arrive at  
\begin{align} \nonumber
\eps^{\min(2,\alpha)} \,  \frac{d}{dt}  \mathcal{E}_{\vartheta_{\eps}} (\eps , \tilde{p}_{\varepsilon}) 
& = {\tilde{p}}_{\varepsilon} \cdot {F}^\Ext_{ \vartheta_\varepsilon} ( {\tilde{p}}_{\varepsilon} )\\
\nonumber
& + {\tilde{p}}_{\varepsilon} \cdot  \left(  \frac{1}{2}  \eps^{\min(2,\alpha)} \,  \left(\frac{d}{dt}  {M}_{ \vartheta_\varepsilon }(\varepsilon) \right) {\tilde{p}}_{\varepsilon} 
-  \varepsilon \langle \Gamma^\Ext_{ \vartheta_\varepsilon },{\tilde{p}}_\varepsilon,{\tilde{p}}_\varepsilon \rangle  \right) \\
\label{Leibencore}
& + \varepsilon \gamma^{2}  {\tilde{p}}_{\varepsilon} \cdot  \mathsf{E}_{{\it 1}}^b (q_\varepsilon) 
  + \eps^{\min(2,\alpha)} \,  {\tilde{p}}_{\varepsilon} \cdot  F_{r} (\eps,q_\varepsilon, {\tilde{p}}_{\varepsilon} )  .
\end{align}
Using that the force term is $(\vartheta,p) \mapsto {F}^{\Ext}_{\vartheta} ( p )$ is gyroscopic
in the sense of Definition~\ref{def-gyro}
we deduce that the first  term of the right hand side  of the equation \eqref{Leibencore} vanishes.
Moreover, going back to the definition of ${M}_\vartheta (\eps)$, 
\begin{equation} \label{EZ2}
\eps^{\min(2,\alpha)} \,  \frac{d}{dt}  {M}_{\vartheta_\varepsilon } (\eps)
=
\eps^2\,  \frac{d}{dt}  M^{\Ext}_{a, \vartheta_\varepsilon}
=  \eps   \frac{\partial M^{\Ext}_{a, \vartheta}}{\partial q}  (\vartheta_\varepsilon )  \cdot  {\tilde{p}}_{\varepsilon} ,
\end{equation}
by taking advantage of the fact that the matrix $M^{\Ext}_{a, \vartheta}$ depends only on  $\vartheta$ and not on the last two coordinates of  $q$. 
Using  \eqref{EZ1} we conclude that the second term of the right hand side of Equation~\eqref{Leibencore} vanishes.
Therefore the equation \eqref{Leibencore} reduces to \eqref{Leibencore2} and the proof of Lemma  \ref{nrj11} is complete.
\end{proof} \qed
Now Corollary~\ref{bd-loin-(i)} and Lemma~\ref{drifterisok} already give us that
$\varepsilon {\tilde{p}}_{\varepsilon}$  is bounded.
Then we use that $F_{r}$  is weakly nonlinear in the sense of Definition~\ref{def-weakly-nonlinear},
that $\mathsf{E}_{{\it 1}}^b$ is weakly gyroscopic in the sense of Definition~\ref{wg} (using Lemma~\ref{banane}), 
Lemma~\ref{coerc} and Gronwall's lemma to get the following result. 
\begin{cor} \label{nrj12}
Let $(q_\eps,T_{\varepsilon})$ as in Theorem~\ref{theo:3}.
Let $\delta>0$. There exists $K>0$ (depending on $\mathcal S_0$, $\Omega$, $p_0$, $\gamma$, $m_{1}$, $\mathcal J_{1}$, $\delta$) and $\varepsilon_{0}>0$ such that for any $\eps$ in $(0,\eps_{0})$, as long as $(\eps,q_\eps)$ belongs to $\mathfrak Q_{\delta,\eps_{0}}$,
one has $| {p}_\varepsilon |_{\R^{3}}  \leq K$.
\end{cor}
Corollary~\ref{nrj12} therefore provides the same estimates for Case (ii) as Corollary~\ref{bd-loin-(i)} for Case (i).
\begin{rem}
A modulated energy argument  was also used in \cite{MS} for a related issue but we would like to emphasize that the modulation occurs for a different part of the energy. 
More precisely, in the paper \cite{MS} an Euler-Vlasov system is introduced as a mean-field model for massless point particles moving in a perfect incompressible fluid. 
This system couples the incompressible Euler equation with a Vlasov equation which describes the dynamics of the massive particles, seen as a dispersed phase. 
This Vlasov equation  is therefore the counterpart of the ordinary differential equation  \eqref{ODE-mass} of Case (i). 
The last section of \cite{MS} deals with the limit where the individual mass of the particles converges to $0$. Then a modulated energy is used in order to obtain a hydrodynamic convergence, that is the convergence of the macroscopic mixture velocity associated by the Biot-Savart law with both fluid vorticity and particles circulations. 
This modulated energy is inspired by the paper \cite{brenier} of Brenier regarding the gyrokinetic limit for the Vlasov-Poisson system.
Referring to the decomposition of the energy in \eqref{pader} and to Eq. (78) in \cite{MS} we sketch the following opposition. In \cite{MS} the potential part of the energy or more exactly of its modulation is crucial. On the contrary we use here a modulation of the kinetic part of the energy in order to deduce the limit of the particle dynamics, the fluid dynamics being subjugated to the particle one. 
\end{rem}
\subsection{Passage to the limit}
Here we prove the convergence of the center of mass $h_{\varepsilon}$ to $\overline{h}$, where $\overline{h}$ is the global solution to \eqref{argh}. \par
First, thanks to  Corollary~\ref{nrj12}, Lemma~\ref{nrj-liminf}  remains true in Case $(ii)$. 
Moreover we have the following counterpart of Lemma~\ref{nrj-cvi} and Lemma~\ref{nrj-cv-disk}.
\begin{lem} \label{nrj-cv}
Let $\varepsilon_{1}>0$,  ${\delta}>0$ and  ${T} >0$, and suppose that for any $\varepsilon $ in $(0,\varepsilon_{1} )$,  
\begin{equation} \label{HypUniformite2}
(\eps,q_\eps )  \in  \mathfrak Q_{{\delta}  ,\eps_{1}}  \; \text{ on  } [0,{T} ]  .
\end{equation}
Then
 $h_\varepsilon \longrightharpoonup \overline{h}$ in $W^{1,\infty}([0,{T}];\mathbb R^2)$ weak-$\star$.
\end{lem}
As for Lemma~\ref{nrj-cv-disk}, the proof of Lemma~\ref{nrj-cv} consists in passing to the weak limit, with the help of all a priori bounds, in each term of \eqref{fnorm2}. 
\begin{proof}
We consider ${\delta}>0$, ${T}>0$ and $\varepsilon_{1}>0$ as above and apply Corollary~\ref{nrj12} on the interval $[0,{T}]$. Reducing $\varepsilon_{1}>0$ if necessary, 
\begin{equation} \label{EstBase}
(|h'_{\varepsilon}| + |\varepsilon \vartheta_{\varepsilon}' | )_{\varepsilon \in (0, \varepsilon_{1})}  \text{ is bounded uniformly }   \text{ on } [0,{T}].
\end{equation}
Our goal is to pass to the limit in each term of \eqref{fnorm2}. 
For what concerns the left hand side we first observe that, thanks to \eqref{EstBase}, $\eps M_{a,\Ext,\vartheta_\varepsilon}$ 
is bounded in $W^{1,\infty}$ whereas $\tilde{p}_\varepsilon'$ is bounded in $W^{-1,\infty}$.
Since $M_{g}$ is constant, it follows that $\big( \varepsilon^{\alpha} M_g + \varepsilon^{2}  M^\Ext_{a, \vartheta_\varepsilon} \big) \tilde{p}_\varepsilon'
\longrightarrow 0$  in $W^{-1,\infty}$.
Next, the term $\varepsilon \langle \Gamma^{\Ext}_{\vartheta_\varepsilon}, {\tilde{p}}_\varepsilon,{\tilde{p}}_\varepsilon\rangle$ converges to $0$ in $L^{\infty}$ since all factors in the brackets are bounded.
In the same way, the terms $\varepsilon \gamma^{2} \mathsf{E}_{{\it 1}}^b(q_\varepsilon)$ and $\varepsilon^{\min(2,\alpha)}  {F}_{r} (\eps,q_\varepsilon, {\tilde{p}}_{\varepsilon})$ converge strongly to $0$ in $L^{\infty}$. \par
Now let us consider the last two coordinates of the remaining term $F^{\Ext}_{\vartheta_\varepsilon} ( {\tilde{p}}_{\varepsilon} )$ in the equation \eqref{fnorm2}, that is, $\gamma (h_\varepsilon'  - \gamma  u^\Omega (h_\varepsilon ) - \varepsilon \gamma u_{c}(q_{\varepsilon}))^\perp - \varepsilon \vartheta_\varepsilon' {\zeta}_{\vartheta_\varepsilon }.$
The last term converges weakly to $0$ in $W^{-1,\infty}$ as seen in Case (i); see \eqref{2-g}.
Hence we infer that $h_\varepsilon' - \gamma u^\Omega (h_\varepsilon)$ converges weakly to $0$ in $W^{-1,\infty}$. 
Then one concludes exactly as for Lemma~\ref{nrj-cv-disk}. \qed
\end{proof} 
%
%
%\ \par
%
%
Once Lemma~\ref{nrj-cv} is established, the conclusion of the proof of Theorem~\ref{theo:1} (provided that ${\mathcal S}_{0}$ is not a disk) follows the exact same lines as in the case of a homogeneous disk. This concludes the proof of Theorem~\ref{theo:1} in this case.
\subsection{The case of a non-homogeneous disk}
\label{Subsec:TCONHD}
We now return to the case ${\mathcal S}_{0}$ is a disk, but this time we handle the case where it is non-homogeneous, or more precisely the case where the center of gravity $h_{\varepsilon}$ is not at the center of the disk $h_{c,\varepsilon}$, which can occur with a non-uniform mass distribution. We aim at proving Theorem~\ref{theo:1disk}.
This adds some extra-difficulties in the analysis which require a separate treatment in the case $\alpha >2$. In the case $\alpha \leq 2$, we can use the above subsection since the degeneracy of $M_{a,\vartheta}^{\Ext}$ does not prevent $M_{\vartheta}(\varepsilon)$ to be uniformly bounded from below, so that we can obtain Corollary~\ref{nrj12} in the same way. \par
From now on we suppose $\alpha>2$ and once again we assume without loss of generality ${\mathcal S}_{0}=\overline{B}(0,1)$. 
We will use yet another slight variant of  Definition~\ref{def-weakly-nonlinear} and  Definition~\ref{def-WNL2}.
\begin{defn} \label{def-WNL2-inh}
Let $\delta>0$ and $\varepsilon_0\in(0,1)$ be given. 
We say that a vector field $F_{\flat}$ in $L^{\infty}_\mathrm{loc}((0,\varepsilon_0) \times \Omega \times \R^3; \R^2)$ is weakly nonlinear if 
there exists $K>0$ depending on  $m$, ${\mathcal J}$, $\gamma$, $\Omega$ and $\delta$
such that for any $\varepsilon$ in $(0,\varepsilon_{0})$, any $x$ in $\Omega$ such that $d(B(x,\varepsilon),\partial \Omega) > \delta$ and any $p$ in $\R^{3}$, 
\begin{equation} \label{ineq-WNL3}
| F_{\flat} (\eps , x ,  p ) |_{\R^{2}} \leq K ( 1 + | p |_{\R^{3}}  + \eps | p |_{\R^{3}}^{2} ) .
\end{equation}
\end{defn}
We will use the modulated variable $\tilde{p}_{\flat,\varepsilon}$ defined by
\begin{equation*}
\tilde{p}_{\flat,\varepsilon} :=  h_{c,\varepsilon}' - \gamma  u^{\Omega}(h_{c,\varepsilon}),
\end{equation*}
and the following normal form.
\begin{prop} \label{Pro-fnorm-ball-inh}
There exists $F_{\flat,r}$ in $L^{\infty}_\mathrm{loc}((0,1) \times\Omega \times \R^3; \R^2)$ such that, for any $\delta>0$, there exists $\varepsilon_{0}$ in $(0,1)$ for which $F_{\flat,r}$ is weakly nonlinear in the sense of Definition~\ref{def-WNL2-inh}, and such that Equation~\eqref{ODE_intro-disk} can be recast as
\begin{align} \label{fnorm2-ball-inh}
\eps^{\alpha}\,  m_{1} h_{\varepsilon}'' + \eps^2  \pi ( \tilde{p}_{\flat,\varepsilon} )'
&= \gamma  \tilde{p}_{\flat,\varepsilon}^\perp
+ \eps^2 F_{\flat,r} (\eps,h_{c,\varepsilon}, \varepsilon \vartheta_\varepsilon' , \tilde{p}_{\flat,\varepsilon})  , \\
\label{FBI2}
h_{c,\varepsilon} - h_{\varepsilon} &=  \varepsilon \zeta_{\vartheta_\varepsilon} , \\
\mathcal J_1 \varepsilon \vartheta_\varepsilon '' &=  \zeta_{\vartheta_\varepsilon}^\perp \cdot  m_1 h_\varepsilon '' ,
\end{align}
\end{prop}
Again the coefficient $\pi$ comes from the fact that $M^{\Ext}_{\flat} =\pi \mbox{\rm Id}_{2}$ (where $M^{\Ext}_{\flat}$ is defined in \eqref{def-Ma}). The proof of Proposition~\ref{Pro-fnorm-ball-inh} is given in Section~\ref{Sec:FormesNormales}. \par
\ \par
Once Proposition~\ref{Pro-fnorm-ball-inh} is obtained, for $\varepsilon \in (0,\varepsilon_{0})$, we multiply \eqref{fnorm2-ball-inh} by $\tilde{p}_{\flat,\varepsilon}$. 
We set 
${\mathcal E}_{\flat}:= \varepsilon^{\alpha}\,  m_{1} |h_{\varepsilon}'|^{2} + \varepsilon^{\alpha+2} {\mathcal J}_{1} |\vartheta_{\varepsilon}'|^{2} + \eps^2  \pi |\tilde{p}_{\flat,\varepsilon}|^{2}$.
Noticing that
\begin{equation*}
m_{1} h''_{\varepsilon} \cdot h'_{c,\varepsilon} = m_{1} h_{\varepsilon}'' \cdot h_{\varepsilon}' + \varepsilon^{2} {\mathcal J}_{1} \vartheta_{\varepsilon}'' \vartheta_{\varepsilon}',
\end{equation*}
we arrive at
\begin{equation*}
\frac{d}{dt} {\mathcal E}_{\flat}(t)
+ \varepsilon^{\alpha} m_{1} h_{\varepsilon}'' \cdot u^{\Omega}(h_{c,\varepsilon})
\leq C \varepsilon^{2} |\tilde{p}_{\flat,\varepsilon}| (1 + |\tilde{p}_{\flat,\varepsilon}| + \varepsilon |\tilde{p}_{\flat,\varepsilon}|^{2} ).
\end{equation*}
The standard energy estimate gives us here that $\varepsilon |h'_{c,\varepsilon}|$ is bounded, and consequently so is $\varepsilon |\tilde{p}_{\flat,\varepsilon}|$.
Integrating over time and using an integration by parts we obtain (as long as the solid stays at positive distance from $\partial \Omega$):
\begin{equation*}
{\mathcal E}_{\flat} (t) - {\mathcal E}_{\flat} (0)
\leq C \varepsilon^{2} + C \int_{0}^{t} {\mathcal E}_{\flat} (s) \, ds
+ C \varepsilon^{\alpha} |h_{\varepsilon}'(t)| + C \varepsilon^{\alpha} |h_{\varepsilon}'(0)|.
\end{equation*}
Thanks to Young's inequality, 
\begin{equation*}
{\mathcal E}_{\flat} (t) \leq C \int_{0}^{t} {\mathcal E}_{\flat} (s) \, ds + C {\mathcal E}_{\flat}(0) + C \varepsilon^{2},
\end{equation*}
so with Gronwall's lemma we finally deduce that ${\mathcal E}_{\flat}(t) \leq C ({\mathcal E}_{\flat} (0) + \varepsilon^{2})$, as long as the solid stays at positive distance from $\partial \Omega$. Since $\alpha >2$, we deduce that $(h'_{c,\varepsilon})_{\varepsilon \in (0,\varepsilon_{0})}$, $(\varepsilon^{\frac{\alpha}{2}-1} h'_{\varepsilon})_{\varepsilon \in (0,\varepsilon_{0})}$ and $(\varepsilon^{\frac{\alpha}{2}} \vartheta'_{\varepsilon})_{\varepsilon \in (0,\varepsilon_{0})}$ are bounded in $L^{\infty}$. Notice in passing that using \eqref{FBI2} and an interpolation argument we deduce  that $(h_{\varepsilon})_{\varepsilon \in (0,\varepsilon_{0})}$ is bounded in $W^{\frac{2}{\alpha},\infty}$; see the discussion after the statement of Theorem~\ref{theo:1disk}.

Then it is straightforward to adapt the previous reasoning to pass to the weak limit in \eqref{fnorm2-ball-inh}. We deduce $\tilde{p}_{\flat,\varepsilon} \rightharpoonup 0$ in $L^{\infty}$ and the conclusion follows. This ends the proof of Theorem~\ref{theo:1disk}.
%
%
%
%
%%%%%%%%%%%%%%%%%%%%%%%%%%%%%%%%%%%%%%%%%%%%%%%%%%%%%%%%%%%%%%%%%%%%%%%%%%%%%%%%%%%%%%%%%%%%%%%%%%%%%%%%%%%%%%%%%%%%%%%%%%%%%%%%%%%%%%
%
%
%
%
%
%
\section{Asymptotic development of the stream and potential functions}
\label{sec-dev-stream}
In this section, we establish asymptotic expansions for the circulation stream function (defined in \eqref{def_stream})
and the Kirchhoff potentials (defined in \eqref{Kir}) in the domain ${\mathcal F}^{\varepsilon}(q)$, as $\varepsilon$ tends to $0^{+}$.
The asymptotic analysis of the Laplace equation when the size of an inclusion goes to $0$ has been deeply studied, see for example \cite{Ilin} and 
\cite{MNP}. However to our knowledge the results of this section are not covered by the literature. \par
%
%
%
%%%%%%%%%%
%
%
\subsection{A few reminders about single-layer potentials}
\label{AFewReminders}
To get the asymptotic expansions mentioned above, we will look for a representation of these  stream and  potential functions as a superposition of single-layer integrals supported by the two connected components $\partial  \mathcal S^{\varepsilon} (q)$ and $\partial \Omega$ of the boundary of the fluid domain $\mathcal F_\varepsilon  (q)$. In this subsection, we give a few reminders about single-layer potentials which we will use in the analysis. We refer for instance to \cite{Folland} and \cite{McLean}.

Below we consider single-layer potentials of the form:
\begin{equation} \label{def-SL}
SL [\mathtt{p}^{\mathcal C}]  := \int_{\mathcal C} \mathtt{p}^{\mathcal C} (y) G(\cdot -y ) \, {\rm d}s(y) , 
\end{equation}
where $\mathcal C$ is a smooth Jordan curve in the plane and $\mathtt{p}^{\mathcal C}$ belongs to the Sobolev space $ H^{-\frac{1}{2}} (\mathcal C)$. Recall that $G$ was defined in \eqref{NewtonianPotential}.
We  say that $\mathcal C$  is the support of the single-layer potential and that $\mathtt{p}^{\mathcal C}$ is a density on $\mathcal C$. \par
\ \par \noindent
{\it Harmonicity and trace}.
The formula \eqref{def-SL} defines a function in the  Sobolev space  $H^{1}_\mathrm{loc} (\R^{2})$, harmonic in $\mathbb R^2\setminus \mathcal C$. In particular, for any  $\mathtt{p}^{\mathcal C}$ in $ H^{-\frac{1}{2}} (\mathcal C)$,  the trace of $SL[\mathtt{p}^{\mathcal C}]$ on $\mathcal C$ is well-defined as a function of the  Sobolev space $H^{\frac12} (\mathcal C)$.

\ \par \noindent
{\it Jump of the derivative and density}.
The density $\mathtt{p}^{\mathcal C} $ is equal to the jump of the normal derivative of $SL [\mathtt{p}^{\mathcal C}]$ across $\mathcal C$.
To state this rigorously let us be specific on the orientation of the normal. 
According to Jordan's theorem, the set $\R^{2} \setminus \mathcal C$ has two connected components,
one bounded (the interior), say $\mathcal O_i$, and the other one unbounded (the exterior), say $\mathcal O_e$. Moreover the curve $\mathcal C$ is the boundary of each component. 
We consider the restrictions of $SL [\mathtt{p}^{\mathcal C}]$: 
\begin{equation*}
u_{i} := SL [\mathtt{p}^{\mathcal C}]_{|{\mathcal O}_{i}} \ \text{ and } \ 
u_{e} := SL [\mathtt{p}^{\mathcal C}]_{|{\mathcal O}_{e}}.
\end{equation*}
Denote $n_i$ (respectively $n_e$) the unit normal on $\mathcal C$ pointing outward of $\mathcal O_i$ (resp. of $\mathcal O_e$). Then
the function $u_i$ (respectively $u_e$) is harmonic in $\mathcal O_i$ (resp. $\mathcal O_e$) and the 
traces of the normal derivatives $\frac{\partial u_i}{\partial n_i}$ and $\frac{\partial u_e}{\partial n_e}$ 
on each side of $\mathcal C$ are well-defined in $H^{-\frac12} (\mathcal C)$ and satisfy
\begin{equation} \label{saut}
\mathtt{p}^{\mathcal C} = \frac{\partial u_i}{\partial  n_i} + \frac{\partial u_e}{\partial n_e} .
\end{equation} \par
In the sequel we will make use of single-layer potentials supported on the external boundary $\partial \Omega$, on the boundary $\partial \mathcal S_\varepsilon  (q)$ of the solid body and on the boundary  $\partial \mathcal S_0$ of the rescaled body as well.  We will not use the notations $n_i$ nor $n_e$ but rather the notation $n$ which always stands for the normal outward the fluid. Hence we will have to particularly take care of the signs when referring to the formula \eqref{saut}.

\ \par \noindent
{\it Kernel and rank}.
We will use the following facts (see  Th. 7.17, Th. 8.12 and Th. 8.16 in \cite{McLean}):
 \begin{gather}
\label{fact1}
\text{The operator } SL \text{ is Fredholm with index zero from }   L^2 (\mathcal C) \text{ to }  H^{1} (\mathcal C), \\
\label{fact3}
\text{If }  \mathtt{p}^{\mathcal C} \in H^{-\frac12} (\mathcal C) \text{ satisfies } \int_{\mathcal C}  \,\mathtt{p}^{\mathcal C} \, {\rm d}s = 0 \text{ and }SL [\mathtt{p}^{\mathcal C}] =0 \text{ then }  \mathtt{p}^{\mathcal C} =0, \\
\label{fact2}
\text{If } {\rm Cap} (\mathcal C) \neq 1, \text{ then for any }  \mathtt{p}^{\mathcal C} \in H^{-\frac12} (\mathcal C),
SL [\mathtt{p}^{\mathcal C}] =0 \text{ implies }  \mathtt{p}^{\mathcal C} =0.
\end{gather}
Above, with some slight abuses of notation, we omit to mention the trace operator on $\mathcal C$ and we write 
$\int_{\mathcal C} \, \mathtt{p}^{\mathcal C}  \, {\rm d}s $ for the duality bracket
$\langle 1,\mathtt{p}^{\mathcal C} \rangle_{ H^{-\frac12} (\mathcal C) ,  H^{\frac12} (\mathcal C)}$. \par
\ \par
These properties have the following consequences in our context. 
We  assume without loss of generality that the logarithmic capacity\footnote{also called external conformal radius or transfinite diameter in other contexts \cite{Pomm}} ${\rm Cap}({\partial \Omega})$ of $\partial  \Omega$ satisfies ${\rm Cap}({\partial  \Omega}) < 1$, using translation and dilatation of the coordinates system if necessary. 
Observe that the monotonicity property of the logarithmic capacity  entails that  ${\rm Cap}({\partial \mathcal S_{0}}) < 1$.
Using this latter property, we deduce the two following results.
\begin{prop} \label{coro-cap}
There exists a unique smooth function $\psi^{\Ext}_{{\it -1}}$  solution to \eqref{model_stream-1}. 
Moreover, it satisfies
\begin{equation} \label{ineed}
\psi^{\Ext}_{{\it -1}} =  SL [   \mathsf{p}_{{\it -1}}^{\Ext} ] , \text{ with } \mathsf{p}_{{\it -1}}^{\Ext} =   \frac{\partial \psi^{\Ext}_{{\it -1}} }{\partial n} .
\end{equation} 
\end{prop}
In potential theory $-\mathsf{p}_{{\it -1}}^{\Ext}$ is called the equilibrium density of $\partial \mathcal S_{0}$. The constant $C^{\Ext}$ in \eqref{model_stream-1} is given by $C^{\Ext} = \frac{1}{2\pi} \ln \left( {\rm Cap}(\partial\mathcal S_{0}) \right)$. Note that the index $-1$ is in italic type to emphasize the fact that it is related to an asymptotic development of $\psi$ in powers of $\varepsilon$ and is not a coordinate.
\begin{prop} \label{coro-pro}
Let $g$ be a smooth  function on $\partial  \mathcal S_0$ such that 
\begin{equation}
\label{plm}
\int_{\partial  \mathcal S_0} \, g   \mathsf{p}_{{\it -1}}^{\Ext} \, {\rm d}s = 0 .
\end{equation}
Then there exists a unique bounded smooth function $f$ such that 
\begin{equation} \label{gener}
 -\Delta f=0 \quad \text{in }  \quad  \R^{2} \setminus\mathcal S_{0} ,  \quad  \text{ and }  \quad  f =g\quad \text{on }\partial\mathcal S_{0} .
\end{equation}
Moreover, there exists a unique smooth density $p^{\partial \mathcal S_{0}} $ in $C^{\infty} (\partial \mathcal S_{0})$ such that 
$f = SL [ p^{\partial \mathcal S_{0}}]$ and
\begin{equation}
  \label{samsoir}
\int_{\partial \mathcal S_0} \, p^{\partial \mathcal S_{0}}  \, {\rm d}s = 0 .
\end{equation}
Finally, $f = O(|  x|_{\R^{2}}^{-1} )$ at infinity and  
\begin{equation}
  \label{samsoir2}
\int_{\partial \mathcal S_0} \,  \frac{\partial  f}{\partial n}  \, {\rm d}s  = 0 .
\end{equation}
\end{prop}
\begin{proof}[Proof of Proposition \ref{coro-cap} and \ref{coro-pro}.]
The uniqueness part of Proposition \ref{coro-cap} and of Proposition  \ref{coro-pro} and the decay at infinity in Proposition  \ref{coro-pro} 
can be established by considering holomorphy at infinity of appropriate functions; see for instance \cite[Prop. 2.74. and Prop. 3.2.]{Folland}.

The existence part of Proposition \ref{coro-cap} is given in \cite[Th. 8.15]{McLean}; it also follows from the properties of the single-layers potentials recalled above, in particular \eqref{fact1} and \eqref{fact2}. 

Regarding the  existence part of Proposition \ref{coro-pro} we proceed in two steps. 

First we prove that the operator which maps $(p^{\partial \mathcal S_{0}} , C) $ in $L^{2} (\partial \mathcal S_{0})  \times \R$
to $( SL [ p^{\partial \mathcal S_{0}}] -C , \int_{\partial  \mathcal S_0} \,p^{\partial \mathcal S_{0}}    \, {\rm d}s)$
in $ H^{1} (\partial \mathcal S_{0})  \times \R$ is invertible. 
To prove this, we observe that this operator is Fredholm with index zero  as a consequence of \eqref{fact1}.
Moreover if $(p^{\partial \mathcal S_{0}} , C)$ is in the kernel of this operator, then 
$SL [ p^{\partial \mathcal S_{0}} - \frac{C}{C^{\Ext} } \mathsf{p}_{{\it -1}}^{\Ext} ] = 0$,
so that according to \eqref{fact2}, 
$p^{\partial \mathcal S_{0}} = \frac{C}{C^{\Ext} } \mathsf{p}_{{\it -1}}^{\Ext} $. 
Now, as a consequence of \eqref{53d-1} and of the second identity of \eqref{ineed},  
\begin{equation} \label{plm+}
\int_{\partial  \mathcal S_0} \, \mathsf{p}_{{\it -1}}^{\Ext}   \, {\rm d}s = -1 . 
\end{equation}
Then using that $\int_{\partial  \mathcal S_0} \, p^{\partial \mathcal S_{0}} \, {\rm d}s= 0$ 
we deduce that $C=0$ and therefore $p^{\partial \mathcal S_{0}} = 0$ as well. 
 
Then $(g, 0)$ is in the image of this operator, that is there exists  $(p^{\partial \mathcal S_{0}} , C) $ in $L^{2} (\partial \mathcal S_{0})  \times \R$  such that 
\begin{equation} \label{plm2}
SL [ p^{\partial \mathcal S_{0}}] - C = g \text{ on } \partial  \mathcal S_0 ,
\end{equation}
and \eqref{samsoir}.
%$\int_{\partial \mathcal S_0} \, p^{\partial \mathcal S_{0}}  \, {\rm d}s = 0$. 
Observing that the trace of the operator $SL$ on $\partial \mathcal S_{0}$ is self-adjoint we infer that 
\begin{align} \label{plm3}
\int_{\partial \mathcal S_0} \, SL [ p^{\partial \mathcal S_{0}}]   \mathsf{p}_{{\it -1}}^{\Ext} \, {\rm d}s = \int_{\partial  \mathcal S_0} \, p^{\partial \mathcal S_{0}} SL [ \mathsf{p}_{{\it -1}}^{\Ext}]  \, {\rm d}s  =
C^{\Ext} \,  \int_{\partial  \mathcal S_0} \, p^{\partial \mathcal S_{0}}  \, {\rm d}s 
 = 0 .
\end{align}
Combining \eqref{plm}, \eqref{plm+}, \eqref{plm2} and \eqref{plm3} we infer that $C=0$. 
 
Finally the smoothness part of Proposition \ref{coro-cap} and of Proposition  \ref{coro-pro} follows from
\cite[Th. 7.16]{McLean} and  \eqref{samsoir2} follows from  \eqref{samsoir}, \eqref{saut} and  the vanishing by integration by parts of the interior contribution.
\qed 
\end{proof}
\ \\
{\it Regular integral operators}.
Since we consider single-layer potentials supported on two disjoint curves and their values on both curves, we will also be led to consider regular integral operators.
We recall below some straightforward results which are useful in the sequel. 
Given $\mathcal C$ a smooth Jordan curve in $\overline{\Omega}$, we introduce, for $\delta >0$,
\begin{equation*} 
{\mathcal C}^{\delta} := \{x \in  \overline{\Omega} \ : \  \text{dist} (x,  \mathcal C) < \delta \} ,
\end{equation*}
and define 
\begin{equation*}
F_\delta :  C^{1} \big( \overline{\Omega} \setminus {\mathcal C}^{\delta} ; H^{\frac{1}{2}} (\mathcal C) \big) \times H^{-\frac{1}{2}} (\mathcal C) \rightarrow C^{1} ( \overline{\Omega} \setminus {\mathcal C}^{\delta} ; \R ),
\end{equation*}
by setting, for any $(b, \mathtt{p}^{\mathcal C} ) $ in $ C^{1} \big( \overline{\Omega} \setminus {\mathcal C}^{\delta} ; H^{\frac{1}{2}} (\mathcal C) \big) \times H^{-\frac{1}{2}} (\mathcal C) $, for any $x $ in $\overline{\Omega} \setminus {\mathcal C}^{\delta}$, 
\begin{equation*}
F_\delta [b,\mathtt{p}^{\mathcal C} ] (x) := \int_{\mathcal C} b(x,y) \mathtt{p}^{\mathcal C} (y) \, {\rm d}s(y).
\end{equation*} %
This will be applied to $b$ defined in a larger set but singular for $x=y$; this motivates our framework for $b$. \par
Next, given another smooth Jordan curve $\tilde{\mathcal C}$  in $\overline{\Omega} \setminus {\mathcal C}^{\delta}$ and 
for $b $ in $C^{1} \big( \overline{\Omega} \setminus {\mathcal C}^{\delta} ; H^{\frac{1}{2}} (\mathcal C) \big)$, we define the operator
\begin{equation*}
F_{\delta,b} : L^{2} (\mathcal C) \rightarrow  H^{1} (\tilde{\mathcal C})               
\end{equation*}
by setting $F_{\delta,b} (\mathtt{p}^{\mathcal C})$ as the trace of $F_\delta [b,\mathtt{p}^{\mathcal C} ]$ on $\tilde{\mathcal C}$. 
We will  make use of the following lemma.
\begin{lem} \label{op-reg}
Let $\delta >0$. The two following properties hold.
\begin{enumerate}[(i)]
\item \label{un} The operator $F_\delta$ is bilinear continuous with a norm less than $1$, in other words: 
for any $(b, \mathtt{p}^{\mathcal C} )$ in  $C^{1} \big( \overline{\Omega} \setminus {\mathcal C}^{\delta} ; H^{\frac{1}{2}} (\mathcal C) \big) \times H^{-\frac{1}{2}} (\mathcal C)$, 
\begin{equation*}
\|  F_\delta [b,\mathtt{p}^{\mathcal C} ] \|_{C^{1} ( \overline{\Omega} \setminus {\mathcal C}^{\delta}  )} 
\leq \|  b \|_{C^{1} \big( \overline{\Omega} \setminus {\mathcal C}^{\delta} ; H^{\frac{1}{2}} (\mathcal C) \big)}  \, 
\|  \mathtt{p}^{\mathcal C}  \|_{H^{-\frac{1}{2}} (\mathcal C) } .
\end{equation*}
\item \label{deux} If  $\tilde{\mathcal C}$  is a smooth Jordan curve in $\overline{\Omega} \setminus {\mathcal C}^{\delta}$ and  $b $ in $C^{1} \big( \overline{\Omega} \setminus {\mathcal C}^{\delta} ; H^{\frac{1}{2}} (\mathcal C) \big)$, the operator  $F_{\delta,b}$ is compact from $L^{2} (\mathcal C)$ to $H^{1} (\tilde{\mathcal C})$.
\end{enumerate}
\end{lem}
The proof of Lemma \ref{op-reg} is elementary and left to the reader. \par
\subsection{Statements of the results}

\subsubsection{Circulation part}
\label{Cev-c}

Let $(\varepsilon, q) $ in $\mathfrak Q$.
We denote $\psi_\varepsilon (q,\cdot)$ in $\mathcal F_\varepsilon  (q)$ the function defined as the solution to the Dirichlet boundary value problem:
\begin{subequations} \label{model_stream}
\begin{alignat}{3}
\label{model_stream-a} -\Delta\psi_\varepsilon  (q,\cdot)&=0&&\text{in }\mathcal F_\varepsilon  (q) ,\\
\label{model_stream-b} \psi_\varepsilon  (q,\cdot)&=C_{\varepsilon} (q)&\quad&\text{on }\partial\mathcal S_\varepsilon  (q),\\
\label{model_stream-c} \psi_\varepsilon  (q,\cdot)&=0&&\text{on }\partial\Omega,
\end{alignat}
\index{BGrecYP4@$\psi_{\varepsilon}(q,\cdot)$: circulatory part of the stream function of the shrinking solid}
where the constant $C_{\varepsilon} (q)$ is such that:
\begin{equation} \label{53d}
\int_{\partial\mathcal S_\varepsilon  (q)} \frac{\partial\psi_\varepsilon}{\partial n} (q,\cdot) \, {\rm d}s= - 1.
\end{equation}
\end{subequations}
Here, $n$ stands for the unit normal vector to $\partial\mathcal S_\varepsilon  (q) \cup\partial\Omega$ directed toward the exterior of $\mathcal F_\varepsilon (q)$. %\par
The function $\psi_\varepsilon  $ is the counterpart, for the case where the size of the solid is of order $\eps$, of the function $\psi$ defined in \eqref{def_stream} in the case where  the size of the solid is of order $1$. For any $q $ in $\mathfrak Q$, the existence and uniqueness of a solution $\psi_\varepsilon (q,\cdot)$ of \eqref{model_stream} is classical. \par
\ \par
To state a result establishing an asymptotic expansion of  $C_{\varepsilon}  (q)$ and of $\frac{\partial\psi_\varepsilon}{\partial n} (q,\cdot) $ on $\partial\mathcal S_\varepsilon  (q)$  as $\varepsilon\to 0^{+}$ we introduce a few notations. \par
\ \par
\noindent
{\it Definition of $\psi^{\Ext}_{{\it 0}}  (q, \cdot) $ and of $P_{{\it 0}} (q, X) $.}
We denote, for any $q:= (\vartheta ,h) $ in $\R \times \Omega$, by $P_{{\it 0}} (q, X)$ the harmonic  polynomial 
\begin{equation} \label{polyh0}
P_{{\it 0}} (q,  X)  := u^{\Omega} (h)^{\perp} \cdot (   R(\vartheta) X - \zeta_{\vartheta})  .
\index{AP4@$P_{{\it 0}} (q, X) $: harmonic  polynomial extending  $\psi^{\Ext}_{{\it 0}}  (q, \cdot) $ in $\mathcal S_0$}
\end{equation}
Let us recall that $\zeta_\vartheta $ is defined in  \eqref{def-zeta_vartheta} in terms of $ \zeta$ defined in 
  \eqref{def-zeta}.

Recalling \eqref{plm+} and the second identity of \eqref{ineed}
we observe that $P_{{\it 0}} (q, X)$  satisfies 
\begin{equation}
\label{arti1}
\int_{\partial  \mathcal S_0} \,  P_{{\it 0}} (q, \cdot)  \mathsf{p}_{{\it -1}}^{\Ext} \, {\rm d}s = 0 .
\end{equation}
Therefore, according to Proposition \ref{coro-pro}
there exists a unique smooth function $\psi^{\Ext}_{{\it 0}}  (q, \cdot)$  satisfying
\begin{subequations} \label{model_stream0}
\begin{alignat}{3}
-\Delta\psi^{\Ext}_{{\it 0}} (q,\cdot)  & =0 & \quad & \text{in }\R^{2} \setminus \mathcal S_0 ,\\
\label{ohoh}
\psi^{\Ext}_{{\it 0}}  (q, \cdot) & = P_{{\it 0}} (q, \cdot ) & &\text{on }\partial\mathcal S_0 ,
\end{alignat}
and vanishing at infinity. 
Moreover
\begin{equation} \label{circ0nul}
\int_{\partial\mathcal S_0}  \frac{\partial\psi^{\Ext}_{{\it 0}}}{\partial n}  \, {\rm d}s  = 0  .
\index{BGrecYP5@$\psi^{\Ext}_{k}  (q, \cdot) $: $k$th-order profile, defined in $\R^{2} \setminus \mathcal S_0$}
\end{equation} 
\end{subequations}
\ \\
\noindent
{\it Definition of $\psi^{\Int}_{{\it 1}}  (q, \cdot) $.}
We also introduce the solution  $\psi^{\Int}_{{\it 1}}  (q, \cdot) $ of
\begin{subequations}
\label{model_stream1-noint}
\begin{alignat}{3}
-\Delta \psi^{\Int}_{{\it 1}} (q,\cdot)  &=0& &\text{in }\Omega ,\\
\label{pasdel}
\psi^{\Int}_{{\it 1}}  (q, \cdot) &=  - (\nabla G) (\cdot -h)  \cdot  \zeta_{\vartheta}  &\quad&\text{on }\partial\Omega .
\index{BGrecYP6@$\psi^{\Int}_{k} (q, \cdot) $: $k$th-order profile, defined in $\Omega$}
\end{alignat}
\end{subequations}
Above $G$ denotes the Newtonian potential  defined in \eqref{NewtonianPotential}.
The function $\psi^{\Int}_{{\it 1}}$ can be expressed thanks to the function $\psi^{\Int}_{{\it 0}}$ defined in \eqref{depsi0}
according to the following formula: 
\begin{equation} \label{29dec-eq}
\forall h,x \in \Omega, \quad \forall  \vartheta \in \R, \quad
D_{x} \psi^{\Int}_{{\it 0}} (h,x) \cdot \zeta_{\vartheta} = \psi^{\Int}_{{\it 1}} (\vartheta,x,h) .
\end{equation}
\begin{proof}[Proof of \eqref{29dec-eq}.]
We first recall  that  $\psi^{\Int}_{{\it 0}}$ is symmetric in its variables. 
Indeed by uniqueness of the Dirichlet problem  \eqref{depsi0} it follows that  for any $h $ in $\Omega$, the following decomposition holds in $ \Omega$:
\begin{equation} \label{deco-routh}
\psi^{\Int}_{{\it 0}} (h,\cdot) = G(\cdot-h ) + G^{\Omega}(h,\cdot ),
\end{equation}
where $G^{\Omega}$ denotes the Green function associated with the  domain $\Omega$ and the homogeneous Dirichlet condition, that is 
$$\Delta G^{\Omega}(h,\cdot )  =  \delta_h \text{ in } \Omega, \ \ 
 G^{\Omega}(h,\cdot ) =  0 \text{  on } \partial\Omega .$$	
Using the decomposition \eqref{deco-routh}, that the Newtonian potential $G$ is even
and the symmetry of $G^{\Omega}$ we deduce 
 \begin{equation} \label{symm-routh}
\forall h,x \in \Omega, \quad \psi^{\Int}_{{\it 0}} (h,x) = \psi^{\Int}_{{\it 0}} (x,h) .
\end{equation}

It follows that
$(D_{x} \psi^{\Int}_{{\it 0}}) (x,h) \cdot \zeta_{\vartheta} = (D_{h} \psi^{\Int}_{{\it 0}})  (h,x) \cdot \zeta_{\vartheta}$. 
Next we observe that  $D_{h} \psi^{\Int}_{{\it 0}} (h,\cdot)\cdot \zeta_{\vartheta}$ satisfies the same Dirichlet problem as
$\psi^{\Int}_{{\it 1}}  (\vartheta,h,\cdot)$, by derivation of  \eqref{depsi0}.
Formula \eqref{29dec-eq} follows then from the uniqueness of solutions to the Dirichlet problem, after switching $h$ and $x$. 
\qed 
\end{proof} 
\noindent
{\it Definition of $\psi^{\Ext}_{{\it 1}}  (q, \cdot) $ and of $P_{{\it 1}} (q, X) $.}
Let us denote, for any $q := (\vartheta ,h)$  in $\R \times \Omega$, by $P_{{\it 1}} (q, X)$ the  polynomial defined by
\begin{multline} \label{polyh1}
P_{{\it 1}} (q,X) := - \frac{1}{2}  \langle R(\vartheta)^t \,  D^2_x \psi^{\Int}_{{\it 0}} (h,h) R(\vartheta), 
T^2 ( \mathsf{p}_{{\it -1}}^{\Ext} ) + X^{\otimes 2} \rangle_{\R^{2 \times 2}} 
\\+  R(\vartheta)^t \,  D_x  \psi^{\Int}_{{\it 1}} (q,h)  \cdot \big(  \zeta - X \big) .
\index{AP5@$P_{{\it 1}} (q,  X) $:  harmonic polynomial extending  $\psi^{\Ext}_{{\it 1}}  (q, \cdot) $ in $\mathcal S_0$}
\end{multline}
Above $D^2_x \psi^{\Int}_{{\it 0}} (h,h)$ denotes the second derivative of $\psi^{\Int}_{{\it 0}} (h,\cdot)$ evaluated in $h$,
$D_x \psi^{\Int}_{{\it 1}} (q,h)$ stands for the derivative of $ \psi^{\Int}_{{\it 1}} (q,\cdot) $ evaluated in $h$
and $X^{\otimes 2}$ stands for the $2\times 2$ matrix $X \otimes X$.
The notation $T^2 ( \mathsf{p}_{{\it -1}}^{\Ext})$ stands for 
\begin{equation} \label{mordre2}
T^2 ( \mathsf{p}_{{\it -1}}^{\Ext} ) 
:= \int_{\partial \mathcal S_0}   \frac{\partial\psi^{\Ext}_{{\it -1}}}{\partial n} (X)   \,   X^{\otimes 2}   \, {\rm d}s(X)  .
\end{equation}
This notation is justified by \eqref{ineed}. %\par
Observe that $P_{{\it 1}} (q, X)$ is harmonic since every monomial of Taylor's expansions of harmonic functions are themselves harmonic.
By  \eqref{def-zeta}, \eqref{def-zeta_vartheta}  and the second identity of \eqref{ineed},  
\begin{equation}
\label{niouniou}
\int_{\partial  \mathcal S_0} \,  \mathsf{p}_{{\it -1}}^{\Ext} \,( \zeta - X) \, {\rm d}s = 0 .
\end{equation}
Thus, by \eqref{plm+} and \eqref{niouniou}, 
\begin{equation} \label{arti2}
\int_{\partial  \mathcal S_0} \,  P_{{\it 1}} (q, \cdot)  \mathsf{p}_{{\it -1}}^{\Ext} \, {\rm d}s = 0 .
\end{equation}
Therefore, according to Proposition \ref{coro-pro}
there exists a unique smooth function   $\psi^{\Ext}_{{\it 1}}  (q, \cdot) $ satisfying 
\begin{subequations} \label{model_stream1}
\begin{alignat}{3}
-\Delta\psi^{\Ext}_{{\it 1}} (q,\cdot)  &=0& &\text{in }\R^{2} \setminus \mathcal S_0 ,\\
\label{ahah-1}
\psi^{\Ext}_{{\it 1}}  (q, \cdot) &= P_{{\it 1}} (q,  \cdot)  &\quad&\text{on }\partial\mathcal S_0 ,
\end{alignat}
and vanishing at infinity. 
Moreover 
\begin{equation} \label{circ1nul}
\int_{\partial\mathcal S_0} \frac{\partial \psi^{\Ext}_{{\it 1}}}{\partial n} (q, \cdot)  \, {\rm d}s  = 0  .
\end{equation}
\end{subequations}
%
%\tilde{}
\ \par
Our main result regarding the potential $\psi_\varepsilon$,  in addition to Lemma~\ref{drift}, is the following.
\begin{prop} \label{dev-psi}
There exist $\mathtt{p}^{\partial \mathcal S_0}_{r}: \mathfrak Q \rightarrow L^2 (\partial \mathcal S_{0} ; \R )$, depending only on  $\mathcal S_0$ and $\Omega$, 
such that, for any $\delta >0$, there exists $\eps_{0}$ in $(0,1)$ for which $\mathtt{p}^{\partial \mathcal S_0}_{r} \in L^\infty \big(\mathfrak Q_{\delta,\eps_{0}} ; L^2 (\partial \mathcal S_{0} ; \R ) \big)$, and such that for any $(\varepsilon , q) $ in $\mathfrak Q_{\delta,\eps_{0}}$
and for any $X $ in $\partial \mathcal S_0$, 
\begin{multline} \label{exp-psi_n}
\frac{\partial\psi_\eps}{\partial n} (q,\eps R(\vartheta) X +h) 
= \frac{1}{\eps} \frac{\partial\psi^{\Ext}_{{\it -1}}}{\partial n} (X)
+ \left(   \frac{\partial \psi^{\Ext}_{{\it 0}} }{\partial n} (q,X)  -  R(\vartheta)^t \, u^\Omega (h) \cdot \tau \right) \\
+ \eps \left(   \frac{\partial\psi^{\Ext}_{{\it 1}}}{\partial n}   -   \frac{\partial P_{{\it 1}}}{\partial n}     \right) (q,X) 
+ \eps^{2} \mathtt{p}^{\partial \mathcal S_0}_{r}(\eps,q, X) .
\end{multline}
\end{prop}
We recall that the set $\mathfrak Q_{\delta,\eps_{0}}$ was defined in \eqref{deqQed}. \par
%
%
%\ \par
%
%
The proofs of Lemma~\ref{drift} and of Proposition~\ref{dev-psi} are gathered in Subsection~\ref{proof-deux}. \par
%
%
%
%
%
%============================================
%
%
%
%
\subsubsection{Potential part}
For any $j=1,2,3$,  for  any $q$ in $\mathcal Q$, we consider the functions $K_{j,\eps}(q,\cdot)$ on $\partial\Omega  \cup \partial \mathcal S_\eps(q)$ given by:
\begin{equation} \label{Def-Kj-eps}
K_{j,\eps} (q,\cdot) := n \cdot  \xi_{j} (q,\cdot) \text{ on } \partial\Omega  \cup \partial \mathcal S_\eps (q) ,
\index{AK2@$K_{j,\eps}  (q,\cdot)$: normal trace of elementary rigid velocities on $\partial \mathcal S_\eps(q)$}
\end{equation}
where $n$ denotes the unit normal to $\partial \mathcal S_\eps(q) \cup \partial \Omega$, pointing outside ${\mathcal F}^\eps(q)$
and the functions $\xi_{j} (q,\cdot)$ are given by the formula \eqref{def-xi-j}. %\par
Then the Kirchhoff potentials $\varphi_{j,\eps} (q,\cdot)$, for $j=1,2,3$, are the unique (up to an additive constant) solutions in $\mathcal F^\eps (q)$ of the following Neumann problem:
\begin{subequations} \label{Kir-eps}
\begin{alignat}{3}
\Delta \varphi_{j,\eps} (q, \cdot) &= 0 & \quad & \text{ in } \mathcal F^\eps(q),\\
\frac{\partial \varphi_{j,\eps}}{\partial n} (q,\cdot)&=K_{j,\eps} (q,\cdot) & \quad & \text{ on }\partial\mathcal S_\eps(q), \\
\frac{\partial\varphi_{j,\eps}}{\partial n}(q,\cdot) & =0 & &\text{ on } \partial \Omega.
\index{BGrecTF5@$\varphi_{j,\eps}(q,\cdot)$ ($j=1,2,3$) :  Kirchhoff's potentials of the shrinking body}
\end{alignat}
\end{subequations}
The  functions $K_{j,\eps}(q,\cdot)$ (respectively $\varphi_{j,\eps} (q,\cdot)$) are the  counterpart, for the case where the size of the solid is of order $\eps$, of the  functions defined in \eqref{Def-Kj} (resp. in \eqref{Kir}) in the case where  the size of the solid is of order $1$. \par
We will use the vector notations:
\begin{equation} \label{Kir-vect-pasc}
\boldsymbol{{\varphi}}_\eps = ({\varphi}_{1, \eps} , \, {\varphi}_{2,\eps},\, {\varphi}_{3,\eps})^{t}
\ \text{ and } \ 
\boldsymbol{K}_\eps = (K_{1,\eps} , \, K_{2,\eps},\, K_{3,\eps})^{t} .
\index{BGrecTF6@$\boldsymbol{{\varphi}}_\eps (q,\cdot)$:  vector containing the three Kirchhoff potentials}
\end{equation}
Our result on the expansion of the Kirchhoff potentials $\varphi_{j,\eps}$ is the following. \par
\begin{prop} \label{dev-phi-pasc}
%
%Let $\delta >0$. There exists  $\eps_{0} $ in $(0,1)$ and 
There exist
\begin{enumerate}[(i)]
\item  $\boldsymbol{\varphi}_{r} : \mathfrak Q \rightarrow L^2 (\partial \mathcal S_{0} ; \R^3)$
and $\boldsymbol{\check{c}}: \mathfrak Q \rightarrow \R^3$
such that for any $(\varepsilon , q) $ in $\mathfrak Q$, with $q=(\vartheta , h)$,
for any $X $ in $\partial \mathcal S_0$, 
 \begin{equation}
\boldsymbol{{\varphi}}_\eps  (q,\eps R(\vartheta) X +h)
= \eps I_\eps \mathcal R (\vartheta)   \Big(\boldsymbol{{\varphi}}^{\Ext}  (X) 
  + \boldsymbol{\check{c}} (\eps,q)  + \eps^{2} \boldsymbol{\varphi}_{r}   (\eps,q, X) \Big) ,
\end{equation}
\item  $\boldsymbol{\mathtt{p}}^{\partial \mathcal S_0}_{r}: \mathfrak Q \rightarrow L^2 (\partial \mathcal S_{0} ; \R^3  )$
such that for any  $(\varepsilon , q) $ in $\mathfrak Q$ with $q=(\vartheta , h)$, for any $X $ in $\partial \mathcal S_0$, 
\begin{equation} \label{exp-phi_n-pasc}
\mathcal R (\vartheta)^{t} \, \frac{\partial \boldsymbol{{\varphi}}_\eps}{\partial \tau} (q,\eps R(\vartheta) X +h) 
= I_\eps   \Big( \frac{\partial \boldsymbol{{\varphi}}^{{\Ext}}}{\partial \tau} (X)  
+ \eps^{2} \boldsymbol{\mathtt{p}}^{\partial \mathcal S_0}_{r} (\eps,q, X) \Big) ,
\end{equation}
\item $\boldsymbol{\mathtt{p}}^{\partial \Omega}_{r} : \mathfrak Q \rightarrow L^2 (\partial \Omega ; \R^3  )$
such that for any  $(\varepsilon , q) $ in $\mathfrak Q$, for any $x $ in $\partial \Omega$, 
\begin{equation} \label{exp-phi_n-pasc-ext}
\frac{\partial \boldsymbol{{\varphi}}_\eps}{\partial \tau} (q, x) 
=  I_\eps \,  \eps^{2} \,   \boldsymbol{\mathtt{p}}^{\partial \Omega}_{r} (\eps,q, x) ,
\end{equation}
\end{enumerate}
and such that for any $\delta>0$, there exists $\varepsilon_{0}$ in $(0,1)$ such that
$\boldsymbol{\varphi}_{r}$ belongs to $L^\infty \big(\mathfrak Q_{\delta,\eps_{0}} ; L^2 (\partial \mathcal S_{0} ; \R^3 ) \big)$,
$\boldsymbol{\check{c}}$ to $L^\infty \big(\mathfrak Q_{\delta,\eps_{0}} ; \R^3 \big)$,
$\boldsymbol{\mathtt{p}}^{\partial \mathcal S_0}_{r}$ to $L^\infty \big(\mathfrak Q_{\delta,\eps_{0}} ; L^2 (\partial \mathcal S_{0} ; \R^3  ) \big)$
and $\boldsymbol{\mathtt{p}}^{\partial \Omega}_{r}$ to $L^\infty \big(\mathfrak Q_{\delta,\eps_{0}} ; L^2 (\partial \Omega ; \R^3  ) \big)$.

Moreover the remainders  $\boldsymbol{\varphi}_{r}$, $\boldsymbol{\mathtt{p}}^{\partial \mathcal S_0}_{r}$ and $\boldsymbol{\mathtt{p}}^{\partial \Omega}_{r}  $
depend only on  $\mathcal S_0$ and $\Omega$.
\end{prop}
We recall that $\boldsymbol{\varphi}^{\Ext}$ was defined in \eqref{Kir-exte}.
The proof of Proposition \ref{dev-phi-pasc} is given in Subsection \ref{sec-proof-dev-phi}. \par
\subsection{Asymptotic expansion of the circulation part: Proof of Proposition~\ref{dev-psi} and of Lemma~\ref{drift}}
\label{proof-deux}
In this subsection we prove Proposition~\ref{dev-psi} and Lemma~\ref{drift}. 
Let, for $-\frac{1}{2} \leq s  \leq 1$, $F_{s}$ denote the following Hilbert space:
\begin{equation*}
F_{s} :=  H^{s} (\partial \mathcal S_{0} )  \times  H^{s} (\partial \Omega )  \times  \R ,
\end{equation*}
We will mainly make use of the indices $s=0$ and $1$ and also for technical reasons of $-\frac{1}{2}$ and $\frac{1}{2}$. 
We will proceed in four steps. 

\subsubsection{First Step. Reduction to integral equations}

We look for the solution  $\psi_\varepsilon (q,\cdot)$ of \eqref{model_stream} as a superposition of two single-layer integrals, one supported on the body's boundary and the other one supported on  $\partial \Omega$. This transforms \eqref{model_stream} in an integral system as follows.

We define, for any $(\eps,q) $ in $\mathfrak Q$ with $q=(\vartheta ,h) $ in $\R \times \Omega$, two operators $K^{\partial  \mathcal S_0} (\eps,q)$ and $K^{\partial  \Omega} (\eps,q)$ respectively from $L^{2}(\partial \Omega)$ to $H^{1} (\partial \mathcal S_{0}) $ and from $L^{2}(\partial \mathcal S_{0})$ to $ H^{1} (\partial \Omega)$, by the following formulas: given  densities $\mathtt{p}^{\partial \Omega}$ and $\mathtt{p}^{\partial  \mathcal S_{0} }$ respectively in $L^{2} (\partial \Omega )$ and $L^{2} (\partial \mathcal S_{0})$,
\begin{gather}
\label{K1}
K^{\partial  \mathcal S_0} (\eps,q)  [ \mathtt{p}^{\partial \Omega}] (\cdot) 
:= SL   [   {p}^{\partial \Omega} ] (\eps R(\vartheta) \cdot +h )  
\text{ on } \partial \mathcal S_{0}, \\
\label{K2}
K^{\partial  \Omega} (\eps,q)  [ \mathtt{p}^{\partial  \mathcal S_0}] (\cdot)
:=  \int_{\partial\mathcal S_0}  \mathtt{p}^{\partial \mathcal S_0} (Y) G\big(\cdot -(\eps R(\vartheta) Y +h) \big) \, {\rm d}s(Y) 
\text{ on } \partial \Omega.
\end{gather}
Thanks to Lemma~\ref{op-reg} \eqref{deux}, the operators $K^{\partial  \mathcal S_0} (\eps,q)$ and $K^{\partial  \Omega} (\eps,q)$ are compact
respectively from $L^{2}(\partial \Omega)$ to $H^{1}(\partial \mathcal S_{0})$ and from $L^{2}(\partial \mathcal S_{0})$ to $H^{1} (\partial \Omega)$. \par
We also introduce for $(\eps,q) $ in $\mathfrak Q$, the operator $\mathfrak{A} (\eps,q) : F_{0} \rightarrow F_{1} $ as follows: for any 
$\mathfrak{p} := (\mathtt{p}^{\partial  \mathcal S_0} , \mathtt{p}^{\partial  \Omega} , C)$ in $F_{0}$,
\begin{multline} \label{op-depart}
\mathfrak{A} (\eps,q)  [ \mathfrak{p} ] 
:=  \Big(   SL [ \mathtt{p}^{\partial  \mathcal S_0}] + K^{\partial \mathcal S_{0}} (\eps,q)   [ \mathtt{p}^{\partial \Omega}] - C , \\
SL [ \mathtt{p}^{\partial \Omega}] +K^{\partial \Omega} (\eps,q)  [ \mathtt{p}^{\partial  \mathcal S_0}]  ,
\int_{\partial  \mathcal S_{0}} \mathtt{p}^{\partial \mathcal S_{0}} \, {\rm d}s \Big) .
\end{multline}
To simplify the notations, we omitted to write the trace operators applied to the single-layers in $K^{\partial  \mathcal S_0} (\eps,q)$, $K^{\partial  \Omega} (\eps,q)$ and $\mathfrak{A} (\eps,q)$.
We also emphasize that the dependence of  $\mathfrak{A} (\eps,q)$ on $(\eps,q)$ occurs only through the compact operators   $K^{\partial  \mathcal S_0} (\eps,q)$ and $K^{\partial  \Omega} (\eps,q) $. \par
\ \par
Now the equation \eqref{model_stream} is transformed into an integral system thanks to the following lemma.
\begin{lem} \label{pot}
For any $(\eps,q)  $ in $\mathfrak Q$,  let
$\mathfrak{p}_{\varepsilon}  (q,\cdot) = \big(p_{\varepsilon}^{\partial \mathcal S_0}    (q,\cdot)  , {p}_{\varepsilon}^{\partial \Omega}  (q,\cdot)   , C_{\varepsilon} (q) \big) $ in $F_{0}$
such that 
\begin{equation} \label{rgroup}
\mathfrak{A} (\eps,q) \big[ \mathfrak{p}_{\varepsilon}  (q,\cdot)   + \big(0,0, G(\eps)\big) \big]  = (0,0,-1) .
 \end{equation}
Consider  the density $ p_{\varepsilon}^{\partial \mathcal S_\varepsilon  (q)} (q,\cdot)$ on $\partial\mathcal S_\varepsilon (q)$ defined through the relation:
\begin{equation} \label{chgv}
\text{for } X \in \partial\mathcal S_{0}, \quad
{p}_{\varepsilon}^{\partial \mathcal S_{0}} (q, X)  :=  \varepsilon p_{\varepsilon}^{\partial \mathcal S_\varepsilon  (q)} (q,\eps R(\vartheta) X +h) .
\end{equation}
Then the function in ${\mathcal F}^{\varepsilon} (q)$
\begin{equation} \label{twoSlayers}
\psi_\varepsilon (q,\cdot) 
:= SL [  p_{\varepsilon}^{\partial \mathcal S_\varepsilon  (q)} (q,\cdot) ] 
+ SL [  p_{\varepsilon}^{\partial \Omega} (q,\cdot  ) ],
\end{equation}
is the solution to \eqref{model_stream}.
Moreover the normal derivative $\frac{\partial\psi_\varepsilon}{\partial n} (q,\cdot) $ on $\partial \mathcal S_\eps (q)$ is given by:
\begin{equation} \label{exp-psi_density}
\text{for } X \in \partial\mathcal S_{0} , \quad 
\frac{\partial\psi_\eps}{\partial n} (q,\eps R(\vartheta) X +h) =  \frac{1}{\varepsilon}{p}_{\varepsilon}^{\partial \mathcal S_{0}}    (q,X) .
\end{equation}
\end{lem}
%-
%-
%
\begin{proof}
First observe that for any densities $p_{\varepsilon}^{\partial \mathcal S_\varepsilon (q)} (q,\cdot) $ in $H^{-\frac12} (\partial\mathcal S_\varepsilon (q))$
and $p_{\varepsilon}^{\partial \Omega} (q,\cdot )  $ in $H^{-\frac12} (\partial\Omega )$, 
the right hand side of \eqref{twoSlayers} is in $H^{1}_\mathrm{loc} (\R^{2})$ and harmonic in $\mathcal F_\varepsilon(q)$ and in $\R^{2} \setminus \mathcal F_\varepsilon(q)$.
In particular the equation  \eqref{model_stream-a} is satisfied when $\psi^{\varepsilon} (q,\cdot)$ is given by \eqref{twoSlayers} without further assumptions about 
$p_{\varepsilon}^{\partial \mathcal S_\varepsilon (q)} (q,\cdot)$ or $p_{\varepsilon}^{\partial \Omega} (q,\cdot)$. \par
Next we write \eqref{rgroup} explicitly in the form:
\begin{subequations} \label{splitte}
\begin{alignat}{3}
\label{in1}
-G(\eps) +  SL [ p_{\varepsilon}^{\partial \mathcal S_0}    (q,\cdot)] + K^{\partial \mathcal S_{0}}  (\eps,q)  [ p_{\varepsilon}^{\partial \Omega}    (q,\cdot)]
&= C_{\varepsilon} (q) &&   \text{  on }  \partial\mathcal S_{0}, \\
\label{in2}
SL [ p_{\varepsilon}^{\partial \Omega}    (q,\cdot)] +K^{\partial \Omega}  (\eps,q)  [ p_{\varepsilon}^{\partial \mathcal S_0}    (q,\cdot)]  
&= 0 & \quad & \text{ on } \partial \Omega , \\
\label{53dnew}
\int_{\partial  \mathcal S_{0}} p_\varepsilon^{\partial \mathcal S_{0}} (q,\cdot) \, {\rm d}s  &= -1 . & 
\end{alignat}
\end{subequations}
Thanks to a change of variable, the identity $G\big(\eps (x -y )\big)= G(\eps) + G (x -y )$,  \eqref{K1}, \eqref{K2} and \eqref{chgv}, 
 \eqref{splitte} can be recast as 
\begin{subequations} \label{splitte2}
\begin{alignat}{3}
\label{99}
SL [  p_{\varepsilon}^{\partial \mathcal S_\varepsilon (q)}    (q,\cdot) ] 
+  SL  [   {p}_{\varepsilon}^{\partial \Omega} (q,\cdot  ) ]
&= C_{\varepsilon} (q) & \quad &  \text{ on }  \partial\mathcal S_\varepsilon  (q) , \\
\label{99bis}
SL [ p_{\varepsilon}^{\partial \mathcal S_\varepsilon (q)}    (q,\cdot) ] 
+ SL [ {p}_{\varepsilon}^{\partial \Omega} (q,\cdot  ) ] &= 0  & \quad & \text{ on }  \partial \Omega , \\
\label{53dnew-ante}
\int_{\partial\mathcal S_\varepsilon  (q)} p_{\varepsilon}^{\partial \mathcal S_\varepsilon  (q)} (q,\cdot) \, {\rm d}s = -1 . & 
\end{alignat}
\end{subequations}
In particular we infer from \eqref{99} and \eqref{99bis} that, when  $\psi_{\varepsilon} (q,\cdot)$ is given by \eqref{twoSlayers} with  $\mathfrak{p}_{\varepsilon}  (q,\cdot) = \big(p_{\varepsilon}^{\partial \mathcal S_0}    (q,\cdot)  , {p}_{\varepsilon}^{\partial \Omega}  (q,\cdot)   , C_{\varepsilon} (q) \big)$ solution to \eqref{rgroup},  the boundary conditions \eqref{model_stream-b} and \eqref{model_stream-c} are satisfied. %\par
Moreover, by uniqueness of the solutions to the Poisson problem:
\begin{equation*}
\Delta \Psi = 0 \text{ in } \mathcal S_\varepsilon (q) , \quad  \Psi =  C_{\varepsilon} (q)  \text{ on } \partial \mathcal S_\varepsilon (q) ,
\end{equation*}
the right hand side of \eqref{twoSlayers} is equal to $C_{\varepsilon} (q)$ in $\mathcal  S_\varepsilon (q)$. \par
The single-layer potential $SL[ {p}_{\varepsilon}^{\partial \Omega} (q,\cdot) ]$
is smooth in a neighborhood of $\partial \mathcal S_\varepsilon  (q)$. Hence, according to \eqref{saut},
when $\psi^{\varepsilon} (q,\cdot)$ is given by  \eqref{twoSlayers},
the density ${p}_{\varepsilon}^{\partial \mathcal S_\varepsilon  (q)} (q,\cdot)$ is equal to the jump  across $\partial\mathcal S_\varepsilon  (q)$ of the normal derivatives of the function equal to 
$ \psi_\varepsilon (q,\cdot)$ in $\mathcal F_\varepsilon  (q)$ and to $ C_{\varepsilon} (q)$ in $\mathcal  S_\varepsilon (q)$, that is 
\begin{equation*}
{p}_{\varepsilon}^{\partial \mathcal S_\varepsilon  (q)}    (q,\cdot) = \frac{\partial\psi_\eps}{\partial n} (q,\cdot)  \text{ on }\partial\mathcal S_\varepsilon  (q) .
\end{equation*}
Hence we obtain \eqref{exp-psi_density} by using \eqref{chgv} and  the condition \eqref{53d} by using \eqref{53dnew-ante}. 
This concludes the proof of Lemma~\ref{pot}.
\qed 
\end{proof}
\subsubsection{Second Step. Construction of an approximate solution}
In this step we describe an approximation $\mathfrak{p}_{\rm app}$ up to order ${O}(\varepsilon^{3})$ of the solution $\mathfrak{p}_{\varepsilon}$ of \eqref{rgroup} and reformulate the equation \eqref{rgroup} in terms of the rest $\mathfrak{p}_{\varepsilon} - \mathfrak{p}_{\rm app}$. 
We first introduce the various terms involved in the approximation.

\ \par
\noindent
{\it Densities on $\partial \mathcal S_{0}$.} 
Recall that the functions $\psi^{\Ext}_{{\it -1}} $, $\psi^{\Ext}_{{\it 0}} (q, \cdot)$ and $\psi^{\Ext}_{{\it 1}} (q, \cdot)$, defined  respectively in \eqref{model_stream-1}, \eqref{model_stream0} and \eqref{model_stream1}, are harmonic in $\mathbb R^2\setminus\mathcal S_0$.

Let $\mathsf{p}^{\Ext}_{{\it -1}}$, $\mathsf{p}^{\Ext}_{{\it 0}} (q, \cdot)$ and $\mathsf{p}^{\Ext}_{{\it 1}}  (q, \cdot)$ be the densities on $\partial \mathcal S_{0}$ 
associated respectively with $\psi^{\Ext}_{{\it -1}} $, $\psi^{\Ext}_{{\it 0}} (q, \cdot)$ and $\psi^{\Ext}_{{\it 1}} (q, \cdot)$ as explained in Propositions~\ref{coro-cap} and \ref{coro-pro}. Thus, in $\mathbb R^2\setminus \mathcal S_0$:
\begin{equation} \label{thatis}
\psi^{\Ext}_{{\it -1}} = SL [ \mathsf{p}^{\Ext}_{{\it -1}}]\quad\text{ and }\quad\psi^{\Ext}_{{\it j}}  (q, \cdot)  = SL [  \mathsf{p}^{\Ext}_{{\it j}}  (q,\cdot) ]\quad (j={\it 0,1}).
\end{equation}
The single layer potentials, as being supported by $\partial\mathcal S_0$, are actually defined in $\mathbb R^2$ and harmonic in $\mathbb R^2\setminus \partial\mathcal S_0$.
The identities above can therefore be used to extend the functions $\psi^{\Ext}_{{\it -1}} $, $\psi^{\Ext}_{{\it 0}} (q, \cdot)$ and $\psi^{\Ext}_{{\it 1}} (q, \cdot)$ 
in $\mathcal S_0$.  More explicitly $\psi^{\Ext}_{{\it -1}}$  is extended by $C^{\Ext}$ in $ \mathcal S_0$, 
 $\psi^{\Ext}_{{\it 0}} (q, \cdot)$ is extended by $P_{{\it 0}} (q, \cdot)$ in $\mathcal S_0$ and
 $\psi^{\Ext}_{{\it 1}} (q, \cdot)$ is  extended by $P_{{\it 1}} (q, \cdot)$ in $\mathcal S_0$, 
 see the definitions \eqref{polyh0} and \eqref{polyh1} of the harmonic polynomials $P_{{\it 0}}$ and $P_{{\it 1}}$.
Moreover using \eqref{saut}  
\begin{align}
\label{rel-sauts}
\mathsf{p}^{\Ext}_{{\it 0}}
= \frac{\partial\psi^{\Ext}_{{\it 0}}}{\partial n}
 - \frac{\partial P_{{\it 0}}}{\partial n} 
 \text{ and }
\mathsf{p}^{\Ext}_{{\it 1}}  
= \frac{\partial\psi^{\Ext}_{{\it 1}}}{\partial n} 
 - \frac{\partial P_{{\it 1}}}{\partial n} .
\end{align}

\ \par
\noindent
{\it Densities on $\partial \Omega$.}  
We follow the same ideas to extend in $\mathbb R^2$ the functions $\psi^{\Int}_{{\it 0}} (h, \cdot)$, $\psi^{\Int}_{{\it 1}} (q, \cdot)$ and $\psi^{\Int}_{{\it 2}} (q, \cdot)$, so far defined and harmonic 
in $\Omega$ (see \eqref{depsi0}, \eqref{model_stream1-noint} and \eqref{model_stream-2-noint}). 
We consider their densities $\mathsf{p}^{\Int}_{{\it 0}} (q,\cdot)$, $\mathsf{p}^{\Int}_{{\it 1}} (q,\cdot)$ and $\mathsf{p}^{\Int}_{{\it 2}} (q,\cdot)$, supported in $\partial\Omega$ and such that, in $\Omega$:
\begin{equation} \label{thatisback}
\psi^{\Int}_{{\it 0}}  (h, \cdot) = SL [  \mathsf{p}^{\Int}_{{\it 0}}   (h,\cdot)  ] , \ \
\psi^{\Int}_{{\it j}}  (q, \cdot) = SL [  \mathsf{p}^{\Int}_{{\it j}}   (q,\cdot) ] \quad (j={\it1,2}). 
\end{equation}
The single layer potentials being defined in $\mathbb R^2$ and harmonic $\mathbb R^2\setminus\partial\Omega$, these identities are used to extend the functions 
$\psi^{\Int}_{{\it 0}} (h, \cdot)$, $\psi^{\Int}_{{\it 1}} (q, \cdot)$ and $\psi^{\Int}_{{\it 2}} (q, \cdot)$ in $\mathbb R^2$. 
More explicitly
 $\psi^{\Int}_{{\it 0}} (h, \cdot)$ is extended by $G(\cdot -h)$ in $\R^{2} \setminus  \Omega$, 
$\psi^{\Int}_{{\it 1}} (q, \cdot)$ is extended by $ - (\nabla G) (\cdot -h)  \cdot  \zeta_{\vartheta}$ in $\R^{2} \setminus  \Omega$, and
$\psi^{\Int}_{{\it 2}} (q, \cdot)$ is extended by $Q_{{\it 2}} (q, \cdot) $, defined in \eqref{defQ2},  in $\R^{2} \setminus \Omega$.

\ \par
\noindent
{\it Definition of $\psi^{\Int}_{{\it 2}}  (q, \cdot) $ and of $Q_{{\it 2}} (q, \cdot) $.}
To define  $\psi^{\Int}_{{\it 2}}  (q, \cdot) $,  we introduce the harmonic function in $\R^{2} \setminus  \Omega$:
\begin{multline} \label{defQ2}
Q_{{\it 2}} (q,x) :=  \frac{1}{2}  \langle R(\vartheta)^t \,  D^2 G  (x-h) R(\vartheta), T^2 ( \mathsf{p}_{{\it -1}}^{\Ext} )  \rangle_{\R^{2 \times 2}} 
  \\+  R(\vartheta)^t \,  \nabla G  (x-h)  \cdot  T^1 (\mathsf{p}^{\Ext}_{{\it 0}}  (q, \cdot)  )    ,
\end{multline}
where
\begin{equation} \label{momentP}
T^1 (\mathsf{p}^{\Ext}_{{\it 0}}  (q, \cdot)  )  := 
  \int_{\partial\mathcal S_0}  X \mathsf{p}^{\Ext}_{{\it 0}} (q,X) \, {\rm d}s(X) . 
\end{equation}
Then we consider  $\psi^{\Int}_{{\it 2}}  (q, \cdot) $ as the solution to 
\begin{subequations}
\label{model_stream-2-noint}
\begin{alignat}{3}
-\Delta\psi^{\Int}_{{\it 2}} (q,\cdot)  &=0& &\text{in }\Omega ,\\
\label{ahah}
\psi^{\Int}_{{\it 2}}  (q, \cdot) &= Q_{{\it 2}} (q,  \cdot)  &\quad&\text{on }\partial  \Omega ,
\end{alignat}
\end{subequations}
extended by $Q_{{\it 2}} (q,x)$ for $x$ in  $\R^{2} \setminus  \Omega$. These functions $\psi^{\Int}_{{\it 2}}  (q, \cdot) $ and $Q_{{\it 2}} (q, \cdot) $ do not appear in the claim of Proposition~\ref{dev-psi} and of Lemma~\ref{drift} but will be useful later.
With the choices above we aim at constructing a solution to \eqref{rgroup} with $p_{\varepsilon}^{\partial \mathcal S_0} (\varepsilon ,q,\cdot)$ and ${p}_{\varepsilon}^{\partial \Omega} (\varepsilon ,q,\cdot)  $ close respectively to
\begin{align} \label{p_app}
\mathtt{p}^{\partial \mathcal S_0}_{\text{app} }  (\varepsilon ,q,\cdot)  
&:=  \mathsf{p}_{{\it -1}}^{\Ext}
 + \varepsilon  \mathsf{p}^{\Ext}_{{\it 0}}   (q,\cdot)
 + \varepsilon^{2}\,  \mathsf{p}^{\Ext}_{{\it 1}}   (q,\cdot) , \\
%\text{ and } 
\nonumber
\mathtt{p}^{\partial \Omega}_{\text{app} }  (\varepsilon ,q,\cdot)
&:= \mathsf{p}^{\Int}_{{\it 0}}  (h,\cdot)
 + \varepsilon \mathsf{p}^{\Int}_{{\it 1}}   (q,\cdot)
 + \varepsilon^{2}  \mathsf{p}^{\Int}_{{\it 2}} (q,\cdot) .
 \end{align}
The corresponding approximation $C_{\text{app}}  (\varepsilon,q)$ of   $C_{\varepsilon}(q)$ is chosen as:
\begin{equation} \label{C_app}
C_\text{app} (\varepsilon,q) := - G(\varepsilon) +  \mathsf{C}_{{\it 0}} (h) + \eps \mathsf{C}_{{\it 1}} (q) + \eps^{2}  \mathsf{C}_{{\it 2}} (q) ,
\end{equation}
with
\begin{align}
\nonumber
\mathsf{C}_{{\it 0}} (h) :=& C^{\Ext} +  \psi^{\Int}_{{\it 0}}  (h, h) , \\
\nonumber
\mathsf{C}_{{\it 1}} (q) :=&2 D_x \psi^{\Int}_{{\it 0}}  (h, h)  \cdot {\zeta}_\vartheta , \\   
\label{C2} 
\mathsf{C}_{{\it 2}} (q) := &\psi^{\Int}_{{\it 2}}  (q, h) 
   + D_{x}   \psi^{\Int}_{{\it 1}}  (q, h)  \cdot  \zeta_{\vartheta} \nonumber
   \\&+ \frac{1}{2}  \langle  R(\vartheta)^t \,  D^2_x \psi^{\Int}_{{\it 0}} (h,h) R(\vartheta), T^2 ( \mathsf{p}_{{\it -1}}^{\Ext} ) \rangle_{\R^{2 \times 2}}  .
\end{align}
Using that $\psi^{\Int}_{{\it 0}}$ is symmetric with respect to its two arguments (see \eqref{symm-routh}), and using \eqref{def-KRS}, we see that the first terms of the expansion above are the same as those claimed in Lemma~\ref{drift}, that is 
\begin{equation} \label{lien-intro}
\mathsf{C}_{{\it 0}} (h) := C^{\Ext} +  2  \psi^{\Omega} (h) 
\ \text{ and } \ 
\mathsf{C}_{{\it 1}} (q) := 2  \psi_{c} (q)   .
\end{equation}
We finally define
\begin{equation} \label{DefPapp}
\mathfrak{p}_\text{app}     (\varepsilon ,q,\cdot)  :=  \big(  \mathtt{p}^{\partial \mathcal S_{0}}_{\text{app}}    (\varepsilon ,q,\cdot)  , \mathtt{p}^{\partial \Omega}_{\text{app} }  (\varepsilon ,q,\cdot)  , C_\text{app} (\varepsilon,q) \big) .
\end{equation}
\ \par
Now the equation \eqref{rgroup} translates as follows.
Let us introduce $g^{\partial \mathcal S_0}(\varepsilon , q,\cdot)$ and $g^{\partial \Omega} (\varepsilon , q,\cdot)$ two functions respectively defined on $\partial \mathcal S_{0}$ and $\partial \Omega$, for
$q= (\vartheta,h)$, by 
\begin{subequations} \label{lesource}
\begin{align}
\label{def-gi}
- g^{\partial \mathcal S_0}  (\varepsilon , q,\cdot)
:=& \sum_{j=0}^{2} \int_{\partial \Omega} \mathsf{p}^{\Int}_{{\it j}} (q,y) \eta_{{\it 3-j}} (\varepsilon, q, \cdot, y)  \, {\rm d}s(y) , \\
\label{def-ge}
- g^{\partial \Omega} (\varepsilon , q,x)
:=& \int_{\partial\mathcal S_0}  \mathsf{p}_{{\it -1}}^{\Ext}  (y) \eta_{{\it 3}} (\varepsilon, (\vartheta,x),- y,h) \, {\rm d}s(y) \nonumber
 \\&+ \sum_{j=0}^{1} \int_{\partial\mathcal S_0} \mathsf{p}^{\Ext}_{{\it j}} (q,y) \eta_{{\it 2-j}}  (\varepsilon, (\vartheta,x),- y,h) \, {\rm d}s(y) ,
\end{align}
\end{subequations}
where, for $N \geq 1$,
\begin{equation} \label{rest-Taylor}
\eta_{{\it N}} (\varepsilon , q,\cdot,y)
 :=  \int_0^{1} \frac{ (1-\sigma)^{N-1} }{(N-1) ! }  D^{N} G (\sigma \eps R(\vartheta) \cdot +h -y)  \cdot (R(\vartheta) \cdot )^{\otimes N} {\rm d}\sigma  .
\end{equation}
Let 
\begin{equation} \label{def-frakg}
\mathfrak{g}  (\varepsilon , q,\cdot)
 := \big(  g^{\partial \mathcal S_0} (\varepsilon , q,\cdot) , g^{\partial \Omega} (\varepsilon , q,\cdot) , 0 \big).
\end{equation}
We can deduce from the definitions of the densities $\mathsf{p}^{\Int}_{{\it j}}$ for $j=0,1,2$, $ \mathsf{p}_{{\it -1}}^{\Ext}$ and $\mathsf{p}^{\Ext}_{{\it j}}$ for $j=0$ and $1$, and from Lemma~\ref{op-reg}, \eqref{deux} that $g^{\partial \mathcal S_0}(\varepsilon , q,\cdot)$ and $g^{\partial \Omega} (\varepsilon , q,\cdot)$ belong respectively to $H^{1} (\partial \mathcal S_{0} )$ and to $H^{1} (\partial \Omega)$. Actually we even have  
\begin{equation} \label{gbd}
\mathfrak{g} \in L^{\infty} (\mathfrak Q^{\delta} ; {F}_{1} ). 
\end{equation}
We can now state the result of this second step.
\begin{lem} \label{jedecoupe}
For any $(\eps,q)$ in $\mathfrak Q$, let $\mathfrak{p}_{r} (\varepsilon ,q,\cdot)$ in $F_{0}$  satisfy:
\begin{equation} \label{rgroup-reste}
\mathfrak{A} (\eps,q)  [ \mathfrak{p}_{r}  (\varepsilon ,q,\cdot)   ] =  \mathfrak{g}  (\varepsilon , q,\cdot)   .
\end{equation}
Then
\begin{align} \label{in3}
\mathfrak{p}_{\varepsilon}  (q,\cdot) 
:= \mathfrak{p}_\text{app}     (\varepsilon ,q,\cdot)  
+ \eps^{3}\, \mathfrak{p}_{r}  (\varepsilon ,q,\cdot),
\end{align}
is solution to \eqref{rgroup}.
\end{lem}
Let us stress in particular that the third  coordinate of the left hand side of  \eqref{rgroup-reste} does not contain the singular term $G(\eps)$ anymore (compare with \eqref{rgroup}) and that the third coordinate of the right hand side of \eqref{rgroup-reste}  is now $0$. 
\begin{proof}
Let $(\eps , q) $ in $ \mathfrak Q$ and $ \mathfrak{p}_{r} (\varepsilon ,q,\cdot) := \big(\mathtt{p}^{\partial \mathcal S_0}_{r} (\varepsilon ,q,\cdot), {p}^{\partial \Omega}_{r} (\varepsilon ,q,\cdot), C_{r} (\varepsilon ,q) \big) $ in $F_{0}$ satisfying \eqref{rgroup-reste} (where $\mathfrak{A}$ is defined in \eqref{op-depart}), that is 
\begin{gather} %\label{splitte-r}
\label{Preum}
SL [ \mathtt{p}^{\partial \mathcal S_0}_{r}    (\varepsilon ,q,\cdot)] + K^{\partial \mathcal S_{0}}  (\eps,q)  [ \mathtt{p}^{\partial \Omega}_{r} (\varepsilon ,q,\cdot)] -C_{r} (\varepsilon ,q)
=   g^{\partial \mathcal S_0} (\varepsilon , q,\cdot) \text{ on } \partial {\mathcal S}_{0},  \\
\label{Vend-fin}
SL [ \mathtt{p}^{\partial \Omega}_{r}    (\varepsilon ,q,\cdot)] +K^{\partial \Omega} (\eps,q)   [ \mathtt{p}^{\partial \mathcal S_0}_{r}    (\varepsilon ,q,\cdot)] 
= g^{\partial \Omega} (\varepsilon , q,\cdot)  \text{ on } \partial \Omega, \\
\label{circu-reste}
\int_{\partial  \mathcal S_{0}} \mathtt{p}^{\partial \mathcal S_0}_{r}   (\varepsilon ,q,\cdot) \, {\rm d}s = 0 .
\end{gather}
Let $\mathfrak{p}_{\varepsilon}  (q,\cdot)   = \big(p_{\varepsilon}^{\partial \mathcal S_0} (q,\cdot), {p}_{\varepsilon}^{\partial \Omega} (q,\cdot), C_{\varepsilon} (q) \big)  $ in $  F_{0}$ given by \eqref{in3}. To prove \eqref{rgroup} we now verify the three parts of \eqref{splitte}. \par

\ \par
Let us start with \eqref{53dnew}.
Using the second equality in \eqref{ineed} and \eqref{rel-sauts}, the fact that $P_{{\it 0}}$ and $P_{{\it 1}}$ are harmonic and the conditions \eqref{53d-1}, \eqref{circ0nul} and \eqref{circ1nul}, 
we arrive at 
\begin{equation} \label{nvelan7}
\int_{\partial\mathcal S_0}  \mathsf{p}_{{\it -1}}^{\Ext} \, {\rm d}s = - 1
\ \text{ and } \ 
\int_{\partial\mathcal S_0}  \mathsf{p}^{\Ext}_{{\it 0}}  (q,\cdot) \, {\rm d}s 
=  \int_{\partial\mathcal S_0}  \mathsf{p}^{\Ext}_{{\it 1}}  (q,\cdot) \, {\rm d}s = 0 .
\end{equation}
We deduce from  \eqref{circu-reste} and   \eqref{nvelan7}   that  the condition \eqref{53dnew} is fulfilled. 

\ \par
Let us now verify \eqref{in1}.
Using \eqref{DefPapp}, \eqref{in3}, \eqref{ineed}, \eqref{model_stream-1-b}, \eqref{thatis},   we obtain,  on $ \partial\mathcal S_{0}$, 
\begin{align} \label{nvelan8}
SL [ p_{\varepsilon}^{\partial \mathcal S_0}  (q,\cdot) ] 
&= C^{\Ext} + \varepsilon  \psi^{\Ext}_{{\it 0}}  (q, \cdot) 
 + \varepsilon^{2} \psi^{\Ext}_{{\it 1}}  (q, \cdot) 
 + \varepsilon^{3} SL [ p^{\partial \mathcal S_0}_{\varepsilon ,r}  (\varepsilon,q,\cdot)] .
\end{align}
It follows from \eqref{K1}, \eqref{DefPapp}, \eqref{in3}, \eqref{thatisback}, Taylor's formula and \eqref{29dec-eq}  that, on $\partial {\mathcal S}_{0}$, 
\begin{multline} \label{nvelan9bis}
K^{\partial \mathcal S_{0}}   (\eps,q)  [ p_{\varepsilon}^{\partial \Omega}    (q,\cdot)] (X)
= \psi^{\Int}_{{\it 0}}  (h, h) 
+ \varepsilon  D_{x}   \psi^{\Int}_{{\it 0}}  (h, h)  \cdot  \big( R(\vartheta) X + \zeta_{\vartheta} \big) \\
+ \varepsilon^{2}  \Big( \psi^{\Int}_{{\it 2}}  (q, h) 
 + D_{x}   \psi^{\Int}_{{\it 1}}  (q, h)  \cdot R(\vartheta) X 
+ \frac{1}{2}  D^{2}_{x}   \psi^{\Int}_{{\it 0}}  (h, h)  \cdot  \big( R(\vartheta) X  , R(\vartheta) X \big) \Big) \\
+\varepsilon^{3} 
\Big( K^{\partial \mathcal S_{0}} (\eps,q) [ \mathtt{p}^{\partial \Omega}_{r}  (\varepsilon , q,\cdot)  ] (X) -g^{\partial \mathcal S_0} (\varepsilon , q,X) \Big) .
\end{multline}
We recall that the function $g^{\partial \mathcal S_0}  (\varepsilon , q,\cdot)$ is defined in \eqref{def-gi}. \par
Gathering \eqref{C_app}, \eqref{in3}, \eqref{nvelan8} and  \eqref{nvelan9bis} we obtain,  on $ \partial\mathcal S_{0} $,
\begin{multline*}
-G(\varepsilon) + SL [ p_{\varepsilon}^{\partial \mathcal S_0}  (q,\cdot) ] 
+ K^{\partial \mathcal S_{0}}   (\eps,q)  [ p_{\varepsilon}^{\partial \Omega}    (q,\cdot)] - C_{\varepsilon} (q) \\ 
= \varepsilon  \big(   \psi^{\Ext}_{{\it 0}}  (q, \cdot) - P_{{\it 0}} (q,  \cdot)   \big) 
\ + \ \varepsilon^{2} \big( \psi^{\Ext}_{{\it 1}}  (q, \cdot) - P_{{\it 1}} (q,  \cdot) \big) \\
 + \varepsilon^{3}   \Big(  
  SL [ p^{\partial \mathcal S_0}_{\varepsilon ,r}  (\varepsilon,q,\cdot)] 
  + K^{\partial \mathcal S_{0}}   (\eps,q)  [ \mathtt{p}^{\partial \Omega}_{r}  (\varepsilon , q,\cdot)  ] 
  - C_{r}(\varepsilon,q)
  - g^{\partial \mathcal S_0} (\varepsilon , q,\cdot) \Big),
\end{multline*}
where $P_{{\it 0}} (q,  \cdot) $ and  $P_{{\it 1}} (q,  \cdot) $  are  the harmonic polynomials defined respectively in \eqref{polyh0} and in \eqref{polyh1}.
Now taking into account the boundary conditions \eqref{ohoh} and \eqref{ahah}, and \eqref{Preum} we deduce  that \eqref{in1} holds true.

\ \par
Finally we move to the verification of \eqref{in2}.
First, using \eqref{DefPapp}, \eqref{in3} and \eqref{thatisback}, we obtain,  on $\partial \Omega$, 
\begin{multline} \label{Vend2}
SL [ p_{\varepsilon}^{\partial \Omega} (q, \cdot ) ] 
= \psi^{\Int}_{{\it 0}}  (h, \cdot) 
+ \eps   \psi^{\Int}_{{\it 1}}  (q, \cdot)  
+  \varepsilon^{2} \, \psi^{\Int}_{{\it 2}}  (q, \cdot)  
+ \varepsilon^{3} \,   SL [ \mathtt{p}^{\partial \Omega}_{r}   (\varepsilon,q,\cdot) ].
\end{multline}
By \eqref{def-zeta_vartheta}, \eqref{def-zeta}, the second equality in \eqref{ineed} and \eqref{mordre2}, 
\begin{gather} \label{rel1}
\int_{\partial\mathcal S_0}     \mathsf{p}_{{\it -1}}^{\Ext} (X)  R(\vartheta) X  \, {\rm d}s (X)  =  - \zeta_{\vartheta}  ,
\\  \label{rel2}
\int_{\partial\mathcal S_0}     \mathsf{p}_{{\it -1}}^{\Ext} (X) X^{\otimes 2} \, {\rm d}s (X)
= T^2 ( \mathsf{p}_{{\it -1}}^{\Ext} ) ,
\end{gather}
On the other hand, using  \eqref{K2}, \eqref{in3}, Taylor's formula, \eqref{def-ge}, \eqref{momentP}, \eqref{circu-reste}, \eqref{nvelan7},  \eqref{rel1},
\eqref{rel2}, we deduce , for $x $ in $\partial \Omega$, 
\begin{multline} \label{Vend1}
 K^{\partial  \Omega} (\eps,q)  [ p_{\varepsilon}^{\partial  \mathcal S_0}(q, \cdot )] (x)
= -G(x-h) + \varepsilon DG(x-h) \cdot \zeta_{\vartheta} \\
+ \varepsilon^2 \, \Big( -   R(\vartheta)^t \,  D G  (x-h)  \cdot  T^1 (\mathsf{p}^{\Ext}_{{\it 0}}  (q, \cdot)  ) \\
\quad \quad \quad \quad \quad \quad \quad \quad - \frac{1}{2}  \langle  R(\vartheta)^t \,  D^2_x G  (x-h) R(\vartheta), T^2 ( \mathsf{p}_{{\it -1}}^{\Ext} )  \rangle	_{\R^{2 \times 2}} \Big)  \\
+ \varepsilon^3 \,  \Big(  K^{\partial \Omega} (\eps,q) [ p^{\partial \mathcal S_0}_{r} (\varepsilon,q,\cdot)](x)
- g^{\partial \Omega} (\varepsilon , q,x)   \Big)  . \quad
\end{multline}
By \eqref{Vend2} and \eqref{Vend1}, the equation \eqref{in2} now reads, for $x $ in $\partial \Omega$, 
\begin{multline} \label{VendBis}
SL [ p_{\varepsilon}^{\partial \Omega} (q, \cdot ) ] (x)  + K^{\partial  \Omega} (\eps,q)  [ p_{\varepsilon}^{\partial  \mathcal S_0}(q, \cdot )] (x)
= \psi^{\Int}_{{\it 0}}  (h, x) -G(x-h) \\
+ \varepsilon   \Big(  \psi^{\Int}_{{\it 1}}  (q, x)   +DG(x-h) \cdot \zeta_{\vartheta} \Big)
+  \varepsilon^2 \, \Big(\psi^{\Int}_{{\it 2}}  (q, x)  - Q_{{\it 2}} (q,x)  \Big) \\
+  \varepsilon^3 \, \Big( SL [ \mathtt{p}^{\partial \Omega}_{r}   (\varepsilon,q,\cdot) ]
+ K^{\partial \Omega}   (\eps,q)  [ p^{\partial \mathcal S_0}_{\varepsilon,r}    (\varepsilon,q,\cdot)](x)
- g^{\partial \Omega} (\varepsilon , q,x)  \Big),
\end{multline}
where $Q_{{\it 2}} (q,x)$ denotes the harmonic polynomial defined in \eqref{defQ2}. \par
Taking now  the boundary conditions \eqref{depsi0}, \eqref{pasdel}, \eqref{ahah} and \eqref{Vend-fin} into account, we deduce from the equation \eqref{VendBis} that   \eqref{in2} holds true.
This concludes the proof of Lemma~\ref{jedecoupe}.
\qed 
\end{proof}
\subsubsection{Third Step. Existence and estimate of the remainders}
In this third step we prove, for $(\eps,q)$ in $\mathfrak Q_{\delta, \eps_{0}}$ with $\delta$ and $\eps_{0}$ positive and small enough, 
the existence of $ \mathfrak{p}_{r}  (\varepsilon ,q,\cdot)$ in $F_{0}$
satisfying \eqref{rgroup-reste} and provide an estimate in $F_{0}$, uniform over  $(\eps,q)  $ in $\mathfrak Q_{\delta, \eps_{0}}$. \par
We will make use of the fact that the the third argument of the right hand side of \eqref{rgroup-reste} vanishes.
Accordingly, we define
\begin{equation} \label{deftildeF1}
\tilde{F}_{1} :=  H^{1} (\partial \mathcal S_{0} )  \times  H^{1} (\partial \Omega )  \times  \{ 0 \} ,
\end{equation}
which is a closed subspace of $F_{1}$ and prove the following result.
\begin{lem} \label{etap3}
Let $\delta > 0$. There exists $\eps_{0}$ in $(0,1)$, such that for any $\mathfrak{g}$ in $L^{\infty} (\mathfrak Q_{\delta,\eps_{0}} ; \tilde{F}_{1} )$, there exists $\mathfrak{p}_{r} $ in $L^{\infty} (\mathfrak Q_{\delta,\eps_{0}} ; F_{0} )$
such that $\mathfrak{p}_{r} (\varepsilon ,q,\cdot)$ solves \eqref{rgroup-reste} for any $(\eps,q)$ in $ \mathfrak Q_{\delta, \eps_{0}}$.
\end{lem}
\begin{proof}[Proof of Lemma~\ref{etap3}.]
To prove Lemma \ref{etap3} let us start with stating a perturbative result. %\par
We will use the notation that given $X$ and $Y$ two Banach spaces, ${\mathcal L}(X;Y)$ is the space of bounded linear operators from $X$ to $Y$.
Now the framework is as follows. 
Let $\delta > 0$. Recall that $\mathfrak Q^{\delta}$ and $\Omega_{\delta}$ were defined in \eqref{MathfrakQdelta} and \eqref{Omega_delta}. We introduce the following families of operators. 
\begin{itemize}
\item First we consider a family of operators in ${\mathcal L}(L^{2}(\partial \Omega); H^{1} (\partial \mathcal S_{0}))$:
\begin{multline} \label{hyp-K}
\tilde{K}^{\partial  \mathcal S_0} \in \text{Lip} \left( \overline{\Omega_{\delta}}; {\mathcal L} \big( L^{2}(\partial \Omega); H^{1} (\partial \mathcal S_{0}) \big) \right)
\text{ such that for all } h \text{ in } \overline{\Omega_{\delta}}, \\
\ \tilde{K}^{\partial  \mathcal S_0} (h)   \text{  is compact from } L^{2} (\partial \Omega )   \text{ to } H^{1} (\partial \mathcal S_{0} ).
\end{multline}
\item Next we consider two families of operators: one in ${\mathcal L}(L^{2}(\partial \Omega); H^{1} (\partial \mathcal S_{0}))$ and the other one in ${\mathcal L}(L^{2}(\partial \mathcal S_{0}); H^{1}(\partial \Omega))$:
\begin{subequations} \label{hyp-T}
\begin{gather}
(T^{\partial  \mathcal S_0} (\eps,q) )_{(\eps,q) \in  \mathfrak Q^{\delta}}
\text{ bounded in } {\mathcal L}(L^{2}(\partial \Omega); H^{1} (\partial \mathcal S_{0})), \\
(T^{\partial  \Omega} (\eps,q) )_{ (\eps,q) \in  \mathfrak Q^{\delta}} 
\text{ bounded in } {\mathcal L}(L^{2}(\partial \mathcal S_{0}); H^{1}(\partial \Omega)).
\end{gather}
\end{subequations}
\end{itemize}
Given these operators we can construct the following one.
For $(\eps,q) $ in $ \mathfrak Q^{\delta}$, let $A(\eps , q): F_{0} \rightarrow F_{1} $ given by the following formula:
for any $\mathfrak p := (\mathtt{p}^{\partial \mathcal S_{0}} , \mathtt{p}^{\partial  \Omega} , C ) $ in $ F_{0}$,
\begin{equation} \label{cplex}
A(\eps , q)  [\mathfrak p  ]  := \big( A(\eps , q)  [\mathfrak p  ]_{i} \big)_{1 \leq i \leq 3} \in F_{1} ,
\end{equation}
with
\begin{subequations} \label{tttt}
\begin{align}
\label{cplex1}
A(\eps , q)  [\mathfrak p  ]_{1} 
&:= SL  [ \mathtt{p}^{\partial \mathcal S_{0}}] + \tilde{K}^{\partial \mathcal S_{0}} (h) [ \mathtt{p}^{\partial \Omega}] 
+ \eps T^{\partial  \mathcal S_0} (\eps,q)  [ \mathtt{p}^{\partial \Omega}]  - C , \\
\label{cplex2}
A(\eps , q)  [\mathfrak p  ]_{2} 
&:= SL  [ \mathtt{p}^{\partial \Omega}] + \eps T^{\partial \Omega} (\eps,q)  [ \mathtt{p}^{\partial \mathcal S_{0}}] , \\
\label{cplex3}
A(\eps , q)  [\mathfrak p  ]_{3}
&:= \int_{\partial  \mathcal S_{0}} \mathtt{p}^{\partial \mathcal S_{0}} \, {\rm d}s .
\end{align}
\end{subequations}
We will use the following  perturbative result.
\begin{lem} \label{abstract}
Let $\delta > 0$ and for $(\eps,q) $ in $ \mathfrak Q^{\delta}$, $A(\eps, q)$ given as above, with assumptions \eqref{hyp-K} and \eqref{hyp-T}.
Then there exists $\eps_{0} $ in $(0,1)$ such that for any $ (\eps,q)$ in $ \mathfrak Q_{\delta, \eps_{0}}$,
$A(\eps , q)  $ is an isomorphism from $F_{0}$ to $F_{1}$ and 
\begin{equation} \label{ti1}
\sup_{ (\eps,q) \in  \mathfrak Q_{\delta, \eps_{0}} } \| A(\eps , q)^{-1}  \|_{{\mathcal L}(F_{1};F_{0})} < \infty . 
\end{equation}
\end{lem}
\begin{proof}
It is straightforward to see that for any  $(\eps,q)  $ in $ \mathfrak Q$, $A (\eps ,q) $ is linear continuous. 
Let $(\eps,q)$ in $ \mathcal Q$, with $q=(\vartheta ,h) $ in $\R \times \Omega$.
Let us introduce, for any $ \mathfrak{p} := (\mathtt{p}^{\partial {\mathcal S}_{0}} , \mathtt{p}^{\partial  \Omega} , C)  $ in $F_{-\frac{1}{2}}$,
 \begin{align*}
L [ \mathfrak{p} ]
&:= (SL [ \mathtt{p}^{\partial \mathcal S_{0}}], SL [ \mathtt{p}^{\partial \Omega}], C) , \\
K(h) [ \mathfrak{p} ] 
&:= ( \tilde{K}^{\partial \mathcal S_{0}} (h) [ \mathtt{p}^{\partial \Omega}] - C, 0 , \int_{\partial  \mathcal S_{0}} \mathtt{p}^{\partial \mathcal S_{0}} \, {\rm d}s - C) , \\
T (\eps,q)   [ \mathfrak{p}]
&:= (T^{\partial  \mathcal S} (\eps,q) [ \mathtt{p}^{\partial \Omega}]  ,T^{\partial \Omega}  (\eps,q) [ \mathtt{p}^{\partial \mathcal S_{0}}] , 0) ,
\end{align*}
so that we can write $A$ in the following form: on $ F_{0}$,
\begin{equation} \label{decoco}
A (\eps ,q) = L +  K(h) + \eps T (\eps , q) .
\end{equation} \par
We first consider the operator $L + K(h)$. According to \eqref{fact1},  the operator $L$ is Fredholm with index zero and since for each $h $ in $\overline{\Omega_{\delta}}$,  $K(h)$ is compact, we deduce that $L + K(h)$ is Fredholm with index zero. It follows that to prove that $L + K(h)$ is an isomorphism, it is sufficient to prove that its kernel is trivial. \par
Consider $ \mathfrak{p} := (\mathtt{p}^{\partial \mathcal S_{0}} , \mathtt{p}^{\partial  \Omega} , C) $ in $F_{-\frac{1}{2}}$ such that $\big(L + K(h) \big) [ \mathfrak{p}]  =0$. Since the logarithmic capacity ${\rm Cap}({\partial \Omega})$ of $\partial  \Omega$ satisfies 
${\rm Cap}({\partial \Omega}) \neq 1$,  according to  \eqref{fact2}, the second equation $SL [ \mathtt{p}^{\partial \Omega}] =0$ implies $ \mathtt{p}^{\partial  \Omega} =0$.
Then, substituting into the first equation,   $SL [ \mathtt{p}^{\mathcal S_{0}}]   = C$, whereas the third equation reads $ \int_{\partial  \mathcal S_{0}} \mathtt{p}^{\partial \mathcal S_{0}} \, {\rm d}s = 0$. Thus according to \eqref{fact3}, we obtain $ \mathtt{p}^{\partial \mathcal S_{0}} = 0$ and thus $C=0$.
This proves that the kernel of  $L + K(h)$  is trivial, and consequently that for any $h $ in $\Omega_{\delta}$, $L + K(h)$ is an isomorphism. \par
Now using that the dependence of $K$ on $h$ is Lipschitz, we deduce that $L + K(h)$ has locally a bounded inverse.
By compactness of $\overline{\Omega_{\delta}}$, it follows that $L + K(h)$ has  a bounded inverse for $h$ running over $\overline{\Omega_{\delta}}$. \par
Since the operators $ (T_{\eps})_{\eps \in (0,1)} $ are bounded in the space of bounded operators from  $F_{0}$ to $F_{1}$ 
we can then easily deduce the result from \eqref{decoco}. 
This concludes the proof of Lemma~\ref{abstract}.
\qed 
\end{proof}

In our case, Lemma~\ref{abstract} is applied as follows. Recalling \eqref{K1}-\eqref{K2} we define, for any $(\eps,q) $ in $ \mathfrak Q$, with 
$q= (\vartheta,h)$, 
\begin{itemize}
\item for any density $\mathtt{p}^{\partial \Omega} $ in $ L^{2} (\partial \Omega )$,
\begin{equation} \label{Eq:DefKTilde}
\tilde{K}^{\partial  \mathcal S_0} (h)  [ \mathtt{p}^{\partial \Omega}]
= {K}^{\partial  \mathcal S_0} (0,0,h)  [ \mathtt{p}^{\partial \Omega}]
= SL  [ \mathtt{p}^{\partial \Omega}] (h)
\end{equation}
as a constant function on $\partial {\mathcal S}_{0}$, and
\begin{equation} \label{Eq:DefTS}
T^{\partial  \mathcal S_0}  (\eps,q)   [ \mathtt{p}^{\partial \Omega}]
:=  \int_{\partial\Omega}  \mathtt{p}^{\partial \Omega}    (y)  \eta_{{\it 1}}  (\varepsilon , q , \cdot ,  y)  \, {\rm d}s(y) 
\ \text{ on } \partial {\mathcal S}_{0}, 
\end{equation}
\item for any density $\mathtt{p}^{\partial \mathcal S_{0}}$ in $ L^{2} (\partial \mathcal S_{0} ) $,
\begin{equation} \label{Eq:DefTOmega}
T^{\partial  \Omega}  (\eps,q)  [ \mathtt{p}^{\partial  \mathcal S_0}]
:= \int_{\partial\mathcal S_0}  \mathtt{p}^{\partial  \mathcal S_0}    (Y)  \eta_{{\it 1}}  (\varepsilon , \vartheta , \cdot , -Y , h)  \, {\rm d}s(y)
\ \text{ on } \partial \Omega.
\end{equation}
\end{itemize}
The following lemma  entails that the hypotheses of Lemma~\ref{abstract} are satisfied.
\begin{lem} \label{abstract-check}
Let $\delta > 0$. With the definitions above, \eqref{hyp-K} and \eqref{hyp-T} hold true.
\end{lem}
\begin{proof}
We use Lemma~\ref{op-reg} with  $ \mathcal C = \partial \Omega$, $b=G$ and $ \mathtt{p}^{\mathcal C} = \mathtt{p}^{\partial \Omega}$ 
to obtain that $\tilde{K}^{\partial  \mathcal S_0}$ satisfies \eqref{hyp-K}.
Next we apply Lemma~\ref{op-reg}, \eqref{un}   for any $ (\eps,q)$ in $ \mathfrak Q_{\delta, \eps_{0}}$,
with $ \mathcal C = \partial \Omega$, $b(x,y)= \eta_{{\it 1}}  (\varepsilon , q , x ,  y) $ and $ \mathtt{p}^{\mathcal C} = \mathtt{p}^{\partial \Omega}$ 
and with $ \mathcal C = \partial \mathcal S_0$, $b(x,y)=  \eta_{{\it 1}}  (\varepsilon , \vartheta , x , -y , h) $ and $ \mathtt{p}^{\mathcal C} = \mathtt{p}_{\mathcal S_0}$ 
to get that $T^{\partial  \mathcal S_0}  (\eps,q)$ and $T^{\partial  \Omega}  (\eps,q)$ satisfy  \eqref{hyp-T}. 
This ends the proof of Lemma~\ref{abstract-check}
\qed 
\end{proof}
Then we consider the operator $A(\eps , q)$ associated with these operators $\tilde{K}^{\partial  \mathcal S_0} (h)$,  $T^{\partial  \mathcal S_0}  (\eps,q)$ and $T^{\partial  \Omega}  (\eps,q)$ as given by \eqref{cplex}-\eqref{tttt}. The next lemma shows that this operator $ A(\eps , q)$ provides the existence of a solution to \eqref{rgroup-reste} with uniform estimates.
\begin{lem} \label{abstract-check2}
Let $\delta > 0$. There exists $\eps_{0} $ in $(0,1)$ such that for any $ (\eps,q)$ in $ \mathfrak Q_{\delta, \eps_{0}}$,
\begin{equation} \label{solulu}
\mathfrak{p}_{r}  (\varepsilon ,q,\cdot)  := A(\eps , q)^{-1} \mathfrak{g}  (\varepsilon , q,\cdot) 
\end{equation}
belongs to $F_{0}$ and solves \eqref{rgroup-reste}. Moreover $ \mathfrak{p}_{r} $ is in $L^{\infty} (\mathfrak Q_{\delta,\eps_{0}} ; F_{0} )$.
\end{lem}
\begin{proof}
Let $\delta > 0$.
Let us first observe that for any  $ (\eps,q)$ in $ \mathfrak Q_\delta$, 
for any $\mathfrak p := (\mathtt{p}^{\partial \mathcal S_{0}} , \mathtt{p}^{\partial  \Omega} , C )
$ in $F_{0}$ satisfying the condition
\begin{equation} \label{condi}
\int_{\partial  \mathcal S_{0}} \mathtt{p}^{\partial \mathcal S_{0}} \, {\rm d}s = 0 , 
\end{equation}
the following equality holds:
\begin{equation*}
\mathfrak{A} (\eps,q) [\mathfrak p ]  = A(\eps , q)  [ \mathfrak p ]  .
\end{equation*}
Indeed \eqref{Eq:DefKTilde}
and first order Taylor expansions yield with \eqref{Eq:DefTS} and \eqref{Eq:DefTOmega} that 
\begin{align*}
K^{\partial  \mathcal S_0} (\eps,q)  [ \mathtt{p}^{\partial \Omega}]  -  \tilde{K}^{\partial  \mathcal S_0} (h)  [ \mathtt{p}^{\partial \Omega}] 
&=  \eps  T^{\partial  \mathcal S_0}  (\eps,q)   [ \mathtt{p}^{\partial \Omega}] , \\
K^{\partial  \Omega} (\eps,q)  [ \mathtt{p}^{\partial \mathcal S_{0}}]
&=  \eps T^{\partial  \Omega} (\eps,q) [ \mathtt{p}^{\partial \mathcal S_{0}}]  .
\end{align*}
We emphasize in particular that the last equality relies on the condition \eqref{condi}. \par
Now, consider $\eps_{0} $ in $(0,1)$ obtained by applying Lemma~\ref{abstract}. 
For any $ (\eps,q)$ in $ \mathfrak Q_{\delta, \eps_{0}}$, consider
$
\mathfrak{p}_{r}  (\varepsilon ,q,\cdot) = 
\big(\mathtt{p}^{\partial \mathcal S_0}_{r}    (\varepsilon ,q,\cdot)  , {p}^{\partial \Omega}_{r}  (\varepsilon ,q,\cdot)   , C_{r} (\varepsilon ,q) \big)
$
given by \eqref{solulu}. It belongs to $F_{0}$ and satisfies \eqref{circu-reste} and consequently
\begin{equation*}
\mathfrak{A} (\eps,q) [\mathfrak \mathtt{p}_{r}  (\varepsilon ,q,\cdot) ]
= A(\eps , q)  [ \mathfrak \mathtt{p}_{r}  (\varepsilon ,q,\cdot) ]  =  \mathfrak{g}  (\varepsilon , q,\cdot) .
\end{equation*}
Moreover we have the estimate
\begin{equation*}
\|  \mathfrak \mathtt{p}_{r}  (\varepsilon ,q,\cdot) \|_{F_{0}} \leq  \|  A(\eps , q)^{-1} \|_{{\mathcal L}(F_{1};F_{0})}
 \, \|  \mathfrak{g}  (\varepsilon , q,\cdot)  \|_{F_{1}} .
\end{equation*}
The estimates \eqref{gbd} and \eqref{ti1} entail that  $\mathfrak{p}_{r}$ is in $L^{\infty} (Q_{\delta,\eps_{0}} ;F_{0} )$, which concludes the proof of Lemma~\ref{abstract-check2}.
\qed 
\end{proof}
 Lemma~\ref{etap3} follows in a straightforward manner.
\qed
\end{proof} 

\subsubsection{Fourth Step. Conclusion}
\begin{proof}[End of proof of Lemma~\ref{drift}]
We apply Lemma~\ref{etap3} to \eqref{def-frakg}. Thanks to \eqref{gbd} the assumption is satisfied. 
Regarding $C_{\varepsilon} (q)$ this yields an expansion actually better than the one stated in Lemma~\ref{drift}, that is,
according to \eqref{C_app} and \eqref{lien-intro} and what precedes, 
there exists $C_r $ in $L^{\infty}(\mathfrak Q_{\delta,\eps_{0}}; \R)$ such that 
\begin{multline}
\label{expansion_mu-now}
C_{\varepsilon}  (q) = - G(\varepsilon) +  C^{\Ext} 
+ \psi^{\Int}_{{\it 0}}  (h, h) 
+ 2\eps D_x \psi^{\Int}_{{\it 0}}  (h, h)  \cdot {\zeta}_\vartheta
+ \eps^{2}\, \mathsf{C}_{{\it 2}} (q) \\
 + \eps^{3}\,  C_r (\varepsilon ,q) ,
\end{multline}
where $\mathsf{C}_{{\it 2}} (q) $ is given by  \eqref{C2}. To prove Lemma~\ref{drift} it is therefore sufficient to observe that $\mathsf{C}_{{\it 2}} (q)$ is bounded uniformly in $\R$ for $(\eps,q)$ in $ \mathfrak Q^{\delta} $ and to redefine $C_r (\varepsilon ,q) $ such that $\eps^{2}\, C_r (\varepsilon ,q) $ is equal to  the sum of the last two terms in \eqref{expansion_mu-now}.
\qed  
\end{proof}
\begin{proof}[End of proof of Proposition  \ref{dev-psi}]
Combining \eqref{exp-psi_density}, \eqref{p_app} and \eqref{in3}, we deduce  that on $\partial \mathcal S_0$
\begin{equation*}
\frac{\partial\psi_\eps}{\partial n} (q,\eps R(\vartheta) \cdot +h) 
= \mathsf{p}_{{\it -1}}^{\Ext}   (\cdot)
 + \varepsilon  \mathsf{p}^{\Ext}_{{\it 0}}   (q,\cdot)
 + \varepsilon^{2}\,  \mathsf{p}^{\Ext}_{{\it 1}}   (q,\cdot) 
 + \eps^{3}\,  \mathtt{p}^{\partial \mathcal S_0}_{r}  (\varepsilon ,q,\cdot) ,
\end{equation*}
with $\mathtt{p}^{\partial \mathcal S_0}_{r}  $ in $L^\infty \big(\mathfrak Q_{\delta,\eps_{0}}  ; L^2 (\partial \mathcal S_{0} ; \R ) \big)$.

Combining this with 
the second equality in 
\eqref{ineed}, \eqref{rel-sauts} and using that $\frac{\partial P_{{\it 0}}}{\partial n} (q, X) = -  R(\vartheta)^t \, u^\Omega (h) \cdot \tau $, for $X $ on 
$ \partial \mathcal S_0$,  as a direct consequence of the definition of $P_{{\it 0}}$ in \eqref{polyh0}, we 
conclude the proof of Proposition~\ref{dev-psi}.
\qed
\end{proof} 
\subsection{Asymptotic expansion of the potential part: Proof of Proposition \ref{dev-phi-pasc}}
\label{sec-proof-dev-phi}
The proof of Proposition~\ref{dev-phi-pasc} is very close to the one of Proposition~\ref{dev-psi}. We will only explain how to transform the (Neumann) problem defining the Kirchhoff potentials into a Dirichlet one, so that the proof of Proposition~\ref{dev-phi-pasc} follows from a tedious adaptation of the steps of the proof of Proposition~\ref{dev-psi} detailed in Subsection~\ref{proof-deux}. \par
We emphasize that the indices below correspond to coordinates in $\R^{3}$ and are in normal font type (while indices related to the order in an asymptotic development in powers of $\varepsilon$ are written in italic type). \par
We consider the functions $\overline{\varphi}_{j,\eps} (q,\cdot)$, for $j=1,2,3$, as the solution to the following Dirichlet boundary value problem in $\mathcal F_\varepsilon(q)$:
\begin{subequations} \label{model_conj-h}
\begin{alignat}{3}
\label{model_conj-h-a} 
- \Delta\overline{\varphi}_{j,\eps}  (q,\cdot) &=0&& \text{in }\mathcal F_\varepsilon  (q) ,\\
\label{model_conj-h-b}
\overline{\varphi}_{j,\eps}  (q,\cdot)&= \overline{K}_{j} (q,\cdot) + c_{j,\varepsilon} (q)
&\quad& \text{on }\partial \mathcal S_\varepsilon  (q),\\
\label{model_conj-h-c}
\overline{\varphi}_{j,\eps}  (q,\cdot) &=0&& \text{on }\partial\Omega,
\index{BGrecTF7@$\overline{\varphi}_{j,\eps} (q,\cdot)$: functions harmonically conjugated to the Kirchhoff potentials ${\varphi}_{j,\eps}(q,\cdot)$, up to a rotation}
\end{alignat}
where the 
functions $ \overline{K}_{j} (q,\cdot)$ are given by
\begin{gather*}
\overline{K}_{j} (q,\cdot) := 
\left\{ \begin{array}{ccc}
\frac{1}{2} |x-h|^2  &   \text{ if }   & j=1 , \\
-R(\vartheta)^t \, (x-h) \cdot e_2&   \text{ if }   &  j=2 ,\\
R(\vartheta)^t \, (x-h) \cdot e_1  &    \text{ if }  &    j=3 ,
\end{array} \right.
\end{gather*}
where $e_1$ and $e_2$ are the unit vectors of the canonical basis,
and the constants $c_{j,\varepsilon} (q)$ are such that:
\begin{equation} \label{circ-conj-h}
\int_{\partial\mathcal S_\varepsilon  (q)} \frac{\partial\overline{\varphi}_{j,\eps}}{\partial n} (q,\cdot) \, {\rm d}s= 0.
\end{equation}
\end{subequations}
Precisely, the constants $c_{j,\varepsilon} (q)$ are given by
\begin{equation} \label{fovoir}
c_{j,\varepsilon} (q) = C_{\varepsilon} (q) \int_{\partial\mathcal S_\varepsilon  (q)}  \frac{\partial\overline{\phi}_{j,\eps}}{\partial n} (q,\cdot) \, {\rm d}s ,
\end{equation}
where $\overline{\phi}_{j,\eps}$, $j=1,2,3$ are the solutions of 
\begin{alignat}{3} \label{fovoir2}
-\Delta\overline{\phi}_{j,\eps} (q,\cdot) &=0&& \text{in }\mathcal F_\varepsilon  (q) , \\
\overline{\phi}_{j,\eps} (q,\cdot) &= \overline{K}_{j} (q,\cdot)  &\quad& \text{on }\partial \mathcal S_\varepsilon  (q), \\
\overline{\phi}_{j,\eps} (q,\cdot) &=0&& \text{on }\partial\Omega .
\end{alignat}
We will use the vector notation:
\begin{equation} \label{Kir-vect}
\boldsymbol{\overline{\varphi}}_\eps :=
(\overline{\varphi}_{1,\eps} , \, \overline{\varphi}_{2,\eps},\, \overline{\varphi}_{3,\eps})^{t} .
\end{equation}
Up to a rotation, the functions $\overline{\varphi}_{j,\eps} (q,\cdot)$ are harmonically conjugated to the Kirchhoff potentials ${\varphi}_{j,\eps}(q,\cdot)$ (see \eqref{Kir-eps}), as shown in the following result.
\begin{lem} \label{IWantTo Sleep}
For any $(\eps,q) $ in $\mathfrak Q$, with $q=(\vartheta ,h)$, there  holds in $\mathcal F_\varepsilon  (q) $,
\begin{equation} \label{lien-har}
\nabla {\varphi}_{j,\eps} (q,\cdot) = \nabla^\perp \check{\varphi}_{j,\eps} (q,\cdot),
\end{equation}
where
\begin{equation} \label{def-check}
\Big(\check{\varphi}_{1,\eps} (q,\cdot), \, \check{\varphi}_{2,\varepsilon} (q,\cdot),\, \check{\varphi}_{3,\varepsilon} (q,\cdot)\Big)
:=  \mathcal R(\vartheta)   \Big(\overline{\varphi}_{1,\eps} (q,\cdot) , \, \overline{\varphi}_{2,\eps} (q,\cdot),\, \overline{\varphi}_{3,\eps} (q,\cdot)\Big) .
\end{equation}
\end{lem}
\begin{proof} 
We recall that for any $(\eps,q) $ in $\mathfrak Q$, the system 
\begin{subequations}
\label{Unik}
\begin{alignat}{3}
\dive u &=0& &\text{in } \mathcal F_\varepsilon  (q) ,\\
\curl u &=0& &\text{in } \mathcal F_\varepsilon  (q) ,\\
u \cdot n  &=0 & \quad&\text{on }\partial \Omega  ,\\
\label{Unik4} 
u \cdot n  &=K_{j} (q,\cdot) &  \quad&\text{on }\partial \mathcal S_\eps (q)  ,\\
\label{Unik5} 
\int_{\partial \mathcal S_\eps (q)} u \cdot \tau \, {\rm d}s &= 0, & &
\end{alignat}
\end{subequations}
has a unique solution $u$, say in $H^1 ({\mathcal F}_\varepsilon  (q) )$.
Then one observes that both $\nabla {\varphi}_{j,\eps} (q,\cdot) $ and $ \nabla^\perp \check{\varphi}_{j,\eps} (q,\cdot)$ solve \eqref{Unik}.
In particular let us emphasize that, on $\partial \mathcal S_\eps (q) $, 
$$
\Big( n \cdot \nabla^\perp \overline{\varphi}_{j,\eps} (q,\cdot) \Big)_{j=1,2,3}
= 
\Big(   \frac{\partial \overline{K}_{j}}{\partial \tau} (q,\cdot) \Big)_{j=1,2,3} 
=  \mathcal R(\vartheta)^{t} \boldsymbol{K} (q,\cdot) ,
$$
so that, for $j=1,2,3$, $\nabla^\perp \check{\varphi}_{j,\eps} (q,\cdot)$ satisfies \eqref{Unik4}, and the condition \eqref{circ-conj-h} ensures that \eqref{Unik5} is satisfied.
\qed
\end{proof} 
In the case without exterior boundary we consider in the same way $\overline{\varphi}^{\Ext}_{j}$ as the solution to 
\begin{subequations} \label{phi-0}
\begin{alignat}{3}
-\Delta\overline{\varphi}^{\Ext}_{j}  &=0& & \text{ in }\R^{2} \setminus \mathcal S_{0} ,\\
\label{ahah-phi}
\overline{\varphi}^{\Ext}_{j}  (\cdot) &= \overline{K}_{j} (0,  \cdot)  + c^{}_{j}&\quad&\text{on }\partial \mathcal S_{0}  , \\
\overline{\varphi}^{\Ext}_{j} (x) & \rightarrow 0 & & \text{ as } |x| \rightarrow +\infty,
\end{alignat}
where the constant $c_{\Ext, j}$ is such that 
\begin{equation} \label{circ-saoule}
\int_{\partial\mathcal S_0} \frac{\partial\overline{\varphi}^{\Ext}_{j}}{\partial n} \, {\rm d}s= 0.
\index{BGrecTF8@$\overline{\varphi}^{\Ext}_{j} $: functions harmonically conjugated to the Kirchhoff potentials ${\varphi}^{\Ext}_{j} $ when $\Omega = \R^2$}
\end{equation}
\end{subequations}
The existence and uniqueness of such a constant $ c_{\Ext ,j}$ is provided by a similar argument as for \eqref{fovoir}-\eqref{fovoir2}.
Proceeding as in the proof of Lemma~\ref{IWantTo Sleep} 
\begin{equation} \label{eqeq}
\nabla {\varphi}^{\Ext}_{j} = \nabla^\perp \overline{\varphi}^{\Ext}_{j} ,
\end{equation}
where the functions $\varphi^{\Ext}_{j}$, for $j=1,2,3$, are the Kirchhoff potentials in $\R^{2} \setminus \mathcal S_0 $ defined in \eqref{Kir-exte}.
As before we introduce the vector notation for the functions $\overline{\varphi}^{\Ext}_{j}$:
\begin{equation} \label{Kir-vect-ext}
\boldsymbol{\overline{\varphi}}^{\Ext} 
:=(\overline{\varphi}^{\Ext}_{1}, \overline{\varphi}^{\Ext}_{2},\, \overline{\varphi}^{\Ext}_{3}) .
\end{equation}
Then, following the strategy of Proposition \ref{dev-psi} we obtain  the following result.
\begin{prop}  \label{dev-phi}
There exist
\begin{enumerate}[(i)]
\item  $\boldsymbol{p}^{\partial \mathcal S_0}_{r} : \mathfrak Q \rightarrow L^2 (\partial \mathcal S_{0} ; \R^3 )$ such that for any $(\varepsilon , q) $ in $\mathfrak Q$, with $q=(\vartheta,h) $, for any $X $ in $\partial \mathcal S_0$, 
\begin{equation} \label{exp-phi_n}
\frac{\partial \boldsymbol{\overline{\varphi}}_\eps}{\partial n} (q,\eps R(\vartheta) X +h) 
= I_\eps  \Big( \frac{\partial \boldsymbol{\overline{\varphi}}^{\Ext}}{\partial n} (X)  
    + \eps^{2} \boldsymbol{p}^{\partial \mathcal S_0}_{r} (\eps,q, X) \Big) ,
\end{equation} 
\item $\boldsymbol{p}^{\partial \Omega}_{r} : \mathfrak Q \rightarrow L^2 (\partial \Omega ; \R^3 )$ such that for any 
$(\varepsilon, q) $ in $\mathfrak Q$, for any $x $ in $\partial \Omega$, 
\begin{equation} \label{exp-phi_n-ext}
\frac{\partial \boldsymbol{\overline{\varphi}}_\eps}{\partial n} (q, x) = I_\eps \,  \eps^{2} \,   \boldsymbol{p}^{\partial \Omega}_{r} (\eps,q, x) .
\end{equation}
and moreover such that for any $\delta >0$, there exists $\varepsilon_{0}$ in $(0,1)$ for which
$\boldsymbol{p}^{\partial \mathcal S_0}_{r} \in L^\infty \big(\mathfrak Q_{\delta,\eps_{0}} ; L^2 (\partial \mathcal S_{0} ; \R^3 ) \big)$
and $\boldsymbol{p}^{\partial \Omega}_{r} \in L^\infty \big(\mathfrak Q_{\delta,\eps_{0}} ; L^2 (\partial \Omega ; \R^3 ) \big)$.

\end{enumerate}
\end{prop}
Proposition~\ref{dev-phi-pasc} then follows from Lemma~\ref{IWantTo Sleep}, \eqref{eqeq} and  Proposition~\ref{dev-phi}. We omit the details.
\section{Proof of the normal forms}
\label{Sec:FormesNormales}
This section is devoted to the proof of the normal forms in Proposition~\ref{Pro-fnormA}, Proposition~\ref{Pro-fnorm} and Proposition~\ref{Pro-fnorm-ball-inh}, as well as 
the expansion of the added inertia in Proposition~\ref{dev-added} that was used to establish Lemma~\ref{kin-eps}.
The proof of the normal forms \eqref{fnorm1}, \eqref{fnorm2} and \eqref{fnorm2-ball-inh} consists first in expanding the functions $M_\varepsilon (q)$, $ \langle \Gamma_\varepsilon (q)  ,p,p\rangle $ and ${F}_\varepsilon (q,p)$ with respect to $\eps$ thanks to the expansions of the previous sections and to Lamb's lemma (Lemma~\ref{blasius}), 
and then in substituting these expansions into \eqref{ODE_intro_eps}. Next, further modifications are needed in order to reach the exact forms \eqref{fnorm1}, \eqref{fnorm2} and \eqref{fnorm2-ball-inh}.
\subsection{Asymptotic expansion of the added inertia and the Christoffel symbols}
\label{sec-dev-mat}
In this subsection, we use the asymptotic developments of Section~\ref{sec-dev-stream} to deduce expansions for the added inertia matrix and for the Christoffel symbols. \par
We begin by giving the expansions in terms of $\varepsilon$ of the inertia matrix  $M_{a,\eps} (q)$ which is the counterpart for the body of size $\eps$ of the added mass $M_a (q) $ defined in \eqref{def-MAq}.
Precisely, it is defined for $(\eps,q)$ in $\mathfrak Q$ by 
\begin{equation} \nonumber
M_{a,\eps} (q) := \int_{\partial \mathcal S_\eps (q)}\boldsymbol\varphi_\eps (q,\cdot)\otimes\frac{\partial\boldsymbol\varphi_\eps}{\partial n}(q,\cdot) \, {\rm d}s 
=\int_{\partial \mathcal S_\eps (q)} \boldsymbol\varphi_\eps (q,\cdot)\otimes \boldsymbol K_\eps (q,\cdot) \, {\rm d}s  .
\end{equation}
The function $ \boldsymbol{{\varphi}}_\eps $ mentioned above is  defined in \eqref{Kir-eps}, \eqref{Kir-vect-pasc}.
Let us also recall that the matrix $M^{\Ext}_{a, \vartheta} $ is defined in  \eqref{def-Maa} and \eqref{g+a_ext}, and $I_{\varepsilon}$ is defined in \eqref{Eq:Ieps}.
The expansion is as follows.
\begin{prop} \label{dev-added}
There exists a function $M_{r}: \mathfrak Q \rightarrow \R^{3 \times 3}$ depending on $\mathcal S_0$ and $\Omega$ such that for any $\delta>0$ there exists $\eps_{0} $ in $(0,1)$ such that
$M_{r} \in L^\infty (\mathfrak Q_{\delta,\eps_{0}} ; \R^{3 \times 3})$ and
 such that for all $(\varepsilon ,q ) $ in $\mathfrak Q$,  with $q=(\vartheta , h)$, 
\begin{equation} \label{dev-mat-ii-proof}
M_{a,\varepsilon} (q) = 
\varepsilon^2 I_\eps  
\Big( M^{\Ext}_{a, \vartheta} 
+ \eps^2 {M}_{r} (\eps ,q )
\Big) I_\eps  .
\end{equation}
\end{prop}
\begin{proof}
Using a change of variable, \eqref{Kir-gras} and \eqref{Kir-vect-pasc} we deduce  that for $(\eps,q)$ in $\mathfrak Q$,
\begin{equation}
\label{change-K}
 \boldsymbol{K}_\eps (q, \eps R(\vartheta) \cdot +h) =I_\eps  \mathcal  R (\vartheta)  \boldsymbol{K} (0, \cdot) \text{ on } \partial\mathcal S_0.
\end{equation}
It follows that
\begin{equation*}
M_{a,\eps} (q) 
=  \eps \int_{\partial\mathcal S_{0}}
 \boldsymbol\varphi_\eps (q, \eps R(\vartheta) \cdot +h )\otimes I_{\eps}  \mathcal R (\vartheta) \boldsymbol K (0,\cdot)
\, {\rm d}s  .
\end{equation*}
We now apply Proposition \ref{dev-phi-pasc}, (i) to get 
\begin{align*}
M_{a,\eps} (q) 
&=  \eps^{2} I_\eps   \Big(  \int_{\partial\mathcal S_{0}}
   \mathcal R (\vartheta)   \Big(\boldsymbol{\varphi}^{\Ext}   +\boldsymbol{\check{c}} (\eps,q)  + \eps^{2} \boldsymbol{\varphi}_{r}   (\eps,q, \cdot) \Big)
\otimes   \mathcal R (\vartheta) \boldsymbol K (0,\cdot) \, {\rm d}s \,   \Big) I_{\eps} 
\\ &=
 \eps^{2} I_\eps   \Big( M^\Ext_{a,\vartheta}  
 +  \eps^{2} \mathcal R (\vartheta) 
 \int_{\partial\mathcal S_{0}}  \boldsymbol{\varphi}_{r}   (\eps,q, \cdot)  
\otimes  \boldsymbol K (0,\cdot) \, {\rm d}s \, \mathcal R (\vartheta)^{t}  \Big)  I_{\eps} ,
\end{align*}
since
$$
\int_{\partial\mathcal S_{0}} \boldsymbol{\check{c}} (\eps,q)  \otimes  \boldsymbol K (0,\cdot) \, {\rm d}s
= \boldsymbol{\check{c}} (\eps,q)  \otimes   \int_{\partial\mathcal S_{0}} \boldsymbol K (0,\cdot) \, {\rm d}s
=0,
$$
and 
$$M^\Ext_{a,\vartheta}    
=   \mathcal R (\vartheta) \int_{\partial\mathcal S_{0}} \boldsymbol{\varphi}^{\Ext} \otimes   \boldsymbol K (0,\cdot) \, {\rm d}s \, 
 \mathcal R (\vartheta)^{t}
,$$
thanks to  \eqref{Kir-ext-boundary}, 
\eqref{def-Maa} and 
 \eqref{g+a_ext}.
Above $\eps_{0} $ belongs to $(0,1)$
and $\boldsymbol{\varphi}_{r}$ is in the space $L^\infty \big(\mathfrak Q_{\delta,\eps_{0}} ; L^2 (\partial \mathcal S_{0} ; \R^3 ) \big)$.
 Then we set 
 $${M}_{r} (\eps ,q ) := 
 \mathcal R (\vartheta) 
 \int_{\partial\mathcal S_{0}}  \boldsymbol{\varphi}_{r}   (\eps,q, \cdot)  
\otimes  \boldsymbol K (0,\cdot) \, {\rm d}s \, \mathcal R (\vartheta)^{t} , $$
and we observe that $M_{r}  $ is in  $L^\infty (\mathfrak Q_{\delta,\eps_{0}} ; \R)$ and depends only on $\mathcal S_0$ and $\Omega$.
This concludes the proof of Proposition \ref{dev-added} and consequently of Lemma~\ref{kin-eps}.
\qed
\end{proof} 
We now consider the Christoffel symbols $\Gamma_\eps^{\rm rot}$ given for  $(\eps,q)$ in $\mathfrak Q  $ and  $p = ({\omega},\ell)  \in\mathbb R^3$, by 
\begin{equation} \label{zozo-eps}
\langle \Gamma_\eps^{\rm rot} (q), p, p\rangle := 
- \begin{pmatrix}0\\
P^\eps_a \end{pmatrix}
\times p
-
\omega M_{a,\eps} (q)\begin{pmatrix}
0 \\
\ell^\perp
\end{pmatrix}       \in\mathbb R^3,
\end{equation}
where $P^\eps_a$ denotes the last two coordinates of $M_{a,\eps}(q)p$.
The formula  \eqref{zozo-eps} is the counterpart for a body of size $\eps$ of the  
Christoffel symbols given by \eqref{zozo} when $\eps=1$.

The next result proves that the leading term of $\Gamma_\eps^{\rm rot}$ is given, up to an appropriate scaling, by the Christoffel symbols  $\langle\Gamma^{\Ext}_{\vartheta} ,p,p\rangle$ of the solid as if it was immersed in a fluid filling the plane. %all the rest of
We recall that $\langle\Gamma^{\Ext}_{\vartheta},p,p\rangle$ is defined in \eqref{Christo-exter}. %\par
Precisely, we have the following result.
\begin{prop} \label{Gamma-int-exp}
%Let $\delta > 0$ be fixed. There exists $\eps_{0}$ in $(0,1)$ and 
There exists $\Gamma^{\rm rot}_{r} : \mathfrak Q \rightarrow  \mathcal{BL} (\R^3 \times \R^3 ; \R^3 )$
depending on  $\mathcal S_0$, $\gamma$  and $\Omega$, such that for any $\delta >0$, there exists $\varepsilon_{0}$ in $(0,1)$ for which
$\Gamma^{\rm rot}_{r} \in L^\infty (\mathfrak Q_{\delta,\eps_{0}} ;  \mathcal{BL} (\R^3 \times \R^3 ; \R^3 ))$ and
such that for any $(\varepsilon ,q) $ in $ \mathfrak Q$, with $q=(\vartheta,h)$,  for any   ${p} = ({\omega},\ell) $ in $\R^3$,
\begin{equation} \label{exp-GammaS}
\langle \Gamma_\eps^{\rm rot} (q),p,p\rangle 
=  \eps I_\eps \big(\langle \Gamma^{\Ext}_{\vartheta} , I_{\varepsilon} p , I_{\varepsilon} p \rangle
 + \eps^{2} \langle \Gamma^{\rm rot}_{r} (\eps,q), I_{\varepsilon} p , I_{\varepsilon} p \rangle  \big) .
\end{equation}
\end{prop}
Proposition \ref{Gamma-int-exp} follows from Proposition \ref{dev-added} by straightforward computations. \par
\ \par
Finally we study the Christoffel symbols $\Gamma_{\eps}^{\Omega}$ given for  $(\eps,q)$ in $\mathfrak Q  $ and  $p=(p_1,p_2,p_3)\in\mathbb R^3$, by 
\begin{equation} \label{DeCadix}
\langle\Gamma_\varepsilon^{\partial \Omega} (q), p, p \rangle
:= \left( \sum_{1\leq k,l\leq 3} (\Gamma_\varepsilon^{\partial \Omega})^j_{k,l}(q) p_k p_l \right)_{1\leq j\leq 3} \in \mathbb R^3 ,
\end{equation}
where for every $j,k,l\in\{1,2,3\}$, we set
\begin{equation*}
(\Gamma_\varepsilon^{\partial \Omega})^j_{k,l}(q) :=
\frac{1}{2} [  \Lambda^{l}_{\eps ,kj}(q)  +  \Lambda^{k}_{\eps ,jl}(q) -\Lambda^{j}_{\eps ,kl}(q)  ] ,
\end{equation*}
with 
\begin{equation*}
\Lambda^{l}_{\eps , kj}(q) :=
\int_{\partial\Omega}
\left(
\frac{\partial\varphi_{j,\varepsilon}}{\partial\tau}  \frac{\partial\varphi_{k,\varepsilon}}{\partial\tau} K_l
\right) 
 (q,\cdot)\, {\rm d}s  .
\end{equation*}
The Christoffel symbols $\Gamma_{\eps}^{\Omega}$ are the counterpart $\Gamma^{\Omega}$ defined in \eqref{Def_Gamma}. They are expanded as follows.
\begin{prop} \label{Gamma-ext-exp}
%
%Let  $\delta > 0$  be fixed. There exists $\eps_{0} $ in $(0,1)$ and 
There exists $\Gamma^{\partial \Omega}_{r}  : \mathfrak Q \rightarrow  \mathcal{BL} (\R^3 \times \R^3 ; \R^3 )$ depending on  $\mathcal S_0$, $\gamma$  and $\Omega$, 
such that for any $\delta >0$, there exists $\varepsilon_{0}$ in $(0,1)$ for which $\Gamma^{\partial \Omega}_{r} \in L^\infty (\mathfrak Q_{\delta,\eps_{0}} ;  \mathcal{BL} (\R^3 \times \R^3 ; \R^3 ))$ and such that for any $(\varepsilon ,q) $ in $ \mathfrak Q$, with $q:=(\vartheta,h)$,  for any   ${p} := ({\omega},\ell) $ in $\R^3$,
\begin{equation} \label{expGammaOmega}
\langle \Gamma_{\eps}^{{\partial \Omega}} (q),p,p\rangle 
=  \eps^{3} I_\eps  \langle \Gamma^{\partial \Omega}_{r} (\eps,q), I_{\varepsilon} p , I_{\varepsilon} p \rangle .
\end{equation}
\end{prop}
\begin{proof}
Proposition \ref{Gamma-ext-exp} follows from Proposition \ref{dev-phi-pasc}, (iii). Indeed, \eqref{DeCadix} can be rewritten as:
\begin{equation*}
\langle\Gamma_\varepsilon^{\partial \Omega} (q), p, p \rangle 
=\int_{\partial\Omega}
\left[ \frac{\partial\boldsymbol\varphi_\varepsilon}{\partial\tau}(\mathbf K_\varepsilon\cdot p) 
\left(\frac{\partial\boldsymbol\varphi_\varepsilon}{\partial\tau}\cdot p\right) 
-\frac{1}{2} \mathbf K_\varepsilon \left(\frac{\partial\boldsymbol\varphi_\varepsilon}{\partial\tau}\cdot p\right)^2
\right] (q, \cdot) \, {\rm d}s.
\end{equation*}
Observe that $\mathbf K_\varepsilon$ is actually independent of $\varepsilon$ on $\partial\Omega$.
According to \eqref{exp-phi_n-pasc-ext}, we obtain:
\begin{multline*}
\langle \Gamma_\varepsilon^{\partial \Omega} (q), p, p \rangle 
\\ = \varepsilon^3 I_\varepsilon \int_{\partial\Omega}
\left[ \boldsymbol{\mathtt{p}}^{\partial \Omega}_{r} (\hat{\mathbf K}_\varepsilon \cdot I_{\varepsilon} p) \left(\boldsymbol{\mathtt{p}}^{\partial \Omega}_{r} \cdot I_{\varepsilon} p\right)
- \frac{1}{2} \hat{\mathbf K}_\varepsilon \left( \boldsymbol{\mathtt{p}}^{\partial \Omega}_{r} \cdot I_{\varepsilon} p\right)^2 \right] (q,\cdot) \, {\rm d}s,
\end{multline*}
 where  $\hat{\mathbf K}_\varepsilon := \varepsilon I_\varepsilon^{-1} \mathbf K_\varepsilon$. This  gives the expected result.
\qed
\end{proof} 
%
%
%
%
%
%
%%%%%%%%%%%%%%%%%%%%%%%%%%%%%%%%%%%
%
\subsection{Asymptotic expansion of $E_\varepsilon $}
We now consider the expansion of $E_\varepsilon$ which is given, for $(\eps,q) $ in $\mathfrak Q$, by
\begin{equation*}
E_{\varepsilon} (q) := 
- \frac{1}{2}\int_{\partial\mathcal S_\eps (q)}\left|\frac{\partial\psi_\eps }{\partial n} (q,\cdot) \right|^2 \boldsymbol{K}_\eps (q,\cdot) \, {\rm d}s.
\end{equation*}
This formula is the counterpart of  \eqref{E-def} for a body of size $\eps$.
We recall that the function $\psi_\varepsilon (q,\cdot)$ is defined in \eqref{model_stream} and the vector field $\boldsymbol{K}_\eps (q,\cdot)$ in  \eqref{Def-Kj-eps}-\eqref{Kir-vect-pasc}.

The first two terms in the asymptotic expansion will be given respectively thanks to two vector fields
$\mathsf{E}_{{\it 0}} (q)$ and $\mathsf{E}_{{\it 1}} (q)$  which we now define. 
First we set
\begin{equation*} 
\mathsf{E}_{{\it 0}} (q) :=
-
\begin{pmatrix}
 u^\Omega (h) \cdot  \zeta_\vartheta \\
 u^\Omega (h)^{\perp} 
\end{pmatrix} ,    \text{ where  }  q=(\vartheta , h).
\end{equation*}
We recall that $u^\Omega$ and ${\zeta}_\vartheta$  were defined in \eqref{DefUOmega} and \eqref{def-zeta_vartheta}-\eqref{def-zeta} respectively. %\par
Next we define $\mathsf{E}_{{\it 1}} (q) $ as 
\begin{equation} \label{decompE1}
\mathsf{E}_{{\it 1}} (q) := \mathsf{E}_{{\it 1}}^a  (q) + \mathsf{E}_{{\it 1}}^b (q) + \mathsf{E}_{{\it 1}}^c (q) ,
\end{equation}
where the three addends are given by the definitions below. \par
\ \par
\noindent
{\it The inertial subprincipal term ${\it \mathsf{E}_{{\it 1}}^a}$}.
The definition of the term $\mathsf{E}_{{\it 1}}^a$ will use some functions of  the entries 
of the matrix $M_{a}^{\Ext} $ defined in \eqref{def-Maa}. 
Let us first recall that we decomposed $M_{a}^{\Ext}$ in \eqref{def-Ma}.
We also define the real traceless symmetric $2 \times 2$ matrix $ {M}^{\dagger}$ defined by 
\begin{equation} \label{def-Mdagger}
{M}^{\dagger} 
 = \big( M^{\dagger}_{i,j} \big)_{1\leqslant i,j \leqslant 2} 
 := \frac12 \Big(  M^\Ext_{\flat} (\perp) + \big(M^\Ext_\flat (\perp) \big)^t  \Big) 
 = \frac12 \Big(  M^\Ext_\flat (\perp) -   (\perp) M^\Ext_\flat  \Big)  ,
\end{equation}
where $(\perp)$  is defined as the $2 \times 2$ matrix 
\begin{equation} \label{mat-perp}
(\perp)  := 
\begin{pmatrix}
0 & -1 \\
1 & 0
\end{pmatrix}.
\end{equation}
The matrix $ {M}^{\dagger} \, $ depends only on  $\mathcal S_0$.
Its coefficients can be described as follows: 
\begin{subequations} \label{def-Mdagger2}
\begin{gather}
M^{\dagger}_{1,1} = - M^{\dagger}_{2,2} = \int_{\partial \mathcal S_0} \frac{\partial \varphi^{\Ext}_{3}}{\partial n} \varphi^{\Ext}_{2} \, {\rm d}s 
, \quad \text{ and  } \\
M^{\dagger}_{1,2} =  M^{\dagger}_{2,1} = \frac12 \int_{\partial \mathcal S_0} \left( \frac{\partial  \varphi^{\Ext}_{3} }{\partial n} \varphi^{\Ext}_{3}
- \frac{\partial \varphi^{\Ext}_{2} }{\partial n}  \varphi^{\Ext}_{2} \right) \, {\rm d}s ,
\end{gather}
\end{subequations}
where the functions $\varphi^{\Ext}_{j}$, for $j=1,2,3$, are the Kirchhoff potentials in $\R^{2} \setminus \mathcal S_0 $ defined in \eqref{Kir-exte}. %\par
Recalling \eqref{def-Ma}, we also consider
\begin{equation} \label{Mflattheta}
M_{\flat,\vartheta}^\Ext := R(\vartheta) M^\Ext_\flat  R(\vartheta)^t \, , \quad
\mu^{\Ext}_{\vartheta} := R(\vartheta) \mu^\Ext , \quad
{M}_{\vartheta}^{\dagger} := R(\vartheta)  {M}^{\dagger}  R(\vartheta)^t .
\end{equation}
Then we define:
\begin{equation} \label{def-E1a}
\mathsf{E}_{{\it 1}}^a  (q) :=
\begin{pmatrix}
u^\Omega (h)^{\perp}  M_\vartheta^{\dagger} \, u^\Omega (h)^{\perp}  \\ 0 \\ 0
\end{pmatrix}  ,    \text{ where  }  q=(\vartheta , h).
\end{equation}
\noindent
{\it The weakly gyroscopic subprincipal term ${\it \mathsf{E}_{{\it 1}}^b}$}.
Let us introduce the geometrical constant $2\times 2$ matrix
\begin{align} \nonumber
\sigma &:=   \int_{\partial \mathcal S_0} \,  \frac{\partial\psi^{\Ext}_{{\it -1}}}{\partial n} (X)  \, X \otimes  X^{\perp} \, {\rm d}s(X)
+ \zeta \otimes \zeta^{\perp}  \\
\label{def-sigma}
&= \int_{\partial \mathcal S_0} \,   \frac{\partial\psi^{\Ext}_{{\it -1}}}{\partial n} (X)  \,  ( X \otimes  X^{\perp} 
- \zeta \otimes \zeta^{\perp} ) \, {\rm d}s(X) ,
\end{align}
which only depends on $\mathcal S_0$. Next we introduce its symmetric part
\begin{equation} \label{anti-sigma}
\sigma^{s} := \frac{1}{2} (\sigma + \sigma^{t}),
\end{equation}
and the associated field force $\mathsf{E}_{{\it 1}}^b (q) $ defined,  for $q = (\vartheta, h)$ in $\R \times \Omega$, by
\begin{equation} \label{DefH1b}
\mathsf{E}_{{\it 1}}^b (q) 
:=
\begin{pmatrix}
 - \langle  D^2_x \psi^{\Int}_{{\it 0}} (h,h)  ,  R(-2\vartheta)  \sigma^s  \rangle_{\R^{2 \times 2}} \\
 0 \\ 0
\end{pmatrix} .
\end{equation}
Let us recall that the function $\psi^{\Int}_{{\it 0}}$ is defined in \eqref{depsi0}.
The main property of $\mathsf{E}_{{\it 1}}^b$ is the following.
\begin{lem} \label{cvwg}
The vector field $\mathsf{E}_{{\it 1}}^b$ in $C^{\infty}(\R\times  \Omega ; \R^3)$ defined by \eqref{DefH1b}
is weakly gyroscopic in the sense of Definition \ref{wg}.
\end{lem}
\begin{proof}
Multiply \eqref{DefH1b} by $\tilde{p}$ and integrate.
The conclusion follows from an integration by parts, crude bounds, Lemma~\ref{banane},
the smoothness of the function $\psi^{\Int}_{{\it 0}}$ and Lemma~\ref{drifterisok}.
\qed
\end{proof}
%
%
%\ \par
\noindent
{\it The drift subprincipal term ${\it \mathsf{E}_{{\it 1}}^c}$}.
Let us introduce the force field $\mathsf{E}_{{\it 1}}^c (q) $  defined, for $q = (\vartheta  , h)$ in $\R\times  \Omega$, by
\begin{equation} \label{DefH1c}
\mathsf{E}_{{\it 1}}^c (q) 
:=
-  \begin{pmatrix}
{\zeta}_\vartheta \cdot   u_c (q) 
\\   (u_c (q))^\perp 
\end{pmatrix} .	
\end{equation}
Above $u_c (q)$ denotes the corrector velocity defined in \eqref{def-uc-perp}. \par
\ \par
Now the goal of this subsection is to establish the following result.
\begin{prop} \label{E-exp}
There exists a function   $E_{r} : \mathfrak Q \rightarrow \R^3$  depending on  $\mathcal S_0$ and $\Omega$,
such that for any $\delta >0$, there exists $\varepsilon_{0}$ in $(0,1)$ for which $E_{r}$ belongs to $L^\infty (\mathfrak Q_{\delta,\eps_{0}}; \R^3)$
and such that for any $(\varepsilon ,q) $ in $ \mathfrak Q_{\delta,\eps_{0}}$, 
\begin{equation} \label{E-dev}
E_\varepsilon(q)= 
I_\eps  \Big(   \mathsf{E}_{{\it 0}} (q) + \varepsilon \mathsf{E}_{{\it 1}} (q) + \eps^2  E_{r}  (\eps,q)  \Big) .
\end{equation}
\end{prop}
\begin{proof}
We proceed in three steps: first we use a change of variable in order to recast $E_{\varepsilon} (q)$ 
as an integral on the fixed boundary $\partial\mathcal S_0$. 
Then we plug the expansion of $\psi^{\varepsilon}$ into this integral. 
Finally we use several times Lamb's lemma in order to compute the terms of the resulting expansion. \par
\ \par
First, thanks to a change of variable and \eqref{change-K}, 
\begin{equation*}
E_{\varepsilon} (q)= - \frac{\eps}{2} I_\eps \mathcal  R (\vartheta)  \int_{\partial\mathcal S_0}\left| \frac{\partial\psi_\eps}{\partial n} (q,\eps R(\vartheta) \cdot +h) \right|^2  \boldsymbol{K} (0, \cdot) \, {\rm d}s ,
\end{equation*}
where $\boldsymbol{K} (q, \cdot)$ is the vector field defined in  \eqref{Kir-gras}. 
Now let $\delta > 0$. 
Using Proposition~\ref{dev-psi} we deduce  that there exists $\eps_{0} $ in $(0,1)$ such that for any $(\varepsilon ,q) $ in $ \mathfrak Q_{\delta,\eps_{0}} $,
\begin{equation} \label{gather-mec1}
E_\varepsilon (q) =  
I_\eps  \mathcal  R (\vartheta)
\Big( \frac{1}{\eps}  \underline{\mathsf{E}}_{{\it -1}}
+ \underline{\mathsf{E}}_{{\it 0}} (q )
+ \varepsilon \underline{\mathsf{E}}_{{\it 1}} (q)
+ \eps^2 \underline{E}_{r}  (\eps,q)  \Big),
\end{equation}
with 
\begin{align} \nonumber
\underline{\mathsf{E}}_{{\it -1}}&:= 
- \frac{1}{2} \int_{\partial\mathcal S_0}
\left|\frac{\partial\psi^{\Ext}_{{\it -1}} }{\partial n}   \right|^2 
\boldsymbol{K} (0, \cdot) \, {\rm d}s , \\
\label{E0def}
\underline{\mathsf{E}}_{{\it 0}} (q )&:= - \int_{\partial\mathcal S_0}
\frac{\partial\psi^{\Ext}_{{\it -1}} }{\partial n}   
\left(  \frac{\partial\psi^{\Ext}_{{\it 0}}}{\partial n} (q,\cdot)  -  R(\vartheta)^t \, u^\Omega (h) \cdot \tau \right)
\boldsymbol{K} (0, \cdot) \, {\rm d}s , \\
\label{gather-mec1demi}
\underline{\mathsf{E}}_{{\it 1}} (q) &:= \underline{\mathsf{E}}_{{\it 1}}^a (q)  + \underline{\mathsf{E}}_{{\it 1}}^b (q) , 
\end{align}
where
\begin{align}
\label{jura}
\underline{\mathsf{E}}_{{\it 1}}^a (q) &:= - \frac{1}{2} \int_{\partial\mathcal S_0}
\left|  \frac{\partial\psi^{\Ext}_{{\it 0}}}{\partial n}  (q,\cdot)  -  R(\vartheta)^t \,  u^\Omega (h) \cdot \tau
   \right|^{2}  \boldsymbol{K} (0, \cdot) \, {\rm d}s , \\
\label{jurb}
\underline{\mathsf{E}}_{{\it 1}}^b (q) &:= -
\int_{\partial\mathcal S_0}  \frac{\partial\psi^{\Ext}_{{\it -1}}}{\partial n}  
\left( \frac{\partial\psi^{\Ext}_{{\it 1}}}{\partial n}  (q,\cdot) -  \frac{\partial P_{{\it 1}}  }{\partial n} (q,\cdot) \right)
\boldsymbol{K} (0, \cdot) \, {\rm d}s,
 \end{align}
and $\underline{E}_{r}  $ in $   L^\infty (\mathfrak Q_{\delta,\eps_{0}} ; \R^3)$
depending only on  $\mathcal S_0$ and $\Omega$. \par
\ \par
We now compute each term thanks to Lamb's lemma.
More precisely we establish the following equalities:
\begin{align}
\label{E-1}
\underline{\mathsf{E}}_{{\it -1}} &= 0 , \\
\label{E0}
\mathcal  R (\vartheta)   \underline{\mathsf{E}}_{{\it 0}} (q) &= {\mathsf{E}}_{\it 0} (q) , \\
\label{E1a}
\mathcal  R (\vartheta)   \underline{\mathsf{E}}_{{\it 1}}^a (q) &= {\mathsf{E}}_{\it 1}^a (q) , \\
\label{E1b}
\mathcal  R (\vartheta)   \underline{\mathsf{E}}_{{\it 1}}^b (q) &= \mathsf{E}_{{\it 1}}^b (q)  + \mathsf{E}_{{\it 1}}^c (q) .
\end{align}
The proof of Proposition~\ref{E-exp} is then concluded after observing that $\mathcal  R (\vartheta)  \underline{E}_{r} $ is also in $L^\infty (\mathfrak Q_{\delta,\eps_{0}} ; \R^3) $  and depends only on  $\mathcal S_0$ and $\Omega$. \par
To simplify the notations we omit to write the dependence on $q$ except if this dependence reduces on $\vartheta$ or $h$.
Similarly we omit to write that the function $\boldsymbol{K} $, its coordinates  $K_j$ and the vector fields $\xi_j$, which appear thanks to Lamb's lemma, are evaluated at $q=0$. \par
\ \par
\noindent
\textbf{Proof of \eqref{E-1}. Computation of $\underline{\mathsf{E}}_{{\it -1}}$}.
We use Lemma \ref{blasius} with  $u=v=  \nabla^\perp \psi^{\Ext}_{{\it -1}} $
and observe that $\nabla^\perp \psi^{\Ext}_{{\it -1}} $ is tangent to $\mathcal S_0$ to  obtain \eqref{E-1}. \par
%
%
%%%%%%%%%%%%%%%%%%%
%
\ \par
\noindent
\textbf{Proof of \eqref{E0}. Computation of $\underline{\mathsf{E}}_{{\it 0}}$}.
We observe that 
\begin{align} \label{obs}
\nabla^\perp \psi^{\Ext}_{{\it -1}} &= -  \frac{\partial\psi^{\Ext}_{{\it -1}}}{\partial n} \tau
\text{ on } \partial\mathcal S_0 ,
\\ \label{obs-aut}
\tau  \cdot  \nabla^\perp \psi^{\Ext}_{{\it 0}} &= -  \frac{\partial\psi^{\Ext}_{{\it 0}}}{\partial n} 
\text{ on } \partial\mathcal S_0 .
\end{align}
Hence for $j=1,2,3$, 
\begin{equation*}
- \int_{\partial\mathcal S_0}
\frac{\partial\psi^{\Ext}_{{\it -1}} }{\partial n}\cdot    \frac{\partial\psi^{\Ext}_{{\it 0}}}{\partial n} K_j \, {\rm d}s 
= - \int_{\partial\mathcal S_0}
\nabla \psi^{\Ext}_{{\it -1}} \cdot   \nabla \psi^{\Ext}_{{\it 0}} K_j \, {\rm d}s ,
\end{equation*}
and we use Lemma~\ref{blasius} with 
$(u,v) = (\nabla^\perp \psi^{\Ext}_{{\it -1}} , \nabla^\perp  \psi^{\Ext}_{{\it 0}} )$
to  obtain
\begin{equation*}
- \int_{\partial\mathcal S_0}
\frac{\partial\psi^{\Ext}_{{\it -1}} }{\partial n}\cdot    \frac{\partial\psi^{\Ext}_{{\it 0}}}{\partial n} K_j \, {\rm d}s 
= - \int_{\partial\mathcal S_0} (\xi_j \cdot \nabla^\perp \psi^{\Ext}_{{\it -1}}  )
(n \cdot \nabla^\perp \psi^{\Ext}_{{\it 0}}  ) \, {\rm d}s .
\end{equation*}
Then we use again \eqref{obs} and 
observe that applying the tangential derivative to  \eqref{ohoh}, taking \eqref{polyh0} into account, yields
\begin{equation} \label{surbord}
n \cdot \nabla^\perp \psi^{\Ext}_{{\it 0}} = \frac{\partial\psi^{\Ext}_{{\it 0}}}{\partial \tau}
=  R(\vartheta)^t \,  u^\Omega (h)^\perp \cdot \tau  \text{ on } \partial\mathcal S_0 .
\end{equation}
Thus
\begin{equation} \label{deez}
- \int_{\partial\mathcal S_0}
\frac{\partial\psi^{\Ext}_{{\it -1}} }{\partial n}\cdot    \frac{\partial\psi^{\Ext}_{{\it 0}}}{\partial n} K_j \, {\rm d}s 
= \int_{\partial\mathcal S_0} \frac{\partial \psi^{\Ext}_{{\it -1}}  }{\partial n}
(\xi_j \cdot \tau ) ( R(\vartheta)^t \,  u^\Omega (h)^\perp \cdot \tau ) \, {\rm d}s .
\end{equation}
Then using \eqref{53d-1}, \eqref{def-zeta} and \eqref{def-zeta_vartheta}, we arrive at 
\begin{equation*} \label{gather-mec3}
\mathcal  R (\vartheta) \underline{\mathsf{E}}_{{\it 0}}
= \mathcal  R (\vartheta) 
\left( \int_{\partial\mathcal S_0}
  \frac{\partial\psi^{\Ext}_{{\it -1}} }{\partial n}   \big( \xi_j \cdot  R(\vartheta)^t \,   u^\Omega (h)^\perp  \big) \, {\rm d}s \right)_j
= -
\begin{pmatrix}
u^\Omega (h) \cdot  \zeta_\vartheta \\
u^\Omega (h)^{\perp} 
\end{pmatrix} = \mathsf{E}_{{\it 0}} .
\end{equation*}
%
%
%%%%%%%%%%%%%%%%%%%%%
%
\ \par
\noindent
\textbf{Proof of  \eqref{E1a}. Computation of $\underline{\mathsf{E}}_{{\it 1}}^a$}.
We start with expanding the square in  \eqref{jura}, to get 
\begin{equation} \label{multisum}
\underline{\mathsf{E}}_{{\it 1}}^a 
= \underline{\mathsf{E}}_{{\it 1}}^{a,1}
+ \int_{\partial\mathcal S_0}  \frac{\partial\psi^{\Ext}_{{\it 0}}}{\partial n}  \big(R(\vartheta)^t \,  u^\Omega (h) \cdot \tau \big)  \boldsymbol{K} \, {\rm d}s
- \frac{1}{2} \int_{\partial\mathcal S_0} \left| R(\vartheta)^t \, u^\Omega (h) \cdot \tau  \right|^{2}  \boldsymbol{K} \, {\rm d}s ,
  \end{equation}
with
\begin{align*}
\underline{\mathsf{E}}_{{\it 1}}^{a,1}
=- \frac{1}{2} \int_{\partial\mathcal S_0}
 \left|  \frac{\partial\psi^{\Ext}_{{\it 0}}}{\partial n}  \right|^{2}   \boldsymbol{K} \, {\rm d}s
=
- \frac{1}{2} \int_{\partial\mathcal S_0}
\left|  \nabla \psi^{\Ext}_{{\it 0}}  \right|^{2}   \boldsymbol{K}  \, {\rm d}s
+ \frac{1}{2} \int_{\partial\mathcal S_0}
\left|  \frac{\partial\psi^{\Ext}_{{\it 0}}}{\partial \tau}  \right|^{2}   \boldsymbol{K} \, {\rm d}s .
\end{align*}
We apply Lemma~\ref{blasius} with $u=v=  \nabla^\perp \psi^{\Ext}_{{\it 0}} $ to get
\begin{equation*}
\frac{1}{2} \int_{\partial\mathcal S_0}
\left|  \nabla \psi^{\Ext}_{{\it 0}}   \right|^{2}   \boldsymbol{K} \, {\rm d}s
= \Big(  \int_{\partial\mathcal S_0}
\big(  \nabla^\perp \psi^{\Ext}_{{\it 0}} \cdot n \big)  
\big(  \nabla^\perp \psi^{\Ext}_{{\it 0}} \cdot  \xi_j  \big) \, {\rm d}s  \Big)_{j=1,2,3}
\end{equation*}
Let us denote by $\underline{\mathsf{E}}_{{\it 1},j}^{a,1}$, $j=1,2,3$,  the coordinates of the vector $\underline{\mathsf{E}}_{{\it 1}}^{a,1}$.
We use \eqref{surbord} to get 
\begin{multline*}
\underline{\mathsf{E}}_{{\it 1},j}^{a,1}=
 -  \int_{\partial\mathcal S_0} 
\big( \xi_j  \cdot \nabla^\perp \psi^{\Ext}_{{\it 0}}  \big)  \big( R(\vartheta)^t \, u^\Omega (h)^{\perp} \cdot \tau   \big) \, {\rm d}s 
\\  + \frac12  \int_{\partial\mathcal S_0} \big(  R(\vartheta)^t \, u^\Omega (h)^{\perp} \cdot \tau
\big)^{2}   K_j    \, {\rm d}s .
 \end{multline*}
Then we decompose $\xi_j  \cdot \nabla^\perp \psi^{\Ext}_{{\it 0}} $ in 
normal and tangential parts and use again \eqref{surbord} to obtain:
\begin{equation*}
\label{E11a}
\underline{\mathsf{E}}_{{\it 1},j}^{a,1}
=
 \int_{\partial\mathcal S_0} \frac{\partial\psi^{\Ext}_{{\it 0}}}{\partial n}
  \big( R(\vartheta)^t \, u^\Omega (h)^{\perp} \cdot \tau \big) (\xi_{j}\cdot \tau ) \, {\rm d}s 
-
\frac12  \int_{\partial\mathcal S_0} \big(  R(\vartheta)^t \, u^\Omega (h)^{\perp} \cdot \tau
\big)^{2}   K_j    \, {\rm d}s .
 \end{equation*}
Now we plug this expression of $\underline{\mathsf{E}}_{{\it 1}}^{a,1}$ into $\eqref{multisum}$ to get
\begin{align*}
\underline{\mathsf{E}}_{{\it 1},j}^a
&= -
\frac12  \int_{\partial\mathcal S_0} \big(  R(\vartheta)^t \, u^\Omega (h)^{\perp}
\big)^{2}   K_j    \, {\rm d}s 
+ \int_{\partial\mathcal S_0}  \frac{\partial\psi^{\Ext}_{{\it 0}}}{\partial n}  \big(R(\vartheta)^t \,  u^\Omega (h) \cdot \tau \big)  K_j \, {\rm d}s \\
& \quad
+ \int_{\partial\mathcal S_0} \frac{\partial\psi^{\Ext}_{{\it 0}}}{\partial n}
  \big( R(\vartheta)^t \, u^\Omega (h)^{\perp} \cdot \tau \big) (\xi_{j}\cdot \tau ) \, {\rm d}s  .
 \end{align*}
We observe that the first term in the right hand side vanishes and we combine the two other ones to get 
\begin{equation*}
\underline{\mathsf{E}}_{{\it 1},j}^a = 
\int_{\partial\mathcal S_0}  \frac{\partial\psi^{\Ext}_{{\it 0}}}{\partial n}
  \big( R(\vartheta)^t \, u^\Omega (h)^{\perp} \big) \cdot \xi_{j} \, {\rm d}s  .
\end{equation*}
Using \eqref{circ0nul} we infer
\begin{equation} \label{gather-mec4}
\underline{\mathsf{E}}_{{\it 1},j}^a = 0 \text{ for } j=2,3.
\end{equation}
Now for $j=1$, we start with observing that 
\begin{equation} \label{tot1}
\underline{\mathsf{E}}_{{\it 1},1}^a 
= \big( R(\vartheta)^t \, u^\Omega (h)^{\perp} \big) \cdot \int_{\partial\mathcal S_0} 
\frac{\partial\psi^{\Ext}_{{\it 0}}}{\partial n} x^{\perp} \, {\rm d}s  .
\end{equation}
To compute the right hand side we introduce the matrix
\begin{equation} \label{overlineM}
\overline{M} := \int_{\partial\mathcal S_0} 
\begin{pmatrix}
\overline{\varphi}^{\Ext}_{2} \\
\overline{\varphi}^{\Ext}_{3} 
\end{pmatrix}
\otimes
\begin{pmatrix}
\frac{\partial    \overline{\varphi}^{\Ext}_{3}  }{\partial n} \\
-\frac{\partial    \overline{\varphi}^{\Ext}_{2}  }{\partial n}
\end{pmatrix}
 \, {\rm d}s ,
\end{equation}
where the functions $\overline{\varphi}^{\Ext}_{j}$, for $j=1,2,3$, defined in \eqref{phi-0}, are harmonic conjugates to the functions $\varphi^{\Ext}_{j}$. %\par
\begin{lem} \label{finalLundi}
For any $q:= (\vartheta ,h) $ in $\R \times \Omega$, 
\begin{equation} \label{Mbarre}
\int_{\partial \mathcal S_0} \frac{\partial\psi^{\Ext}_{{\it 0}}}{\partial n} (q,x) x^{\perp} \, {\rm d}s(x)
= \overline{M}  R(\vartheta)^t \, u^\Omega (h)^{\perp} .
\end{equation}
\end{lem}
\begin{proof}
First it follows from \eqref{ahah-phi} that, on $\partial \mathcal S_{0}$, 
\begin{equation} \label{Mbarre0}
x^{\perp} = 
\begin{pmatrix}
\overline{\varphi}^{\Ext}_{2} - c^{\Ext}_{2}  \\
\overline{\varphi}^{\Ext}_{3}  - c^{\Ext}_{3}
\end{pmatrix}
.
\end{equation}
We now express the stream function $\psi^{\Ext}_{{\it 0}}  (q, \cdot) $ thanks to the functions $ \overline{\varphi}^{\Ext}_{3}$ and $\overline{\varphi}^{\Ext}_{2} $.
Let  $q:= (\vartheta ,h) $ in $\R \times \Omega$.
On $\partial  \mathcal S_0$, it follows from \eqref{polyh0}, \eqref{ohoh}, \eqref{phi-0} and Proposition~\ref{coro-cap} 
that there exists $c$ in $\R$ such that, on $\overline{\R^{2} \setminus \mathcal S_0 }$, 
\begin{equation*}
\psi^{\Ext}_{{\it 0}}  (q, \cdot) 
=
R(\vartheta)^t \, u^\Omega (h)^{\perp} \cdot 
\begin{pmatrix}
 \overline{\varphi}^{\Ext}_{3}   \\
-  \overline{\varphi}^{\Ext}_{2} 
\end{pmatrix}
+ c \; \psi^{\Ext}_{{\it -1}}   .
\end{equation*}
Then using \eqref{53d-1}, \eqref{circ0nul} and   \eqref{circ-saoule} we obtain $c=0$.
Thus for any $q:= (\vartheta ,h) $ in $\R \times \Omega$,  on $\overline{\R^{2} \setminus \mathcal S_0 }$, 
\begin{equation} \label{MbarreBis}
\psi^{\Ext}_{{\it 0}}  (q, \cdot) 
= R(\vartheta)^t \, u^\Omega (h)^{\perp} \cdot 
\begin{pmatrix}
 \overline{\varphi}^{\Ext}_{3}   \\
-  \overline{\varphi}^{\Ext}_{2} 
\end{pmatrix} .
\end{equation}
Substituting \eqref{Mbarre0} and \eqref{MbarreBis} into the left hand side of \eqref{Mbarre} 
and using again \eqref{circ0nul} establishes Lemma~\ref{finalLundi}.
\qed
\end{proof} 
Then, combining  \eqref{tot1} and \eqref{Mbarre}, we obtain:
\begin{equation*}
\underline{\mathsf{E}}_{{\it 1},1}^a=   \big(  R(\vartheta)^t \, u^\Omega (h)^{\perp} \big) \cdot \overline{M}  R(\vartheta)^t \, u^\Omega (h)^{\perp} .
\end{equation*}
Let us now connect the matrices $M^\dagger$ defined in \eqref{def-Mdagger} and $\overline{M}$ defined in \eqref{overlineM}.
Using integrations by parts and \eqref{eqeq}, we get, for any $i,j=2,3$, 
\begin{align*}
\int_{\partial\mathcal S_0}  
\frac{\partial{\overline{\varphi}^{\Ext}_{i}}}{\partial n} \overline{\varphi}^{\Ext}_{j} \, {\rm d}s 
&= \int_{\partial \mathcal S_0} 
\frac{\partial{{\varphi}^{\Ext}_{i}}}{\partial n} {\varphi}^{\Ext}_{j} {\rm d}x .
\end{align*}
Combining this with \eqref{def-Mdagger2} yields 
\begin{equation} \label{dagger-over}
M^\dagger = \frac12 ( \overline{M} +  \overline{M}^t ) .
\end{equation}
Recalling the definition of $M_\vartheta^{\dagger}$ in \eqref{Mflattheta}, we deduce that 
\begin{equation} \label{gather-mec5}
\underline{\mathsf{E}}_{{\it 1},1}^a =  u^\Omega (h)^{\perp}  M_\vartheta^{\dagger} \, u^\Omega (h)^{\perp} .
\end{equation}
Gathering \eqref{def-E1a}, \eqref{gather-mec4} and \eqref{gather-mec5} we obtain \eqref{E1a}. \par
%
%%%%%%%%%%%%%%%%%%%%%
%
\ \par
\noindent
\textbf{Proof of  \eqref{E1b}. Computation of $\underline{\mathsf{E}}_{{\it 1}}^b$}.
We start by splitting $\underline{\mathsf{E}}_{{\it 1}}^b$ into two parts as follows:
\begin{equation*}
\underline{\mathsf{E}}_{{\it 1}}^b  = 
- \int_{\partial\mathcal S_0}  \frac{\partial\psi^{\Ext}_{{\it -1}}}{\partial n} 
     \frac{\partial\psi^{\Ext}_{{\it 1}}}{\partial n}
     \boldsymbol{K} \, {\rm d}s
+ \int_{\partial\mathcal S_0}  \frac{\partial\psi^{\Ext}_{{\it -1}}}{\partial n}   \frac{\partial P_{{\it 1}}}{\partial n} 
   \boldsymbol{K} \, {\rm d}s.
\end{equation*}
Using \eqref{obs} and \eqref{obs-aut}, we see that the first term of the right hand side above is equal to 
\begin{equation*}
- \int_{\partial\mathcal S_0} \nabla^\perp \psi^{\Ext}_{{\it -1}} \cdot \nabla^\perp \psi^{\Ext}_{{\it 1}} \boldsymbol{K} \, {\rm d}s .
\end{equation*}
We denote $\underline{\mathsf{E}}_{{\it 1},j}^b$, $j=1,2,3$, the coordinates of $\underline{\mathsf{E}}_{{\it 1}}^{b}$.
We apply Lemma~\ref{blasius} with $u= \nabla^\perp \psi^{\Ext}_{{\it -1}}$ 
and $v= \nabla^\perp \psi^{\Ext}_{{\it 1}}$ for any $j=1,2,3$, to get
\begin{equation*}
\underline{\mathsf{E}}_{{\it 1},j}^b = 
-  \int_{\partial\mathcal S_0} \frac{\partial\psi^{\Ext}_{{\it 1}}}{\partial \tau}   \xi_j \cdot  \nabla^\perp \psi^{\Ext}_{{\it -1}} \, {\rm d}s 
+ \int_{\partial\mathcal S_0}  \frac{\partial\psi^{\Ext}_{{\it -1}}}{\partial n}   \frac{\partial P_{{\it 1}}}{\partial n} 
  K_j \, {\rm d}s .
\end{equation*}
We now use that, on $\partial\mathcal S_0$,
\begin{equation*}
\xi_j \cdot  \nabla^\perp \psi^{\Ext}_{{\it -1}}
= - \frac{\partial\psi^{\Ext}_{{\it -1}}}{\partial n} \xi_j \cdot \tau
\ \text{ and } \
\frac{\partial\psi^{\Ext}_{{\it 1}}}{\partial \tau} =  \frac{\partial P_{{\it 1}}}{\partial \tau},
\end{equation*}
the last identity being a consequence of \eqref{ahah-1}, to deduce that 
\begin{equation} \label{expp}
\underline{\mathsf{E}}_{{\it 1},j}^b = 
\int_{\partial\mathcal S_0}  \frac{\partial\psi^{\Ext}_{{\it -1}}}{\partial n}  \xi_j \cdot \nabla P_{{\it 1}} \, {\rm d}s .
\end{equation}
Thanks to the expression of  $P_{{\it 1}}$ in \eqref{polyh1}, 
\begin{equation} \label{alacon}
\underline{\mathsf{E}}_{{\it 1},j}^b  = 
- \langle   D^2_x \psi^{\Int}_{{\it 0}} (h,h) ,  R(\vartheta) A^{1}_j  R(\vartheta)^{t} \rangle_{\R^{2 \times 2}}
- D_x  \psi^{\Int}_{{\it 1}} (q,h)  \cdot  R(\vartheta) A^{2}_j    ,
\end{equation}
where
\begin{equation*}
A^{1}_j  :=  \int_{\partial\mathcal S_0}  \frac{\partial\psi^{\Ext}_{{\it -1}}}{\partial n}x \otimes  \xi_j  \, {\rm d}s 
\ \text{ and } \ 
A^{2}_j  :=  \int_{\partial\mathcal S_0}  \frac{\partial\psi^{\Ext}_{{\it -1}}}{\partial n} \xi_j  \, {\rm d}s .
\end{equation*}
\ \par
\noindent
%$\bullet$ \ 
{\it Case $j=1$.} Consider the first term in the right hand side of \eqref{alacon}.
Using \eqref{def-sigma} we see that $A_1^1 = \sigma -  \zeta \otimes \zeta^{\perp}$ and 
we observe that, since $D^2_x \psi^{\Int}_{{\it 0}} (h,h) $ is symmetric, 
\begin{equation*}
\langle   D^2_x \psi^{\Int}_{{\it 0}} (h,h) ,  R(\vartheta) \sigma R(\vartheta)^{t} \rangle_{\R^{2 \times 2}} 
= \langle   D^2_x \psi^{\Int}_{{\it 0}} (h,h) ,  R(\vartheta)  \sigma^s R(\vartheta)^t \rangle,
\end{equation*}
where $\sigma^s$ is the symmetric part of $\sigma$ defined in \eqref{anti-sigma}.
Then using that $ \sigma^s$ is a traceless symmetric $2 \times 2$ matrix, 
\begin{equation*}
\langle   D^2_x \psi^{\Int}_{{\it 0}} (h,h) ,  R(\vartheta) \sigma R(\vartheta)^{t} \rangle_{\R^{2 \times 2}} 
= \langle   D^2_x \psi^{\Int}_{{\it 0}} (h,h) ,  R(-2\vartheta)  \sigma^s \rangle = - \mathsf{E}_{{\it 1},1}^b (q) ,
\end{equation*}
where $\mathsf{E}_{{\it 1},1}^b (q) $ denotes the first coordinate of the vector field $\mathsf{E}_{{\it 1}}^b (q) $ defined in  \eqref{DefH1b}.
Therefore we obtain for $j=1$,  %the first term in the right hand side of \eqref{alacon} is equal to  
\begin{equation*}
- \langle   D^2_x \psi^{\Int}_{{\it 0}} (h,h) ,  R(\vartheta) A^{1}_j  R(\vartheta)^{t} \rangle_{\R^{2 \times 2}} =
\mathsf{E}_{{\it 1},1}^b (q)
+ \langle   D^2_x \psi^{\Int}_{{\it 0}} (h,h) ,  \zeta_{\vartheta} \otimes \zeta_{\vartheta}^{\perp} \rangle_{\R^{2 \times 2}} .
\end{equation*}
Concerning the second term in the right hand side of \eqref{alacon}, we use $A_1^2 = - \zeta^{\perp}$ (see \eqref{def-zeta}) to get that for $j=1$,
\begin{equation*}
- D_x  \psi^{\Int}_{{\it 1}} (q,h)  \cdot  R(\vartheta) A^{2}_j =
D_x  \psi^{\Int}_{{\it 1}} (q,h)  \cdot  \zeta^{\perp}_{\vartheta} .
\end{equation*}
Thus 
\begin{equation*}
\underline{\mathsf{E}}_{{\it 1},1}^{b} =  \mathsf{E}_{{\it 1},1}^b (q)
+ \langle   D^2_x \psi^{\Int}_{{\it 0}} (h,h) ,  \zeta_{\vartheta} \otimes \zeta_{\vartheta}^{\perp} \rangle_{\R^{2 \times 2}}
+ D_x  \psi^{\Int}_{{\it 1}} (q,h)  \cdot  \zeta^{\perp}_{\vartheta} .
\end{equation*} 
The last two terms in the right hand side can be expressed in terms of the corrector velocity $u_{c} (q)$ defined in \eqref{def-uc-perp}, as follows from the following statement.
\begin{lem} \label{jedepe}
For any $q=(\vartheta , h) $ in $\Omega \times  \R$, 
\begin{equation} \label{def-uc}
 u_c (q) = 
\Big( 
 D^2_x \psi^{\Int}_{{\it 0}} (h,h) \cdot {\zeta}_\vartheta 
+ D_x  \psi^{\Int}_{{\it 1}} (q,h) 
\Big)^\perp .
\end{equation}
\end{lem}
\begin{proof}
From the definition of $\psi_c$ in \eqref{PsiC} and the one of $u_{c} (q)$ in \eqref{def-uc-perp} we deduce that for any $q=(\vartheta, h)$ in $\Omega \times \R$,
\begin{equation*}
u_c (q) = \Big(  D^2_x \psi^{\Int}_{{\it 0}} (h,h) \cdot {\zeta}_\vartheta
+ D^2_{xh} \psi^{\Int}_{{\it 0}} (h,h) \cdot {\zeta}_\vartheta \Big)^\perp  ,
\end{equation*}
which yields \eqref{def-uc} thanks to  \eqref{29dec-eq}.
\qed
\end{proof} 
Hence we finally obtain
\begin{equation} \label{nrv1} 
\underline{\mathsf{E}}_{{\it 1},1}^b  =  \mathsf{E}_{{\it 1},1}^b  (q) - \zeta_{\vartheta} \cdot  u_c (q) .
\end{equation}
%
%
%
%$\bullet$ \ 
{\it Case $j=2$ or $3$.}
In this case, $A^{1}_j  =  - \zeta \otimes \xi_j $ and    $A^{2}_j  =  -  \xi_j  $,
and therefore
\begin{align*}
\underline{\mathsf{E}}_{{\it 1},j}^b
&=  \langle   D^2_x \psi^{\Int}_{{\it 0}} (h,h) ,  R(\vartheta)  ( \zeta \otimes   \xi_j  ) R(\vartheta)^{t} 
	\rangle_{\R^{2 \times 2}}
+   D  \psi^{\Int}_{{\it 1}} (q,h) \cdot R(\vartheta)  \xi_j  \\
&=  \Big(   D^2_x \psi^{\Int}_{{\it 0}} (h,h)  \cdot   \zeta_{\vartheta}  +   D  \psi^{\Int}_{{\it 1}} (q,h)  \Big) \cdot  R(\vartheta)  \xi_j   \\
&= - R(\vartheta)^{t} \, u_{c} (q)^{\perp} \cdot   \xi_j  .
\end{align*}
Thus 
\begin{align} \label{nrv2} 
R(\vartheta) ( \underline{\mathsf{E}}_{{\it 1},1}^b  )_{j=2,3} & =   - u_{c} (q)^{\perp} .
\end{align}
Gathering \eqref{DefH1b}, \eqref{DefH1c}, \eqref{nrv1} and \eqref{nrv2} we obtain \eqref{E1b}. This ends the proof of Proposition~\ref{E-exp}
\qed
\end{proof} 
%
%
%%%%%%%%%%%%%%%%%%%%%%%%%%%%%%%%%%%%%%%%%%
%
%
%
%
\subsection{Asymptotic expansion of $B_\varepsilon$}
We now tackle the expansion of $B_\varepsilon$ which is given, for $(\eps,q)$ in $\mathfrak Q $,  by
\begin{equation*} \label{B-def-eps}
B_{\eps}(q) :=  \int_{\partial\mathcal S_{\eps}(q)}
\frac{\partial\psi_{\eps}}{\partial n} (q,\cdot) \left( \boldsymbol K_{\eps}  (q,\cdot) \times 
\frac{\partial\boldsymbol\varphi_{\eps}}{\partial \tau} (q,\cdot)\right) \, {\rm d}s .
\end{equation*}
This formula is the counterpart of  \eqref{B-def} for a body of size $\eps$.
Let us recall that the Kirchhoff potentials $\boldsymbol\varphi_{\eps}$ are defined in \eqref{Kir-eps}-\eqref{Kir-vect-pasc}. \par
The expansion that we obtain for $B_{\eps}(q)$ is given in the following statement where $B^{\Ext}_{\vartheta}$ is defined in \eqref{def-Bext}, $M_{\vartheta}^{\dagger}$ in \eqref{Mflattheta}, $I_{\varepsilon}$ in \eqref{Eq:Ieps} and where:
\begin{equation} \label{DefB1}
\mathsf{B}_{{\it 1}} (q) :=
\begin{pmatrix}
0 \\ 
- 2 M_{\vartheta}^{\dagger}  u^\Omega (h)^{\perp}
\end{pmatrix}  \text{ for }  q=(\vartheta , h).
\end{equation}
\begin{prop}  \label{B-exp}
%
%Let $\delta > 0$. There exists $\eps_{0} $ in $(0,1)$ and a function
There exists $B_{r} : \mathfrak Q \rightarrow \R^3$  depending only on $\mathcal S_0$ and $\Omega$,  
 such that for any $\delta >0$, there exists $\varepsilon_{0}$ in $(0,1)$ for which $B_{r} \in L^\infty (\mathfrak Q_{\delta,\eps_{0}}; \R^3)$
and such that for any $(\varepsilon ,q) $ in $ \mathfrak Q_{\delta,\eps_{0}}$, where $q=(\vartheta , h)$,
\begin{equation} \label{B-dev}
B_\varepsilon  (q) = 
\eps I_\eps^{-1}  
\Big( B^{\Ext}_{\vartheta}
+ \varepsilon \mathsf{B}_{{\it 1}} (q)
+ \eps^2 B_{r} (\eps,q)  \Big).
\end{equation}
\end{prop}
\begin{proof}
We proceed as in the proof of Proposition \ref{E-exp}. 
Let us state the following formula which is useful several times in the sequel: 
\begin{equation} \label{vprod25}
\text{for any } (p_a ,p_b) \in  \R^3 \times \R^{3} , \quad 
\eps \, p_a \times {p}_b  = I_\eps  \Big(  (I_\eps  {p}_a) \times (I_\eps   p_{b}) \Big) .
\end{equation}
By a change of variable, using \eqref{vprod25} and \eqref{change-K}, we arrive at 
\begin{multline*} \label{B-def-eps-cv}
B_{\eps}(q) = \\ \eps I_\eps^{-1}   \mathcal  R (\vartheta)
\int_{\partial  \mathcal S_0}\frac{\partial\psi_{\eps}}{\partial n}   (q,\eps R(\vartheta) \cdot +h) 
\left( \boldsymbol{K} (0, \cdot) \times 
\mathcal  R (\vartheta)^t \, \frac{\partial\boldsymbol\varphi_{\eps}}{\partial \tau} (q,\eps R(\vartheta) \cdot +h)\right) \, {\rm d}s .
\end{multline*}
Now let $\delta > 0$. 
We use Proposition~\ref{dev-psi} and Proposition~\ref{dev-phi-pasc} to obtain that there exists $\eps_{0} $ in $(0,1)$ such that for any $(\varepsilon ,q) $ in $ \mathfrak Q_{\delta,\eps_{0}} $, 
\begin{equation*}
B_\varepsilon   (q) = 
\eps I_\eps^{-1}   \mathcal  R (\vartheta)
\Big( \underline{\mathsf{B}}_{{\it 0}} 
+ \varepsilon \underline{\mathsf{B}}_{{\it 1}} (q)
+ \eps^2 B_{r} (\eps,q)  \Big),
\end{equation*}
with
\begin{align*}
\underline{\mathsf{B}}_{{\it 0}}   &:=
  \int_{\partial  \mathcal S_0} 
 \frac{\partial\psi^{\Ext}_{{\it -1}}}{\partial n}   
 \left( \boldsymbol K (0, \cdot)  \times 
\frac{\partial\boldsymbol\varphi^{{\Ext}}}{\partial \tau} \right) 
 \, {\rm d}s 
 ,
\\
\underline{\mathsf{B}}_{{\it 1}} (q) &:=
 \int_{\partial  \mathcal S_0}
\Big(  \frac{\partial\psi^{\Ext}_{{\it 0}}}{\partial n}  
- R(\vartheta)^t \,  u^\Omega (h) \cdot \tau
 \Big) 
\left( \boldsymbol K (0, \cdot)  \times 
\frac{\partial\boldsymbol\varphi^{{\Ext}}}{\partial \tau}  \right)
\, {\rm d}s 
 ,
\end{align*}
and $B_{r} $ in $ L^\infty (\mathfrak Q_{\delta,\eps_{0}} ; \R^{3} )$   depending only on  $\mathcal S_0$ and $\Omega$. \par
%
%\ \par
%
We now compute each term thanks to Lamb's lemma. More precisely we will prove the following equalities:
\begin{align}
\label{B0}
\mathcal  R (\vartheta) \underline{\mathsf{B}}_{{\it 0}}  &=  B^{\Ext}_{\vartheta}, \\
\label{B1}
\mathcal  R (\vartheta) \underline{\mathsf{B}}_{{\it 1}}  &= {\mathsf{B}}_{{\it 1}} .
\end{align}
As in the proof of Proposition \ref{B-exp} we will omit to write the dependence on $q$, except if this dependence reduces to a dependence on $\vartheta$ or $h$, and it will be understood  that the   functions $\boldsymbol{K} $, its coordinates  $K_j$ and the vector fields $\xi_j$
are evaluated at $q=0$. \par
\ \par
\noindent
\textbf{Proof of  \eqref{B0}. Computation of $\underline{\mathsf{B}}^{0}$}.
For $j=1,2,3$, we denote by $\underline{\mathsf{B}}_{{\it 0},j}$, the coordinates of  $\underline{\mathsf{B}}_{{\it 0}}  $.
By  \eqref{obs}, 
\begin{align*}
\underline{\mathsf{B}}_{{\it 0},1}    &= 
  \int_{\partial  \mathcal S_0}
 \frac{\partial\psi^{\Ext}_{{\it -1}}}{\partial n}  
\frac{\partial\varphi^{\Ext}_{3}}{\partial \tau}  K_{2} 
\, {\rm d}s 
 -
 \int_{\partial  \mathcal S_0}
 \frac{\partial\psi^{\Ext}_{{\it -1}}}{\partial n}  
\frac{\partial\varphi^{\Ext}_{2}}{\partial \tau}  K_{3} 
\, {\rm d}s 
 \\ &= -
   \int_{\partial  \mathcal S_0}
 \nabla^{\perp}\psi^{\Ext}_{{\it -1}}
\cdot
\nabla\varphi^{\Ext}_{3} K_{2} 
\, {\rm d}s 
 +
 \int_{\partial  \mathcal S_0}
 \nabla^{\perp}\psi^{\Ext}_{{\it -1}}
\cdot
\nabla\varphi^{\Ext}_{2} K_{3} 
\, {\rm d}s .
\end{align*}
Then we use Lemma~\ref{blasius} twice, with 
$(u,v)= ( \nabla^\perp \psi^{\Ext}_{{\it -1}} , \nabla \varphi^{\Ext}_{2} )$
and with $(u,v)= ( \nabla^\perp \psi^{\Ext}_{{\it -1}} , \nabla \varphi^{\Ext}_{3} )$, 
 \eqref{53d-1} and \eqref{obs} to obtain 
\begin{equation*}
\underline{\mathsf{B}}_{{\it 0},1}
=  \int_{\partial  \mathcal S_0}  \frac{\partial\psi^{\Ext}_{{\it -1}}}{\partial n} 
\Big(    ( \tau \cdot \xi_{2} )  ( n \cdot \xi_{3} ) -    ( \tau \cdot \xi_{3} )  ( n \cdot \xi_{2} ) \Big) \, {\rm d}s 
= - 1 .
\end{equation*}
Proceeding in the same way and using \eqref{def-zeta}, we arrive at
\begin{align*}
\underline{\mathsf{B}}_{{\it 0},2}  
= \int_{\partial  \mathcal S_0}
 \frac{\partial\psi^{\Ext}_{{\it -1}}}{\partial n} \Big(    ( \tau \cdot \xi_{3} )  ( n \cdot \xi_{1} ) -    ( \tau \cdot \xi_{1} )  ( n \cdot \xi_{3} ) \Big)\, {\rm d}s 
=  \zeta^{\perp} \cdot  \xi_{2} ,
\end{align*}
and
$\underline{\mathsf{B}}_{{\it 0},3} =  \zeta^{\perp} \cdot  \xi_{3} $.
This gives \eqref{B0}.\par
\ \par
\noindent
\textbf{Proof of  \eqref{B1}. Computation of $\underline{\mathsf{B}}_{{\it 1}}$}.
Let us start with the first coordinate $\underline{\mathsf{B}}_{{\it 1},1}$ of $\underline{\mathsf{B}}_{{\it 1}}$, that is:
\begin{align*}
\underline{\mathsf{B}}_{{\it 1},1} &=
- \int_{\partial  \mathcal S_0}
\Big(  \frac{\partial\psi^{\Ext}_{{\it 0}}}{\partial n}   
- R(\vartheta)^t \,  u^\Omega (h) \cdot \tau
 \Big) 
\left( \frac{\partial\varphi^{\Ext}_{2}}{\partial \tau}  K_3  -
\frac{\partial\varphi^{\Ext}_{3}}{\partial \tau} K_2   \right)
\, {\rm d}s 
\\ &= \underline{\mathsf{B}}_{{\it 1},1}^a +  \underline{\mathsf{B}}_{{\it 1},1}^b +  \underline{\mathsf{B}}_{{\it 1},1}^c  ,
\end{align*}
with 
\begin{align*}
\underline{\mathsf{B}}_{{\it 1},1}^a   &:=
- \int_{\partial  \mathcal S_0}
 \frac{\partial\psi^{\Ext}_{{\it 0}}}{\partial n}   
\frac{\partial\varphi^{\Ext}_{2}}{\partial \tau}  K_3
\, {\rm d}s ,
\\ \underline{\mathsf{B}}_{{\it 1},1}^b &:= 
\int_{\partial  \mathcal S_0}
\frac{\partial\psi^{\Ext}_{{\it 0}}}{\partial n}   
\frac{\partial\varphi^{\Ext}_{3}}{\partial \tau} K_2 
\, {\rm d}s ,
\\  \underline{\mathsf{B}}_{{\it 1},1}^c  &:=
 \int_{\partial  \mathcal S_0}
 R(\vartheta)^t \,  u^\Omega (h) \cdot 
\left( \frac{\partial\varphi^{\Ext}_{2}}{\partial \tau}  K_3  -
\frac{\partial\varphi^{\Ext}_{3}}{\partial \tau} K_2   \right) \tau
\, {\rm d}s .
\end{align*}
We start with
\begin{align*}
\underline{\mathsf{B}}_{{\it 1},1}^a   &=
 \int_{\partial  \mathcal S_0}
\nabla^{\perp} \psi^{\Ext}_{{\it 0}}
\cdot
\nabla \varphi^{\Ext}_{2} K_3
\, {\rm d}s 
-
\int_{\partial  \mathcal S_0}
 \frac{\partial\psi^{\Ext}_{{\it 0}}}{\partial  \tau}   
 K_{2} K_3
\, {\rm d}s .
\end{align*}
We use Lemma~\ref{blasius} with $u= \nabla^\perp \psi^{\Ext}_{{\it 0}} $ and $v = \nabla\varphi^{\Ext}_{2}$
to  obtain 
\begin{align*}
\int_{\partial  \mathcal S_0}
\nabla^{\perp} \psi^{\Ext}_{{\it 0}}
\cdot  \nabla \varphi^{\Ext}_{2} K_3 \, {\rm d}s 
&=
\int_{\partial  \mathcal S_0}
\big( \nabla^{\perp} \psi^{\Ext}_{{\it 0}} \cdot \xi_{3} \big)
\big( \nabla \varphi^{\Ext}_{2} \cdot  n \big) \, {\rm d}s  \\ 
&\quad + 
\int_{\partial  \mathcal S_0}
\big( \nabla^{\perp} \psi^{\Ext}_{{\it 0}} \cdot n \big)
\big( \nabla \varphi^{\Ext}_{2} \cdot  \xi_{3} \big) \, {\rm d}s v\\ 
&=
 \int_{\partial  \mathcal S_0}
\left(  \frac{\partial\psi^{\Ext}_{{\it 0}}}{\partial  \tau}   K_{3}
- \frac{\partial\psi^{\Ext}_{{\it 0}}}{\partial  n}   \xi_{3} \cdot \tau 
 \right) K_{2} \, {\rm d}s \\
+ 
\int_{\partial  \mathcal S_0}
 \frac{\partial\psi^{\Ext}_{{\it 0}}}{\partial  \tau}   
 K_{2} K_3 \, {\rm d}s 
& \quad + 
\int_{\partial  \mathcal S_0}  \frac{\partial\psi^{\Ext}_{{\it 0}}}{\partial  \tau}   
\frac{\partial\varphi^{\Ext}_{2}}{\partial  \tau}    (  \xi_{3} \cdot \tau )
\, {\rm d}s .
\end{align*}
Therefore
\begin{align*}
\underline{\mathsf{B}}_{{\it 1},1}^a   &=
 \int_{\partial  \mathcal S_0}
\left(  \frac{\partial\psi^{\Ext}_{{\it 0}}}{\partial  \tau}   K_{3}
-
 \frac{\partial\psi^{\Ext}_{{\it 0}}}{\partial  n}   \xi_{3} \cdot \tau 
 \right) K_{2}
\, {\rm d}s 
+
\int_{\partial  \mathcal S_0}  \frac{\partial\psi^{\Ext}_{{\it 0}}}{\partial  \tau}   
\frac{\partial\varphi^{\Ext}_{2}}{\partial  \tau}    (  \xi_{3} \cdot \tau )
\, {\rm d}s .
\end{align*}
By switching the indexes  $2$ and $3$ we obtain 
\begin{align*}
\underline{\mathsf{B}}_{{\it 1},1}^b   &=
- \int_{\partial  \mathcal S_0}
\left(  \frac{\partial\psi^{\Ext}_{{\it 0}}}{\partial  \tau}   K_{2}
-
 \frac{\partial\psi^{\Ext}_{{\it 0}}}{\partial  n}   \xi_{2} \cdot \tau 
 \right) K_{3}
\, {\rm d}s 
-
\int_{\partial  \mathcal S_0}  \frac{\partial\psi^{\Ext}_{{\it 0}}}{\partial  \tau}   
\frac{\partial\varphi^{\Ext}_{3}}{\partial  \tau}    (  \xi_{2} \cdot \tau )
\, {\rm d}s .
\end{align*}
We sum these two terms, observe that $ -  (\xi_{3} \cdot \tau ) K_{2} +  ( \xi_{2} \cdot \tau ) K_{3} = K_{2}^{2} + K_{3}^{2}=1$ and use \eqref{circ0nul} to get
\begin{equation*}
\int_{\partial  \mathcal S_0} (K_{2}^{2} + K_{3}^{2} )  \frac{\partial\psi^{\Ext}_{{\it 0}}}{\partial  n}  \, {\rm d}s   
= \int_{\partial  \mathcal S_0}    \frac{\partial\psi^{\Ext}_{{\it 0}}}{\partial  n}  \, {\rm d}s    = 0 .
\end{equation*}
We deduce
\begin{align*}
\underline{\mathsf{B}}_{{\it 1},1}^a + \underline{\mathsf{B}}_{{\it 1},1}^b   =
 \int_{\partial  \mathcal S_0} \frac{\partial\psi^{\Ext}_{{\it 0}}}{\partial  \tau} 
\left( 
\frac{\partial\varphi^{\Ext}_{2}}{\partial  \tau}   (  \xi_{3} \cdot \tau )
-
\frac{\partial\varphi^{\Ext}_{3}}{\partial  \tau}   (  \xi_{2} \cdot \tau ) 
\right)   \, {\rm d}s .
\end{align*}
Now, using \eqref{surbord}, we obtain 
\begin{align}
\nonumber
\underline{\mathsf{B}}_{{\it 1},1}^a + \underline{\mathsf{B}}_{{\it 1},1}^b   &=
- \int_{\partial  \mathcal S_0} R(\vartheta)^t \, u^\Omega (h) \cdot 
 \Big( \frac{\partial\varphi^{\Ext}_{2}}{\partial  \tau} (  \xi_{3} \cdot \tau )
- \frac{\partial\varphi^{\Ext}_{3}}{\partial  \tau}   (  \xi_{2} \cdot \tau ) 
\Big) n  \, {\rm d}s \\
\label{marre1} &= R(\vartheta)^t \, u^\Omega (h) \cdot
\int_{\partial  \mathcal S_0} 
\Big( 
 \frac{\partial\varphi^{\Ext}_{2}}{\partial  \tau} (  \xi_{2} \cdot n)
+ \frac{\partial\varphi^{\Ext}_{3}}{\partial  \tau}   (  \xi_{3} \cdot n) 
\Big) n  \, {\rm d}s .
\end{align}
On the other hand we observe that 
\begin{equation} \label{marre2}
\underline{\mathsf{B}}_{{\it 1},1}^c  = 
R(\vartheta)^t \, u^\Omega (h) \cdot 
\int_{\partial  \mathcal S_0} 
\Big( 
 \frac{\partial\varphi^{\Ext}_{2}}{\partial  \tau} (  \xi_{2} \cdot \tau)
+
\frac{\partial\varphi^{\Ext}_{3}}{\partial  \tau}   (  \xi_{3} \cdot \tau) 
\Big) \tau  \, {\rm d}s .
\end{equation}
By \eqref{marre1} and \eqref{marre2},  
\begin{equation*}
\underline{\mathsf{B}}_{{\it 1},1} =  R(\vartheta)^t \, u^\Omega (h) \cdot
\int_{\partial  \mathcal S_0} 
\Big[ \frac{\partial\varphi^{\Ext}_{2}}{\partial  \tau}   \xi_{2}
 +   \frac{\partial\varphi^{\Ext}_{3}}{\partial  \tau}  \xi_{3}\Big]  \, {\rm d}s = 0 ,
\end{equation*}
thanks to an integration by parts.\par
\ \par
Let us now consider the second coordinate $\underline{\mathsf{B}}_{{\it 1},2}$ of $\underline{\mathsf{B}}_{{\it 1}}$, that is:
\begin{align*}
\underline{\mathsf{B}}_{{\it 1},2}
&= \int_{\partial  \mathcal S_0}
\Big(  \frac{\partial\psi^{\Ext}_{{\it 0}}}{\partial n}   
- R(\vartheta)^t \,  u^\Omega (h) \cdot \tau
 \Big) 
\left( \frac{\partial\varphi^{\Ext}_{1}}{\partial \tau}  K_3  -
\frac{\partial\varphi^{\Ext}_{3}}{\partial \tau} K_1   \right)
\, {\rm d}s \\
&= \underline{\mathsf{B}}_{{\it 1},2}^a +  \underline{\mathsf{B}}_{{\it 1},2}^b +  \underline{\mathsf{B}}_{{\it 1},2}^c  ,
\end{align*}
with 
\begin{align*}
\underline{\mathsf{B}}_{{\it 1},2}^a
&:= - \int_{\partial  \mathcal S_0}
 \frac{\partial\psi^{\Ext}_{{\it 0}}}{\partial n}   
\frac{\partial\varphi^{\Ext}_{3}}{\partial \tau}  K_1
\, {\rm d}s , \\
\underline{\mathsf{B}}_{{\it 1},2}^b
&:= \int_{\partial  \mathcal S_0}
\frac{\partial\psi^{\Ext}_{{\it 0}}}{\partial n}   
\frac{\partial\varphi^{\Ext}_{1}}{\partial \tau} K_3 
\, {\rm d}s , \\
\underline{\mathsf{B}}_{{\it 1},2}^c
&:= \int_{\partial  \mathcal S_0}
R(\vartheta)^t \,  u^\Omega (h) \cdot 
\left( \frac{\partial\varphi^{\Ext}_{3}}{\partial \tau}  K_1  -
\frac{\partial\varphi^{\Ext}_{1}}{\partial \tau} K_3   \right)  \tau
\, {\rm d}s .
\end{align*}
Proceeding as above with $\underline{\mathsf{B}}_{{\it 1},1}^a$ and $\underline{\mathsf{B}}_{{\it 1},1}^b $,  we arrive at 
\begin{align*}
\underline{\mathsf{B}}_{{\it 1},2}^a
&= \int_{\partial  \mathcal S_0}
\left(  \frac{\partial\psi^{\Ext}_{{\it 0}}}{\partial  \tau}   K_{1}
- \frac{\partial\psi^{\Ext}_{{\it 0}}}{\partial  n}   \xi_{1} \cdot \tau 
 \right) K_{3} \, {\rm d}s 
+ \int_{\partial  \mathcal S_0}  \frac{\partial\psi^{\Ext}_{{\it 0}}}{\partial  \tau}   
\frac{\partial\varphi^{\Ext}_{3}}{\partial  \tau}    (  \xi_{1} \cdot \tau ) \, {\rm d}s , \\ 
\underline{\mathsf{B}}_{{\it 1},2}^b
&= - \int_{\partial  \mathcal S_0}
\left(  \frac{\partial\psi^{\Ext}_{{\it 0}}}{\partial  \tau}   K_{3}
- \frac{\partial\psi^{\Ext}_{{\it 0}}}{\partial  n}   \xi_{3} \cdot \tau  \right) K_{1} \, {\rm d}s 
- \int_{\partial  \mathcal S_0}  \frac{\partial\psi^{\Ext}_{{\it 0}}}{\partial  \tau}   
\frac{\partial\varphi^{\Ext}_{1}}{\partial  \tau}    (  \xi_{3} \cdot \tau ) \, {\rm d}s .
\end{align*}
We sum these two terms, observing that
$( \xi_{1} \cdot \tau )K_{3}  - (  \xi_{3} \cdot \tau )K_1 = x^{\perp} \cdot  \xi_{2}$,
to get 
\begin{multline} \label{macr2}
\underline{\mathsf{B}}_{{\it 1},2}^a + \underline{\mathsf{B}}_{{\it 1},2}^b
=
- \int_{\partial  \mathcal S_0} (  x^{\perp} \cdot  \xi_{2}   ) \frac{\partial\psi^{\Ext}_{{\it 0}}}{\partial  n}  \, {\rm d}s  \\
+ \int_{\partial  \mathcal S_0} \frac{\partial\psi^{\Ext}_{{\it 0}}}{\partial  \tau} 
\left( 
\frac{\partial\varphi^{\Ext}_{3}}{\partial  \tau}   (  \xi_{1} \cdot \tau )
- \frac{\partial\varphi^{\Ext}_{1}}{\partial  \tau}   (  \xi_{3} \cdot \tau ) 
\right)   \, {\rm d}s .
\end{multline}
Using  \eqref{Mbarre} we obtain
\begin{equation} \label{macr}
 \int_{\partial  \mathcal S_0} (  x^{\perp} \cdot  \xi_{2}   ) \frac{\partial\psi^{\Ext}_{{\it 0}}}{\partial  n}  \, {\rm d}s
= \overline{M}  R(\vartheta)^t \, u^\Omega (h)^{\perp} \cdot   \xi_{2} ,
\end{equation}
with $ \overline{M} $ given by  \eqref{overlineM}. 
We also use \eqref{surbord} to modify the second term in the right hand side of \eqref{macr2}  and then get
\begin{align*}
\underline{\mathsf{B}}_{{\it 1},2}^a + \underline{\mathsf{B}}_{{\it 1},2}^b   &=
- \overline{M}  R(\vartheta)^t \, u^\Omega (h)^{\perp} \cdot   \xi_{2}
 \\ &\quad - \int_{\partial  \mathcal S_0}
 R(\vartheta)^t \, u^\Omega (h)  \cdot 
\left( 
\frac{\partial\varphi^{\Ext}_{3}}{\partial  \tau}   (  \xi_{1} \cdot \tau )
-
\frac{\partial\varphi^{\Ext}_{1}}{\partial  \tau}   (  \xi_{3} \cdot \tau ) 
\right) n  \, {\rm d}s .
\end{align*}
Adding $\underline{\mathsf{B}}_{{\it 1},2}^c$,  we arrive at 
\begin{equation} \label{pff1}
 \underline{\mathsf{B}}_{{\it 1},2} =
- \overline{M}  R(\vartheta)^t \, u^\Omega (h)^{\perp} \cdot   \xi_{2}
- R(\vartheta)^t \, u^\Omega (h)  \cdot \underline{\mathsf{B}}_{{\it 1},2}^d ,
\end{equation}
with
\begin{align*}
\underline{\mathsf{B}}_{{\it 1},2}^d 
&:= \int_{\partial  \mathcal S_0} \frac{\partial\varphi^{\Ext}_{3}}{\partial  \tau}   
\Big(  ( \xi_{1}^{\perp} \cdot  n) n + ( \xi_{1}^{\perp} \cdot \tau )  \tau   \Big)   \, {\rm d}s \\
& {\hskip 1cm} +  \int_{\partial  \mathcal S_0} \frac{\partial\varphi^{\Ext}_{1}}{\partial  \tau}  
\Big(  - ( \xi_{3}^{\perp} \cdot  n) n - ( \xi_{3}^{\perp} \cdot \tau )  \tau   \Big)  \, {\rm d}s \\
&= \int_{\partial  \mathcal S_0} \frac{\partial\varphi^{\Ext}_{3}}{\partial  \tau}   
 \xi_{1}^{\perp}      \, {\rm d}s
- \int_{\partial  \mathcal S_0}  \frac{\partial\varphi^{\Ext}_{1}}{\partial  \tau}     \xi_{3}^{\perp}   \, {\rm d}s .
\end{align*}
Using an integration by parts we see that the second term of the right hand side  above vanishes and the first term gives
\begin{equation*}
\int_{\partial \mathcal S_0} \frac{\overline\varphi^{\Ext}_{3}}{\partial \tau} (x) x^{\perp} \, {\rm d}s(x)
= - \int_{\partial  \mathcal S_0} \varphi^{\Ext}_{3} \; n \, {\rm d}s
= -  \overline{M}  \xi_2 (0,\cdot) .
\end{equation*}
Therefore
\begin{equation} \label{pff2} 
R(\vartheta)^t \, u^\Omega (h)  \cdot \underline{\mathsf{B}}_{{\it 1},2}^d =  \overline{M}^t \,   R(\vartheta)^t \, u^\Omega (h)^{\perp}  \cdot \xi_{2} .
\end{equation}
Gathering \eqref{pff1} and  \eqref{pff2} and using \eqref{dagger-over},  we arrive at 
\begin{equation} \label{foufou2}
\underline{\mathsf{B}}_{{\it 1},2} = - 2 M^{\dagger}   R(\vartheta)^t \, u^\Omega (h)^{\perp}  \cdot \xi_{2} .
\end{equation}
Proceeding in the same way for the third coordinate, using 
\begin{equation*}
(  \xi_{1} \cdot \tau )K_{2}  - (  \xi_{2} \cdot \tau )K_1 = - x^{\perp} \cdot  \xi_{3} \text{ and }
\int_{\partial \mathcal S_0} \frac{\partial\overline\varphi^{\Ext}_{2}}{\partial \tau} (x) x^{\perp} \, {\rm d}s(x) 
= \overline{M}  \xi_3  (0,\cdot) ,
\end{equation*} 
 we deduce that 
\begin{equation} \label{foufou3}
\underline{\mathsf{B}}_{{\it 1},3} = - 2 M^{\dagger}   R(\vartheta)^t \, u^\Omega (h)^{\perp}  \cdot \xi_{3} .
\end{equation}
Combining \eqref{foufou2} and \eqref{foufou3} and recalling the definition of $M_{\vartheta}^{\dagger} $ in \eqref{Mflattheta}, we deduce  \eqref{B1}.
This ends the proof of Proposition~\ref{B-exp}.
\qed
\end{proof} 
%
%
%
%%%%%%%%%%%%%%%%%%%%%%%%%%%%%%%%%%%%%%%%%%
%
%
\subsection{End of the proof of the normal forms} 
\label{eqfnorm}

To prove  Proposition~\ref{Pro-fnormA} and Proposition~\ref{Pro-fnorm} we have to put Equation~\eqref{ODE_intro} under the normal forms \eqref{fnorm1} and \eqref{fnorm2}. 
We focus on the more delicate Case (ii). Case (i) can be proved with the same strategy with some simplifications,
since the proof of the normal form \eqref{fnorm2} corresponding to Case (ii) actually requires to perform additional manipulations. At the end of this section, we add a few words about Proposition~\ref{Pro-fnorm-ball-inh}. \par
By  \eqref{def-upsilon} the equation \eqref{ODE_intro} reads 
\begin{equation} \label{ODE_intro_epsB}
M_\varepsilon (q_\varepsilon)   q_\varepsilon'' 
=  \gamma^2 E_{\varepsilon} (q_\varepsilon) + \gamma q_\varepsilon' \times B_{\eps} (q_\varepsilon)  
-  \langle \Gamma_\varepsilon  (q_\varepsilon),q_\varepsilon',q_\varepsilon'\rangle .
\end{equation}
The proof now consists in substituting the previous expansions of $\Gamma_\varepsilon$, $E_{\varepsilon}$ and $B_{\eps}$ into the right hand side of \eqref{ODE_intro_epsB} and to rely on some crucial cancellations. 
Let $\delta > 0$. Using the decomposition \eqref{NiouGamma} of $\Gamma$, the definition \eqref{def-hat} of $p_{\varepsilon}$, the expansions \eqref{Gamma-int-exp} and \eqref{expGammaOmega} for the Christoffel symbols, \eqref{E-dev} for the electric field, \eqref{B-dev} for the magnetic field and the relation \eqref{vprod25} we obtain that,
for $\eps_{0}$ in $(0,1)$ small enough, as long as $(\varepsilon ,q_{\varepsilon})$ belongs to $\mathfrak Q_{\delta,\eps_{0}}$:
\begin{multline} \label{mesaoulent}
\gamma^2 E_{\varepsilon} (q_\varepsilon) + \gamma q'_\varepsilon \times B_{\eps} (q_\varepsilon)  
 - \langle\Gamma_\varepsilon  (q_\varepsilon), q'_\varepsilon, q'_\varepsilon\rangle 
 = I_\eps  \Big[ \Big( \gamma^2 \mathsf{E}_{{\it 0}}  (q_\varepsilon) +\gamma p_\varepsilon  \times B^{\Ext}_{\vartheta_\varepsilon}  \Big) \\
+ \eps \Big(  \gamma^2 \mathsf{E}_{{\it 1}}  (q_\varepsilon) +\gamma p_\varepsilon \times \mathsf{B}_{{\it 1}}  (q_\varepsilon)  
- \langle \Gamma^{\Ext}_{\vartheta_\varepsilon} , p_\varepsilon , p_\varepsilon \rangle \Big) 
 + \eps^2 \check{F}_{r} (\varepsilon, q_\varepsilon, p_\varepsilon ) \Big], 
\end{multline}
where 
\begin{multline*}
\check{F}_{r} (\varepsilon, q_\varepsilon, p_\varepsilon ) =  \gamma^2  E_{r}  (\eps,q_\varepsilon) + \gamma p_\varepsilon \times B_{r}  (\eps,q_\varepsilon) 
- \eps \langle \Gamma^{\rm rot}_{r} (\eps,q_\varepsilon), p_\varepsilon , p_\varepsilon \rangle  \\
- \eps \langle \Gamma^{\partial  \Omega}_{r} (\eps,q_\varepsilon), p_\varepsilon , p_\varepsilon \rangle .
\end{multline*}
Then recalling that ${F}^{\Ext}_{\vartheta_\varepsilon}$  is defined in \eqref{force-ext}, we observe that the zero order term in the right hand side of \eqref{mesaoulent} (in terms of powers of $\varepsilon$) can be recast as follows:
\begin{equation} \label{EB1}
\gamma^2 \mathsf{E}_{{\it 0}}  (q_\varepsilon) + \gamma p_{\varepsilon} \times B^{\Ext}_{\vartheta_\varepsilon}  (q_\varepsilon)
= {F}^{\Ext}_{\vartheta_\varepsilon} ( \varepsilon \vartheta_\varepsilon',  h_\varepsilon'  - \gamma  u^\Omega (h_\varepsilon ) )  .
\end{equation}
Now, in order to deal with the subprincipal term of the right hand side of  \eqref{mesaoulent}, let us state the following crucial lemma, where we consider only the part $\mathsf{E}_{{\it 1}}^a (q_\varepsilon)$ defined in \eqref{def-E1a} of the decomposition \eqref{decompE1} of the term $\mathsf{E}_{{\it 1}}(q_\varepsilon)$.
\begin{lem} \label{lemma-grav}
The following holds:
\begin{equation} \label{grav}
\gamma^2 \mathsf{E}_{{\it 1}}^a  (q_{\varepsilon}) + \gamma p_{\varepsilon}  \times {\mathsf{B}}_{{\it 1}}  (q_{\varepsilon})
- \langle \Gamma^{\Ext}_{\vartheta}, p_{\varepsilon} , p_{\varepsilon} \rangle  
=  - \langle \Gamma^{\Ext}_{\vartheta}, \hat{p}_{\varepsilon}, \hat{p}_{ \varepsilon} \rangle ,
\end{equation}
where 
\begin{equation} \label{smodu}
{\hat{p}}_{\varepsilon} := (\varepsilon \vartheta_\varepsilon',  h_\varepsilon'  - \gamma  u^\Omega (h_\varepsilon ) )^{t} .
\end{equation}
\end{lem}
\begin{rem}
As for \eqref{EB1}, this relation is algebraic, in the sense that it does not rely on $p_{\varepsilon} = I_{\varepsilon} q'_{\varepsilon}$ or on the fact that $q_{\varepsilon}$ satisfies \eqref{ODE_intro}.
\end{rem}
\begin{proof}[Proof of Lemma~\ref{lemma-grav}.]
We will recast the second and third terms of the left hand side in terms of the matrix $M^\dagger$ defined in \eqref{def-Mdagger}.
Let us start with  the Christoffel term. 
Using the definition of $M^{\Ext}_{a, \vartheta}$ in \eqref{g+a_ext}
and the decomposition of $M_{a}^{\Ext}$ in \eqref{def-Ma},  we arrive at 
\begin{equation} \label{Ma}
M^{\Ext}_{a, \vartheta} =
\begin{pmatrix}
m^\#   & \mu_{\vartheta}^t \\
\mu_{\vartheta} & M_{\flat,\vartheta}^\Ext   
\end{pmatrix} ,
\end{equation}
with $M_{\flat,\vartheta}^\Ext$ and $\mu_{\vartheta}$ as in \eqref{Mflattheta}. 
In particular we infer from \eqref{Christo-exter} that
\begin{equation*}
\langle \Gamma^{\Ext}_{\vartheta}, p_{\varepsilon}, p_{\varepsilon} \rangle 
=  \begin{pmatrix}
 - ( M^\Ext_{\flat,\vartheta_{\varepsilon}} h_{\varepsilon}')^\perp \cdot h_{\varepsilon}' \\
 (\eps \vartheta_{\varepsilon}')^2 \mu_{\vartheta_{\varepsilon}}^\perp  +  \eps \vartheta_{\varepsilon}'  \big( (M^\Ext_{\flat,\vartheta_{\varepsilon}} h'_{\varepsilon})^\perp  -  M^\Ext_{\flat,\vartheta_{\varepsilon}}  (h_{\varepsilon}')^\perp \big)
\end{pmatrix}  .
\end{equation*}
It remains to recast this expression thanks to the matrix ${M}^{\dagger}$ defined in \eqref{def-Mdagger}. This is done thanks to the following elementary identities: recalling \eqref{mat-perp}  we have 
for any $\vartheta $ in $\R$ and $X $ in $\R^2$, 
\begin{align}
\label{flatVSdagger}
 (M^\Ext_{\flat,\vartheta} X )^\perp \cdot X &= X^\perp \cdot {M}^{\dagger}_\vartheta  X^\perp , \\
\label{flatVSdagger2}
 (M^\Ext_{\flat,\vartheta} X)^\perp   -  M^\Ext_{\flat,\vartheta} X^\perp &= -2 {M}^{\dagger}_\vartheta  X , \\
\label{flatVSdagger3}
 (\perp) {M}^{\dagger}_\vartheta  (\perp) & =  {M}^{\dagger}_\vartheta .
\end{align}
Therefore 
\begin{equation} \label{detai2-pre}
\langle \Gamma^{\Ext}_{\vartheta}, p_{\varepsilon}, p_{\varepsilon} \rangle 
=  \begin{pmatrix}
  - (h_{\varepsilon}')^\perp  \cdot M_\vartheta^{\dagger} \, (h_{\varepsilon}')^\perp \\
  (\eps \vartheta_{\varepsilon}')^2 \mu_{\vartheta_{\varepsilon}}^\perp  - 2\eps \vartheta_{\varepsilon}'  M_\vartheta^{\dagger} \, h_{\varepsilon}'
\end{pmatrix}
.
\end{equation}
Now using \eqref{DefB1}, \eqref{flatVSdagger3} and the fact that for any $\vartheta $ in $\R$,  $ {M}^{\dagger}_\vartheta$ is symmetric, we obtain
\begin{equation}
\label{detai1}
 p_{\varepsilon} \times {\mathsf{B}}_{{\it 1}}  (q_{\varepsilon})
=  
\begin{pmatrix}
-  2 (h_{\varepsilon}')^\perp  \cdot M_{\vartheta_{\varepsilon}}^{\dagger} \,  u^\Omega (h_{\varepsilon})^\perp \\
- 2  (\eps \vartheta_{\varepsilon}')^2  \big( M_{\vartheta_{\varepsilon}}^{\dagger} \,  u^\Omega (h_{\varepsilon})^\perp \big)^\perp
\end{pmatrix}  
= 
\begin{pmatrix}
-  2 (h_{\varepsilon}')^\perp  \cdot M_{\vartheta_{\varepsilon}}^{\dagger} \, u^\Omega (h_{\varepsilon})^\perp \\
- 2  (\eps \vartheta_{\varepsilon}')^2 M_{\vartheta_{\varepsilon}}^{\dagger} \,  u^\Omega (h_{\varepsilon})
\end{pmatrix} .
\end{equation}
Now it suffices to combine \eqref{detai1},  \eqref{detai2-pre} and  \eqref{def-E1a} to deduce  \eqref{grav}. 
\qed
\end{proof} 
As a consequence, combining \eqref{mesaoulent}, \eqref{EB1} and \eqref{grav}, we arrive at  
\begin{multline} \label{vasy1}
\gamma^2 E_{\varepsilon} (q_{\varepsilon}) + \gamma q_{\varepsilon}' \times B_{\eps} (q_{\varepsilon})  
- \langle \Gamma_\varepsilon  (q_{\varepsilon}), q_{\varepsilon}', q_{\varepsilon}' \rangle \\
=  I_\eps \Big[ {F}^{\Ext}_{\vartheta} (\hat{p}_{\varepsilon})
+  \varepsilon \big\{ - \langle \Gamma^{\Ext}_{\vartheta}, \hat{p}_{\varepsilon}, \hat{p}_{\varepsilon} \rangle
+ \gamma^{2} \big( \mathsf{E}_{{\it 1}}^b (q_{\varepsilon})  + \mathsf{E}_{{\it 1}}^c (q_{\varepsilon})  \big) \big\}
+ \eps^2 F_{r} (\eps, q_{\varepsilon}, \hat{p}_{\varepsilon})  \Big] .
\end{multline}
Moreover $F_{r} $ belongs to $ L^{\infty}_\mathrm{loc}(\mathfrak Q_{\delta,\varepsilon_{0}} \times \R^3; \R^3)$,  depends on  $\mathcal S_0$, $\gamma$  and $\Omega$ and is weakly nonlinear in the sense of 
Definition \ref{def-weakly-nonlinear}. 
%
%
%
%
%
%######
%
%
%
%
%
Next the part $\mathsf{E}_{{\it 1}}^c$ of the subprincipal term in \eqref{vasy1} can be absorbed by the principal term up to a modification of size $\varepsilon$ of the arguments (that is, thanks to the second oder modulation). 
More precisely,  by  \eqref{def-Bext} and \eqref{DefH1c}, 
\begin{equation} \label{vasy3}
{F}_{\Ext , \vartheta_{\varepsilon}} ( {\tilde{p}}_{\varepsilon} ) =
{F}_{\Ext , \vartheta_{\varepsilon}} (\hat{p}_{\varepsilon}) + \varepsilon \gamma^{2} \mathsf{E}_{{\it 1}}^c (q_{\varepsilon}) ,
\end{equation}
where $\tilde{p}_{\varepsilon}$ is given by \eqref{dmodu}.
Thus we deduce from \eqref{vasy1} and \eqref{vasy3} that 
\begin{multline} \label{AUX-tot}
\gamma^2 E_{\varepsilon} (q_{\varepsilon}) + \gamma q_{\varepsilon}' \times B_{\eps} (q_{\varepsilon})  
- \langle \Gamma_\varepsilon  (q_{\varepsilon}), q_{\varepsilon}', q_{\varepsilon}' \rangle \\
= I_\eps  \Big[ {F}^{\Ext}_{\vartheta_{\varepsilon}} (\tilde{p}_{\varepsilon})
- \varepsilon \langle \Gamma^{\Ext}_{\vartheta_{\varepsilon}}, \tilde{p}_{\varepsilon}, \tilde{p}_{\varepsilon} \rangle
+ \varepsilon \gamma^{2} \mathsf{E}_{{\it 1}}^b (q_{\varepsilon}) 
+ \eps^2 \hat{F}_{r} (\eps, q_{\varepsilon}, \tilde{p}_{\varepsilon})  \Big],
\end{multline}
where the term $\hat{F}_{r}$ is defined by
\begin{multline*}
\hat{F}_{r} (\eps, q_{\varepsilon}, \tilde{p}_{\varepsilon})
:= \check{F}_{r} (\eps, q_{\varepsilon}, \tilde{p}_{\varepsilon} + \varepsilon  \gamma p_{c}(q_{\varepsilon})) 
- 2  \gamma  \langle \Gamma^{\Ext}_{\vartheta_{\varepsilon}}, \tilde{p}_{\varepsilon}, p_{c}(q_{\varepsilon}) \rangle \\
- \varepsilon  \gamma^{2} \langle \Gamma^{\Ext}_{\vartheta_{\varepsilon}}, p_{c}(q_{\varepsilon}), p_{c}(q_{\varepsilon}) \rangle ,
\end{multline*}
where $p_{c}(q_{\varepsilon}):=(0, u_{c}(q_{\varepsilon}))$. One can easily check that $\hat{F}_{r}$ is still weakly nonlinear. \par
\ \par
Using Proposition~\ref{dev-added} and \eqref{AUX-tot}, and recalling the notation \eqref{TrueMatrix}, we recast 
the equation \eqref{ODE_intro_epsB} as follows:
\begin{multline} \label{passto2}
\eps^{\min(2,\alpha)} \, \Big( {M}_{\vartheta_\varepsilon} (\eps) + \eps^{4-\min (2,\alpha)} \, M_{r}  (\eps,q_\varepsilon)  \Big) {p}_\varepsilon' 
 = F^{\Ext}_{\vartheta_\varepsilon} ({\tilde{p}}_{\varepsilon} )
- \varepsilon  \langle \Gamma^{\Ext}_{\vartheta_\varepsilon}, \tilde{p}_\varepsilon, \tilde{p}_\varepsilon \rangle
\\ + \varepsilon \gamma^{2} \mathsf{E}_{{\it 1}}^{b} (q_{\varepsilon}) 
+ \eps^2 \hat{F}_{r} ( \eps  , q_\varepsilon,\tilde{p}_\varepsilon).
\end{multline}
We need to perform further modifications on this equation in order to achieve the normal forms \eqref{fnorm1}-\eqref{fnorm2} exactly, due to the fact that the mass matrix in \eqref{passto2}  
contains some extra lower-order terms, and that the time derivative is applied to $ {p}_\varepsilon $ rather than to $\tilde{p}_\varepsilon$.
To deal with the first discrepancy, reducing $\eps_{0} $ in $(0,1)$ if necessary, we simply multiply \eqref{passto2} by the matrix
\begin{equation*}
{M}_\vartheta (\eps) \Big({M}_{\vartheta} (\eps)+  \eps^{4-\min (2,\alpha)} \, M_{r}  (\eps,q) \Big)^{-1} .
\end{equation*}
On the other hand, for the second discrepancy,  we observe that  the modulated translation velocity $\tilde{\ell}_\varepsilon := h'-\gamma (u^\Omega (h_{\varepsilon}) + \eps u_{c} (q_{\varepsilon}))$  satisfies
\begin{multline*}
\tilde{\ell}_\varepsilon' =  h_\varepsilon'' -  \gamma \tilde{\ell}_{\varepsilon}  \cdot (\nabla u^\Omega ) (h_{\varepsilon}) 
- \gamma^2 (u^\Omega (h_{\varepsilon}) + \eps u_{c} (q_{\varepsilon}) )  \cdot \nabla u^\Omega (h_{\varepsilon}) \\
- \gamma D_{\vartheta} u_{c} (q_{\varepsilon}) \eps \vartheta_{\varepsilon}' 
- \eps \gamma D_{h} u_{c} (q_{\varepsilon}) \cdot  \big( \tilde{\ell}_{\varepsilon} +   \gamma (u^\Omega (h_{\varepsilon}) + \eps u_{c} (q_{\varepsilon}) ) \big)  .
\end{multline*}
Thus we obtain \eqref{fnorm2} with $F_{r}$ in $ L^{\infty}_\mathrm{loc}(\mathfrak Q_{\delta,\eps_{0}} \times \R^3; \R^3)$ weakly nonlinear in the sense
of Definition \ref{def-weakly-nonlinear}. This ends the proof of Proposition~\ref{Pro-fnorm}. \par
\ \par
%
%
%%%%%%%%%%%%%%
%
Starting from \eqref{OID1}, the proof of Proposition~\ref{Pro-fnorm-ball-inh} is similar to the one of Proposition~\ref{Pro-fnorm}, with some simplifications, since in this case $\tilde{M}^{{\mathcal S}_{0}}_{\flat}=\pi \mbox{\rm Id}_{2}$ and consequently $M^{\dagger}=0$. It follows that ${\mathsf{E}}_{{\it 1}}^{a}=0$ and ${\mathsf{B}}_{{\it 1}}=0$. \par
Now we expand $\tilde{M}_{\flat,\varepsilon}$, $\Gamma_{\flat,\varepsilon}$, $E_{\flat,\varepsilon}$ and $\tilde{B}_{1,\varepsilon}$ (which depend merely on $h_{c}$ and $\varepsilon$) in terms of $\varepsilon$.
Note that the last two coordinates of ${\mathsf{E}}_{{\it 1}}^{b}$ (which are the only ones to be relevant here, recall \eqref{Eflat}) are zero and that ${\mathsf{E}}_{{\it 1}}^{c}$ gives the term $ - u_{c}(q)^{\perp}$. Noting that $u^{\Omega}(h) + \varepsilon u_{c}(q) = u^{\Omega}(h_{c}) + O(\varepsilon^{2})$ and recalling \eqref{RelB1} we infer
\begin{equation*}
E_{\flat,\varepsilon}= - u^{\Omega}(h_{c})^{\perp} + \varepsilon^{2} E_{\flat,r} \text{ and } \tilde{B}_{1,\varepsilon} = -1 +  \varepsilon^{2} B_{\flat,r},
\end{equation*}
with $E_{\flat,r}= E_{\flat,r}(\eps,h_{c,\varepsilon})$ weakly nonlinear in the sense of Definition~\ref{def-WNL2-inh} and $B_{\flat,r}= B_{\flat,r} (\eps,h_{c,\varepsilon})$ bounded as long as $h_{c,\varepsilon}$ is away from $\partial \Omega$. On the other side, one finds that $\tilde{M}_{\flat} = \varepsilon^{2} \pi \mbox{\rm Id}_{2} + O(\varepsilon^{4})$ and $\Gamma_{\flat,\varepsilon} = O(\varepsilon^{4})$. The conclusion follows easily and this ends the proof of Proposition~\ref{Pro-fnorm-ball-inh}.
%
%
%
%%%%%%%%%%%%%%
%
%
%
%
%
%
\begin{acknowledgements}
The authors  thank the Agence Nationale de la Recherche,  Project IFSMACS, grant ANR-15-CE40-0010 for a partial financial support.
The first and third authors were also partially supported by the Agence Nationale de la Recherche, Project DYFICOLTI, grant ANR-13-BS01-0003-01, the second author by the Agence Nationale de la Recherche, Project OPTIFORM,  grant ANR-12-BS01-0007-04 and the third author by the Agence Nationale de la Recherche, Project BORDS, grant ANR-16-CE40-0027-01.
The third author was also partially supported by  the Emergences Project ``Instabilities in Hydrodynamics" funded by the Mairie de Paris and the Fondation Sciences Math\'ematiques de Paris.
\end{acknowledgements}

\end{document}